\documentclass[10pt,a4paper]{article}
\usepackage[utf8]{inputenc}
\usepackage[T1]{fontenc}
\usepackage{geometry}
\usepackage{amsthm}
\usepackage[french,english]{babel}
\usepackage{amssymb}
\usepackage{lmodern}
\usepackage{amsmath,amsfonts}
\usepackage{mathrsfs} 
\usepackage{dsfont}
\usepackage{stmaryrd}
\DeclareMathOperator{\dive}{div} 
\usepackage{graphicx}
\usepackage{fullpage}
\usepackage[nottoc,notlot,notlof]{tocbibind}
\usepackage[colorlinks=true,urlcolor=black,linkcolor=black,citecolor=black]{hyperref}
\numberwithin{equation}{section}
\usepackage{mathtools,array}
\usepackage{verbatimbox}
\usepackage{color}
\usepackage{bm}
\usepackage{tikz-cd}
\usepackage{pst-node}
\usetikzlibrary{matrix}
\usepackage{mathtools}
\usepackage{verbatimbox}
\usepackage{diagbox}
\usepackage{array}
\newcolumntype{C}{>{$\displaystyle} c <{$}}
\usepackage{makecell}
\usepackage{float}
\usepackage{tabu}
\usepackage{cancel}
\usepackage{lscape}
\usepackage[babel=true]{csquotes}
\usepackage{bm}
\usepackage{xcolor}
\usepackage[new]{old-arrows}
\usepackage{stackengine}
\makeatletter
\def\env@dmatrix{\hskip -\arraycolsep
	\let\@ifnextchar\new@ifnextchar
	\def\arraystretch{2}%
	\array{*{\c@MaxMatrixCols}{>{\displaystyle}c}}%
}
\usepackage{soul}

\makeatother

\DeclareFontShape{OMX}{cmex}{m}{n}{
	<-7.5> cmex7
	<7.5-8.5> cmex8
	<8.5-9.5> cmex9
	<9.5-> cmex10
}{}
\SetSymbolFont{largesymbols}{normal}{OMX}{cmex}{m}{n}
\SetSymbolFont{largesymbols}{bold}  {OMX}{cmex}{m}{n}

\begin{document}

	\renewcommand{\thefootnote}{\fnsymbol{footnote}}
	
	\title{Morse Index Stability of Willmore Immersions I}

	\author{Alexis Michelat\footnote{Institute of Mathematics, EPFL B, Station 8, CH-1015 Lausanne, Switzerland.\hspace{.5em} \href{alexis.michelat@epfl.ch}{alexis.michelat@epfl.ch}}\and \;Tristan \selectlanguage{french}Rivière\selectlanguage{english}\footnote{Department of Mathematics, ETH Zentrum, CH-8093 Zürich, Switzerland.\hspace{2.85em}\href{mailto:tristan.riviere@math.ethz.ch}{tristan.riviere@math.ethz.ch}}}
	\date{\today}
	
	\maketitle
	
	\vspace{-0.5em}
	
	\begin{abstract}
		In (\cite{riviere_morse_scs}), F. Da Lio, M. Gianocca, and T. Rivière developped a new method to show upper semi-continuity results in geometric analysis—which they applied to conformally invariant Lagrangians in dimension $2$ (that include harmonic maps). In this article, we apply this method to show that the sum of the Morse index and the nullity of Willmore immersions is upper semi-continuous, provided that the limiting immersions and the bubbles are free of branch points. Our result covers the case of degenerating Riemann surfaces for which the ratio of the second residue (considered by P. Laurain and T. Rivière in \cite{lauriv1}) and the length of the minimal shrinking geodesic of the underlying sequence of Riemann surfaces is smaller than a universal constant.  
	\end{abstract}

	\tableofcontents
	\vspace{0cm}
	\begin{center}
	\end{center}

	\theoremstyle{plain}
	\newtheorem*{theorem*}{Theorem}
	\newtheorem{theorem}{Theorem}[section]
	\newenvironment{theorembis}[1]
	{\renewcommand{\thetheorem}{\ref{#1}$'$}%
		\addtocounter{theorem}{-1}%
		\begin{theorem}}
		{\end{theorem}}
	\renewcommand*{\thetheorem}{\Alph{theorem}}
	\newtheorem{lemme}[theorem]{Lemma}
	\newtheorem*{lemme*}{Lemma}
	\newtheorem{propdef}[theorem]{Definition-Proposition}
	\newtheorem*{propdef*}{Definition-Proposition}
	\newtheorem{prop}[theorem]{Proposition}
	\newtheorem{cor}[theorem]{Corollary}
	\theoremstyle{definition}
	\newtheorem*{definition}{Definition}
	\newtheorem{defi}[theorem]{Definition}
	\newtheorem{rem}[theorem]{Remark}
	\newtheorem*{rem*}{Remark}
	\newtheorem{rems}[theorem]{Remarks}
	\newtheorem{remimp}[theorem]{Important Remark}
	\newtheorem{exemple}[theorem]{Example}
	\newtheorem{defi2}{Definition}
	\newtheorem{propdef2}[defi2]{Proposition-Definition}
	\newtheorem{remintro}[defi2]{Remark}
	\newtheorem{remsintro}[defi2]{Remarks}
	\newtheorem{conj}{Conjecture}
	\newtheorem{question}{Open Question}
	\renewcommand\hat[1]{%
		\savestack{\tmpbox}{\stretchto{%
				\scaleto{%
					\scalerel*[\widthof{\ensuremath{#1}}]{\kern-.6pt\bigwedge\kern-.6pt}%
					{\rule[-\textheight/2]{1ex}{\textheight}}
				}{\textheight}%
			}{0.5ex}}%
		\stackon[1pt]{#1}{\tmpbox}
	}
	\parskip 1ex
	\newcommand{\totimes}{\ensuremath{\,\dot{\otimes}\,}}
	\newcommand{\vc}[3]{\overset{#2}{\underset{#3}{#1}}}
	\newcommand{\conv}[1]{\ensuremath{\underset{#1}{\longrightarrow}}}
	\newcommand{\A}{\ensuremath{\vec{A}}}
	\newcommand{\B}{\ensuremath{\vec{B}}}
	\newcommand{\C}{\ensuremath{\mathbb{C}}}
	\newcommand{\D}{\ensuremath{\nabla}}
	\newcommand{\Disk}{\ensuremath{\mathbb{D}}}
	\newcommand{\E}{\ensuremath{\vec{E}}}
	\newcommand{\I}{\ensuremath{\mathbb{I}}}
	\newcommand{\Q}{\ensuremath{\vec{Q}}}
	\newcommand{\loc}{\ensuremath{\mathrm{loc}}}
	\newcommand{\z}{\ensuremath{\bar{z}}}
	\newcommand{\hh}{\ensuremath{\mathscr{H}}}
	\newcommand{\h}{\ensuremath{\vec{h}}}
	\newcommand{\vol}{\ensuremath{\mathrm{vol}}}
	\newcommand{\hs}[3]{\ensuremath{\left\Vert #1\right\Vert_{\mathrm{H}^{#2}(#3)}}}
	\newcommand{\R}{\ensuremath{\mathbb{R}}}
	\renewcommand{\P}{\ensuremath{\mathbb{P}}}
	\newcommand{\N}{\ensuremath{\mathbb{N}}}
	\newcommand{\Z}{\ensuremath{\mathbb{Z}}}
	\newcommand{\p}[1]{\ensuremath{\partial_{#1}}}
	\newcommand{\Res}{\ensuremath{\mathrm{Res}}}
	\newcommand{\lp}[2]{\ensuremath{\mathrm{L}^{#1}(#2)}}
	\renewcommand{\wp}[3]{\ensuremath{\left\Vert #1\right\Vert_{\mathrm{W}^{#2}(#3)}}}
	\newcommand{\wpn}[3]{\ensuremath{\Vert #1\Vert_{\mathrm{W}^{#2}(#3)}}}
	\newcommand{\np}[3]{\ensuremath{\left\Vert #1\right\Vert_{\mathrm{L}^{#2}(#3)}}}
	\newcommand{\hp}[3]{\ensuremath{\left\Vert #1\right\Vert_{\mathrm{H}^{#2}(#3)}}}
	\newcommand{\ck}[3]{\ensuremath{\left\Vert #1\right\Vert_{\mathrm{C}^{#2}(#3)}}}
	\newcommand{\hardy}[2]{\ensuremath{\left\Vert #1\right\Vert_{\mathscr{H}^{1}(#2)}}}
	\newcommand{\lnp}[3]{\ensuremath{\left| #1\right|_{\mathrm{L}^{#2}(#3)}}}
	\newcommand{\npn}[3]{\ensuremath{\Vert #1\Vert_{\mathrm{L}^{#2}(#3)}}}
	\newcommand{\nc}[3]{\ensuremath{\left\Vert #1\right\Vert_{C^{#2}(#3)}}}
	\renewcommand{\Re}{\ensuremath{\mathrm{Re}\,}}
	\renewcommand{\Im}{\ensuremath{\mathrm{Im}\,}}
	\newcommand{\dist}{\ensuremath{\mathrm{dist}}}
	\newcommand{\diam}{\ensuremath{\mathrm{diam}\,}}
	\newcommand{\leb}{\ensuremath{\mathscr{L}}}
	\newcommand{\supp}{\ensuremath{\mathrm{supp}\,}}
	\renewcommand{\phi}{\ensuremath{\vec{\Phi}}}
	\renewcommand{\H}{\ensuremath{\vec{H}}}
	\renewcommand{\L}{\ensuremath{\vec{L}}}
	\renewcommand{\lg}{\ensuremath{\mathscr{L}_g}}
	\renewcommand{\ker}{\ensuremath{\mathrm{Ker}}}
	\renewcommand{\epsilon}{\ensuremath{\varepsilon}}
	\renewcommand{\bar}{\ensuremath{\overline}}
	\newcommand{\s}[2]{\ensuremath{\langle #1,#2\rangle}}
	\newcommand{\pwedge}[2]{\ensuremath{\,#1\wedge#2\,}}
	\newcommand{\bs}[2]{\ensuremath{\left\langle #1,#2\right\rangle}}
	\newcommand{\scal}[2]{\ensuremath{\langle #1,#2\rangle}}
	\newcommand{\sg}[2]{\ensuremath{\left\langle #1,#2\right\rangle_{\mkern-3mu g}}}
	\newcommand{\n}{\ensuremath{\vec{n}}}
	\newcommand{\ens}[1]{\ensuremath{\left\{ #1\right\}}}
	\newcommand{\lie}[2]{\ensuremath{\left[#1,#2\right]}}
	\newcommand{\g}{\ensuremath{g}}
	\newcommand{\dzeta}{\ensuremath{\det\hphantom{}_{\kern-0.5mm\zeta}}}
	\newcommand{\e}{\ensuremath{\vec{e}}}
	\newcommand{\f}{\ensuremath{\vec{f}}}
	\newcommand{\ig}{\ensuremath{|\vec{\mathbb{I}}_{\phi}|}}
	\newcommand{\ik}{\ensuremath{\left|\mathbb{I}_{\phi_k}\right|}}
	\newcommand{\w}{\ensuremath{\vec{w}}}
	\newcommand{\hooklongrightarrow}{\lhook\joinrel\longrightarrow}
	\renewcommand{\tilde}{\ensuremath{\widetilde}}
	\newcommand{\vg}{\ensuremath{\mathrm{vol}_g}}
	\newcommand{\im}{\ensuremath{\mathrm{W}^{2,2}_{\iota}(\Sigma,N^n)}}
	\newcommand{\imm}{\ensuremath{\mathrm{W}^{2,2}_{\iota}(\Sigma,\R^3)}}
	\newcommand{\timm}[1]{\ensuremath{\mathrm{W}^{2,2}_{#1}(\Sigma,T\R^3)}}
	\newcommand{\tim}[1]{\ensuremath{\mathrm{W}^{2,2}_{#1}(\Sigma,TN^n)}}
	\renewcommand{\d}[1]{\ensuremath{\partial_{x_{#1}}}}
	\newcommand{\dg}{\ensuremath{\mathrm{div}_{g}}}
	\renewcommand{\Res}{\ensuremath{\mathrm{Res}}}
	\newcommand{\un}[2]{\ensuremath{\bigcup\limits_{#1}^{#2}}}
	\newcommand{\res}{\mathbin{\vrule height 1.6ex depth 0pt width
			0.13ex\vrule height 0.13ex depth 0pt width 1.3ex}}
    \newcommand{\antires}{\mathbin{\vrule height 0.13ex depth 0pt width 1.3ex\vrule height 1.6ex depth 0pt width 0.13ex}}
	\newcommand{\ala}[5]{\ensuremath{e^{-6\lambda}\left(e^{2\lambda_{#1}}\alpha_{#2}^{#3}-\mu\alpha_{#2}^{#1}\right)\left\langle \nabla_{\e_{#4}}\vec{w},\vec{\mathbb{I}}_{#5}\right\rangle}}
	\setlength\boxtopsep{1pt}
	\setlength\boxbottomsep{1pt}
	\newcommand\norm[1]{%
		\setbox1\hbox{$#1$}%
		\setbox2\hbox{\addvbuffer{\usebox1}}%
		\stretchrel{\lvert}{\usebox2}\stretchrel*{\lvert}{\usebox2}%
	}
	\allowdisplaybreaks
	\newcommand*\mcup{\mathbin{\mathpalette\mcapinn\relax}}
	\newcommand*\mcapinn[2]{\vcenter{\hbox{$\mathsurround=0pt
				\ifx\displaystyle#1\textstyle\else#1\fi\bigcup$}}}
	\def\Xint#1{\mathchoice
		{\XXint\displaystyle\textstyle{#1}}%
		{\XXint\textstyle\scriptstyle{#1}}%
		{\XXint\scriptstyle\scriptscriptstyle{#1}}%
		{\XXint\scriptscriptstyle\scriptscriptstyle{#1}}%
		\!\int}
	\def\XXint#1#2#3{{\setbox0=\hbox{$#1{#2#3}{\int}$ }
			\vcenter{\hbox{$#2#3$ }}\kern-.58\wd0}} 
	\def\ddashint{\Xint=}
	\newcommand{\dashint}[1]{\ensuremath{{\Xint-}_{\mkern-10mu #1}}}
	\newcommand\ccancel[1]{\renewcommand\CancelColor{\color{red}}\cancel{#1}}
	\newcommand\colorcancel[2]{\renewcommand\CancelColor{\color{#2}}\cancel{#1}}
	\newcommand{\abs}[1]{\left\lvert #1 \right \rvert}
	
	\renewcommand{\thetheorem}{\thesection.\arabic{theorem}}

 \section{Introduction}

 \subsection{History and Background}

 Let $\Sigma$ be a closed Riemann surface and fix some integer $n\geq 3$. For all immersion $\phi:\Sigma\rightarrow \R^n$, let $g=\phi^{\ast}g_{\R^n}$ be the induced metric on $\Sigma$, where $g_{\R^n}$ is the flat metric on $\R^n$. The Willmore energy of $\phi$ is the quantity
 \begin{align}
     W(\phi)=\int_{\Sigma}|\H|^2d\vg,
 \end{align}
 where $\H$ is the mean-curvature vector, given locally by 
 \begin{align}
     \H=\frac{1}{2}\sum_{i,j=1}^2g^{i,j}\,\vec{\I}_{i,j},
 \end{align}
 where $\vec{\I}_{i,j}$ is the second fundamental form of the immersion $\phi$. The Willmore energy, that first appeared in the work of Poisson and Sophie Germain (\cite{poisson,germain3}) is conformally invariant thanks to the work of Blaschke and Thomsen in the early $1920$s (\cite{blaschke,thomsen}). In other words, $W$ is invariant by isometries, dilations and inversions, provided that the inversion does not modify the topology of the surface $\phi(\Sigma)$. More precisely, the quantity
 \begin{align}
     \int_{\Sigma}|\h_0|_{\mathrm{WP}}^2d\vg=\int_{\Sigma}\left(|\H|^2-K_g\right)d\vg,
 \end{align}
 is conformally invariant for \emph{all} conformal transformations, where $K_g$ is the Gauss curvature, and $\h_0=2\,\vec{\I}(\p{z}\phi,\p{z}\phi)\,dz^2$ is the Weingarten tensor, and for all quadratic differential $\alpha=f(z)dz^2$ and $\beta=h(z)dz^2$, we define in a conformal (or complex) chart the \emph{Weil-Petersson product} of $\alpha$ and $\beta$ by
 \begin{align}
     \s{\alpha}{\beta}_{\mathrm{WP}}=\Re\left(e^{-4\lambda}f(z)\bar{h(z)}\right),
 \end{align}
 where $g=e^{2\lambda}|dz|^2$. 
 What about Willmore? He was unaware of the previous work and reintroduced this definition in his $1965$ article (\cite{willmore1})—and reaped all honours for conjecturing that a particular revolution torus (the Willmore torus, which is no other than the stereographic projection of the Clifford torus, a minimal surface of $S^3$) has the least energy amongst \emph{all} tori.
 A \emph{Willmore surface} is a critical point of $W$, and classically satisfies the equation
 \begin{align}
     \Delta_g^N\H-2|\H|^2\H+\mathscr{A}(\H)=0
 \end{align}
 where $\Delta_g^N$ is the normal Laplacian, and $\mathscr{A}$ is the Simons operator, such that for all $\vec{X}\in\R^n$,
 \begin{align*}
     \mathscr{A}(\vec{X})=\sum_{i,j=1}^2e^{-4\lambda}\s{\vec{\I}_{i,j}}{\vec{X}}\vec{\I}_{i,j}.
 \end{align*}
 In codimension $1$ (if $n=3$), the equation reduces to
 \begin{align}
     \Delta_gH+2H(H^2-K_g)=0,
 \end{align}
 where $\H=H\,\n$, and $\n:\Sigma\rightarrow S^2\subset \R^3$ is the unit normal. If $n\geq 3$ is arbitrary, then the unit normal $\n:\Sigma\rightarrow \mathscr{G}_{n-2}(\R^n)$ is a map taking values in the Grassmannian space of $(n-2)$-vectors. It is related to the previous tensors as follows:
 \begin{align*}
     \int_{\Sigma}|d\n|_g^2d\vg=2\int_{\Sigma}|\H|^2d\vg+2\int_{\Sigma}|\h_0|_{\mathrm{WP}}^2d\vg=\int_{\Sigma}(4|\H|^2-2K_g)d\vg.
 \end{align*}
 Thanks to Gauss-Bonnet theorem, the Dirichlet energy of the unit normal is variationally equivalent to the Willmore energy, and if $A$ designs the full fundamental form, we also have
 \begin{align*}
     \int_{\Sigma}|A|^2d\vg=\int_{\Sigma}|d\n|_g^2d\vg.
 \end{align*}

 \subsection{Morse Index Stability}

 In this article, we are concerned with issues of Morse index stability, that are related to min-max problems (\cite{eversion}). In particular, a natural development of the theory developed here would be to apply it to the \emph{viscosity method} (first developed in the case of minimal surfaces by T. Rivière \cite{viscosity}). Due to the extreme technical difficulty associated branched Willmore surfaces (and the length of the present article and the companion paper \cite{eigenvalue_annuli}), we will concentrate ourselves on the case where the limiting surface does not have singularities. However, notice that the most technical part of the two articles, namely the estimates in neck regions, is proved in complete generality for all order or branch points (\cite{eigenvalue_annuli}). 

 Compared to the lower semi-continuity Morse index estimates, upper semi-continuity results are typically much more complicated due to the need of obtaining precise estimates of the sequence of solutions in regions of loss of compactness. The main technical issue that one must needs overcome is that in neck regions, there could be (infinitely many) negative variations. If the underlying Riemann surfaces degenerates and the image geodesics have long enough length, our proof indicates that there could be negative variations located in the neck regions, which would violate the upper semi-continuity inequality. 
 
 The strategy in our proof closely follows the approach developed in \cite{riviere_morse_scs} by Da Lio-Gianocca-Rivière. Let us explain the main steps of this method, whose flexibility shall become apparent in our article. The approach can be decomposed into four major steps.

 \subsubsection{Neck Regions}

 As a means to simplify the exposition of the method of \cite{riviere_morse_scs}, we restrict to the simplest example of harmonic maps with values into $S^3$. Fix a closed Riemann surface $\Sigma$, and a Riemannian metric $h$ on $\Sigma$  of constant curvature of unit volume. The Dirichlet energy of a map $\varphi:\Sigma\rightarrow S^3$ is defined by 
 \begin{align}
     E(u)=\frac{1}{2}\int_{\Sigma}|d\varphi|_h^2d\mathrm{vol}_{h}.
 \end{align}
 A (stationary) harmonic map is a critical point of $E$ for variations of the type $\ens{\varphi_t=\pi_{S^3}(\varphi+t\,u)}_{-\epsilon<t<\epsilon}$, where $u:\Sigma\rightarrow \R^4$ is an arbitrary smooth map. 
 
 Let $\ens{\varphi_k}_{k\in\N}:\Sigma\rightarrow S^3$ be a sequence of harmonic maps into $S^3$ such that
 \begin{align*}
     \limsup_{k\rightarrow \infty}E(\varphi_k)<\infty.
 \end{align*}
 Then, by classical results (and up to a subsequence) $\ens{\varphi_k}_{k\in\N}$ converges in $W^{1,2}$ towards a macroscopic harmonic map $\varphi_{\infty}:\Sigma\rightarrow S^3$ and a finite number of \enquote{bubbles.} In order to simplify the presentation, assume that $\Sigma=S^2=\C\cup\ens{\infty}$, and assume that there is a single bubble. Then, we have
 \begin{align*}
     \varphi_k\conv{k\rightarrow \infty}\varphi_{\infty}\qquad \text{in}\;\, C^{l}_{\mathrm{loc}}(S^2\setminus\ens{0})\;\,\text{for all}\;\, l\in\N,
 \end{align*}
 and there exists a sequence $\ens{\rho_k}_{k\in\N}$ such that for all $\alpha>0$
 \begin{align*}
     \varphi_k(\delta_k\,\cdot\,)_{B(0,\alpha^{-1})}\conv{k\rightarrow \infty}\psi_{\infty}\;\,\text{in}\;\, C^{l}(B(0,\alpha^{-1}))\;\,\text{for all}\;\, l\in\N,
 \end{align*}
 where $v_{\infty}:S^2\rightarrow S^3$ is a harmonic map. The second derivative of $E$ is given by
 \begin{align*}
     Q_{\varphi}(u)=D^2E(\varphi)(u,u)=\int_{\Sigma}\left(|d\varphi|_h^2-|d\varphi|_h^2|u|^2\right)d\mathrm{vol}_h.
 \end{align*}
 The quantization of energy states that
 \begin{align}\label{quanta_harm}
     E(\varphi_k)\conv{k\rightarrow \infty}E(\varphi_{\infty})+E(\psi_{\infty}).
 \end{align}
 Equivalently, \eqref{quanta_harm} holds if and only if 
 \begin{align}\label{noneck}
     \lim_{\alpha\rightarrow 0}\limsup_{k\rightarrow \infty}\int_{\Omega_k(\alpha)}|d\varphi_k|_h^2d\mathrm{vol}_h=0,
 \end{align}
 where for all $\alpha>0$, $\Omega_k(\alpha)=B_{\alpha}\setminus\bar{B}_{\alpha^{-1}\rho_k}(0)$. 
 Thanks to an $\epsilon$-regularity result, one proves that for all $x\in \Omega_k(\alpha)$
 \begin{align}\label{trivial_eps}
     |\D \varphi_k(x)|\leq \frac{C}{|x|}\np{\D \varphi_k}{2}{\Omega_k(2\alpha)}.
 \end{align}
 As we have said it above, the main part of the proof is to show that the neck regions contribute positively to the second derivative. The first important estimate is a weighted Poincaré inequality.
 \begin{theorem}[Lemma IV.$1$, \cite{riviere_morse_scs}]\label{eigenvalue_harmonic_annulus}
     Let $0<a<b<\infty$, and let $\Omega=B_b\setminus\bar{B}_a(0)$. For all $u\in W^{1,2}_0(\Omega)$, we have
     \begin{align*}
         \int_{\Omega}|\D u|^2dx\geq \frac{\pi^2}{\log^2\left(\frac{b}{a}\right)}\int_{\Omega}\frac{u^2}{|x|^2}dx.
     \end{align*}
 \end{theorem}
 Therefore, a trivial estimate shows that for all $u\in W^{1,2}_0(\Omega_k(\alpha))$,
 \begin{align*}
     Q_{\varphi_k}(u)\geq \left(\frac{\pi^2}{\log^2\left(\frac{\alpha^2}{\delta_k}\right)}-C\np{d\varphi_k}{2}{\Omega_k(\alpha)}^2\right)\int_{\Omega_k(\alpha)}\frac{|u|^2}{|x|^2}dx.
 \end{align*}
 We see that the no-neck energy \eqref{noneck} is not sufficient to show that $Q_{\varphi_k}(u)\geq 0$, and \eqref{trivial_eps} needs to be refined.  
 
 \subsubsection{Step 1: Lorentz Space Energy Quantization and Pointwise Estimate in Neck Regions}

 A flexible method to prove energy quantization results was first developed by Lin-Rivière (\cite{linriv,linriv_GL1,linriv_GL2}). It consists in first proving an energy quantization in the weak $L^2$ space $L^{2,\infty}$ (also known as Marcinkiewicz space), and showing that the $L^{2,1}$ energy is bounded. The Banach space $L^{2,1}$ is a Lorentz space or a special case Orlicz space which is the pre-dual of $L^{2,\infty}$ (for more details, refer to the appendix of \cite{pointwise}). If $\ens{\varphi_k}_{k\in\N}:\Sigma\rightarrow \R^3$ is a sequence of harmonic maps such that
 \begin{align*}
     \np{d\varphi_k}{2,\infty}{\Sigma}\conv{k\rightarrow \infty}0,
 \end{align*}
 and
 \begin{align*}
     \np{d\varphi_k}{2,1}{\Sigma}\leq C,
 \end{align*}
 then the $L^{2,1}/L^{2,\infty}$ duality shows that
 \begin{align*}
     E(\varphi_k)=\frac{1}{2}\int_{\Sigma}|d\varphi_k|_h^2d\mathrm{vol}_h\leq \np{d\varphi_k}{2,1}{\Sigma}\np{d\varphi_k}{2,\infty}{\Sigma}\leq C\np{d\varphi_k}{2,\infty}{\Sigma}\conv{k\rightarrow\infty}0.
 \end{align*}
 For a precise definition of Lorentz spaces, see the appendix of \cite{pointwise}. In this article, we will mostly need the following estimates:
 \begin{align}\label{norms}
 \left\{\begin{alignedat}{1}
     &\np{\frac{1}{|x|}}{2,\infty}{\R^2}=2\sqrt{\pi}\\
     &\np{\frac{1}{|x|}}{2}{B_b\setminus\bar{B}_a(0)}=\sqrt{2\pi}\sqrt{\log\left(\frac{b}{a}\right)}\\
     &\np{\frac{1}{|x|}}{2,1}{B_b\setminus\bar{B}_a(0)}=4\sqrt{\pi}\left(\log\left(\frac{b}{a}\right)+\log\left(1+\sqrt{1-\left(\frac{a}{b}\right)^2}\right)\right).
     \end{alignedat}\right.
 \end{align}
 Generalising \cite[Lemma III.$1$]{riviere_morse_scs} to Jacobian-type systems, we show the following result, that is proven for a smaller range of exponents $0<\beta<1$ in \cite{riviere_morse_scs}.
 \begin{theorem}
     Let $0<2\,a<b<\infty$, $\Omega=B_b\setminus\bar{B}_a(0)$, and for all $0<t\leq 1$, let $\Omega_t=B_{t\,b}\setminus\bar{B}_{t^{-1}a}(0)$. Let $A\in W^{1,(2,1)}(\Omega,\R^n)$ and $B\in W^{1,2}(\Omega,M_n(\R))$, and assume that
     \begin{align}
         \Delta A=\D^{\perp}B\cdot \D A\qquad\text{in}\;\,\Omega.
     \end{align}
     Let 
     \begin{align*}
         \Lambda=\frac{1}{2\pi}\int_{\Omega_{1/2}}\frac{|\D A|}{|x|}dx.
     \end{align*}
     Then, for all $0<\beta<1$, and for all $0<\delta<1$, there exists $\epsilon_{\beta,\delta}>0$ and a \emph{universal} constant $C<\infty$ such that the hypothesis 
     \begin{align*}
         \int_{\Omega}|\D B|^2\leq \epsilon_{\beta,\delta}
     \end{align*}
     implies that for all $z\in \Omega_{1/2}$,
     \begin{align*}
         \np{\D A}{2}{B_{2|z|}\setminus\bar{B}_{|z|}(0)}\leq C\left(\left(\frac{|z|}{b}\right)^{\beta}+\left(\frac{a}{|z|}\right)^{\beta}\right)\np{\D A}{2}{\Omega}+\frac{\sqrt{2\pi\log(2)}}{\log\left(\frac{b}{4a}\right)}\left(1+\delta\right)\left(\Lambda+C\np{\D A}{2}{\Omega}\right).
     \end{align*}
 \end{theorem}
 Now, assuming that a strong $L^{2,1}$ energy quantization holds, namely
 \begin{align}\label{quanta_L21_harmonic}
     \lim_{\alpha\rightarrow 0}\limsup_{k\rightarrow \infty}\np{\D \varphi_k}{2,1}{\Omega_k(\alpha)}=0,
 \end{align}
 we deduce by the $L^{2,1}$ duality that
 \begin{align*}
     \Lambda_{\alpha,k}=\frac{1}{2\pi}\int_{\Omega_k(\alpha)}\frac{|\D u_k|}{|x|}dx\leq \frac{1}{2\pi}\np{\D \varphi_k}{2,1}{\Omega_k(\alpha)}\np{\frac{1}{|x|}}{2,\infty}{\Omega_k(\alpha)}\leq \frac{1}{\sqrt{\pi}}\np{\D \varphi_k}{2,1}{\Omega_k(\alpha)},
 \end{align*}
 which implies thanks to the $\epsilon$-regularity \eqref{trivial_eps} that for all $x\in\Omega_k(\alpha)$
 \begin{align}\label{neck_harmonic1}
     |\D \varphi_k(x)|\leq \frac{C}{|x|}\left(\left(\frac{|x|}{b}\right)^{\beta}+\left(\frac{a}{|x|}\right)^{\beta}+\frac{1}{\log\left(\frac{\alpha^2}{\delta_k^2}\right)}\right)\np{\D \varphi_k}{2,1}{\Omega_k(2\alpha)}.
 \end{align}
 Now, to conclude the proof, one needs to prove a weight Poincaré inequality in neck regions.
  
  \subsubsection{Step 2: Weighted Poincaré Estimates in Neck Regions}

  If one adds the natural weight appearing in the previous inequality from \eqref{ineq_harmonic_annulus}, the constant does not converges to $0$ as the conformal class diverges to $\infty$.
 
 \begin{theorem}[Lemma IV.$1$, \cite{riviere_morse_scs}]\label{neck_harmonic2}
     Let $0<a<b<\infty$ and $\Omega=B_b\setminus\bar{B}_a(0)\subset \R^2$. For all $\beta>0$, there exists a constant $C_{\beta}>0$ such that for all $u\in W^{1,2}_0(\Omega)$,
    \begin{align*}
        \int_{\Omega}|\D u|^2dx\geq C_{\beta}\int_{\Omega}\frac{u^2}{|x|^2}\left(\left(\frac{|x|}{b}\right)^{\beta}+\left(\frac{a}{|x|}\right)^{\beta}\right)dx.
    \end{align*}
 \end{theorem}

 \subsubsection{Step 3: Final Estimate}

 Fix some $0<\beta<1$.  Thanks to \eqref{neck_harmonic1} and \eqref{neck_harmonic2}, there exists $\lambda_0>0$ and $\lambda_1>0$ such that for all $\alpha>0$ and for all $k\in\N$, we have for all $u\in W^{1,2}_0(\Omega_k(\alpha))$
 \begin{align*}
     Q_{\varphi_k}(u)\geq \left(\lambda_0-\lambda_1\np{\D \varphi_k}{2,1}{\Omega_k(2\alpha)}\right)\int_{\Omega_k(\alpha)}\frac{|u|^2}{|x|^2}\left(\left(\frac{|x|}{\alpha}\right)^{\beta}+\left(\frac{\alpha^{-1}\delta_k}{|x|}\right)^{\beta}+\frac{1}{\log^2\left(\frac{\alpha^2}{\delta_k}\right)}\right)dx.
 \end{align*}
 Thanks to the $L^{2,1}$ energy quantization \eqref{quanta_L21_harmonic}, there exists $\alpha_0>0$ such that for all $0<\alpha<\alpha_0$ and $k\in\N$ large enough
 \begin{align}\label{est_final_harmonic}
     Q_{\varphi_k}(u)\geq \frac{\lambda_0}{2}\int_{\Omega_k(\alpha)}\frac{|u|^2}{|x|^2}\left(\left(\frac{|x|}{\alpha}\right)^{\beta}+\left(\frac{\alpha^{-1}\delta_k}{|x|}\right)^{\beta}+\frac{1}{\log^2\left(\frac{\alpha^2}{\delta_k}\right)}\right)dx,
 \end{align}
 which finally shows that variations located in the neck regions contribute positively to the second derivative. 

 \subsubsection{Step 4: Sylvester's Law of Inertia and Final Argument}

 We see that the estimate in neck region suggests the introduction of a weight $\omega_{k,\alpha}$ on $\Sigma$ such that
 \begin{align*}
     \omega_{k,\alpha}(x)=\frac{1}{|x|^2}\left(\left(\frac{|x|}{\alpha}\right)^{\beta}+\left(\frac{\alpha^{-1}\delta_k}{|x|}\right)^{\beta}+\frac{1}{\log^2\left(\frac{\alpha^2}{\delta_k}\right)}\right)\quad \text{for all}\;\, x\in \Omega_k(\alpha).
 \end{align*}
 Interpolating on $\Sigma$ with constant non-zero functions, Sylvester's Law of Intertia asserts that both operators $\leb_{g_k}=-\Delta_h-|d\varphi_k|_h^2$ and $\mathcal{L}_{k,\alpha}=-\omega_{k,\alpha}^{-1}\left(\Delta_h+|d\varphi_k|_h^2\right)$ have the same Morse index. Finally, if $u_k$ is a negative eigenvalue, the normalisation
 \begin{align*}
     \int_{\Sigma}|u_k|^2\omega_{k,\alpha}d\vg=1,
 \end{align*}
 together with a lower-estimate for the smallest eigenvalue (\cite[Lemma IV.$4$]{riviere_morse_scs}) permits to show that $\ens{u_k}_{k\in\N}$ is bounded in $W^{1,2}(\Sigma,\R^n)$, which permits to extract a converging subsequence. Using the estimate \eqref{est_final_harmonic}, one proves that the variation cannot be located in neck regions, which permits to link them to negative (or null) variations of the limiting harmonic map or its bubble. This concludes the proof of the main theorem of \cite{riviere_morse_scs}. 
 
 \subsection{Da Lio-Gianocca-Rivière's Method Applied to Willmore Immersions}
 
 \subsubsection{Generalities on Willmore Immersions}
 
 Fix a Riemann surface $\Sigma$, and let $\{\phi_k\}_{k\in\N}$ be a sequence of immersions from $\Sigma$ to $\R^3$, such that
 \begin{align*}
     \limsup_{k\rightarrow \infty}W(\phi_k)<\infty.
 \end{align*}
 Furthermore, assume that the conformal class of $\Sigma$ is bounded, or more generally, that the condition on the second residue $\vec{\gamma}_1$ of \cite{lauriv1} holds (see \eqref{second_residue_def} for an explicit definition). 
 Then, thanks to Bernard-Rivière work \cite{quanta} (or Laurain-Rivière's \cite{lauriv1}), the energy quantization holds. Assume for simplicity that there is a single \enquote{bubble.} Then, there exists $p\in \Sigma$ and a branched Willmore immersion $\phi_{\infty}:\Sigma\rightarrow \R^n$ such that (modulo diffeomorphisms in the domain and conformal transformations in the target space) up to a subsequence
 \begin{align*}
     \phi_k\conv{k\rightarrow \infty}\phi_{\infty}\;\,\text{in}\;\, C^l_{\mathrm{loc}}(\Sigma\setminus\ens{p})\;\,\text{for all}\;\, l\in\N,
 \end{align*}
 and a Willmore sphere $\vec{\Psi}_{\infty}:S^2=\C\cup\ens{\infty}\rightarrow \R^n$ such that in a conformal chart $B(0,1)\subset \C$ centred in $p$, there exists $\ens{\rho_k}_{k\in\N}\subset (0,\infty)$ such that $\rho_k\conv{k\rightarrow \infty}0$, and for all $\alpha>0$
 \begin{align*}
     e^{-\lambda_k}(\phi_k(\rho_k\,\cdot\,)-\phi_k(0))_{B(0,\alpha^{-1})}\conv{k\rightarrow \infty}\vec{\Psi}_{\infty}\quad\text{in}\;\, C^l(B(0,\alpha^{-1}))\;\,\text{for all}\;\, l\in\N.
 \end{align*}
 The quantization of energy asserts that
 \begin{align}
     \lim_{k\rightarrow \infty}W(\phi_k)=W(\phi_{\infty})+W(\vec{\Psi}_{\infty}).
 \end{align}
 It is equivalent to the \emph{no-neck energy property}, namely
 \begin{align}
     \lim_{\alpha\rightarrow 0}\limsup_{k\rightarrow\infty}\int_{\Omega_k(\alpha)}|\D\n_k|^2dx=0.
 \end{align}
 where $\Omega_k(\alpha)=B_{\alpha}\setminus\bar{B}_{\alpha^{-1}\rho_k}(0)$ is the neck region, the annular region connecting the macroscopic surface to its bubble. Now, for all vectorial variation $\w:\Sigma\rightarrow \R^n$, the second derivative of $\phi$ is given by \eqref{simpler_d3} by
 \begin{align}\label{simpler_d20}
        &Q_{\phi}(\w)=\int_{\Sigma}\bigg\{\frac{1}{2}\left|\Delta_g^N\w+\mathscr{A}(\w)\right|^2+2|\s{d\w}{\H}|_g^2-\left(|d\w|_g^2-16|\partial\phi\totimes\partial\w|_{\mathrm{WP}}^2\right)|\H|^2\nonumber\\
    &+32\s{\partial^N\w\totimes\partial^N\H}{\partial\phi\totimes\partial\w}_{\mathrm{WP}}
    -8\s{\partial\w\totimes\partial\w}{\s{\H}{\h_0}}_{\mathrm{WP}}+8\s{d\phi}{d\w}_g\s{\partial\phi\totimes\partial\w}{\s{\H}{\h_0}}_{\mathrm{WP}}\nonumber\\
    &+8\,\{\phi,\w\}_g\s{i\,\partial\phi\totimes\partial\w}{\s{\H}{\h_0}}_{\mathrm{WP}}
    \bigg\}d\vg.
 \end{align}
 The Morse index is the number of negative eigenvalue of the associated fourth-order elliptic operator $\mathcal{L}_g$ acting on normal variations:
 \begin{align}\label{fourth_order_operator0}
    &\mathcal{L}_g=\leb_g^2+4\,d^{\ast_g}\left(\s{d(\,\cdot\,)}{\H}\H\right)-2\,d^{\ast_g}\left(|\H|^2\,d(\,\cdot\,)\right)+8\s{\s{\h_0}{\,\cdot\,}}{\s{\h_0}{\,\cdot\,}}_{\mathrm{WP}}|\H|^2\nonumber\\
    &+16\,\ast_g\, d\,\Re\left(g^{-1}\otimes \s{\h_0}{\,\cdot\,}\otimes\bar{\partial}^N\H\right)
    -16\,\s{\partial^N(\,\cdot\,)\totimes\partial^N\H}{\h_0}_{\mathrm{WP}}+16\,\ast_g\,d\,\Re\left(g^{-1}\otimes\s{\H}{\h_0}\otimes\partial(\,\cdot\,)\right)\nonumber\\
    &+8\s{\s{\H}{\h_0}}{\s{\h_0}{\,\cdot\,}}_{\mathrm{WP}}\H
    +8\s{\H}{\,\cdot\,}\Re\left(g^{-2}\otimes \s{\H}{\bar{\h_0}}\h_0\right).
\end{align}
Thanks to the identity
\begin{align}\label{diag0}
    Q_{\phi}(\vec{w})=\frac{1}{2}\int_{\Sigma}\s{\vec{w}}{\mathcal{L}_g\vec{w}\,}d\vg,
\end{align}
if for all $\lambda\in\R$, we define 
\begin{align*}
    \mathcal{E}(\lambda)=W^{2,2}(\Sigma,\R^n)\cap\ens{\w:\mathcal{L}_g\w=\lambda\,\w},
\end{align*}
then the Morse index of $\phi$ is defined by 
\begin{align*}
    \mathrm{Ind}^{\mathcal{L}}_W(\phi)=\dim\bigoplus_{\lambda<0}\mathcal{E}(\lambda),
\end{align*}
while the nullity is the dimension of the Kernel of $\mathcal{L}_g$:
\begin{align}
    \mathrm{Null}_W^{\mathcal{L}}(\phi)=\dim\mathrm{Ker}(\mathcal{L}_g).
\end{align}
One of the main technical difficulties or \cite{eigenvalue_annuli} is to generalise Theorem \ref{neck_harmonic2} to a family of second-order differential operators. First, using the analysis of \cite{pointwise}, if $\Omega_k(\alpha)$ is a neck region, there exists an integer $m\geq 1$ such that for $\alpha>0$ small enough and $k$ large enough,
\begin{align*}
    e^{\lambda_k}=e^{\mu_k}|z|^{m-1}\qquad\text{for all}\;\, z\in \Omega_k(\alpha),
\end{align*}
where $\mu_k$ is a uniformly bounded function on $B(0,\alpha)$. Restricting to normal variations $\w=u\,\n$, we see that the highest-order term in $Q_{\phi}(u)$ is given by
\begin{align*}
    \frac{1}{2}\int_{\Sigma}(\Delta_{g_k}u)^2d\mathrm{vol}_{g_k}=\frac{1}{2}\int_{\Sigma}e^{-2\mu_k}|z|^{2-2m}(\Delta u)^2|dz|^2. 
\end{align*}
Now, making the change of variable $u=|z|^{m-1}v$ in neck regions, we get
\begin{align*}
    \frac{1}{2}\int_{\Omega_k(\alpha)}(\Delta_{g_k}u)^2d\mathrm{vol}_{g_k}&=\frac{1}{2}\int_{\Omega_k(\alpha)}e^{-2\mu_k}\left(\Delta v+2(m-1)\frac{x}{|x|^2}\cdot \D v+\frac{(m-1)^2}{|x|^2}v\right)^2dx\\
    &\geq C\int_{\Omega_k(\alpha)}\left(\Delta v+2(m-1)\frac{x}{|x|^2}\cdot \D v+\frac{(m-1)^2}{|x|^2}v\right)^2dx.
\end{align*}
Therefore, we are led to study the properties of the second-order elliptic operator with  \emph{regular singularities} (\cite[($2.10$) p.$25$]{PR}) whose adjoint operator already appeared in \cite{index4}
\begin{align*}
    \leb_m=\Delta+2(m-1)\frac{x}{|x|^2}\cdot \D+\frac{(m-1)^2}{|x|^2},
\end{align*}
and more precisely, the first eigenvalue of the fourth-order elliptic operator
\begin{align*}
    \leb_m^{\ast}\leb_m=\Delta^2+2(m^2-1)\frac{1}{|x|^2}-4(m^2-1)\left(\frac{x}{|x|^2}\right)^t\cdot\D^2(\,\cdot\,)\cdot\left(\frac{x}{|x|^2}\right)+\frac{(m^2-1)^2}{|x|^4}.
\end{align*}

\subsubsection{General Estimates for a Family of Elliptic Fourth-Order Operators}

The generalisation of Theorem \ref{neck_harmonic2} is far from straightforward, and due to length considerations, we prove it  in a separate paper (see \cite{eigenvalue_annuli}). Furthermore, an additional eigenvalue problem needs to be considered in order to treat the derivatives of lower order. Contrary to Theorem \ref{neck_harmonic2}, the inequality is only verified when the conformal class of the annulus is large enough. 

\begin{theorem}[\cite{eigenvalue_annuli}]\label{pde_lemma1}
    For all $m\geq 1$, there exists $0<\lambda_m<\infty$ such that for all $0<a<b<\infty$ such that
    \begin{align}
        \log\left(\frac{b}{a}\right)\geq \lambda_m,
    \end{align}
    and let $\Omega=B_b\setminus\bar{B}_a(0)\subset \R^2$. Then, for all $u\in W^{2,2}_0(\Omega)$, we have
    \begin{align}
        \int_{\Omega}\left(\Delta u+2(m-1)\frac{x}{|x|^2}\cdot \D u+\frac{(m-1)^2}{|x|^2}u\right)^2dx\geq \frac{4m^2\pi^2}{\log^2\left(\frac{b}{a}\right)}\int_{\Omega}\frac{u^2}{|x|^4}dx,
    \end{align}
    and
    \begin{align}
        \int_{\Omega}\left(\Delta u+2(m-1)\frac{x}{|x|^2}\cdot \D u+\frac{(m-1)^2}{|x|^2}u\right)^2dx\geq \frac{\left(4m^2+\dfrac{\pi^2}{\log^2\left(\frac{b}{a}\right)}\right)\dfrac{\pi^2}{\log^2\left(\frac{b}{a}\right)}}{4(m^2+1)+\dfrac{2\pi^2}{\log^2\left(\frac{b}{a}\right)}}\int_{\Omega}\frac{|\D u|^2}{|x|^2}dx.
    \end{align}
\end{theorem}

The second estimate is the generalisation of \cite[Lemma IV.1]{riviere_morse_scs} (see Theorem \ref{neck_harmonic2}), which is as expected much more technical (especially for $m>1$).
\begin{theorem}[Theorem A, Theorem B \cite{eigenvalue_annuli}]\label{pde_lemma2}
    Let $0<a<b<\infty$, and define $\Omega=B_b\setminus\bar{B}_a(0)\subset \R^2$. Then, for all $\dfrac{1}{2}<\beta<\infty$, there exists  $C_{\beta}<\infty$ such that for all $u\in W^{2,2}_0(\Omega)$,
    \begin{align*}
        \int_{\Omega}\frac{u^2}{|x|^4}\left(\left(\frac{|x|}{b}\right)^{4\beta}+\left(\frac{a}{|x|}\right)^{4\beta}\right)dx\leq C_{\beta}\int_{\Omega}\left(\Delta u\right)^2dx
    \end{align*}
    and for all $\sqrt{2}-1<\beta<\infty$, there exists $C_{\beta}'<\infty$ such that 
    \begin{align*}
        \int_{\Omega}\frac{|\D u|^2}{|x|^2}\left(\left(\frac{|x|}{b}\right)^{2\beta}+\left(\frac{a}{|x|}\right)^{2\beta}\right)dx\leq C_{\beta}'\int_{\Omega}(\Delta u)^2dx.
    \end{align*}
    For all $m>1$ and for all $\alpha>0$, there exists $C_{\alpha,m}<\infty$ such that for all $u\in W^{2,2}_0(\Omega)$,
    \begin{align*}
        \int_{\Omega}\left(\frac{u^2}{|x|^4}+\frac{|\D u|^2}{|x|^2}\right)\left(\left(\frac{|x|}{b}\right)^{2\alpha}+\left(\frac{a}{|x|}\right)^{2\alpha}\right)dx\leq C_{\alpha,m}\int_{\Omega}\left(\Delta u+2(m-1)\frac{x}{|x|^2}\cdot \D u+\frac{(m-1)^2}{|x|^2}u\right)^2dx.
    \end{align*}
\end{theorem}
Therefore, we are brought to define the following weight (where $\dfrac{1}{2}<\beta<1$ is fixed once and for all) on any annulus $\Omega=B_b\setminus\bar{B}_a(0)\subset \R^2$ (that we identify to an annulus of $\Sigma$ by a fixed chart), that is nothing else than a constant interpolation of the natural weight given by the previous two estimates:
\begin{align}\label{weight_def}
    \omega_{a,b}(x)&=\left\{\begin{alignedat}{2}
        &\frac{1}{b^{4m}}\left(1+\left(\frac{a}{b}\right)^{4\beta}+\frac{1}{\log^2\left(\frac{b}{a}\right)}\right)\qquad&& \text{for all}\;\, x\in \Sigma\setminus\bar{B}_b(0)\\
        &\frac{1}{|x|^{4m}}\left(\left(\frac{|x|}{b}\right)^{4\beta}+\left(\frac{a}{b}\right)^{4\beta}+\frac{1}{\log^2\left(\frac{b}{a}\right)}\right)\qquad&& \text{for all}\;\, x\in B_b\setminus\bar{B}_a(0)\\
        &\frac{1}{a^{4m}}\left(1+\left(\frac{a}{b}\right)^{4\beta}+\frac{1}{\log^2\left(\frac{b}{a}\right)}\right)\qquad &&\text{for all}\;\, x\in B_{a}(0)
        \end{alignedat}\right.
\end{align}
Here, we see additional difficulties arising when $m\geq 2$. Indeed, $\omega_{\alpha^{-1}\delta_k,\alpha}$ converges to $\omega_{0,\alpha}$ as $k\rightarrow \infty$, where 
\begin{align*}
    \omega_{0,\alpha}=\left\{\begin{alignedat}{2}
        &\frac{1}{\alpha^{4m}}\qquad&&\text{for all}\;\,x\in \Sigma\setminus\bar{B}_{\alpha}(0)\\
        &\frac{1}{\alpha^{4\beta}}\frac{1}{|x|^{4(m-\beta)}}\qquad&&\text{for all}\;\,x\in B_{\alpha}(0)
    \end{alignedat}\right.
\end{align*}
Since $0<\beta<1$, provided that $m\geq 2$, our weight $\omega_{0,\alpha}$ does not belong to $L^{1}_{\mathrm{loc}}(\Sigma)$. Therefore, normalising an eigenvalue as $1$ with respect to our weight will not allow us to take a strong limit in the quantity
\begin{align}
    \int_{\Sigma}u_k^2\,\omega_{\alpha^{-1}\delta_k,\alpha}\,d\mathrm{vol}_{g_k}=1.
\end{align}
This is one of the reasons why in the asymptotic analysis, we restrict to the case $m=1$. Determining the exact (possible) defect in the upper semi-continuity of the (extended) Morse index will be the topic of a future work (\cite{morse_willmore_II}). However, our present analysis permits to treat the case of degenerating Riemann surfaces (\cite{lauriv1}), as in \cite{riviere_morse_scs} in the case of conformally invariant problems in dimension $2$—that include harmonic maps. 

Before stating our main theorem, let us give a few additional comments the possible extensions of this work.

Although those eigenvalue estimates or weighted Poincaré inequality seem to restrict to dimension $2$, Theorem \ref{pde_lemma2} easily generalises to dimension $4$ for the Laplace operator thanks to the Sobolev embedding $W^{2,2}(\R^4)\hookrightarrow W^{1,4}(\R^4)\hookrightarrow L^p(\R^4)$ for all $p<\infty$. In this example, using the conformal invariance (notice that $\leb_m$ is not conformally invariant in dimension $2$ for all $m>1$), the proof is an easy adaptation of the proof of \cite[Lemma IV.$1$]{riviere_morse_scs}. More surprisingly, Theorem \ref{pde_lemma1} generalises in dimension $4$ for the first eigenvalue of the bilaplacian $\Delta^2$. Furthermore, the minimiser of the associated problem is radial (provided that the conformal class of the annulus is large enough). 

We believe that our ODE analysis is \emph{universal} in the sense that it should allow us to treat a large class of elliptic operators of order $4$ in all dimensions. Indeed, the algebraic systems obtained in higher dimension $d\geq 3$ always reduce  to one of our $2$-dimensional systems, which permits to obtain a lower estimate of the first eigenvalue of the bilaplacian in annuli in all dimension. In particular, we deduce the following weighted Poincaré estimate in all dimension.

 \begin{theorem}[Theorem C \cite{morse_biharmonic}]\label{poincare_gen0}
        Let $d\geq 3$, let $0<a<b<\infty$, and let $\Omega=B_b\setminus\bar{B}_a(0)\subset \R^d$.Then, for all  $u\in W^{2,2}_0(\Omega)$, we have
        \begin{align}\label{poincare_dim_gen0}
            \int_{\Omega}(\Delta u)^2dx\geq \left(\frac{d^2}{4}+\frac{\pi^2}{\log^2\left(\frac{b}{a}\right)}\right)\left(\frac{(d-4)^2}{4}+\frac{\pi^2}{\log^2\left(\frac{b}{a}\right)}\right)\int_{\Omega}\frac{u^2}{|x|^4}dx.
        \end{align}
        Furthermore, for all $d\geq 3$ and for all $u\in W^{2,2}_0(\Omega)\subset W^{2,2}(\R^d)$, we have
        \begin{align}\label{poincare_dim_gen1}
            \int_{\Omega}(\Delta u)^2dx\geq \mu_1(\Omega)\int_{\Omega}\frac{|\D u|^2}{|x|^2}dx,
        \end{align}
        where
        \begin{align*}
            \mu_1(\Omega)>\left\{\begin{alignedat}{2}
                &\frac{25+\dfrac{104\pi^2}{\log^2\left(\frac{b}{a}\right)}+\dfrac{16\pi^4}{\log^2\left(\frac{b}{a}\right)}}{36+\dfrac{16\pi^2}{\log^2\left(\frac{b}{a}\right)}}\qquad&& \text{for}\;\,d=3\\
                &\frac{9+\dfrac{10\pi^2}{\log^2\left(\frac{b}{a}\right)}+\dfrac{\pi^4}{\log^2\left(\frac{b}{a}\right)}}{3+\dfrac{\pi^2}{\log^2\left(\frac{b}{a}\right)}}\qquad&&\text{for}\;\, d=4\\
                &\frac{d^2(d-4)^2+((d-2)^2+4)\dfrac{8\pi^2}{\log^2\left(\frac{b}{a}\right)}+\dfrac{16\pi^4}{\log^4\left(\frac{b}{a}\right)}}{4(d-4)^2+\dfrac{16\pi^2}{\log^2\left(\frac{b}{a}\right)}}\qquad &&\text{for}\;\, d\geq 5.
            \end{alignedat}\right.
        \end{align*}
    \end{theorem}
    
    \begin{rem}
    \begin{enumerate}
    \item[(1)]
    For $d=4$, the estimate reads
    \begin{align*}
        \lambda_1(\Delta^2,B_b\setminus\bar{B}_a(0))>\left(4+\frac{\pi^2}{\log^2\left(\frac{b}{a}\right)}\right)\frac{\pi^2}{\log^2\left(\frac{b}{a}\right)}>\frac{4\pi^2}{\log^2\left(\frac{b}{a}\right)},
    \end{align*}
    and the upper estimate shows that
    \begin{align*}
        \lambda_1(\Delta^2,B_b\setminus\bar{B}_a(0))<\left(4+\frac{4\pi^2}{\log^2\left(\frac{b}{a}\right)}\right)\frac{4\pi^2}{\log^2\left(\frac{b}{a}\right)}\conv{\frac{b}{a}\rightarrow \infty}0,
    \end{align*}
    where $\lambda_1>0$ is the first positive eigenvalue for the weighted problem
    \begin{align*}
        \Delta^2u=\lambda_1\frac{u}{|x|^4}.
    \end{align*}
    Notice that the power $4$ is not related to the dimension but to the order of the differential operator.
    This degeneracy is expected in critical dimension where concentration phenomenons occur.
    \item[(2)] Thanks to an upper bound on the first eigenvalue $\lambda_1(\Omega)$ appearing on the right-hand side of \eqref{poincare_dim_gen0}, one can recover classical Rellich inequalities (in the full space $\R^d$) in dimension $d\geq 5$ and $d=3$ (for the subspace $W^{2,2}(\R^3)\cap\ens{u:u(0)=0}$) with optimal constants. Likewise, \eqref{poincare_dim_gen1} gives another proof of the Hardy-Rellich inequality with optimal constants in all dimension $d\geq 3$ (see \cite{hardy-rellich1, hardy-rellich2, survey_cazacu}). 
    \end{enumerate}
    \end{rem}

\subsubsection{Main Results}

As we have explained above, due to length and technical difficulties, we restrict ourselves to the case $m=1$ in the asymptotic analysis. First, let us define the notion of bubble convergence for Willmore immersions of bounded energy and possibly degenerating conformal class.

As in \cite{riviere_morse_scs}, we can also treat the case of surfaces of degenerating conformal class, provided that a suitable bound holds on the second residue. Recall that for a Willmore immersion $\phi:\Sigma\rightarrow \R^n$, the second residue associated to the simple closed curve $\gamma$ is given by (\cite[Equation (10)]{classification})
\begin{align}\label{second_residue_def}
    \vec{\gamma}_1(\phi,\gamma)=\frac{1}{4\pi}\Im\int_{\gamma}\phi \wedge\left(\partial\H+|\H|^2\partial\phi+\,g^{-1}\otimes\s{\H}{\h_0}\otimes\bar{\partial}\phi\right)+g^{-1}\otimes \h_0\wedge \bar{\partial}\phi.
\end{align}

\begin{defi}
    Let $\Sigma$ be a closed Riemann surface, $\{\phi_k\}_{k\in\N}$ be a sequence of Willmore immersions from $\Sigma$ into $\R^n$, and for all $k\in\N$, let $g_k=\phi^{\ast}g_{\R^n}$. If 
    \begin{align*}
        \limsup_{k\rightarrow \infty}W(\phi_k)<\infty,
    \end{align*}
    we say that $\{\phi_k\}_{k\in\N}$ \emph{bubble converges} to $(\phi_{\infty}^1,\cdots,\phi_{\infty}^m,\vec{\Psi}_1,\cdots,\vec{\Psi}_p,\vec{\chi}_1,\cdots,\vec{\chi}_q)$ provided that
    \begin{enumerate}
        \item[($1$)] The sequence $\{(\Sigma,g_k)\}_{k\in\N}$ converges (in the sense of Deligne-Mumford) to a nodal surface $(\widetilde{\Sigma},\widetilde{h})$ of connected components $(\widetilde{\Sigma}_1,\cdots,\widetilde{\Sigma}_m)$.
        \item[($2$)] For all $1\leq l\leq m$, the map $\phi_{\infty}^l:\Sigma_l\rightarrow \R^n$ is a Willmore immersion such that (up to conformal transformations and diffeomorphisms in the domain) 
    \begin{align*}
        \phi_k\conv{k\rightarrow \infty}(\phi_{\infty}^1,\cdots,\phi_{\infty}^m)\qquad \text{in}\;\, C^{r}_{\mathrm{loc}}\left(\Sigma\setminus\ens{p_k^1,\cdots,p_k^{2N}}\right)\;\,\text{for all}\;\, r\in \N,
    \end{align*}
    where $\ens{p_k^1,\cdots,p_k^{2N}}$ converge to the $2N$ ($N\in\N$) cusps of $\widetilde{\Sigma}$.
    \item[($3$)] If $\Gamma_k=\ens{\gamma_k^1,\cdots,\gamma_{k}^N}\subset \Sigma$ are the shrinking geodesics of $(\Sigma,g_k)$, then there are no concentration points on $\Gamma_k$.
    \item[($4$)] There exists a finite number of concentration points $\ens{a_k^1,\cdots,a_k^p}\subset \Sigma$ and $\ens{b_k^1,\cdots,b_k^q}\subset \Sigma$ such that (up to conformal maps and diffeomorphisms) for all $1\leq i\leq p$ and $1\leq j\leq q$
    \begin{align*}
        \phi_k(\rho_k^{1,i}z+a_k^i)\conv{k\rightarrow \infty}\vec{\Psi}_i\qquad \text{in}\;\, C^r_{\mathrm{loc}}\left(\widehat{\C}\setminus \mathscr{S}_i^1\right)\;\,\text{for all}\;\,r\in\N,
    \end{align*}
    and 
    \begin{align*}
        \phi_k(\rho_k^{2,i}z+a_k^i)\conv{k\rightarrow \infty}\vec{\chi}_j\qquad \text{in}\;\, C^r_{\mathrm{loc}}(\widehat{\C}\setminus \mathscr{S}_j^2)\;\,\text{for all}\;\,r\in\N,
    \end{align*}
    where $\{\rho_k^{\alpha,\beta}\}_{k\in\N}\subset (0,\infty)$ is such that $\rho_k^{\alpha,\beta}\conv{k\rightarrow \infty}0$ and $\mathscr{S}_{\alpha}^{\beta}$ are discrete subsets of $\widehat{\C}=\C\cup\ens{\infty}\simeq S^2$.
    \item[($5$)] The exists integers $m_1,\cdots,m_q\geq 1$ such that
    \begin{align}
        \lim_{k\rightarrow \infty}W(\phi_k)=\sum_{l=1}^mW(\phi_{\infty}^l)+\sum_{i=1}^pW(\vec{\Psi}_i)+\sum_{j=1}^q\left(W(\vec{\chi}_j)-4\pi m_j\right).
    \end{align}
    \end{enumerate}
\end{defi}
\begin{rem}
\begin{enumerate}
    \item[($1$)]     The difference between $\{\vec{\Psi}_i\}_{1\leq i\leq p}$ and $\ens{\vec{\chi}_j}_{1\leq j\leq q}$ is that the first family comes from compact bubbles, whilst the second family is made of inverted non-compact bubbles. For example, in the case $\Sigma=S^2$ and $W(\phi_k)=16\pi$ for all $k\in\N$ N. Marque (\cite{marque_minimal_bubbling}) found an example where $p=0$ and $q=1$, where $\vec{\phi}_{\infty}$ is the inversion of the L\'{o}pez surface (a minimal surface with two ends: one embedded end (\emph{i.e.} of multiplicity $1$ with zero logarithmic growth) and one end of multiplicity $3$) and $\vec{\chi}_1$ is the inversion of the Enneper surface $\vec{\zeta}:\C\rightarrow \R^3$, a minimal surface with a single end of multiplicity $3$. In particular, we have $W(\vec{\chi}_1)=12\pi$, $m_1=3$, and
    \begin{align*}
        16\pi=W(\phi_k)\conv{k\rightarrow \infty}W(\phi_{\infty})+(W(\vec{\chi}_1)-12\pi)=W(\phi_{\infty}).
    \end{align*}
    To our knowledge, no explicit example of energy quantization for higher genus surface was constructed to this day.
    \item[($2$)] Thanks to the main theorem of \cite{lauriv1}, the bubble convergence and the energy identity hold provided that for all $1\leq l\leq N$
    \begin{align*}
        \lim_{k\rightarrow \infty}\frac{\vec{\gamma}_1(\phi_k,\gamma_k^l)}{\sqrt{\leb(\gamma_k^l)}}=0,
    \end{align*}
    where $\leb(\gamma_k^l)$ is the length of $\gamma_k^l$ with respect to the metric $g_k$.
\end{enumerate}
\end{rem}

\renewcommand*{\thetheorem}{\Alph{theorem}}
\setcounter{theorem}{0}
\begin{theorem}\label{main_theorem}
    Let $n\geq 3$, let $\Sigma$ be a closed Riemann surface, and assume that $\{\phi_k\}_{k\in\N}$ bubble converges to $(\phi_{\infty}^1,\cdots,\phi_{\infty}^m,\vec{\Psi}_1,\cdots,\vec{\Psi}_p,\vec{\chi}_1,\cdots,\vec{\chi}_q)$, where all Willmore immersions are \emph{unbranched}, and let $\ens{\gamma_k^1,\cdots,\gamma_k^N}\subset \Sigma$ be the set of shrinking geodesics of $(\Sigma,g_k)$. There exists a universal constant $\Lambda_n>0$ such that the bound
    \begin{align}\label{smallness_residue}
        \limsup_{k\rightarrow \infty}\max_{1\leq l\leq N}\frac{|\vec{\gamma}_1(\phi_k,\gamma_k^l)|}{\leb(\gamma_k^l)}\leq \Lambda_n
    \end{align}
    implies that 
    \begin{align}\label{scs_unbranched}
        \limsup_{k\rightarrow \infty}\left(\mathrm{Ind}_W(\phi_k)+\mathrm{Null}_W(\phi_k)\right)\leq \sum_{l=1}^m\mathrm{Ind}_W^0(\phi_{\infty}^l)+\sum_{i=1}^p\mathrm{Ind}_W^0(\vec{\Psi}_{i})+\sum_{j=1}^q\mathrm{Ind}_W^0(\vec{\chi}_{j})<\infty
    \end{align}
    where $\mathrm{Ind}^0_W=\mathrm{Ind}_W+\mathrm{Null}_W$.
\end{theorem}
\setcounter{theorem}{9}
\renewcommand*{\thetheorem}{\thesection.\arabic{theorem}}
\begin{rem}
    Notice that hypothesis \eqref{smallness_residue} implies that 
    \begin{align*}
         \lim_{k\rightarrow \infty}\max_{1\leq l\leq N}\frac{|\vec{\gamma}_1(\phi_k,\gamma_k^l)|}{\sqrt{\leb(\gamma_k^l)}}=0.
    \end{align*}
    Therefore, the analysis of \cite{lauriv1} implies that the quantization of energy holds for our sequence of immersions. Namely, 
    \begin{align*}
        \lim_{k\rightarrow \infty}W(\phi_k)=\sum_{l=1}^nW(\phi_{\infty}^l)+\sum_{i=1}^pW(\vec{\Psi}_i)+\sum_{j=1}^q\left(W(\vec{\chi}_j)-4\pi\hspace{0.1em}m_j\right)
    \end{align*}
    for some integer $m_1,\cdots,m_q\geq 1$ such that $4\pi\hspace{0.1em}m_j\leq W(\vec{\chi}_j)$ for all $1\leq j\leq m$.
\end{rem}

As expected, the hypotheses to get the lower semi-continuity of the Morse index (see \cite{index3}) are weaker than those of Theorem \ref{main_theorem}, and we do not need to assume smoothness at the limit.
\setcounter{theorem}{1}
\renewcommand*{\thetheorem}{\Alph{theorem}}
\begin{theorem}\label{main_theoremII}
    Let $n\geq 3$, let $\Sigma$ be a closed Riemann surface, and assume that $\{\phi_k\}_{k\in\N}$ bubble converges to $(\phi_{\infty}^1,\cdots,\phi_{\infty}^m,\vec{\Psi}_1,\cdots,\vec{\Psi}_p,\vec{\chi}_1,\cdots,\vec{\chi}_q)$, and let $\ens{\gamma_k^1,\cdots,\gamma_k^N}\subset \Sigma$ be the set of shrinking geodesics of $(\Sigma,g_k)$. Assume that
    \begin{align}\label{smallness_residue_II}
        \lim_{k\rightarrow \infty}\max_{1\leq l\leq N}\frac{|\vec{\gamma}_1(\phi_k,\gamma_k^l)|}{\sqrt{\leb(\gamma_k^l)}}=0.
    \end{align}
    Then, we have 
    \begin{align}\label{sci_general}
        \sum_{l=1}^m\mathrm{Ind}_W(\phi_{\infty}^l)+\sum_{i=1}^p\mathrm{Ind}_W(\vec{\Psi}_{i})+\sum_{j=1}^q\mathrm{Ind}_W(\vec{\chi}_{j})\leq \liminf_{k\rightarrow \infty}\mathrm{Ind}_W(\phi_k)
    \end{align}
\end{theorem}
\setcounter{theorem}{10}
\renewcommand*{\thetheorem}{\thesection.\arabic{theorem}}
\begin{rem}
    Let us point out that the definition of the Morse index is different in the case of branched immersions since normal variations are generally not integrable. Therefore, we define the Morse index in the classical sense as the number of negative vectorial variations of the second derivative.
\end{rem}

\subsubsection{Open Problems and Future Work}

Besides the afore-mentioned extension of our work to branched immersions (\cite{morse_willmore_II}) and (extrinsic or intrinsic) biharmonic maps in dimension $4$ (\cite{morse_biharmonic}), the most natural extension would be to the viscosity method for Willmore surfaces (\cite{eversion,morse_viscosity}), namely for solutions of min-max problems related for the Willmore energy. The most famous example is the \emph{min-max sphere eversion} (\cite{kusner,brakke}), that historically consists in using  the Willmore energy as a gradient flow to find the \enquote{optimal} sphere eversion—a sphere eversion is a continuous path of immersions linking the sphere equipped with its standard orientation to the antipodal sphere (\cite{smale})). 

Due to the \emph{universality property} (see \cite{mondinonguyen}) of the Willmore enery as the only non-trivial conformally invariant Lagrangian (since the integral of the Gauss curvature is constant by virtue of Gauss-Bonnet theorem—for the total curvature, \emph{i.e.} the integral of the absolute value of the Gauss curvature, Tobias Ekholm (\cite{ekholm}) showed that the total curvature is at most $8\pi$, but it is not know whether this value is optimal or not), it seems fairly natural to use this Lagrangian as a distance function with respect to the standard embedding of the sphere. Indeed, a classical rigidity result shows that $W(\phi)>4\pi$, unless $\phi:S^2\rightarrow\R^3$ is an embedding of the round sphere (in which case $W(\phi)=4\pi$ since $H=-1$ and $\mathrm{Vol}(S^2)=4\pi$ for the round metric).

Another direction is to study compactness properties provided that a bound on the Morse index is given. The results above might help us restrict the number of possible bubblings and extend the previously known compactness bounds. 

\subsection{Morse Index Stability in Geometric Analysis: History and Perspectives}

Let us draw a short panorama of recent results related to (lower or upper) semi-continuity of the Morse index for harmonic maps and minimal surfaces.

In the Almgren-Pitts min-max theory of minimal surface, the first notable result was proven by F. C. Marques and A. Neves (\cite{marquesmorse}) who showed an upper Morse index bound of min-max minimal hypersurfaces, that can be interpreted as an lower semi-continuity result for the Morse index. One of the applications of this result is to simplify the proof of the Willmore conjecture (\cite{marqueswillmore}). The proof of lower index estimates by Marques-Neves (\cite{marques_lower_bound})—that corresponds to an upper-semi continuity theorem—is much more technical and uses the \emph{Multiplicity One Conjecture} proven by X. Zhou (\cite{zhou_xin}), that asserts that for a generic set of metrics (the bumpy metrics of B. White: \cite{bumpy_white1,bumpy_white2}). Marques-Montezuma-Neves (\cite{marques_morse_inequalities}) later extended this result to prove the general Morse inequalities (see \cite{marques_survey_morse} for more details). For the alternative construction using the Allen-Cahn functional of M. Guaraco and Gaspar-Guaraco (\cite{guaraco1, guaraco2}), O. Chodosh and C. Mantoulidis showed in dimension $3$ the equality between the dimension of the min-max family and the Morse index for \enquote{bumpy metrics} (\cite{chodosh_mantoulidis}; other upper semi-continuity results include the work of F. Hiesmayr (\cite{morse_allen-cahn1}) assuming that the limiting surface is two-sided and the work P. Gaspar (\cite{morse_allen-cahn2}) in the general case). 

As it appears in the afore-mentioned literature, the \emph{Multiplicity One Conjecture} is a central property that is often crucial to obtain upper index bound. In the framework of the Allen-Cahn functional, this property was proven in the case of one-parameter families by C. Bellettini (\cite{bellettini_multiplity_one1,bellettini_multiplity_one2}). When it comes to the \emph{viscosity method} of T. Rivière (\cite{viscosity}) for the construction of ($2$-dimensional) minimal surfaces in arbitrary codimension (for which only the lower semi-continuity of the index is known \cite{lower}), the \emph{Multiplicity One Conjecture} was proven by A. Pigati and T. Rivière (\cite{multiplicity}; see also \cite{pigati1}). It is likely that the method developed in \cite{riviere_morse_scs} will lead to a proof of the upper semi-continuity of the Morse index in the viscosity method for minimal surfaces (while the current work should apply to the viscosity method). Other natural directions would be to apply the method to the Yang-Mills-Higgs model that permits to construct codimension $2$ minimal surfaces as in the work of A. Pigati and D. Stern \cite{pigati_stern}, or Morse index stability results in the viscosity method for free boundary minimal surface of A. Pigati (\cite{pigati_viscosity_boundary}). Finally, let us mention the work of M. Karpukhin and D. Stern (\cite{stern_harmonic,stern_harmonic2}; see also the earlier work of M. Karpukhin \cite{karpukhin}) on harmonic maps that uses Morse index estimates for the Ginzburg-Landau functional and has applications to eigenvalues estimates and maximisers of the conformal volume (introduce by Li-Yau in \cite{lieyau}) that uses min-max constructions that enter in the framework of \emph{min-max hierarchies} \cite{hierarchies,hierarchies2}).

\section{The Second Derivative of the Willmore Energy}

Let $\Sigma$ be a closed Riemann surface, $(M^n,h)$ be a complete Riemannian manifold, and let $\phi:\Sigma\rightarrow M^n$ be smooth immersion. Define
\begin{align*}
    W(\phi)=\int_{\Sigma}|\H|^2d\vg,
\end{align*}
where $g=\phi^{\ast}h$ is the induced metric, and $\H$ is the mean curvature, given by 
\begin{align*}
    \H=\frac{1}{2}\sum_{i,j=1}^2\vec{\I}_{i,j},
\end{align*}
where $\vec{\I}$ is the second fundamental form.

\subsection{First Variation}

We assume that $(M^n,h)$ is isometrically embedded into $\R^m$ for $m$ large enough. If $\w\in C^{\infty}(\Sigma,\R^m)$ is any variation, let $\{\phi_t\}_{-\epsilon<t<\epsilon}$ such that $\phi_0=\phi$ and $\dfrac{d}{dt}\phi_t=\w$. Without loss of generality, we can assume that the following identity holds for all $-\epsilon<t<\epsilon$ 
\begin{align*}
	\D_{\frac{d}{dt}}\frac{d}{dt}\phi_t=0.
\end{align*}
 It is clear in $\R^n$ since we can choose $\phi_t=\phi+t\,\w$.

 \subsubsection{Metric}

In local coordinates, defining
\begin{align*}
	g_{i,j}=\s{\p{x_i}\phi}{\p{x_j}\phi},
\end{align*}
we have $\det(g)=g_{1,1}\,g_{2,2}-g_{1,2}^2$. Therefore, if $\{\phi_t\}_{t\in (-\epsilon,\epsilon)}$ is a variation of $\phi$ and $\vec{w}=\dfrac{d}{dt}\left(\phi_t\right)_{|t=0}$, by the classical symmetry lemma, we have
\begin{align*}
	\D_{\frac{d}{dt}\phi_t}\p{z}\phi_t=\D_{\p{z}\phi_t}\frac{d}{dt}\phi_t.
\end{align*}
If $\w_t=\dfrac{d}{dt}\phi_t$, we deduce that
\begin{align*}
	\frac{d}{dt}g_{i,j}=\s{\D_{\p{x_i}\phi_t}\w_t}{\p{x_j}\phi_t}+\s{\D_{\p{x_j}\phi_t}\w_t}{\phi_t}.
\end{align*}
Therefore, we have
\begin{align*}
	\frac{d}{dt}\det(g_t)&=\frac{d}{dt}\left(g_{1,1}g_{2,2}-g_{1,2}\right)=2\,g_{2,2}\s{\D_{\p{x_1}\phi_t}\w_t}{\p{x_1}\phi_t}+2\,g_{1,2}\s{\D_{\p{x_2}\phi_t}\w_t}{\p{x_2}\phi_t}\\
	&-2\,g_{1,2}\left(\s{\D_{\p{x_1}\phi_t}\w}{\p{x_2}\phi_t}+\s{\D_{\p{x_2}\phi_t}\w_t}{\p{x_1}\phi_t}\right)\\
	&=2\s{d\phi_t}{d\w_t}_{g_t}\det(g_t),
\end{align*}
since
\begin{align*}
	g^{i,j}=(-1)^{i+j}\frac{g_{i+1,j+1}}{\det(g)}
\end{align*}
using $\Z_2$ indices, which implies that for all smooth functions $\varphi,\psi:\Sigma\rightarrow \R$
\begin{align*}
	\s{d\varphi}{d\psi}_g=\sum_{i,j=1}^{2}g^{i,j}\p{x_i}\varphi\cdot\p{x_j}\psi.
\end{align*}
We also get
\begin{align*}
	\frac{d}{dt}\sqrt{\det(g_t)}=\s{d\phi_t}{d\w_t}_{g_t}\sqrt{\det(g_t)}.
\end{align*}
Now, we deduce that
\begin{align*}
	\frac{d}{dt}g^{i,j}=-2\s{d\phi_t}{d\w_t}_{g_t}g^{i,j}+\frac{(-1)^{i+j}}{\det(g_t)}\left(\s{\D_{\p{x_{i+1}}\phi_t}\w_t}{\p{x_{j+1}}\phi_t}+\s{\D_{\p{x_{j+1}}\phi_t}\w_t}{\p{x_{i+1}}\phi_t}\right).
\end{align*}
At $t=0$, we get
\begin{align*}
    \frac{d}{dt}g^{1,1}=-2\,e^{-4\lambda}\left(\s{\D_{\e_1}\w}{\e_1}+\s{\D_{\e_2}\w}{\e_2}\right)+2\,e^{-4\lambda}\s{\D_{\e_2}\w}{\e_2}=-2\,e^{-4\lambda}\s{\D_{\e_1}\w}{\e_1},
\end{align*}
and likewise, we find that for all $i,j=1,2$
\begin{align*}
    \frac{d}{dt}g^{i,j}=-e^{-4\lambda}\left(\s{\D_{\e_i}\w}{\e_j}+\s{\D_{\e_j}\w}{\e_i}\right).
\end{align*}

\subsubsection{Second Fundamental Form}

Notice that 
\begin{align*}
	\vec{\I}(\p{x_i}\phi_t,\p{x_j}\phi_t)=\D_{\p{x_i}\phi_t}^{N}\p{x_j}\phi_t=\D_{\p{x_i}\phi_t}\p{x_j}\phi_t-\D_{\p{x_i}\phi_t}^{\top}\p{x_j}\phi_t.
\end{align*}
For all vector-fields $X,Y$, if
\begin{align*}
	\D_X^{\top}Y=\lambda_1\p{x_1}\phi_t+\lambda_2\p{x_2}\phi_t,
\end{align*}
we get
\begin{align*}
	&\s{\D_XY}{\p{x_1}\phi_t}=g_{1,1}\lambda_1+g_{1,2}\lambda_2\\
	&\s{\D_{X}Y}{\p{x_2}\phi_t}=g_{1,2}\lambda_1+g_{2,2}\lambda_2.
\end{align*}
Therefore, we can rewrite the system as 
\begin{align*}
	\begin{pmatrix}
		\s{\D_XY}{\p{x_1}\phi_t}\\
		\s{\D_XY}{\p{x_2}\phi_t}
	\end{pmatrix}
    =\begin{pmatrix}
    	g_{1,1} & g_{1,2}\\
    	g_{1,2} & g_{2,2}
    \end{pmatrix}\begin{pmatrix}
    \lambda_1\\
    \lambda_2
\end{pmatrix}
\end{align*}
which means that
\begin{align*}
	\begin{pmatrix}
		\lambda_1\\
		\lambda_2
	\end{pmatrix}=\begin{pmatrix}
		g^{1,1} & g^{1,2}\\
		g^{1,2} & g^{2,2}
	\end{pmatrix}\begin{pmatrix}
	\s{\D_XY}{\p{x_1}\phi_t}\\
	\s{\D_XY}{\p{x_2}\phi_t}
\end{pmatrix}=\begin{pmatrix}
g^{1,1}\s{\D_XY}{\p{x_1}\phi_t}+g^{1,2}\s{\D_XY}{\p{x_2}\phi_t}\\
g^{1,2}\s{\D_XY}{\p{x_1}\phi_t}+g^{2,2}\s{\D_XY}{\p{x_2}\phi_t}
\end{pmatrix}.
\end{align*}
In other words, we have
\begin{align*}
	\vec{\I}(\p{x_i}\phi_t,\p{x_j}\phi_t)&=\D_{\p{x_i}\phi_t}\p{x_j}\phi_t-\left(g^{1,1}\s{\D_{\p{x_i}\phi_t}\p{x_j}\phi_t}{\p{x_1}\phi_t}+g^{1,2}\s{\D_{\p{x_i}\phi_t}\p{x_j}\phi_t}{\p{x_2}\phi_t}\right)\p{x_1}\phi_t\\
	&-\left(g^{1,2}\s{\D_{\p{x_i}\phi_t}\p{x_j}\phi_t}{\p{x_1}\phi_t}+g^{2,2}\s{\D_{\p{x_i}\phi_t}\p{x_j}\phi_t}{\p{x_2}\phi_t}\right)\p{x_2}\phi_t.
\end{align*}
We first estimate
\begin{align*}
	\D_{\frac{d}{dt}}\D_{\p{x_i}\phi_t}\p{x_j}\phi_t=\D_{\p{x_i}\phi_t}\D_{\p{x_j}\phi_t}\w_t+R(\w_t,\p{x_i}\phi_t)\p{x_j}\phi_t.
\end{align*}
Therefore, we have
\begin{align*}
	&\pi_{\n_t}\left(\D_{\frac{d}{dt}}\vec{\I}(\p{x_i}\phi_t,\p{x_j}\phi_t)\right)=\D_{\p{x_i}\phi_t}^N\D_{\p{x_j}\phi_t}\w_t+\pi_{\n}\left(R(\w_t,\p{x_i}\phi_t)\p{x_j}\phi_t\right)\\
	&-\left(g^{1,1}\s{\D_{\p{x_i}\phi_t}\p{x_j}\phi_t}{\p{x_1}\phi_t}+g^{1,2}\s{\D_{\p{x_i}\phi_t}\p{x_j}\phi_t}{\p{x_2}\phi_t}\right)\D_{\p{x_1}\phi_t}^N\w\\
	&-\left(g^{1,2}\s{\D_{\p{x_i}\phi_t}\p{x_j}\phi_t}{\p{x_1}\phi_t}+g^{2,2}\s{\D_{\p{x_i}\phi_t}\p{x_j}\phi_t}{\p{x_2}\phi_t}\right)\D_{\p{x_2}\phi_t}^N\w_t\\
    &=\D_{\p{x_i}\phi_t}^N\D_{\p{x_j}\phi_t}w_t-\D_{\D_{\p{x_i}\phi_t}^{\top}\p{x_j}\phi_t}^N\w_t+\pi_{\n}\left(R(\w_t,\p{x_i}\phi_t)\p{x_j}\phi_t\right)
\end{align*}
where $\pi_{\n_t}$ is the normal projection. Furthermore, we have
\begin{align*}
    &\D_{\p{x_j}\phi_t}^{\top}\w_t=\left(g^{1,1}\s{\D_{\p{x_j}\phi_t}\w_t}{\p{x_1}\phi_t}+g^{1,2}\s{\D_{\p{x_j}\phi_t}\w_t}{\p{x_2}\phi_t}\right)\p{x_1}\phi_t\\
    &+\left(g^{1,2}\s{\p{x_j}\w_t}{\p{x_1}\phi_t}+g^{2,2}\s{\D_{\p{x_j}\phi_t}\w_t}{\p{x_2}\phi_t}\right)\p{x_2}\phi_t.
\end{align*}
Therefore, we get
\begin{align*}
    &\D_{\p{x_i}\phi_t}^N\D_{\p{x_j}\phi_t}\w_t=\D_{\p{x_i}\phi_t}^N\D_{\p{x_j}\phi_t}^N\w_t+\D^N_{\p{x_i}\phi_t}\D^{\top}_{\p{x_j}\phi_t}\w_t\\
    &=\D_{\p{x_i}\phi_t}^N\D_{\p{x_j}\phi_t}^N\w_t+\left(g^{1,1}\s{\p{x_j}\w_t}{\p{x_1}\phi_t}+g^{1,2}\s{\p{x_j}\w_t}{\p{x_2}\phi_t}\right)\vec{\I}(\p{x_1}\phi_t,\p{x_i}\phi_t)\\
    &-\left(g^{1,2}\s{\p{x_j}\w_t}{\p{x_1}\phi_t}+g^{2,2}\s{\p{x_j}\w_t}{\p{x_2}\phi_t}\right)\vec{\I}(\p{x_2}\phi_t,\p{x_i}\phi_t)\\
    &=\D_{\p{x_i}\phi_t}^N\D_{\p{x_j}\phi_t}^N\w_t+\vec{\I}(\p{x_i}\phi_t,\D_{\p{x_j}\phi_t}^{\top}\w_t).
\end{align*}
Finally, if 
\begin{align*}
    \D^2_{X,Y}=\D_X^{N}\D_Y^N-\D_{\D_X^{\top}Y}^N,
\end{align*}
we get
\begin{align}
    \pi_{\n_t}\left(\D_{\frac{d}{dt}}\vec{\I}(\p{x_i}\phi_t,\p{x_j}\phi_t)\right)=\D^2_{\p{x_i}\phi_t,\p{x_j}\phi_t}\w_t+\vec{\I}(\p{x_i}\phi_t,\D_{\p{x_j}\phi_t}^{\top}\w_t)+\pi_{\n_t}\left(R(\w_t,\p{x_i}\phi_t)\p{x_j}\phi_t\right).
\end{align}

\subsubsection{Codazzi Identity}

Now, using the definition of the covariant derivative, we get for all tangent vector fields 
\begin{align*}
	\D_{X}^{N}\vec{\I}(Y,Z)&=\D^{N}_X\left(\vec{\I}(Y,Z)\right)-\vec{\I}(\D^{\top}_XY,Z)-\vec{\I}(Y,\D^{\top}_XZ)
\end{align*}
while the Codazzi-Mainardi identity gives
\begin{align*}
	\D_{X}^{N}\vec{\I}(Y,Z)&=\D^{\perp}_Y\vec{\I}(X,Z)+\pi_{\n_t}\left(R(X,Y)Z\right).
\end{align*}

\subsubsection{First Derivative of the Mean Curvature}

We have
\begin{align*}
    \H_{g_t}=\frac{1}{2}\sum_{i,j=1}^2g^{i,j}\,\vec{\I}(\p{x_i}\phi_t,\p{x_j}\phi_t).
\end{align*}
Therefore, we get
\begin{align*}
    &\pi_{\n_t}\left(\D_{\frac{d}{dt}}\H_{g_t}\right)=\frac{1}{2}\sum_{i,j=1}^2\left(-2\s{d\phi_t}{d\w_t}_{g_t}g^{i,j}\right.\\
    &\left.+\frac{(-1)^{i+j}}{\det(g_t)}\left(\s{\D_{\p{x_{i+1}}\phi_t}\w_t}{\p{x_{j+1}}\phi_t}+\s{\D_{\p{x_{j+1}}\phi_t}\w_t}{\p{x_{i+1}}\phi_t}\right)\right)
     \vec{\I}(\p{x_i}\phi_t,\p{x_j}\phi_t)\\
    &
    +\frac{1}{2}\sum_{i,j=1}^2g^{i,j}\left(\D^2_{\p{x_i}\phi_t,\p{x_j}\phi_t}\w_t+\vec{\I}(\p{x_i}\phi_t,\D_{\p{x_j}\phi_t}^{\top}\w_t)+\pi_{\n_t}\left(R(\w_t,\p{x_i}\phi_t)\p{x_j}\phi_t\right)\right)\\
    &=-\s{d\phi_t}{d\w_t}_{g_t}\H_{g_t}+\frac{1}{2}\sum_{i,j=1}^2\frac{(-1)^{i+j}}{\det(g_t)}\left(\s{\D_{\p{x_{i+1}}\phi_t}\w_t}{\p{x_{j+1}}\phi_t}+\s{\D_{\p{x_{j+1}}\phi_t}\w_t}{\p{x_{i+1}}\phi_t}\right)\vec{\I}(\p{x_i}\phi_t,\p{x_j}\phi_t)\\
    &+\frac{1}{2}\Delta_g^N\w_t+\frac{1}{2}\sum_{i,j=1}^2g^{i,j}\,\vec{\I}(\p{x_i}\phi_t,\D_{\p{x_j}\phi_t}^{\top}\w_t)+\pi_{\n_t}\left(\frac{1}{2}\sum_{i,j=1}^2g^{i,j}R(\w,\p{x_i}\phi_t)\p{x_j}\phi_t\right).
\end{align*}
Notice that if $\mathscr{A}$ is the Simons operator, we have
\begin{align}\label{H_compact_first_derivative}
    \pi_{\n_t}\left(\D_{\frac{d}{dt}}\H_{g_t}\right)=\frac{1}{2}\left(\Delta_g^N\w+\mathscr{A}(\w)+\mathscr{R}(\w)\right),
\end{align}
where
\begin{align*}
    \mathscr{R}(\w)=\sum_{i,j=1}^2g^{i,j}R(\w,\p{x_i}\phi_t)\p{x_j}\phi_t.
\end{align*}
is a curvature operator.

\subsection{Second Variation}

\subsubsection{Second Derivative of the Metric}

 Now, noticing that 
\begin{align}\label{nb-1}
	\s{d\phi_t}{d\w_t}_{g_t}d\mathrm{vol}_{g_t}=\frac{d}{dt}\left(d\mathrm{vol}_{g_t}\right),
\end{align}
we will compute the second derivative of the area form. First, since $\D_{\frac{d}{dt}}\frac{d}{dt}\phi_t=0$, and by the symmetry lemma
\begin{align*}
	\D_{\frac{d}{dt}\phi_t}\p{z}\phi_t=\D_{\p{z}\phi_t}\frac{d}{dt}\phi_t,
\end{align*}
we get
\begin{align}\label{nb0}
	\frac{d^2}{dt^2}\left(g_{i,j}\right)_{|t=0}&=\frac{d}{dt}\left(\s{\D_{\p{x_i}\phi_t}\w_t}{\p{x_j}\phi_t}+\s{\D_{\p{x_j}\phi_t}\w_t}{\p{x_i}\phi_t}\right)_{t=0}\nonumber\\
	&=\s{\D_{\w}\D_{\e_i}\w}{\e_j}+2\s{\D_{\e_i}\w}{\D_{\e_j}\w}+\s{\D_{\w}\D_{\e_j}\w}{\e_i}\nonumber\\
	&=2\s{\D_{e_i}\w}{\D_{\e_i}\w}+\s{R(\w,\e_i)\w}{\e_j}+\s{R(\w,\e_j)\w}{\e_i}\nonumber\\
	&=2\s{\D_{\e_i}\w}{\D_{\e_j}\w}-2\s{R(\w,\e_i)\e_j}{\w},
\end{align}
where we wrote $\e_i=\p{x_i}\phi$, and used the symmetry of the curvature tensor. Therefore, we get
\begin{align*}
	&\frac{d^2}{dt^2}\left(\det(g_t)\right)_{|t=0}=\frac{d}{dt}\left(g_{1,1}g_{2,2}-g_{1,2}^2\right)\\
	&=g_{2,2}\left(\frac{d^2}{dt^2}g_{1,1}\right)+g_{1,1}\frac{d^2}{dt^2}\left(g_{2,2}\right)+2\left(\frac{d}{dt}g_{1,1}\right)\left(\frac{d}{dt}g_{2,2}\right)-2\,g_{1,2}\left(\frac{d}{dt}g_{1,2}\right)-2\left(\frac{d}{dt}g_{1,2}\right)^2\\
	&=2\,g_{2,2}\s{\D_{\e_1}\w}{\D_{\e_1}\w}+2\,g_{1,1}\s{\D_{\e_2}\w}{\D_{\e_2}\w}-4\,g_{1,2}\s{\D_{\e_1}\w}{\D_{\e_2}\w}\\
	&-2\,g_{2,2}\s{R(\w,\e_1)\e_1}{\w}-2\,g_{1,1}\s{R(\w,\e_2)\e_2}{\w}-2\,g_{1,2}\s{R(\w,\e_1)\e_2}{\w}-2\,g_{2,1}\s{R(\w,\e_2)\e_1}{\w}\\
	&+8\s{\D_{\e_1}\w}{\e_1}\s{\D_{\e_2}\w}{\e_2}-2\left(\s{\D_{\e_1}\w}{\e_2}+\s{\D_{\e_2}\w}{\e_1}\right)^2\\
	&=\left(2|d\w|_g^2-2\sum_{i,j=1}^{2}g^{i,j}\s{R(\w,\e_i)\e_j}{\w}\right.\\
 &\left.+\frac{1}{\det(g)}\left(8\s{\D_{\e_1}\w}{\e_1}\s{\D_{\e_2}\w}{\e_2}-2\left(\s{\D_{\e_1}\w}{\e_2}+\s{\D_{\e_2}\w}{\e_1}\right)^2\right)\right)\det(g)
\end{align*}
We see that the term 
\begin{align*}
	\mathscr{R}(\w,\w)=\sum_{i,j=1}^{2}g^{i,j}\s{R(\w,\e_i)\e_j}{\w}
\end{align*}
is a Ricci-like curvature, where we take the Ricci curvature of $\varphi(\Sigma)\subset (M^m,h)\subset \R^n$, evaluated at the couple of vector-field $(\w,\w)$ of $M^m$. The second component is better expressed using conformal (of complex) coordinates. 
We recall that the Cauchy-Riemann operators are given by 
\begin{align*}
	\p{z}&=\frac{1}{2}\left(\p{x_1}-i\,\p{x_2}\right)\\
	\p{\z}&=\frac{1}{2}\left(\p{x_1}+i\,\p{x_2}\right).
\end{align*}
Therefore, we deduce that
\begin{align*}
	\p{x_1}&=\p{z}+\p{\z}\\
	\p{x_2}&=i(\p{z}-\p{\z}).
\end{align*}
These identities imply that
\begin{align*}
	&\s{\D_{\e_1}\w}{\e_1}\s{\D_{\e_2}\w}{\e_2}=-\s{\p{z}\w+\p{\z}\w}{\p{z}\phi+\p{\z}\phi}\s{\p{z}\w-\p{\z}\w}{\p{z}\phi-\p{\z}\phi}\\
	&=-4\,\Re\left(\s{\p{z}\phi}{\p{z}\w}\right)^2+4\,\Re\left(\s{\p{z}\phi}{\p{\z}\w}\right)^2
\end{align*}
Likewise, we get
\begin{align*}
	\s{\D_{\e_1}\w}{\e_2}+\s{\D_{\e_2}\w}{\e_1}&=i\s{\p{z}\phi-\p{\z}\phi}{\p{z}\w+\p{\z}\w}+i\s{\p{z}\phi+\p{\z}\phi}{\p{z}\w-\p{\z}\w}\\
	&=2i\left(\s{\p{z}\phi}{\p{z}\w}-\s{\p{\z}\phi}{\p{\z}\w}\right)=-4\,\Im\left(\s{\p{z}\phi}{\p{z}\w}\right).
\end{align*}
Therefore, we finally get
\begin{align*}
	&4\s{{\D_{\e_1}\w}}{\e_1}\s{\D_{\e_2}\w}{\e_2}-\left(\s{\D_{\e_1}\w}{\e_2}+\s{\D_{\e_2}\w}{\e_1}\right)^2\\
	&=-16\,\Re\left(\s{\p{z}\phi}{\p{z}\w}\right)^2+16\,\Re\left(\s{\p{z}\phi}{\p{\z}\w}\right)^2-16\,\Im\left(\s{\p{z}\phi}{\p{z}\w}\right)^2\\
	&=16\,\Re\left(\s{\p{z}\phi}{\p{\z}\w}\right)^2-16|\s{\p{z}\phi}{\p{z}\w}|^2.
\end{align*}
We finally deduce that 
\begin{align}\label{nb1}
	\frac{d^2}{dt^2}\det(g_t)=2\left(|d\w|_g^2+16|\partial\phi\totimes\bar{\partial}\w|^2_{\mathrm{WP}}-16|\partial\phi\totimes\partial\w|_{\mathrm{WP}}^2+\mathscr{R}(\w,\w)\right)\det(g),
\end{align}
where the Weil-Petersson norm of a quadratic differential $\alpha=f(z)dz^2$ is given by 
\begin{align*}
	|\alpha|_{\mathrm{WP}}^2=g^{-2}\otimes \alpha\otimes \bar{\alpha}=e^{-4\lambda}|f(z)|^2.
\end{align*}
Using the formula $(f(u))''=f'(u)u''+f''(u)(u')$, we deduce that 
\begin{align}\label{nb2}
	&\frac{d^2}{dt^2}\left(\sqrt{\det(g_t)}\right)=\frac{1}{2\sqrt{\det(g)}}\frac{d^2}{dt^2}\det(g_t)-\frac{1}{4(\det(g))^{\frac{3}{2}}}\left(\frac{d}{dt}\det(g_t)\right)^2\nonumber\\
	&=\left(|d\w|_g^2+16|\partial\phi\totimes\bar{\partial}\w|^2_{\mathrm{WP}}-16|\partial\phi\totimes\partial\w|_{\mathrm{WP}}^2+\mathscr{R}(\w,\w)\right)\sqrt{\det(g)}-\frac{1}{4(\det(g))^{\frac{3}{2}}}\left(2\s{d\phi}{d\w}_g\det(g)\right)^2\nonumber\\
	&=\left(|d\w|_g^2-\s{d\phi}{d\w}_g^2+16|\partial\phi\totimes\bar{\partial}\w|^2_{\mathrm{WP}}-16|\partial\phi\totimes\partial\w|_{\mathrm{WP}}^2+\mathscr{R}(\w,\w)\right)\sqrt{\det(g)}.
\end{align}
Notice that in the notations of \cite{lower}, we have
\begin{align*}
	16|\partial\phi\totimes\bar{\partial}\w|^2_{\mathrm{WP}}-16|\partial\phi\totimes\partial\w|_{\mathrm{WP}}^2=\frac{1}{2}|d\phi\totimes d\w+d\w\totimes d\phi|_g^2.
\end{align*}
In fact, this expression can be simplified since in a conformal chart $\s{d\phi}{d\w}_g^2=4\,\Re\left(\s{\D_{\e_z}\w}{\e_{\z}}\right)$, which implies that
\begin{align}\label{nb2_bis}
    \frac{d^2}{dt^2}\left(\sqrt{\det(g_t)}\right)_{|t=0}=\left(|d\w|_g^2-16|\partial\phi\totimes\partial\w|_{\mathrm{WP}}^2+\mathscr{R}(\w,\w)\right)\sqrt{\det(g)}.
\end{align}

\subsubsection{Second Derivative of the Mean Curvature}

Recall that
\begin{align}\label{basic1}
	\D_{\frac{d}{dt}}^N\vec{\I}(\p{x_i}\phi_t,\p{x_j}\phi_t)=\D^{N}_{\p{x_i}\phi_t}\D_{\p{x_j}\phi_t}
\w_t-\D^N_{\D^{\top}_{\p{x_i}\phi_t}\p{x_j}\phi_t}\w_t+\pi_{\n_t}\left(R(\w_t,\p{x_i}\phi_t)\p{x_j}\phi_t\right),
\end{align}
and
that
\begin{align}\label{basic2}
	\frac{d}{dt}\left(g^{i,j}\right)_{|t=0}=-e^{-4\lambda}\left(\s{\D_{\e_i}\w}{\e_j}+\s{\D_{\e_j}\w}{\e_i}\right).
\end{align}
We have
\begin{align}\label{new_der0}
	&\D_{\frac{d}{dt}}^N\D^{N}_{\p{x_i}\phi_t}\vec{v}_t=\pi_{\n_t}\left(\D_{\frac{d}{dt}}\D_{\p{x_i}\phi_t}\vec{v}_t\right)-\D_{\frac{d}{dt}}^N\D_{\p{x_i}\phi_t}^{\top}\vec{v}_t\nonumber\\
	&=\D^N_{\p{x_i}\phi_t}\D_{\frac{d}{dt}}\vec{v}_t+\pi_{\n_t}\left(R(\w_t,\p{x_i}\phi_t)\vec{v}_t\right)-\D_{\frac{d}{dt}}^N\left(\left(g^{1,1}\s{\D_{\p{x_i}\phi_t}\vec{v}_t}{\p{x_1}\phi_t}+g^{1,2}\s{\D_{\p{x_i}\phi_t}\vec{v}_t}{\p{x_2}\phi_t}\right)\p{x_1}\phi_t\right.\nonumber\\
	&\left.+\left(g^{1,2}\s{\D_{\p{x_i}\phi_t}\vec{v}_t}{\p{x_1}\phi_t}+g^{2,2}\s{\D_{\p{x_i}\phi_t}\vec{v}_t}{\p{x_2}\phi_t}\right)\p{x_2}\phi_t\right)\nonumber\\
	&=\D^N_{\p{x_i}\phi_t}\D_{\frac{d}{dt}}\vec{v}_t+\pi_{\n_t}\left(R(\w_t,\p{x_i}\phi_t)\vec{v}_t\right)-\left(g^{1,1}\s{\D_{\p{x_i}\phi_t}\vec{v}_t}{\p{x_1}\phi_t}+g^{1,2}\s{\D_{\p{x_i}\phi_t}\vec{v}_t}{\p{x_2}\phi_t}\right)\D_{\p{x_1}\phi_t}^N\w_t\nonumber\\
	&-\left(g^{1,2}\s{\D_{\p{x_i}\phi_t}\vec{v}_t}{\p{x_1}\phi_t}+g^{2,2}\s{\D_{\p{x_i}\phi_t}\vec{v}_t}{\p{x_2}\phi_t}\right)\D_{\p{x_2}\phi_t}^N\w_t.
\end{align}
Therefore, applying the identity \eqref{new_der0} to $\w_t=\D_{\p{x_j}\phi_t}\w_t$ we deduce that
\begin{align}\label{new_der1}
	&\D_{\frac{d}{dt}}^N\D_{\p{x_i}\phi_t}^N\D_{\p{x_j}\phi_t}\w_t=\D^N_{\p{x_i}\phi_t}\D_{\frac{d}{dt}}^N\D_{\p{x_j}\phi_t}\w_t+\pi_{\n}\left(R(\w,\e_i)\D_{\e_j}\w\right)\nonumber\\
	&-e^{-2\lambda}\s{\D_{\e_i}\D_{\e_j}\w}{\e_1}\D_{\e_1}^N\w-e^{-2\lambda}\s{\D_{\e_i}\D_{\e_j}\w}{\e_2}\D_{\e_2}^N\w\nonumber\\
	&=\D_{\p{x_i}\phi_t}^N\left(R(\w_t,\p{x_j}\phi_t)\w_t\right)+\pi_{\n}\left(R(\w,\e_i)\D_{\e_j}\w\right)
	-e^{-2\lambda}\s{\D_{\e_i}\D_{\e_j}\w}{\e_1}\D_{\e_1}^N\w-e^{-2\lambda}\s{\D_{\e_i}\D_{\e_j}\w}{\e_2}\D_{\e_2}^N\w.
\end{align}
Then, we have
\begin{align}\label{new_der2}
	\D_{\p{x_i}\phi_t}\left(R(\w_t,\p{x_j}\phi_t)\w_t\right)=\D_{\e_i}R(\w,\e_j)\w+R(\D^{\top}_{\e_i}\w,\e_j)\w+R(\w,\D_{\e_i}^{\top}\e_j)\w+R(\w,\e_j)\D_{\e_i}^{\top}\w,
\end{align}
which shows by \eqref{new_der1} and \eqref{new_der2} that
\begin{align}\label{new_der3}
	&\D_{\frac{d}{dt}}^N\left(\D_{\p{x_i}\phi_t}^N\D_{\p{x_j}\phi_t}\w_t\right)_{|t=0}=-e^{-2\lambda}\s{\D_{\e_i}\D_{\e_j}\w}{\e_1}\D_{\e_1}^N\w-e^{-2\lambda}\s{\D_{\e_i}\D_{\e_j}\w}{\e_2}\D_{\e_2}^N\w\nonumber\\
	&+\pi_{\n}\left(\D_{\e_i}R(\w,\e_j)\w+R(\D^{\top}_{\e_i}\w,\e_j)\w+R(\w,\D_{\e_i}^{\top}\e_j)\w+R(\w,\e_j)\D_{\e_i}^{\top}\w+R(\w,\e_i)\D_{\e_j}\w\right)\nonumber\\
 &=-\D_{\D^{\top}\e_i\D_{\e_j}\w}^N\w
 +\pi_{\n}\left(\D_{\e_i}R(\w,\e_j)\w+R(\D^{\top}_{\e_i}\w,\e_j)\w+R(\w,\D_{\e_i}^{\top}\right.\e_j)\w+R(\w,\e_j)\D_{\e_i}^{\top}\w\nonumber\\
 &\left.+R(\w,\e_i)\D_{\e_j}\w\right).
\end{align}
Then, we compute by \eqref{new_der0} since $g^{i,j}_{|t=0}=e^{-4\lambda}\delta_{i,j}$ and using \eqref{basic2}
\begin{align*}
	&\D^N_{\frac{d}{dt}}\D^N_{\D^{\top}_{\p{x_i}\phi_t}\p{x_j}\phi_t}\w_t=\D_{\frac{d}{dt}}^N\left(\left(g^{1,1}\s{\D_{\p{x_i}\phi_t}\p{x_j}\phi_t}{\p{x_1}\phi_t}+g^{1,2}\s{\D_{\p{x_i}\phi_t}\p{x_j}\phi_t}{\p{x_2}\phi_t}\right)\D_{\p{x_1}\phi_t}^N\w_t\right.\\
	&\left.+\left(g^{1,2}\s{\D_{\p{x_i}\phi_t}\p{x_j}\phi_t}{\p{x_1}\phi_t}+g^{2,2}\s{\D_{\p{x_i}\phi_t}\p{x_j}\phi_t}{\p{x_2}\phi_t}\right)\D_{\p{x_2}\phi_t}^N\w_t\right)\\
	&=e^{-2\lambda}\s{\D_{\e_i}\e_j}{\e_1}\left(\pi_{\n}\left(R(\w,\e_1)\w\right)-e^{-2\lambda}\s{\D_{\e_1}\w}{\e_1}\D_{\e_1}^N\w-e^{-2\lambda}\s{\D_{\e_1}\w}{\e_2}\D^{N}_{\e_2}\w\right)\\
	&+e^{-2\lambda}\s{\D_{\e_i}\e_j}{\e_2}\left(\pi_{\n}\left(R(\w,\e_2)\w\right)-e^{-2\lambda}\s{\D_{\e_2}\w}{\e_1}\D_{\e_1}^N\w-e^{-2\lambda}\s{\D_{\e_2}\w}{\e_2}\D_{\e_2}^N\w\right)\\
	&+\left(-2e^{-4\lambda}\s{\D_{\e_1}\w}{\e_1}\s{\D_{\e_i}\e_j}{\e_1}-e^{-4\lambda}\left(\s{\D_{\e_1}\w}{\e_2}+\s{\D_{\e_2}\w}{\e_1}\right)\s{\D_{\e_i}\e_j}{\e_2}\right.\\
	&\left.+e^{-2\lambda}\s{\D_{\frac{d}{dt}}\D_{\p{x_i}\phi_t}\p{x_j}\phi_t}{\p{x_1}\phi_t}+e^{-2\lambda}\s{\D_{\e_i}\e_j}{\D_{\e_1}\w}\right)\D^{N}_{\e_1}\w\\
	&+\left(-2e^{-4\lambda}\s{\D_{\e_2}\w}{\e_2}\s{\D_{\e_i}\e_j}{\e_2}-e^{-4\lambda}\left(\s{\D_{\e_1}\w}{\e_2}+\s{\D_{\e_2}\w}{\e_1}\right)\s{\D_{\e_i}\e_j}{\e_1}\right.\\
	&\left.+e^{-2\lambda}\s{\D_{\frac{d}{dt}}\D_{\p{x_i}\phi_t}\p{x_j}\phi_t}{\p{x_2}\phi_t}+e^{-2\lambda}\s{\D_{\e_i}\e_j}{\D_{\e_2}\w}\right)\D_{\e_2}^N\w
\end{align*}
Since 
\begin{align*}
	\D_{\frac{d}{dt}}\D_{\p{x_i}\phi_t}\p{x_j}\phi_t=\D_{\e_i}\D_{\e_j}\w+R(\w,\e_i)\e_j,
\end{align*}
we deduce that
\begin{align}\label{new_der4}
	&\left(\D_{\frac{d}{dt}}^N\D_{\D_{\p{x_i}\phi_t}^{\top}\p{x_j}\phi_t}^N\w_t\right)_{|t=0}=e^{-2\lambda}\s{\D_{\e_i}\e_j}{\e_1}\left(\pi_{\n}\left(R(\w,\e_1)\w\right)-e^{-2\lambda}\s{\D_{\e_1}\w}{\e_1}\D_{\e_1}^N\w\right.\nonumber\\
    &\left.-e^{-2\lambda}\s{\D_{\e_1}\w}{\e_2}\D^{N}_{\e_2}\w\right)\nonumber\\
	&+e^{-2\lambda}\s{\D_{\e_i}\e_j}{\e_2}\left(\pi_{\n}\left(R(\w,\e_2)\w\right)-e^{-2\lambda}\s{\D_{\e_2}\w}{\e_1}\D_{\e_1}^N\w-e^{-2\lambda}\s{\D_{\e_2}\w}{\e_2}\D_{\e_2}^N\w\right)\nonumber\\
	&+\left(-2e^{-4\lambda}\s{\D_{\e_1}\w}{\e_1}\s{\D_{\e_i}\e_j}{\e_1}-e^{-4\lambda}\left(\s{\D_{\e_1}\w}{\e_2}+\s{\D_{\e_2}\w}{\e_1}\right)\s{\D_{\e_i}\e_j}{\e_2}\right.\nonumber\\
	&\left.+e^{-2\lambda}\s{\D_{\e_i}\D_{\e_j}\w+R(\w,\e_i)\e_j}{\e_1}+e^{-2\lambda}\s{\D_{\e_i}\e_j}{\D_{\e_1}\w}\right)\D^{N}_{\e_1}\w\nonumber\\
	&+\left(-2e^{-4\lambda}\s{\D_{\e_2}\w}{\e_2}\s{\D_{\e_i}\e_j}{\e_2}-e^{-4\lambda}\left(\s{\D_{\e_1}\w}{\e_2}+\s{\D_{\e_2}\w}{\e_1}\right)\s{\D_{\e_i}\e_j}{\e_1}\right.\nonumber\\
	&\left.+e^{-2\lambda}\s{\D_{\e_i}\D_{\e_j}\w+R(\w,\e_i)\e_j}{\e_2}+e^{-2\lambda}\s{\D_{\e_i}\e_j}{\D_{\e_2}\w}\right)\D_{\e_2}^N\w.
\end{align}
We can simplify this expression a lot. Indeed, 
\begin{align}\label{new_der41}
    &e^{-4\lambda}\s{\D_{\e_i}\e_j}{\e_1}\left(\s{\D_{\e_1}\w}{\e_1}\D_{\e_1}^N\w+\s{\D_{\e_1}\w}{\e_2}\D^N_{\e_2}\w\right)\nonumber\\
    &+e^{-4\lambda}\s{\D_{\e_i}\e_j}{\e_2}\left(\s{\D_{\e_2}\w}{\e_1}\D^{N}_{\e_1}\w+\s{\D_{\e_2}\w}{\e_1}\D^N_{\e_2}\w\right)\nonumber\\
    &=e^{-2\lambda}\s{\D_{\e_i}\e_j}{\e_1}\D^N_{\D^{\top}_{\e_1}\w}\w+e^{-2\lambda}\s{\D_{\e_i}\e_j}{\e_2}\D^N_{\D_{\e_2}^{\top}\w}\w=\D^N_{\D_{\D_{\e_i}^{\top}\e_j}^{\top}\w}\w,
\end{align}
while
\begin{align}\label{new_der42}
    e^{-2\lambda}\s{\D_{\e_i}\D_{\e_j}\w}{\e_1}\D^N_{\e_1}\w+e^{-2\lambda}\s{\D_{\e_i}\D_{\e_j}\w}{\e_2}\D_{\e_2}^N\w=\D^N_{\D^{\top}_{\e_i}\D_{\e_j}\w}\w.
\end{align}
Finally, we have
\begin{align}\label{new_der43}
    &e^{-4\lambda}\bigg\{2\s{\D_{\e_1}\w}{\e_1}\s{\D_{\e_i}\e_j}{\e_1}+\left(\s{\D_{\e_1}\w}{\e_2}+\s{\D_{\e_2}\w}{\e_1}\right)\s{\D_{\e_i}\e_j}{\e_2}\bigg\}\D^N_{\e_1}\w\nonumber\\
    &+e^{-4\lambda}\bigg\{2\s{\D_{\e_2}\w}{\e_2}\s{\D_{\e_i}\e_j}{\e_2}+\left(\s{\D_{\e_1}\w}{\e_2}+\s{\D_{\e_2}\w}{\e_1}\right)\s{\D_{\e_i}\e_j}{\e_1}\bigg\}\D^N_{\e_2}\w\nonumber\\
    &=e^{-4\lambda}\s{\D_{\e_i}\e_j}{\e_1}\left(2\s{\D_{\e_1}\w}{\e_1}\D^N_{\e_1}\w+\left(\s{\D_{\e_1}\w}{\e_2}+\s{\D_{\e_2}\w}{\e_1}\right)\D_{\e_2}^N\w\right)\nonumber\\
    &+e^{-4\lambda}\s{\D_{\e_i}\e_j}{\e_2}\left(\left(\s{\D_{\e_1}\w}{\e_2}+\s{\D_{\e_2}\w}{\e_1}\right)\D^N_{\e_1}\w+2\s{\D_{\e_2}\w}{\e_2}\D_{\e_2}^N\w\right)\nonumber\\
    &=e^{-4\lambda}\s{\D_{\e_i}\e_j}{\e_1}\left(e^{2\lambda}\D^N_{\D_{\e_1}^{\top}\w}\w+\s{\D_{\e_1}\w}{\e_1}\D_{\e_1}^{N}\w+\s{\D_{\e_2}\w}{\e_1}\D_{\e_2}^N\w\right)\nonumber\\
    &+e^{-4\lambda}\s{\D_{\e_i}\e_j}{\e_2}\left(e^{2\lambda}\D^N_{\D_{\e_2}^{\top}\w}\w+\s{\D_{\e_1}\w}{\e_2}\D^N_{\e_1}\w+\s{\D_{\e_2}\w}{\e_2}\D_{\e_2}^{\perp}\w\right)\nonumber\\
    &=\D^N_{\D_{\D_{\e_i}^{\top}\e_j}^{\top}\w}\w+e^{-2\lambda}\s{\D_{\e_1}\w}{\D_{\e_i}^{\top}\e_j}\D_{\e_1}^N\w+e^{-2\lambda}\s{\D_{\e_2}\w}{\D_{\e_i}^{\top}\e_j}\D_{\e_2}^{N}\w.
\end{align}
Therefore, we have by \eqref{new_der4}, \eqref{new_der41}, \eqref{new_der42}, and \eqref{new_der43} 
\begin{align}\label{new_der44}
    &\left(\D_{\frac{d}{dt}}^N\D_{\frac{d}{dt}}^N\D^N_{\D^{\top}_{\p{x_i}\phi_t}\p{x_j}\phi_t}\w_t\right)_{|t=0}=-2\,\D^N_{\D_{\D_{\e_i}^{\top}\e_j}^{\top}\w}\w+\D^N_{\D^{\top}\e_i\D_{\e_j}\w}\w+\mathscr{R}_{i,j}^0(\w,d\w)\nonumber\\
    &-e^{-2\lambda}\s{\D_{\e_1}\w}{\D_{\e_i}^{\top}\e_j}\D_{\e_1}^N\w-e^{-2\lambda}\s{\D_{\e_2}\w}{\D_{\e_i}^{\top}\e_j}\D_{\e_2}^N\w
    +e^{-2\lambda}\s{\D_{\e_i}\e_j}{\D_{\e_1}\w}\D_{\e_1}^N\w\nonumber\\
    &+e^{-2\lambda}\s{\D_{\e_i}\e_j}{\D_{\e_2}\w}\D_{\e_2}^{N}\w\nonumber\\
    &=-2\,\D^N_{\D_{\D_{\e_i}^{\top}\e_j}^{\top}\w}\w+\D^N_{\D^{\top}_{\e_i}\D_{\e_i}\w}\w+e^{-2\lambda}\s{\vec{\I}_{i,j}}{\D_{\e_1}^N\w}\D^N_{\e_1}\w+e^{-2\lambda}\s{\vec{\I}_{i,j}}{\D_{\e_2}^N\w}\D_{\e_2}^N\w+\mathscr{R}_{i,j}^0(\w,d\w).
\end{align}
Finally, we have by \eqref{new_der4} and \eqref{new_der44}
\begin{align}\label{new_der5}
	&\D_{\frac{d}{dt}}^N\,\pi_{\n_t}\left(R(\w_t,\p{x_i}\phi_t)\p{x_j}\phi_t\right)=\D_{\w}R(\w,\e_i)\e_j+R(\w,\D_{\e_i}\w)\e_j+R(\w,\e_i)\D_{\e_j}\w\nonumber\\
 &-\D_{\frac{d}{dt}}^N\left(R(\w_t,\p{x_i}\phi_t)\p{x_j}\phi_t\right)\nonumber\\
	&=\D_{\w}R(\w,\e_i)\e_j+R(\w,\D_{\e_i}\w)\e_j+R(\w,\e_i)\D_{\e_j}\w-e^{-2\lambda}\s{R(\w,\e_i)\e_j}{\e_1}\D_{\e_1}^{N}\w\nonumber\\
 &-e^{-2\lambda}\s{R(\w,\e_i)\e_j}{\e_2}\D_{\e_2}^{N}\w.
\end{align}
Thanks to \eqref{new_der5}, \eqref{new_der4}, and \eqref{new_der5}, we deduce that
\begin{align}\label{new_der5_bis}
	&\D_{\frac{d}{dt}}^N\D_{\frac{d}{dt}}^N\vec{\I}(\p{x_i}\phi_t,\p{x_j}\phi_t)=-\D_{\D^{\top}_{\e_i}\D_{\e_j}\w}^N\w-\bigg(-2\,\D^N_{\D_{\D_{\e_i}^{\top}\e_j}^{\top}\w}\w+\D^N_{\D^{\top}\e_i\D_{\e_i}\w}\w\nonumber\\
 &+e^{-2\lambda}\s{\vec{\I}_{i,j}}{\D_{\e_1}^N\w}\D^N_{\e_1}\w+e^{-2\lambda}\s{\vec{\I}_{i,j}}{\D_{\e_2}^N\w}\D_{\e_2}^N\w\bigg)
 +\mathscr{R}_{i,j}(\w,d\w)\nonumber\\
 &=-2\,\D^N_{\D^{\top}_{\e_i}\D_{\e_i}\w}\w+2\,\D^N_{\D^{\top}_{\D^{\top}_{\e_i}\e_j}\w}\w-e^{-2\lambda}\s{\vec{\I}_{i,j}}{\D_{\e_1}^N\w}\D_{\e_1}^N\w-e^{-2\lambda}\s{\vec{\I}_{i,j}}{\D_{\e_2}^N\w}\D^N_{\e_2}\w++\mathscr{R}_{i,j}(\w,d\w).
\end{align}
for some curvature component $\mathscr{R}_{i,j}$. Therefore, we have
\begin{align}\label{new_der6}
    &e^{-2\lambda}\left(\D_{\frac{d}{dt}}^N\D_{\frac{d}{dt}}^N\vec{\I}(\p{x_1}\phi_t,\p{x_1}\phi_t)+\D_{\frac{d}{dt}}^N\D_{\frac{d}{dt}}^N\vec{\I}(\p{x_2}\phi_t,\p{x_2}\phi_t)\right)_{|t=0}=-2\,\D^N_{(\Delta_g\w)^{\top}}\w+2\,\D^N_{\D^{\top}_{(\Delta_g\phi)^{\top}}\w}\w\nonumber\\
    &-\s{\H}{\D_{\e_1}^N\w}\D_{\e_1}^N\w-\s{\H}{\D_{\e_2}^N\w}\D_{\e_2}^N\w\nonumber\\
    &=-2\,\D^N_{(\Delta_g\w)^{\top}}\w-2\,e^{-2\lambda}\s{\H}{\D_{\e_1}^N\w}\D_{\e_1}^N\w-2\,e^{-2\lambda}\s{\H}{\D_{\e_2}^N\w}\D_{\e_2}^N\w
\end{align}
since $2\H=\Delta_g\phi$ is normal.

Now, recall that by \eqref{nb-1}, \eqref{nb0}, and \eqref{nb1} 
\begin{align*}
	&\frac{d}{dt}\left(g_{i,j}\right)_{|t=0}=\s{\D_{\e_i}\w}{\e_j}+\s{\D_{\e_j}\w}{\e_i}\\
	&\frac{d^2}{dt^2}\left(g_{i,j}\right)_{|t=0}=2\s{\D_{\e_i}\w}{\D_{\e_j}\w}-2\s{R(\w,\e_i)\e_j}{\w}\\
	&\frac{d}{dt}\det(g_t)_{|t=0}=2\s{d\phi}{d\w}_ge^{4\lambda}\\
	&\frac{d^2}{dt^2}\det(g_t)_{|t=0}=2\left(|d\w|_g^2+\s{d\phi}{d\w}_g^2-16|\partial\phi\totimes\partial\w|_{\mathrm{WP}}^2+\mathscr{R}(\w,\w)\right)e^{4\lambda}.
\end{align*}
Therefore, we have
\begin{align*}
	\frac{d}{dt}\frac{1}{\det(g_t)}&=-2\s{d\phi}{d\w}_ge^{-4\lambda}\\
	\frac{d^2}{dt^2}\frac{1}{\det(g_t)}&=-2\left(|d\w|_g^2+\s{d\phi}{d\w}_g^2-16|\partial\phi\totimes\partial\w|_{\mathrm{WP}}^2+\mathscr{R}(\w,\w)\right)e^{-4\lambda}+8\,\s{d\phi}{d\w}_g^2e^{-4\lambda}\\
	&=-2\left(|d\w|_g^2-3\s{d\phi}{d\w}_g^2-16|\partial\phi\totimes\partial\w|_{\mathrm{WP}}^2+\mathscr{R}(\w,\w)\right)e^{-4\lambda}
\end{align*}
Therefore, we have
\begin{align*}
	&\frac{d^2}{dt^2}g^{i,j}=\frac{d^2}{dt^2}\left((-1)^{i+j}\frac{g_{i+1,j+1}}{\det(g_t)}\right)=2(-1)^{i+j}e^{-4\lambda}\left(\s{\D_{\e_{i+1}}\w}{\D_{\e_{j+1}}\w}-\s{R(\w,\e_{i+1})\e_{j+1}}{\w}\right)\\
	&-4\s{d\phi}{d\w}_ge^{-4\lambda}(-1)^{i+j}\left(\s{\D_{\e_{i+1}}\w}{\e_{j+1}}+\s{\D_{\e_{j+1}}\w}{\e_{i+1}}\right)\\
	&-2\left(|d\w|_g^2-3\s{d\phi}{d\w}_g^2-16|\partial\phi\totimes\partial\w|_{\mathrm{WP}}^2+\mathscr{R}(\w,\w)\right)e^{-2\lambda}(-1)^{i+j} \delta_{i,j}.
\end{align*}
Since
\begin{align*}
	\D_{\frac{d}{dt}}^N\vec{\I}(\p{x_i}\phi_t,\p{x_j}\phi_t)=\D^2_{\e_i,\e_j}\w+\vec{\I}(\e_i,\D_{\e_j}^{\top}\w)+\pi_{\n}\left(R(\w,\e_i)\e_j\right),
\end{align*}
we get
\begin{align}\label{formula_der2_H}
	&\D_{\frac{d}{dt}}^N\D_{\frac{d}{dt}}^N\sum_{i,j=1}^2g^{i,j}\vec{\I}_{i,j}=e^{-2\lambda}\left(\D^{\perp}_{\frac{d}{dt}}\D_{\frac{d}{dt}}^{\perp}\vec{\I}(\p{x_1}\phi_t,\p{x_1}\phi_t)+\D^{\perp}_{\frac{d}{dt}}\D_{\frac{d}{dt}}^{\perp}\vec{\I}(\p{x_2}\phi_t,\p{x_2}\phi_t)\right)\nonumber\\
	&-2\sum_{i,j=1}^2\left(\s{\D_{\e_i}\w}{\e_j}+\s{\D_{\e_j}\w}{\e_i}\right)\left(\D^2_{\e_i,\e_j}\w+\vec{\I}(\e_i,\D_{\e_j}^{\top}\w)+\pi_{\n}\left(R(\w,\e_i)\e_j\right)\right)\nonumber\\
	&+\sum_{i,j=1}^2\bigg\{2(-1)^{i+j}e^{-4\lambda}\left(\s{\D_{\e_{i+1}}\w}{\D_{\e_{j+1}}\w}-\s{R(\w,\e_{i+1})\e_{j+1}}{\w}\right)\nonumber\\
	&-4\s{d\phi}{d\w}_ge^{-4\lambda}\left(\s{\D_{\e_{i+1}}\w}{\e_{j+1}}+\s{\D_{\e_{j+1}}\w}{\e_{i+1}}\right)\bigg\}\vec{\I}_{i,j}\nonumber\\
	&-2\left(|d\w|_g^2-3\s{d\phi}{d\w}_g^2-16|\partial\phi\totimes\partial\w|_{\mathrm{WP}}^2+\mathscr{R}(\w,\w)\right)e^{-2\lambda}\left(\vec{\I}_{1,1}+\vec{\I}_{2,2}\right).
\end{align}
Since $\H=\dfrac{1}{2}e^{-2\lambda}\left(\vec{\I}_{1,1}+\vec{\I}_{2,2}\right)$, we deduce that
\begin{align}\label{formula_der2_H2}
	&\D_{\frac{d}{dt}}^N\D_{\frac{d}{dt}}^N\sum_{i,j=1}^2g^{i,j}\vec{\I}_{i,j}=e^{-2\lambda}\left(\D^{\perp}_{\frac{d}{dt}}\D_{\frac{d}{dt}}^{\perp}\vec{\I}(\p{x_1}\phi_t,\p{x_1}\phi_t)+\D^{\perp}_{\frac{d}{dt}}\D_{\frac{d}{dt}}^{\perp}\vec{\I}(\p{x_2}\phi_t,\p{x_2}\phi_t)\right)\nonumber\\
	&-2\,e^{-4\lambda}\sum_{i,j=1}^2\left(\s{\D_{\e_i}\w}{\e_j}+\s{\D_{\e_j}\w}{\e_i}\right)\left(\D^2_{\e_i,\e_j}\w+\vec{\I}(\e_i,\D_{\e_j}^{\top}\w)+\pi_{\n}\left(R(\w,\e_i)\e_j\right)\right)\nonumber\\
	&+\sum_{i,j=1}^2\bigg\{2(-1)^{i+j}e^{-4\lambda}\left(\s{\D_{\e_{i+1}}\w}{\D_{\e_{j+1}}\w}-\s{R(\w,\e_{i+1})\e_{j+1}}{\w}\right)\nonumber\\
	&-4\s{d\phi}{d\w}_ge^{-4\lambda}(-1)^{i+j}\left(\s{\D_{\e_{i+1}}\w}{\e_{j+1}}+\s{\D_{\e_{j+1}}\w}{\e_{i+1}}\right)\bigg\}\vec{\I}_{i,j}\nonumber\\
	&-4\left(|d\w|_g^2-3\s{d\phi}{d\w}_g^2-16|\partial\phi\totimes\partial\w|_{\mathrm{WP}}^2+\mathscr{R}(\w,\w)\right)\H.
\end{align}
First compute
\begin{align}\label{lemme_complexe1}
	\left\{\begin{alignedat}{1}
		&\s{\D_{\e_1}\w}{\e_1}=\s{\D_{\e_z}\w+\D_{\e_{\z}}\w}{\e_z+\e_{\z}}=\frac{1}{2}e^{2\lambda}\s{d\phi}{d\w}_g+2\,\Re\left(\partial\phi\totimes\partial\w\right)\\
		&\s{\D_{\e_2}\w}{\e_2}=\s{i(\D_{\e_z}-\D_{\e_{\z}})\w}{i(\e_z-\e_{\z})}=\frac{1}{2}e^{2\lambda}\s{d\phi}{d\w}_g-2\,\Re\left(\partial\phi\totimes\partial\w\right)\\
		&\s{\D_{\e_1}\w}{\e_2}+\s{\D_{\e_2}\w}{\e_1}=i\s{\D_{\e_z}\w+\D_{\e_{\z}}\w}{\e_z-\e_{\z}}+i\s{\D_{\e_{z}}\w-\D_{\e_{\z}}\w}{\e_z+\e_{\z}}\\
		&=2\,i\left(\s{\D_{\e_z}\w}{\e_z}-\s{\D_{\e_{\z}}\w}{\e_{\z}}\right)=-4\,\Im\left(\partial\phi\totimes\partial\w\right).
	\end{alignedat}\right.
\end{align}
On the other hand,
\begin{align}\label{lemme_complexe2}
	\left\{\begin{alignedat}{1}
		&e^{-2\lambda}\vec{\I}_{1,1}=e^{-2\lambda}\vec{\I}(\e_z+\e_{\z},\e_z+\e_{\z})=\H+\,\Re(\H_0)\\
		&e^{-2\lambda}\vec{\I}_{2,2}=\H-\Re\left(\H_0\right)\\
		&e^{-2\lambda}\vec{\I}_{1,2}=-\Im(\H_0).
	\end{alignedat}\right.
\end{align}
Therefore, we have
\begin{align}\label{nder2_1}
	&e^{-4\lambda}\sum_{i,j=1}^2(-1)^{i+j}\left(\s{\D_{\e_{i+1}}\w}{\e_{j+1}}+\s{\D_{\e_{j+1}}\w}{\e_{i+1}}\right)\vec{\I}_{i,j}=\left(\s{d\phi}{d\w}_g-4\,\Re\left(g^{-1}\otimes\partial\phi\totimes\partial\w\right)\right)\nonumber\\
 &\times \left(\H+\Re(\H_0)\right)+\left(\s{d\phi}{d\w}_g+4\,\Re\left(g^{-1}\otimes\partial\phi\totimes\partial\w\right)\right)\left(\H-\Re(\H_0)\right)\nonumber\\
	&-8\,\Im\left(g^{-1}\otimes\partial\phi\totimes\partial\w\right)\Im(\H_0)\nonumber\\
	&=2\s{d\phi}{d\w}_g\H-8\,\Re\left(g^{-1}\otimes\partial\phi\totimes\partial\w\right)\Re(\H_0)-8\,\Im\left(g^{-1}\otimes\partial\phi\totimes\partial\w\right)\Im(\H_0)\nonumber\\
	&=2\s{d\phi}{d\w}_g\H-8\,\s{\partial\phi\totimes\partial\w}{\h_0}_{\mathrm{WP}}.
\end{align}
Then, we have
\begin{align}\label{lemme_complexe3}
	\left\{\begin{alignedat}{1}
		\s{\D_{\e_1}\w}{\D_{\e_1}\w}&=\frac{1}{2}e^{2\lambda}|d\w|_g^2+2\,\Re\left(\partial\w\totimes\partial\w\right)\\
		\s{\D_{\e_2}\w}{\D_{\e_2}\w}&=\frac{1}{2}e^{2\lambda}|d\w|_g^2-2\,\Re\left(\partial\w\totimes\partial\w\right)\\
		\s{\D_{\e_1}\w}{\D_{\e_2}\w}&=-2\,\Im\left(\partial\w\totimes\partial\w\right).
	\end{alignedat}\right.
\end{align}
Therefore, we have
\begin{align}\label{nder2_2}
	&e^{-4\lambda}\sum_{i,j=1}^2(-1)^{i+j}\s{\D_{\e_{i+1}}\w}{\D_{\e_{j+1}}\w}\vec{\I}_{i,j}=\left(\frac{1}{2}|d\w|_g^2-2\,\Re\left(g^{-1}\otimes\partial\w\totimes\partial\w\right)\right)\left(\H+\Re(\H_0)\right)\nonumber\\
	&+\left(\frac{1}{2}|d\w|_g^2+2\,\Re\left(g^{-1}\otimes\partial\w\totimes\partial\w\right)\right)\left(\H-\Re(\H_0)\right)\nonumber\\
	&-4\,\Im(g^{-1}\otimes \partial\w\totimes\partial\w)\Im(\H_0)=|d\w|_g^2\H-4\,\s{\partial\w\totimes\partial\w}{\h_0}_{\mathrm{WP}}.
\end{align}
Then, we have
\begin{align}\label{lemme_complexe4}
	\left\{\begin{alignedat}{1}
		\D^{2}_{\e_1,\e_1}\w&=2\,\Re\left(\D^{2}_{\e_z,\e_{\z}}\w\right)+2\,\Re\left(\D^2_{\e_z,\e_z}\w\right)\\
		&=\frac{1}{2}e^{2\lambda}\Delta_g^N\w+2\,\Re\left(\D^2_{\e_z,\e_z}\w\right)\\
		\D^2_{\e_2,\e_2}\w&=\frac{1}{2}e^{2\lambda}\Delta_g^N\w-2\,\Re\left(\D^2_{\e_z,\e_z}\w\right)\\
		\D^2_{\e_1,\e_2}\w+\D^2_{\e_2,\e_1}\w&=-4\,\Im\left(\D^2_{\e_z,\e_z}\w\right),
	\end{alignedat}\right.
\end{align}
which implies by \eqref{lemme_complexe1} and \eqref{lemme_complexe4} that
\begin{align}\label{nder2_3}
	&e^{-4\lambda}\sum_{i,j=1}^2\left(\s{\D_{\e_i}\w}{\e_j}+\s{\D_{\e_j}\w}{\e_i}\right)\D^2_{\e_i,\e_j}\w\nonumber\\
	&=2\left(\frac{1}{2}\s{d\phi}{d\w}_g+2\,\Re\left(g^{-1}\otimes\partial\phi\totimes\partial\w\right)\right)\left(\frac{1}{2}\Delta_g^N\w+2\,\Re\left(e^{-2\lambda}\D^2_{\e_z,\e_z}\w\right)\right)\nonumber\\
	&+2\left(\frac{1}{2}\s{d\phi}{d\w}_g-2\,\Re\left(g^{-1}\otimes\partial\phi\totimes\partial\w\right)\right)\left(\frac{1}{2}\Delta_g^N\w-2\,\Re\left(e^{-2\lambda}\D^2_{\e_z,\e_z}\w\right)\right)\nonumber\\
	&+16\,\Im\left(g^{-1}\otimes\partial\phi\totimes\partial\w\right)\Im\left(e^{-2\lambda}\D^2_{\e_z,\e_z}\w\right)
	=\s{d\phi}{d\w}_g\Delta_g^N\w+16\bs{\D^2_{\partial,\partial}\w}{\partial\phi\totimes\partial\w}_{\mathrm{WP}}.
\end{align}
Since
\begin{align*}
	&\vec{\I}(\e_1,\D_{\e_2}^{\top}\w)=e^{-2\lambda}\s{\D_{\e_2}\w}{\e_1}\vec{\I}_{1,1}+e^{-2\lambda}\s{\D_{\e_2}\w}{\e_2}\vec{\I}_{1,2}\\
	&\vec{\I}(\e_2,\D_{\e_1}^{\top}\w)=e^{-2\lambda}\s{\D_{\e_1}\w}{\e_1}\vec{\I}_{1,2}+e^{-2\lambda}\s{\D_{\e_1}\w}{\e_2}\vec{\I}_{2,2}
\end{align*}
Then, we compute by \eqref{lemme_complexe1} and \eqref{lemme_complexe2}
\begin{align}\label{nder2_4part1}
	&e^{-4\lambda}\sum_{i,j=1}^2\left(\s{\D_{\e_i}\w}{\e_j}+\s{\D_{\e_j}\w}{\e_i}\right)\vec{\I}(\e_i,\D_{\e_j}^{\top}\w)=2\,e^{-2\lambda}\s{\D_{\e_1}\w}{\e_1}\left(e^{-4\lambda}\s{\D_{\e_1}\w}{\e_1}\vec{\I}_{1,1}\right.\nonumber\\
 &\left.+e^{-4\lambda}\s{\D_{\e_1}\w}{\e_2}\vec{\I}_{1,2}\right)
	+2\,e^{-2\lambda}\s{\D_{\e_2}\w}{\e_2}\left(e^{-4\lambda}\s{\D_{\e_2}\w}{\e_1}\vec{\I}_{1,2}+e^{-4\lambda}\s{\D_{\e_2}\w}{\e_2}\vec{\I}_{2,2}\right)\nonumber\\
	&+e^{-2\lambda}\left(\s{\D_{\e_1}\w}{\e_2}+\s{\D_{\e_2}\w}{\e_1}\right)\left(e^{-4\lambda}\left(\s{\D_{\e_1}\w}{\e_1}+\s{\D_{\e_2}\w}{\e_2}\right)\vec{\I}_{1,2}+e^{-4\lambda}\s{\D_{\e_2}\w}{\e_1}\vec{\I}_{1,1}\right.\nonumber\\
 &\left.+e^{-4\lambda}\s{\D_{\e_1}\w}{\e_2}\vec{\I}_{2,2}\right)\nonumber\\
    &=e^{-4\lambda}\bigg\{2\s{\D_{\e_1}\w}{\e_1}^2+\s{\D_{\e_2}\w}{\e_1}\left(\s{\D_{\e_1}\w}{\e_2}+\s{\D_{\e_2}\w}{\e_1}\right)\bigg\}e^{-2\lambda}\vec{\I}_{1,1}\nonumber\\
    &+e^{-4\lambda}\bigg\{2\s{\D_{\e_2}\w}{\e_2}^2+\s{\D_{\e_1}\w}{\e_2}\left(\s{\D_{\e_1}\w}{\e_2}+\s{\D_{\e_2}\w}{\e_1}\right)\bigg\}e^{-2\lambda}\vec{\I}_{2,2}\nonumber\\
    &+e^{-4\lambda}\bigg\{2\s{\D_{\e_1}\w}{\e_1}\s{\D_{\e_1}\w}{\e_2}+2\s{\D_{\e_2}\w}{\e_2}\s{\D_{\e_2}\w}{\e_1}\nonumber\\
    &+\left(\s{\D_{\e_1}\w}{\e_2}+\s{\D_{\e_2}\w}{\e_1}\right)\left(\s{\D_{\e_1}\w}{\e_1}+\s{\D_{\e_2}\w}{\e_2}\right)\bigg\}e^{-2\lambda}\vec{\I}_{1,2}.
\end{align}
Now, recall that by definition (\cite[$2.147$ p.$115$]{helein}), for all smooth functions $a,b:\R^2\rightarrow \R^n$
\begin{align}\label{jac1}
    \{a,b\}=\p{x_1}a\cdot\p{x_2}b-\p{x_1}a\cdot \p{x_1}b=\D^{\perp}a\cdot \D b,
\end{align}
where $\D^{\perp}=(-\p{x_2},\p{x_1})$. On the other hand, since $\p{z}=\dfrac{1}{2}\left(\p{x_1}-i\,\p{x_2}\right)$, we have
\begin{align}\label{jac2}
    -4\,\Im\left(\p{z}a\cdot \p{z}b\right)=-\Im\left(\left(\p{x_1}a-i\,\p{x_2}a\right)\cdot \left(\p{x_1}b-i\,\p{x_2}b\right)\right)=\p{x_1}a\cdot\p{x_2}b+\p{x_2}a\cdot \p{x_2}b.
\end{align}
Therefore, we get by \eqref{jac1} and \eqref{jac2}
\begin{align}\label{jac3}
\left\{\begin{alignedat}{1}
    &\p{x_1}a\cdot \p{x_2}b=\frac{1}{2}\{a,b\}-2\,\Im\left(\p{z}a\cdot\p{z}b\right)\\
    &\p{x_2}a\cdot\p{x_1}b=-\frac{1}{2}\{a,b\}-2\,\Im\left(\p{z}a\cdot\p{z}b\right).
    \end{alignedat}\right.
\end{align}
Therefore, we deduce by \eqref{jac3} that
\begin{align}\label{jac4}
    \left\{\begin{alignedat}{1}
        &\s{\D_{\e_1}\w}{\e_2}=-\frac{1}{2}\{\phi,\w\}-2\,\Im\left(\partial \phi\totimes\partial \w\right)\\
        &\s{\D_{\e_2}\w}{\e_1}=\frac{1}{2}\{\phi,\w\}-2\,\Im\left(\partial\phi\totimes\partial\w\right).
    \end{alignedat}\right.
\end{align}
Therefore, we have by \eqref{nder2_4part1} and \eqref{jac4}
\begin{align}\label{nder2_4part2}
    &e^{-4\lambda}\bigg\{2\s{\D_{\e_1}\w}{\e_1}^2+\s{\D_{\e_2}\w}{\e_1}\left(\s{\D_{\e_1}\w}{\e_2}+\s{\D_{\e_2}\w}{\e_1}\right)\bigg\}e^{-2\lambda}\vec{\I}_{1,1}\nonumber\\
    &+e^{-4\lambda}\bigg\{2\s{\D_{\e_2}\w}{\e_2}^2+\s{\D_{\e_1}\w}{\e_2}\left(\s{\D_{\e_1}\w}{\e_2}+\s{\D_{\e_2}\w}{\e_1}\right)\bigg\}e^{-2\lambda}\vec{\I}_{2,2}\nonumber\\
    &+e^{-4\lambda}\bigg\{2\s{\D_{\e_1}\w}{\e_1}\s{\D_{\e_1}\w}{\e_2}+2\s{\D_{\e_2}\w}{\e_2}\s{\D_{\e_2}\w}{\e_1}\nonumber\\
    &+\left(\s{\D_{\e_1}\w}{\e_2}+\s{\D_{\e_2}\w}{\e_1}\right)\left(\s{\D_{\e_1}\w}{\e_1}+\s{\D_{\e_2}\w}{\e_2}\right)\bigg\}e^{-2\lambda}\vec{\I}_{1,2}\nonumber\\
    &=\bigg\{2\left(\frac{1}{2}\s{d\phi}{d\w}_g+2\,\Re\left(g^{-1}\otimes\partial\phi\totimes\partial\w\right)\right)^2\nonumber\\
    &-4\left(\frac{1}{2}\{\phi,\w\}_g-2\,\Im\left(g^{-1}\otimes\partial\phi\totimes\partial\w\right)\right)\Im\left(g^{-1}\otimes\partial\phi\totimes\partial\w\right)\bigg\}\left(\H+\Re\left(\H_0\right)\right)\nonumber\\
    &+\bigg\{2\left(\frac{1}{2}\s{d\phi}{d\w}_g-2\,\Re\left(g^{-1}\otimes\partial\phi\totimes\partial\w\right)\right)^2\nonumber\\
    &-4\left(-\frac{1}{2}\{\phi,\w\}_g-2\,\Im\left(g^{-1}\otimes\partial\phi\totimes\partial\w\right)\right)\Im\left(g^{-1}\otimes\partial\phi\totimes\partial\w\right)\bigg\}\left(\H-\Re\left(\H_0\right)\right)\nonumber\\
    &-\bigg\{2\left(\frac{1}{2}\s{d\phi}{d\w}_g+2\,\Re\left(g^{-1}\otimes\partial\phi\totimes\partial\w\right)\right)\left(-\frac{1}{2}\{\phi,\w\}_g-2\,\Im\left(g^{-1}\otimes\partial\phi\totimes\partial\w\right)\right)\nonumber\\
    &+2\left(\frac{1}{2}\s{d\phi}{d\w}_g-2\,\Re\left(g^{-1}\otimes\partial\phi\totimes\partial\w\right)\right)\left(\frac{1}{2}\{\phi,\w\}_g-2\,\Im\left(g^{-1}\otimes\partial\phi\totimes\partial\w\right)\right)\nonumber\\
    &-4\,\s{d\phi}{d\w}_g\Im\left(g^{-1}\otimes\partial\phi\totimes\partial\w\right)\bigg\}\Im(\H_0)\nonumber\\
    &=\left(\s{d\phi}{d\w}_g^2+16|\partial\phi\totimes\partial\w|_{\mathrm{WP}}^2\right)\H\nonumber\\
    &+\bigg\{8\s{d\phi}{d\w}_g\,\Re\left(g^{-1}\otimes\partial\phi\totimes\partial\w\right)-4\{\phi,\w\}_g\,\Im\left(g^{-1}\otimes\partial\phi\totimes\partial\w\right)\bigg\}\Re(\H_0)\nonumber\\
    &+\bigg\{8\s{d\phi}{d\w}_g\,\Im\left(g^{-1}\otimes\partial\phi\totimes\partial\w\right)+4\{\phi,\w\}_g\,\Re\left(g^{-1}\otimes\partial\phi\totimes\partial\w\right)\bigg\}\Im(\H_0)\nonumber\\
    &=\left(\s{d\phi}{d\w}_g^2+16|\partial\phi\totimes\partial\w|_{\mathrm{WP}}^2\right)\H+8\s{d\phi}{d\w}_g\,\s{\partial\phi\totimes\partial\w}{\h_0}_{\mathrm{WP}}-4\,\{\phi,\w\}_g\s{i\,\partial\phi\totimes\partial\w}{\h_0}_{\mathrm{WP}},
\end{align}
where we used the formula
\begin{align*}
    -\Re(a)\cdot\Im(b)+\Im(a)\cdot\Re(b)=\frac{1}{4i}\Big\{-(a+\bar{a})\cdot(b-\bar{b})+(a-\bar{a})\cdot (b+\bar{b})\Big\}=\Im\left(a\cdot\bar{b}\right),
\end{align*}
so that
\begin{align*}
    &-\Im\left(g^{-1}\otimes\partial\phi\totimes\partial\w\right)\Re(\H_0)+\Re\left(g^{-1}\otimes\partial\phi\totimes\partial\w\right)\Im(\H_0)=\Im\left(g^{-2}\otimes \left(\partial\phi\totimes\partial\w\right)\otimes\bar{\h_0}\right)\\
    &=-\Re\left(g^{-2}\otimes \left(i\,\partial\phi\totimes\partial\w\right)\otimes\bar{\h_0}\right)
    =\s{i\,\partial\phi\totimes\partial\w}{\h_0}_{\mathrm{WP}}
\end{align*}
since $\Im(a)=-\,\Re(i\,a)$ for all $a\in \C^n$. Therefore, we deduce by \eqref{nder2_4part1}  and \eqref{nder2_4part2} that
\begin{align}\label{nder2_4}
    &e^{-4\lambda}\sum_{i,j=1}^2\left(\s{\D_{\e_i}\w}{\e_j}+\s{\D_{\e_j}\w}{\e_i}\right)\vec{\I}(\e_i,\D_{\e_j}^{\top}\w)=\left(\s{d\phi}{d\w}_g^2+16|\partial\phi\totimes\partial\w|_{\mathrm{WP}}^2\right)\H\nonumber\\
    &+8\s{d\phi}{d\w}_g\,\s{\partial\phi\totimes\partial\w}{\h_0}_{\mathrm{WP}}-4\,\{\phi,\w\}_g\s{i\,\partial\phi\totimes\partial\w}{\h_0}_{\mathrm{WP}}.
\end{align}
Therefore, we deduce by \eqref{formula_der2_H}, \eqref{new_der6}, \eqref{nder2_1}, \eqref{nder2_2}, \eqref{nder2_3}, and \eqref{nder2_4} that
\begin{align}
    &2\,\left(\D_{\frac{d}{dt}}^N\D_{\frac{d}{dt}}^N\H_{g_t}\right)_{|t=0}=\D_{\frac{d}{dt}}^N\D_{\frac{d}{dt}}^N\left(\sum_{i,j=1}^2g^{i,j}\vec{\I}(\p{x_i}\phi_t,\p{x_j}\phi_t)\right)_{|t=0}\nonumber\\
    &=-2\,\D^N_{(\Delta_g\w)^{\top}}\w-2\,e^{-2\lambda}\s{\H}{\D_{\e_1}^N\w}\D_{\e_1}^N\w-2\,e^{-2\lambda}\s{\H}{\D_{\e_2}^N\w}\D_{\e_2}^N\w\nonumber\\
    &-2\left(\s{d\phi}{d\w}_g\Delta_g^N\w+16\bs{\D^2_{\partial,\partial}\w}{\partial\phi\totimes\partial\w}_{\mathrm{WP}}\right)\nonumber\\
    &-2\left(\left(\s{d\phi}{d\w}_g^2+16|\partial\phi\totimes\partial\w|_{\mathrm{WP}}^2\right)\H
    +8\s{d\phi}{d\w}_g\,\s{\partial\phi\totimes\partial\w}{\h_0}_{\mathrm{WP}}
    -4\,\{\phi,\w\}_g\s{i\,\partial\phi\totimes\partial\w}{\h_0}_{\mathrm{WP}}\right)\nonumber\\
    &+2\left(|d\w|_g^2\H-4\,\s{\partial\w\totimes\partial\w}{\h_0}_{\mathrm{WP}}\right)
    -4\s{d\phi}{d\w}_g\left(2\s{d\phi}{d\w}_g\H-8\,\s{\partial\phi\totimes\partial\w}{\h_0}_{\mathrm{WP}}\right)\nonumber\\
    &-4\left(|d\w|_g^2-3\s{d\phi}{d\w}_g^2-16|\partial\phi\totimes\partial\w|_{\mathrm{WP}}^2\right)\H+\mathscr{R}_2(\w,d\w)\nonumber\\
    &=-2\s{d\phi}{d\w}_g\Delta_g^N\w-32\s{\D^2_{\partial,\partial}\w}{\partial\phi\totimes\partial\w}_{\mathrm{WP}}-2\,\D^N_{(\Delta_g\w)^{\top}}\w\nonumber\\
    &-2\,e^{-2\lambda}\s{\H}{\D_{\e_1}^N\H}\D_{\e_1}^N\w-2\,e^{-2\lambda}\s{\H}{\D_{\e_2}^N\w}\D_{\e_2}^N\w-8\s{\partial\w\totimes\partial\w}{\h_0}_{\mathrm{WP}}+16\s{d\phi}{d\w}_g\s{\partial\phi\totimes\partial\w}{\h_0}_{\mathrm{WP}}\nonumber\\
    &+8\,\{\phi,\w\}_g\s{i\,\partial\phi\totimes\partial\w}{\h_0}_{\mathrm{WP}}+\left(-2|d\w|_g^2+2\s{d\phi}{d\w}_g^2+32|\partial\phi\totimes\partial\w|^2_{\mathrm{WP}}\right)\H+\mathscr{R}_2(\w,d\w)
\end{align}
for some curvature operator $\mathscr{R}$. Therefore, we have
\begin{align}\label{end_result_d21}
    &2\s{\D_{\frac{d}{dt}}^N\D_{\frac{d}{dt}}^N\H_{g_t}}{\H_{g_t}}_{|t=0}=-2\s{d\phi}{d\w}_g\bs{\Delta_g^N\w}{\H}-32\s{\s{\D^2_{\partial,\partial}\w}{\H}}{\partial\phi\totimes\partial\w}_{\mathrm{WP}}\nonumber\\
    &-2\s{\s{d\phi}{\Delta_g\w}}{\s{d\w}{\H}}_g
    -2|\s{d\w}{\H}|_g^2-8\s{\partial\w\totimes\partial\w}{\s{\H}{\h_0}}_{\mathrm{WP}}+16\s{d\phi}{d\w}_g\s{\partial\phi\totimes\partial\w}{\s{\H}{\h_0}}_{\mathrm{WP}}\nonumber\\
    &+8\,\{\phi,\w\}_g\s{i\,\partial\phi\totimes\partial\w}{\s{\H}{\h_0}}_{\mathrm{WP}}
    +\left(-2|d\w|_g^2+2\s{d\phi}{d\w}_g^2+32|\partial\phi\totimes\partial\w|_{\mathrm{WP}}^2\right)|\H|^2+\s{\mathscr{R}_2(\w,d\w)}{\H}.
\end{align}
Since
\begin{align}\label{end_result_d22}
    &2\left(\D_{\frac{d}{dt}}^N\H_{g_t}\right)_{|t=0}=\D_{\frac{d}{dt}}^N\left(\sum_{i,j=1}^2g^{i,j}\vec{\I}(\p{x_i}\phi_t,\p{x_j}\phi_t)\right)\nonumber\\
    &=-e^{-4\lambda}\sum_{i,j=1}^2\left(\s{\D_{\e_i}\w}{\e_j}+\s{\D_{\e_j}\w}{\e_i}\right)\vec{\I}_{i,j}+e^{-2\lambda}\left(\D_{\frac{d}{dt}}^{N}\vec{\I}(\p{x_1}\phi_t,\p{x_1}\phi_t)+\D_{\frac{d}{dt}}^{N}\vec{\I}(\p{x_2}\phi_t,\p{x_2}\phi_t)\right)\nonumber\\
    &=-e^{-4\lambda}\sum_{i,j=1}^2\left(\s{\D_{\e_i}\w}{\e_j}+\s{\D_{\e_j}\w}{\e_i}\right)\vec{\I}_{i,j}+e^{-2\lambda}\left(\D^{2}_{\e_1,\e_1}\w+\vec{\I}(\e_1,\D_{\e_1}^{\top}\w\right)\nonumber\\
    &+\D^2_{\e_2,\e_2}\w+\vec{\I}(\e_2,\D_{\e_2}^{\top}\w)+\pi_{\n}\left(R(\w,\e_1)\e_1+R\left(\w,\e_2)\e_2\right)\right)\nonumber\\
    &=-e^{-4\lambda}\sum_{i,j=1}^2\left(\s{\D_{\e_i}\w}{\e_j}+\s{\D_{\e_j}\w}{\e_i}\right)\vec{\I}_{i,j}+\Delta_g^N\w\nonumber\\
    &+e^{-2\lambda}\left(\s{\D_{\e_1}\w}{\e_1}\vec{\I}_{1,1}+\left(\s{\D_ {\e_1}\w}{\e_2}+\s{\D_{\e_2}\w}{\e_1}\right)\vec{\I}_{1,2}+\s{\D_{\e_2}\w}{\e_2}\vec{\I}_{2,2}\right)+\mathscr{R}(\w)\nonumber\\
    &=\Delta_g^N\w-\frac{1}{2}e^{-4\lambda}\sum_{i,j=1}^2\left(\s{\D_{\e_i}\w}{\e_j}+\s{\D_{\e_j}\w}{\e_i}\right)\vec{\I}_{i,j}+\mathscr{R}(\w)=\Delta_g^N\w+\mathscr{A}(\w)+\mathscr{R}(\w),
\end{align}
where $\mathscr{A}$ is the Simons operator, and where we used the identities:
\begin{align*}
    &g^{i,j}_{|t=0}=e^{-2\lambda}\delta_{i,j}\\
    &\frac{d}{dt}g^{i,j}_{|t=0}=-e^{-4\lambda}\left(\s{\D_{\e_i}\w}{\e_j}+\s{\D_{\e_j}\w}{\e_i}\right)\\
    &\D_{\frac{d}{dt}}^N\vec{\I}(\p{x_i}\phi_t,\p{x_j}\phi_t)_{|t=0}=\D_{\e_i,\e_j}^N+\vec{\I}(\e_i,\D_{\e_j}^{\top}\w)+\pi_{\n}\left(R(\w,\e_i)\e_j\right).
\end{align*}
Notice that thanks to \eqref{lemme_complexe1} and \eqref{lemme_complexe2}
\begin{align*}
    &e^{-4\lambda}\sum_{i,j=1}^2\left(\s{\D_{\e_i}\w}{\e_j}+\s{\D_{\e_j}\w}{\e_i}\right)\vec{\I}_{i,j}=\left(\s{d\phi}{d\w}_g+4\,\Re\left(\partial\phi\totimes\partial\w\right)\right)\left(\H+\Re(\H_0)\right)\\
    &+\left(\s{d\phi}{d\w}_g-4\,\Re\left(\partial\phi\totimes\partial\w\right)\right)\left(\H-\Re(\H_0)\right)+8\,\Im\left(g^{-1}\otimes\partial\phi\totimes\partial\w\right)\Im(\H_0)\\
    &=2\s{d\phi}{d\w}_g\H+8\s{\partial\phi\totimes\partial\w}{\h_0}_{\mathrm{WP}},
\end{align*}
which shows that
\begin{align}\label{simons}
    \mathscr{A}(\w)=-\s{d\phi}{d\w}_g\H-4\s{\partial\phi\totimes\partial\w}{\h_0}_{\mathrm{WP}}.
\end{align}
Thanks to \eqref{end_result_d21} and \eqref{end_result_d22}, we deduce that
\begin{align}\label{end_result_d23}
    &\frac{d^2}{dt^2}|\H_{g_t}|^2_{|t=0}=\left(2\left|\D_{\frac{d}{dt}}^N\H_{g_t}\right|^2+2\bs{\D_{\frac{d}{dt}}^N\D_{\frac{d}{dt}}^N\H_{g_t}}{\H_{g_t}}\right)_{|t=0}\nonumber\\
    &=\frac{1}{2}\left|\Delta_g^N+\mathscr{A}(\w)+\mathscr{R}(\w)\right|^2-2\s{d\phi}{d\w}_g\bs{\Delta_g^N\w}{\H}-32\s{\s{\D^2_{\partial,\partial}\w}{\H}}{\partial\phi\totimes\partial\w}_{\mathrm{WP}}\nonumber\\
    &-2\s{\s{d\phi}{\Delta_g\w}}{\s{d\w}{\H}}_g
    -2|\s{d\w}{\H}|_g^2-8\s{\partial\w\totimes\partial\w}{\s{\H}{\h_0}}_{\mathrm{WP}}+16\s{d\phi}{d\w}_g\s{\partial\phi\totimes\partial\w}{\s{\H}{\h_0}}_{\mathrm{WP}}\nonumber\\
    &+8\,\{\phi,\w\}_g\s{i\,\partial\phi\totimes\partial\w}{\s{\H}{\h_0}}_{\mathrm{WP}}
    +\left(-2|d\w|_g^2+2\s{d\phi}{d\w}_g^2+32|\partial\phi\totimes\partial\w|_{\mathrm{WP}}^2\right)|\H|^2+\s{\mathscr{R}_2(\w,d\w)}{\H}.
\end{align}
Therefore, we finally deduce that
\begin{align}\label{end_result_d24}
    &\frac{d^2}{dt^2}\left(|\H_{g_t}|^2d\mathrm{vol}_{g_t}\right)_{|t=0}=\left(\frac{d^2}{dt^2}|\H_{g_t}|^2\right)^2d\vg+2\,\frac{d}{dt}\left(|\H_{g_t}|^2\right)_{|t=0}\frac{d}{dt}\left(d\mathrm{vol}_{g_t}\right)_{|t=0}+|\H|^2\frac{d^2}{dt^2}\left(d\mathrm{vol}_{g_t}\right)_{|t=0}\nonumber\\
    &=\bigg\{\frac{1}{2}\left|\Delta_g^N\w+\mathscr{A}(\w)+\mathscr{R}(\w)\right|^2-\colorcancel{2\s{d\phi}{d\w}_g\s{\Delta_g^N\w}{\H}}{red}-32\s{\s{\D^2_{\partial,\partial}\w}{\H}}{\partial\phi\totimes\partial\w}_{\mathrm{WP}}\nonumber\\
    &-2\s{\s{d\phi}{\Delta_g\w}}{\s{d\w}{\H}}_g
    -2|\s{d\w}{\H}|_g^2-8\s{\partial\w\totimes\partial\w}{\s{\H}{\h_0}}_{\mathrm{WP}}+16\s{d\phi}{d\w}_g\s{\partial\phi\totimes\partial\w}{\s{\H}{\h_0}}_{\mathrm{WP}}\nonumber\\
    &+8\,\{\phi,\w\}_g\s{i\,\partial\phi\totimes\partial\w}{\s{\H}{\h_0}}_{\mathrm{WP}}
    +\left(-2|d\w|_g^2+2\s{d\phi}{d\w}_g^2+32|\partial\phi\totimes\partial\w|_{\mathrm{WP}}^2\right)|\H|^2+\s{\mathscr{R}_2(\w,d\w)}{\H}\nonumber\\
    &+2\s{d\phi}{d\w}_g\s{\colorcancel{\Delta_g^N\w}{red}+\mathscr{A}(\w)+\mathscr{R}(\w)}{\H}+\left(|d\w|_g^2-16|\partial\phi\totimes\partial\w|_{\mathrm{WP}}^2\right)|\H|^2\bigg\}d\vg.
\end{align}
Thanks to \eqref{simons}, we deduce that
\begin{align*}
    2\s{d\phi}{d\w}_g\s{\mathscr{A}(\w)}{\H}=-2\s{\phi}{d\w}_g^2|\H|^2-8\s{d\phi}{d\w}_g\s{\partial\phi\totimes\partial\w}{\s{\H}{\h_0}}_{\mathrm{WP}},
\end{align*}
which shows that 
\begin{align}\label{end_result_d25}
    &\frac{d^2}{dt^2}\left(|\H_{g_t}|^2d\mathrm{vol}_{g_t}\right)_{|t=0}=\bigg\{\frac{1}{2}\left|\Delta_g^N\w+\mathscr{A}(\w)+\mathscr{R}(\w)\right|^2-32\s{\s{\D^2_{\partial,\partial}\w}{\H}}{\partial\phi\totimes\partial\w}_{\mathrm{WP}}\nonumber\\
    &-2\s{\s{d\phi}{\Delta_g\w}}{\s{d\w}{\H}}_g
    -2|\s{d\w}{\H}|_g^2-8\s{\partial\w\totimes\partial\w}{\s{\H}{\h_0}}_{\mathrm{WP}}+8\s{d\phi}{d\w}_g\s{\partial\phi\totimes\partial\w}{\s{\H}{\h_0}}_{\mathrm{WP}}\nonumber\\
    &+8\,\{\phi,\w\}_g\s{i\,\partial\phi\totimes\partial\w}{\s{\H}{\h_0}}_{\mathrm{WP}}
    -\left(|d\w|_g^2-16|\partial\phi\totimes\partial\w|_{\mathrm{WP}}^2\right)|\H|^2\nonumber\\
    &+\s{\mathscr{R}_2(\w,d\w)+2\s{d\phi}{d\w}_g\,\mathscr{R}(\w)}{\H}\bigg\}d\vg.
\end{align}
Finally, we deduce the following result.
\begin{theorem}
    Let $\phi:\Sigma\rightarrow (M^n,h)$ be a Willmore immersion. Then, for all variation $\w:\Sigma\rightarrow TM^m$, we have
    \begin{align}\label{d2_willmore}
    &    Q_{\phi}(\w)=D^2W(\phi)(\w,\w)=\int_{\Sigma}\bigg\{\frac{1}{2}\left|\Delta_g^N\w+\mathscr{A}(\w)+\mathscr{R}(\w)\right|^2-32\s{\s{\D^2_{\partial,\partial}\w}{\H}}{\partial\phi\totimes\partial\w}_{\mathrm{WP}}\nonumber\\
    &-2\s{\s{d\phi}{\Delta_g\w}}{\s{d\w}{\H}}_g
    -2|\s{d\w}{\H}|_g^2-8\s{\partial\w\totimes\partial\w}{\s{\H}{\h_0}}_{\mathrm{WP}}+8\s{d\phi}{d\w}_g\s{\partial\phi\totimes\partial\w}{\s{\H}{\h_0}}_{\mathrm{WP}}\nonumber\\
    &+8\,\{\phi,\w\}_g\s{i\,\partial\phi\totimes\partial\w}{\s{\H}{\h_0}}_{\mathrm{WP}}
    -\left(|d\w|_g^2-16|\partial\phi\totimes\partial\w|_{\mathrm{WP}}^2\right)|\H|^2\nonumber\\
    &+\s{\mathscr{R}_3(\w,d\w)}{\H}\bigg\}d\vg,
    \end{align}
    where $\mathscr{R}_3$ is a curvature operator such that
    \begin{align*}
        |\mathscr{R}_3(\w,d\w)|\leq C\left(|\D R||\w|^2+|R|\left(|d\w|_g+|\vec{\I}|_g|\w|\right)|\w|\right)
    \end{align*}
    for some universal constant $C<\infty$.
\end{theorem}

From now, on, we assume that $M^n=\R^n$ equipped with its standard metric now. We see that it is not obvious that the second derivative corresponds to the index of a self-adjoint operator. 
Notice that in conformal coordinates, there holds
\begin{align*}
    \D^2_{\partial,\partial}\w=\left(\D_{\p{z}\phi}^N\D_{\p{z}\phi}^N\w-2(\p{z}\lambda)\D_{\p{z}\phi}\w\right)dz^2=e^{2\lambda}\D_{\p{z}\phi}^N\left(e^{-2\lambda}\D_{\p{z}\phi}\w\right)dz^2,
\end{align*}
which implies that
\begin{align*}
    &\int_{\Sigma}\s{\s{\H}{\D^{2}_{\partial,\partial}\w}}{\partial\phi\totimes\partial\w}_{\mathrm{WP}}d\vg=\Re\left(\int_{\Sigma}\bs{\bs{\H}{\D_{\p{z}\phi}\left(e^{-2\lambda}\D_{\p{z}\phi}^N\w\right)}}{\s{\p{\z}\phi}{\D_{\p{\z}\phi}\w}}|dz|^2\right)\\
    &=-\Re\left(\int_{\Sigma}e^{-2\lambda}\s{\D_{\p{z}\phi}^N\w}{\D_{\p{z}\phi}^N\H}\s{\p{\z}\phi}{\D_{\p{\z}\phi}\w}|dz|^2\right)-\Re\left(\s{\D_{\p{z}\phi}\w}{\H}\s{e^{-2\lambda}\D_{\p{z}\phi}\p{\z}\phi}{\D_{\p{\z}\phi}\w}|dz|^2\right)\\
    &-\Re\left(\int_{\Sigma}\s{\D_{\p{z}\phi}^N\w}{\H}\s{\p{\z}\phi}{e^{-2\lambda}\D_{\p{z}\phi}\D_{\p{\z}\phi}\w}|dz|^2\right)\\
    &=-\int_{\Sigma}\s{\partial^N\w\totimes\partial^N\H}{\partial\phi\otimes\partial\w}_{\mathrm{WP}}d\vg-\frac{1}{8}\int_{\Sigma}|\s{d\w}{\H}|_g^2d\vg-\frac{1}{16}\int_{\Sigma}\s{\s{d\w}{\H}}{\s{d\phi}{\Delta_g\w}}_gd\vg.
\end{align*}
Therefore, we have
\begin{align*}
    &\int_{\Sigma}\left(-32\s{\s{\D^2_{\partial,\partial}\w}{\H}}{\partial\phi\totimes\partial\w}_{\mathrm{WP}}
    -2\s{\s{d\phi}{\Delta_g\w}}{\s{d\w}{\H}}_g\right)d\vg=\\
    &32\int_{\Sigma}\s{\partial^N\w\totimes\partial^N\H}{\partial\phi\totimes\partial\w}_{\mathrm{WP}}d\vg
    +4\int_{\Sigma}|\s{d\w}{\H}|_g^2d\vg,
\end{align*}
and the second derivative becomes
\begin{align}
    &Q_{\phi}(\w)=\int_{\Sigma}\bigg\{\frac{1}{2}\left|\Delta_g^N\w+\mathscr{A}(\w)\right|^2+2|\s{d\w}{\H}|_g^2-\left(|d\w|_g^2-16|\partial\phi\totimes\partial\w|_{\mathrm{WP}}^2\right)|\H|^2\nonumber\\
    &+32\s{\partial^N\w\totimes\partial^N\H}{\partial\phi\totimes\partial\w}_{\mathrm{WP}}
    -8\s{\partial\w\totimes\partial\w}{\s{\H}{\h_0}}_{\mathrm{WP}}+8\s{d\phi}{d\w}_g\s{\partial\phi\totimes\partial\w}{\s{\H}{\h_0}}_{\mathrm{WP}}\nonumber\\
    &+8\,\{\phi,\w\}_g\s{i\,\partial\phi\totimes\partial\w}{\s{\H}{\h_0}}_{\mathrm{WP}}
    \bigg\}d\vg.
\end{align}
We deduce that the following simpler expression holds.
\begin{theorem}
    Let $\phi:\Sigma\rightarrow\R^n$ be a smooth Willmore immersion. Then, for all $\vec{w}\in W^{2,2}(\Sigma,\R^n)$, we have
    \begin{align}\label{simpler_d2}
        &Q_{\phi}(\w)=\int_{\Sigma}\bigg\{\frac{1}{2}\left|\Delta_g^N\w+\mathscr{A}(\w)\right|^2+2|\s{d\w}{\H}|_g^2-\left(|d\w|_g^2-16|\partial\phi\totimes\partial\w|_{\mathrm{WP}}^2\right)|\H|^2\nonumber\\
    &+32\s{\partial^N\w\totimes\partial^N\H}{\partial\phi\totimes\partial\w}_{\mathrm{WP}}
    -8\s{\partial\w\totimes\partial\w}{\s{\H}{\h_0}}_{\mathrm{WP}}+8\s{d\phi}{d\w}_g\s{\partial\phi\totimes\partial\w}{\s{\H}{\h_0}}_{\mathrm{WP}}\nonumber\\
    &+8\,\{\phi,\w\}_g\s{i\,\partial\phi\totimes\partial\w}{\s{\H}{\h_0}}_{\mathrm{WP}}
    \bigg\}d\vg.
    \end{align}
    If $\w$ is assume to be normal, then
    \begin{align}\label{simpler_d3}
        &Q_{\phi}(\w)=\int_{\Sigma}\bigg\{\frac{1}{2}\left|\Delta_g^N\w+\mathscr{A}(\w)\right|^2+2|\s{d\w}{\H}|_g^2-\left(|d\w|_g^2-16|\partial\phi\totimes\partial\w|_{\mathrm{WP}}^2\right)|\H|^2\nonumber\\
    &+32\s{\partial^N\w\totimes\partial^N\H}{\partial\phi\totimes\partial\w}_{\mathrm{WP}}
    -8\s{\partial\w\totimes\partial\w}{\s{\H}{\h_0}}_{\mathrm{WP}}+8\s{d\phi}{d\w}_g\s{\partial\phi\totimes\partial\w}{\s{\H}{\h_0}}_{\mathrm{WP}}
    \bigg\}d\vg.
    \end{align}
\end{theorem}

\begin{theorem}\label{admissible_variations}
    Let $\phi:\Sigma\rightarrow \R^n$ be a branched Willmore immersion. Then, a variation $\w\in W^{2,2}(\Sigma,\R^n)$ such that $|d\w|_g\in L^{\infty}(\Sigma,d\mathrm{vol}_{g_0})$ is admissible if and only if
    \begin{align*}
        \Delta_g^N\w\in L^2(\Sigma,d\vg).
    \end{align*}
    Furthermore, for all such variations, the formula \eqref{simpler_d2} holds.
\end{theorem}
\begin{proof}
    Thanks to the analysis of \cite{beriviere}, at a branched point of order $m$, there exists a bounded function $\mu$ such that
    \begin{align*}
        e^{\lambda}=e^{\mu}|z|^{m-1},
    \end{align*}
    and there exists an integer $\alpha\in \Z$ such that $\alpha\leq m-1$ and $\vec{C_0}\in \C^n\setminus\ens{0}$ such that
    \begin{align*}
        \H=\Re\left(\frac{\vec{C}_0}{z^{\alpha-1}}\right)+O(|z|^{1-\alpha}\log|z|).
    \end{align*}
    Thanks to the Codazzi equation, assuming for simplicity that $\mu(0)=0$, we deduce that for some $\vec{A}_1\in \C^n$
    \begin{align*}
        \h_0=-\frac{(\alpha-1)}{2m}\vec{C}_0\,z^{m-\alpha-1}\z^m+\vec{A}_1z^{m-1}+O(|z|^{m}\log|z|).
    \end{align*}
    In particular, we have
    \begin{align*}
        |\s{\partial^N\w\totimes\partial\H}{\partial\phi\totimes\partial\w}_{\mathrm{WP}}e^{2\lambda}|\leq Ce^{\lambda}|\D\H|\leq \frac{C}{|z|}
    \end{align*}
    which is integrable, while the other terms can all be trivially bounded but
    \begin{align*}
        \frac{1}{2}\int_{\Sigma}\left|\Delta_g^N\w+\mathscr{A}(\w)\right|^2d\vg.
    \end{align*}
    Since the identity \eqref{simons} implies that
    \begin{align*}
        \int_{\Sigma}|\mathscr{A}(\w)|^2d\vg\leq C\int_{\Sigma}|A|^2|d\w|_g^2d\vg<\infty,
    \end{align*}
    the conclusion of the theorem follows.
\end{proof}

\subsection{Formula for normal variations}

For a normal variation, we have
\begin{align*}
    \mathscr{A}(\w)=e^{-4\lambda}\sum_{i,j=1}^2\s{\vec{\I}_{i,j}}{\w}\vec{\I}_{i,j}, 
\end{align*}
and
\begin{align*}
    &\s{d\phi}{d\w}_g=-2\s{\H}{\w}\\
    &\partial\phi\totimes\partial\w=-\frac{1}{2}\s{\h_0}{\w}\\
    &\{\phi,\w\}=\D^{\perp}\phi\cdot\D\w=-\p{x_2}\phi\cdot\D_{\e_1}\w+\p{x_1}\phi\cdot\D_{\e_2}\w=\s{\vec{\I}_{1,2}}{\w}-\s{\vec{\I}_{2,1}}{\w}=0.
\end{align*}
Therefore, the second derivative becomes
\begin{align*}
    &Q_{\phi}(\w)=\frac{1}{2}\int_{\Sigma}\left|\Delta_g^N\w+2\s{\w}{\H}\H+2\s{\s{\w}{\h_0}}{\h_0}_{\mathrm{WP}}\right|^2d\vg+2\int_{\Sigma}|\s{d\w}{\H}|_g^2d\vg\nonumber\\
    &-\int_{\Sigma}\left(|d\w|_g^2
    -4|\s{\h_0}{\w}|_{\mathrm{WP}}^2\right)|\H|^2d\vg
    -16\int_{\Sigma}\s{\partial^N\w\totimes\partial^N\H}{\s{\h_0}{\w}}_{\mathrm{WP}}d\vg\nonumber\\
    &-8\int_{\Sigma}\s{\partial\w\totimes\partial\w}{\s{\H}{\h_0}}_{\mathrm{WP}}d\vg
    +8\int_{\Sigma}\s{\H}{\w}\s{\s{\h_0}{\w}}{\s{\H}{\h_0}}_{\mathrm{WP}}d\vg.
\end{align*}
The first three components (and the fifth component) obviously induce a self-adjoint operator. 
We have
\begin{align*}
    &\int_{\Sigma}\s{\partial^N\w\totimes\partial^N\H}{\s{\h_0}{\w}}_{\mathrm{WP}}d\vg=\Re\int_{\Sigma}e^{-2\lambda}\bs{\bs{\D_{\p{z}\phi}\w}{\D_{\p{z}\phi}^N\H}}{\bs{\bar{\h_0}}{\w}}|dz|^2\\
    &=-\Re\int_{\Sigma}e^{-2\lambda}\bs{\bs{\w}{\left(\D^N_{\p{z}\phi}\D_{\p{z}\phi}^N-2(\p{z}\lambda)\D_{\p{z}\phi}^N\right)\H}}{\bs{\bar{\h_0}}{\w}}|dz|^2\\
    &-\Re\int_{\Sigma}e^{-2\lambda}\bs{\bs{\w}{\D_{\p{z}\phi}^N\H}}{\bs{\bar{\D_{\p{\z}\phi}^N\h_0}}{\w}}|dz|^2-\Re\int_{\Sigma}e^{-2\lambda}\bs{\bs{\w}{\D_{\p{z}\phi}^N\H}}{\bs{\bar{\h_0}}{\D_{\p{z}}^N\w}}|dz|^2\\
    &=-\int_{\Sigma}\s{\s{\D^2_{\partial,\partial}\H}{\w}}{\s{\h_0}{\w}}_{\mathrm{WP}}d\vg-\int_{\Sigma}|\s{\bar{\partial}\h_0}{\w}|_{\mathrm{WP}}^2d\vg-\int_{\Sigma}\s{\s{\partial^N\H}{\w}}{\bar{\partial}^N\w\totimes\h_0}_{\mathrm{WP}}d\vg,
\end{align*}
where we used the Codazzi identity $\bar{\partial}^N\h_0=g\otimes \partial^N\H$. Therefore, we have
\begin{align*}
    &\int_{\Sigma}\s{\partial^N\w\totimes \partial^N\H}{\s{\h_0}{\w}}_{\mathrm{WP}}d\vg=\frac{1}{2}\int_{\Sigma}\s{\partial^N\w\totimes \partial^N\H}{\s{\h_0}{\w}}_{\mathrm{WP}}d\vg\nonumber\\
    &-\frac{1}{2}\int_{\Sigma}\s{\s{\partial^N\H}{\w}}{\h_0\totimes \bar{\partial}^N\w}_{\mathrm{WP}}d\vg
    -\frac{1}{2}\int_{\Sigma}\s{\s{\D^2_{\partial,\partial}\H}{\w}}{\s{\h_0}{\w}}_{\mathrm{WP}}-\frac{1}{2}\int_{\Sigma}|\s{\bar{\partial}\h_0}{\w}|_{\mathrm{WP}}^2d\vg.
\end{align*}
Therefore, if we introduce the symmetric bilinear form $B_{\phi}$ on normal function in $W^{2,2}(\Sigma,\R^n)$ defined by 
\begin{align}\label{bilinear}
    &B_{\phi}(\vec{v},\vec{w})=\frac{1}{2}\int_{\Sigma}\s{\leb_g\vec{v}}{\leb_g\vec{w}}d\vg+2\int_{\Sigma}\s{\s{d\vec{v}}{\H}}{\s{d\vec{w}}{\H}}_gd\vg\nonumber\\
    &-\int_{\Sigma}\left(\s{d\vec{v}}{d\w}_g-4\s{\s{\h_0}{\vec{v}}}{\s{\h_0}{\vec{w}}}_{\mathrm{WP}}\right)|\H|^2d\vg-8\int_{\Sigma}\s{\partial^N\vec{v}\totimes\partial^N\H}{\s{\h_0}{\w}}_{\mathrm{WP}}d\vg\nonumber\\
    &-8\int_{\Sigma}\s{\partial^N\vec{w}\totimes\partial^N\H}{\s{\h_0}{\vec{v}}}_{\mathrm{WP}}d\vg-8\int_{\Sigma}\s{\partial\vec{v}\totimes\partial\vec{w}}{\s{\H}{\h_0}}_{\mathrm{WP}}d\vg\nonumber\\
    &+4\int_{\Sigma}\s{\H}{\vec{v}}\s{\s{\h_0}{\vec{w}}}{\s{\H}{\h_0}}_{\mathrm{WP}}d\vg+4\int_{\Sigma}\s{\H}{\vec{w}}\s{\s{\h_0}{\vec{v}}}{\s{\H}{\h_0}}_{\mathrm{WP}}d\vg,
\end{align}
then
\begin{align}\label{symmetric}
    B_{\phi}(\vec{w},\vec{w})=Q_{\phi}(\vec{w}),
\end{align}
and we deduce that provided that
\begin{align}\label{fourth_order_operator}
    &\mathcal{L}_g=\leb_g^2+4\,d^{\ast_g}\left(\s{d(\,\cdot\,)}{\H}\H\right)-2\,d^{\ast_g}\left(|\H|^2\,d(\,\cdot\,)\right)+8\s{\s{\h_0}{\,\cdot\,}}{\s{\h_0}{\,\cdot\,}}_{\mathrm{WP}}|\H|^2\nonumber\\
    &+16\,\ast_g\, d\,\Re\left(g^{-1}\otimes \s{\h_0}{\,\cdot\,}\otimes\bar{\partial}^N\H\right)
    -16\,\s{\partial^N(\,\cdot\,)\totimes\partial^N\H}{\h_0}_{\mathrm{WP}}+16\,\ast_g\,d\,\Re\left(g^{-1}\otimes\s{\H}{\h_0}\otimes\partial(\,\cdot\,)\right)\nonumber\\
    &+8\s{\s{\H}{\h_0}}{\s{\h_0}{\,\cdot\,}}_{\mathrm{WP}}\H
    +8\s{\H}{\,\cdot\,}\Re\left(g^{-2}\otimes \s{\H}{\bar{\h_0}}\h_0\right)
\end{align}
the following identity holds
\begin{align}\label{diag}
    Q_{\phi}(\vec{w})=\frac{1}{2}\int_{\Sigma}\s{\vec{w}}{\mathcal{L}_g\vec{w}\,}d\vg.
\end{align}
Furthermore, thanks to \eqref{bilinear} and \eqref{symmetric}, we deduce that $\mathcal{L}_g$ is a self-adjoint operator. In codimension $1$, taking $\vec{w}=u\,\n$ for some $u\in W^{2,2}(\Sigma)$, we get
\begin{align*}
    Q_{\phi}(u)=\frac{1}{2}\int_{\Sigma}u\,\mathcal{L}_gu\,d\vg,
\end{align*}
where
\begin{align}\label{operator1_normal}
    \mathcal{L}_g&=\leb_g^2+2\,d^{\ast_g}\left(H^2\,d(\,\cdot\,)\right)+16\,\ast_g\,d\,\Re\left(g^{-1}\otimes h_0\otimes \partial H\right)-16\s{\partial(\,\cdot\,)\otimes\partial H}{h_0}_{\mathrm{WP}}\nonumber\\
    &+16\,\ast_g\,d\,\Re\left(H\,g^{-1}\otimes h_0\otimes \partial(\,\cdot\,)\right)
    +24|H|^2|h_0|_{\mathrm{WP}}^2.
\end{align}
We can simplify this formula to
\begin{align}\label{operator2_normal}
\mathcal{L}_g&=\leb_g^2+2\,d^{\ast_g}\left(H^2\,d(\,\cdot\,)\right)+16\,\ast_g\,d\,\Re\left(g^{-1}\otimes h_0\otimes \partial \left(H\,\cdot\,\right)\right)-16\s{\partial(\,\cdot\,)\otimes\partial H}{h_0}_{\mathrm{WP}}
    +24|\H|^2|\h_0|_{\mathrm{WP}}^2.
\end{align}
Using the formula $\leb_g=\Delta_g+|A|^2$, we can also rewrite this operator as
\begin{align}\label{operator3_normal}
    \mathcal{L}_g&=\Delta_g^2+|A|^2\Delta_g+2\s{d|A|^2}{d(\,\cdot)}_g+2\,d^{\ast_g}\left(H^2\,d(\,\cdot\,)\right)+\left(|A|^4+\Delta_g|A|^2+24|H|^2|h_0|_{\mathrm{WP}}^2\right)\nonumber\\
    &+16\,\ast_g\,d\,\Re\left(g^{-1}\otimes h_0\otimes \partial \left(H\,\cdot\,\right)\right)-16\s{\partial(\,\cdot\,)\otimes\partial H}{h_0}_{\mathrm{WP}}.
\end{align}

\subsection{The Definition of the Morse Index}

\begin{defi}
Let $\Sigma$ be a closed Riemann surface, and $\phi:\Sigma\rightarrow \R^n$ be a branched Willmore immersion. Define the following Sobolev space
    \begin{align*}
        \mathscr{E}(\phi)=W^{2,2}(\Sigma,\R^n)\cap\ens{\vec{w}:|d\w|_g\in L^{\infty}(\Sigma),\;\, \leb_g\vec{w}\in L^2(\Sigma,d\vg)}.
    \end{align*}
    Define the index of $\phi$ by
    \begin{align*}
        \mathrm{Ind}_W(\phi)=\sup\ens{\dim(V): V\subset \mathscr{E}(\phi)\;\,\text{such that}\;\, Q_{\phi}(\vec{v})<0\;\,\text{for all}\;\, \vec{v}\in V\setminus\ens{0}},
    \end{align*}
    and the nullity of $\phi$ by 
    \begin{align*}
        \mathrm{Null}_W(\phi)=\sup\ens{\dim(V):V\subset \mathscr{E}(\phi)\;\, \text{such that}\;\, {Q_{\phi}}_{|V}=0}.
    \end{align*}
    Finally, the space of normal variations is defined by
    \begin{align*}
        \mathrm{Var}^N(\phi)=\mathscr{E}(\phi)\cap\ens{\w:\s{\w}{d\phi}=0}.
    \end{align*}
\end{defi}

\begin{rem}
    Thanks to Theorem \ref{admissible_variations}, we deduce that 
    \begin{align*}
        \mathscr{E}(\phi)=W^{2,2}(\Sigma,\R^n)\cap\ens{\vec{w}:|d\w|_g\in L^{\infty}(\Sigma),\;\, \Delta_g^N\vec{w}\in L^2(\Sigma,d\vg)}.
    \end{align*}
    Furthermore, notice that
    \begin{align*}
        \D_{\p{\z}}^N\w=\D_{\p{\z}}\w-\D^{\top}_{\p{\z}}\w=\D_{\p{\z}}\w-2e^{-2\lambda}\s{\D_{\p{\z}}\w}{\e_{\z}}\e_z-2e^{-2\lambda}\s{\D_{\p{\z}}\w}{\e_z}\e_{\z},
    \end{align*}
    which implies that
    \begin{align*}
        \Delta_g^N\w&=4\,e^{-2\lambda}\,\Re\left(\D_{\p{z}}^N\D_{\p{\z}}^N\w\right)=(\Delta_g\w)^N-8\,e^{-4\lambda}\Re\left(\s{\D_{\p{\z}}\w}{\e_{\z}}\vec{\I}(\e_z,\e_z)+\s{\D_{\p{\z}}\w}{\e_z}\vec{\I}(\e_z,\e_{\z})\right)\\
        &=(\Delta_g\w)^N-4\s{\partial\phi\totimes\partial\w}{\h_0}_{\mathrm{WP}}-\s{d\phi}{d\w}_g\H,
    \end{align*}
    which shows that 
    \begin{align*}
        \mathscr{E}(\phi)=W^{2,2}(\Sigma,\R^n)\cap\ens{\vec{w}:|d\w|_g\in L^{\infty}(\Sigma),\;\, (\Delta_g\vec{w})^N\in L^2(\Sigma,d\vg)}.
    \end{align*}
\end{rem}

\begin{theorem}\label{diagonalisation}
    Let $\Sigma$ be a closed Riemann surface, and $\phi:\Sigma\rightarrow \R^n$ be a smooth Willmore immersion, $g=\phi^{\ast}g_{\R^n}$, and 
    \begin{align*}
        &\mathcal{L}_g=\leb_g^2+4\,d^{\ast_g}\left(\s{d(\,\cdot\,)}{\H}\H\right)-2\,d^{\ast_g}\left(|\H|^2\,d(\,\cdot\,)\right)+8\s{\s{\h_0}{\,\cdot\,}}{\s{\h_0}{\,\cdot\,}}_{\mathrm{WP}}|\H|^2\nonumber\\
    &+16\,\ast_g\, d\,\Re\left(g^{-1}\otimes \s{\h_0}{\,\cdot\,}\otimes\bar{\partial}^N\H\right)
    -16\,\s{\partial^N(\,\cdot\,)\totimes\partial^N\H}{\h_0}_{\mathrm{WP}}+16\,\ast_g\,d\,\Re\left(g^{-1}\otimes\s{\H}{\h_0}\otimes\partial(\,\cdot\,)\right)\nonumber\\
    &+8\s{\s{\H}{\h_0}}{\s{\h_0}{\,\cdot\,}}_{\mathrm{WP}}\H
    +8\s{\H}{\,\cdot\,}\Re\left(g^{-2}\otimes \s{\H}{\bar{\h_0}}\h_0\right)
    \end{align*}
    be the operator acting on normal variations $\w \in \mathrm{Var}^N(\phi)$. For all $\lambda\in \R$, denote by $\mathscr{E}(\lambda)\subset \mathrm{Var}^N(\phi)\subset W^{2,2}(\Sigma,\R^n)$ the space of eigenfunctions associated to $\lambda$. Then, we define the index of $\phi$ relative to $\mathcal{L}$ by
    \begin{align*}
        \mathrm{Ind}_W^{\mathcal{L}}(\phi)=\dim\bigoplus_{\lambda<0}\mathscr{E}(\lambda),
    \end{align*}
    and
    \begin{align*}
        \mathrm{Null}_W^{\mathcal{L}}(\phi)=\mathrm{Ker}(\mathscr{L}).
    \end{align*}
    Then, the following identities hold true:
    \begin{align}\label{eigen_bound}
        \left\{\begin{alignedat}{1}
           \mathrm{Ind}_W(\phi)&=\mathrm{Ind}_W^{\mathcal{L}}(\phi)<\infty\\
           \mathrm{Null}_W(\phi)&=\mathrm{Null}_W^{\mathcal{L}}(\phi)<\infty.
        \end{alignedat}\right.
    \end{align}
\end{theorem}
\begin{proof}
    \textbf{Case 1: The Codimension $1$ Case.}
    
    \textbf{Step 1: Spectral decomposition of $\mathcal{L}$.}

    Let $f\in W^{2,2}(\Sigma)$ and consider the following equation on
    \begin{align*}
        \mathcal{L}u=f.
    \end{align*}
    Consider for all $\lambda>0$ the minimisation problem
    \begin{align}\label{min_eigenvalue}
        \inf\ens{Q_{\lambda}(u)=Q_{\phi}(u)+\frac{\lambda}{2}\int_{\Sigma}u^2\,d\vg-\frac{1}{2}\int_{\Sigma}f\,u\,d\vg: u\in W^{2,2}(\Sigma)}.
    \end{align}
    We have 
    \begin{align*}
        \partial\w=\partial(u\,\n)=\partial u\,\n+u\,\p{z}\n\,dz,
    \end{align*}
    which shows that
    \begin{align*}
        \partial\w\totimes\partial\w=(\partial u)^2+u^2\s{\p{z}\n}{\p{z}\n}dz^2.
    \end{align*}
    We have
    \begin{align*}
        \p{z}\n=2e^{-2\lambda}\s{\p{z}\n}{\e_{\z}}\e_z+2e^{-2\lambda}\s{\p{z}\n}{\e_z}\e_{\z}=-H\,\e_z-H_0\e_{\z}
    \end{align*}
    so that
    \begin{align*}
        \s{\p{z}\n}{\p{z}\n}=e^{2\lambda}H\,H_0,
    \end{align*}
    and
    \begin{align*}
        \partial\w\totimes\partial\w=(\partial u)^2+u^2\,H\,h_0.
    \end{align*}
    Finally, we have
    \begin{align*}
        \int_{\Sigma}\s{\partial \w\totimes\partial\w}{\s{\H}{\h_0}}_{\mathrm{WP}}d\vg=\int_{\Sigma}\s{\partial u\otimes\partial u}{h_0}_{\mathrm{WP}}H\,d\vg+\int_{\Sigma}H^2|h_0|_{\mathrm{WP}}^2u^2d\vg.
    \end{align*}
    Therefore, we get
    \begin{align*}
        Q_{\phi}(u)&=\frac{1}{2}\int_{\Sigma}\left(\Delta_gu+|A|^2u\right)^2d\vg+2\int_{\Sigma}|du|_g^2H^2d\vg
        -\int_{\Sigma}\left(|du|_g^2-4|h_0|_{\mathrm{WP}}^2u^2\right)H^2d\vg\\
        &-16\int_{\Sigma}\s{\partial u\otimes\partial H}{h_0}_{\mathrm{WP}}u\,d\vg
        -8\int_{\Sigma}\s{\partial u\otimes\partial u}{h_0}_{\mathrm{WP}}H\,d\vg-8\int_{\Sigma}H^2|h_0|_{\mathrm{WP}}^2u^2d\vg\\
        &+8\int_{\Sigma}H^2|h_0|_{\mathrm{WP}}^2u^2d\vg\\
        &=\frac{1}{2}\int_{\Sigma}(\Delta_gu+|A|^2u)^2d\vg+\int_{\Sigma}\left(|du|_g^2+4|h_0|_{\mathrm{WP}}^2u^2\right)H^2d\vg-8\int_{\Sigma}\s{\partial u\otimes \partial u}{h_0}_{\mathrm{WP}}H\,d\vg\\
        &-16\int_{\Sigma}\s{\partial u\otimes\partial H}{h_0}_{\mathrm{WP}}u\,d\vg.
    \end{align*}
    Using the inequality 
    \begin{align*}
        (a+b)^2\geq \frac{1}{2}a^2-b^2,
    \end{align*}
    we deduce that
    \begin{align}\label{ineq_d21}
        &Q_{\phi}(u)\geq \frac{1}{4}\int_{\Sigma}(\Delta_gu)^2d\vg+\int_{\Sigma}\left(|du|_g^2H^2+\left(4H^2|h_0|_{\mathrm{WP}}^2-|A|^4\right)u^2\right)\,d\vg\nonumber\\
        &-8\int_{\Sigma}\s{\partial u\otimes \partial u}{h_0}_{\mathrm{WP}}H\,d\vg-16\int_{\Sigma}\s{\partial u\otimes\partial H}{h_0}_{\mathrm{WP}}u\,d\vg\nonumber\\
        &\geq \frac{1}{4}\int_{\Sigma}(\Delta_gu)^2d\vg+\int_{\Sigma}\left(|du|_g^2H^2+\left(4H^2|h_0|_{\mathrm{WP}}^2-|A|^4\right)u^2\right)\,d\vg\nonumber\\
        &-2\int_{\Sigma}|du|_g^2|h_0|_{\mathrm{WP}}|H|d\vg-2\int_{\Sigma}|h_0|_{\mathrm{WP}}^2u^2d\vg-2\int_{\Sigma}|du|_g^2H^2d\vg-16\int_{\Sigma}\s{\partial u\otimes\partial H}{h_0}_{\mathrm{WP}}u\,d\vg\nonumber\\
        &\geq \frac{1}{4}\int_{\Sigma}(\Delta_gu)^2d\vg-2\int_{\Sigma}|du|_g^2|A|^2d\vg-2\int_{\Sigma}|A|^4u^2d\vg-16\int_{\Sigma}\s{\partial u\otimes\partial H}{h_0}_{\mathrm{WP}}u\,d\vg
    \end{align}
    We have
    \begin{align}\label{ineq_d22}
        \int_{\Sigma}|du|_g^2|A|^2d\vg&=-\int_{\Sigma}u\,\Delta_gu\,|A|^2d\vg-\int_{\Sigma}u\s{du}{d|A|^2}_gd\vg=-\int_{\Sigma}u\,\Delta_gu\,|A|^2d\vg\nonumber\\
        &+\frac{1}{2}\int_{\Sigma}\Delta_g|A|^2\,u^2\,d\vg\nonumber\\
        &\leq \frac{1}{16}\int_{\Sigma}(\Delta_gu)^2d\vg+4\int_{\Sigma}|A|^4u^2d\vg+\frac{1}{2}\int_{\Sigma}\Delta_g|A|^2\,u^2d\vg,
    \end{align}
    while 
    \begin{align*}
       \int_{\Sigma}\s{\partial u\otimes\partial H}{h_0}_{\mathrm{WP}}u d\vg&=\int_{\Sigma}e^{-2\lambda}\frac{1}{2}\p{z}(u^2)\p{z}H\bar{h_0}=-\frac{1}{2}\int_{\Sigma}u^2\left(\p{z}(e^{-2\lambda}\p{z}H)\bar{h_0}+|\p{z}H|^2\right)|dz|^2\\
       &=-\frac{1}{2}\int_{\Sigma}u^2\left(\s{\D_{\partial,\partial}^2H}{h_0}_{WP}+\frac{1}{4}|dH|_g^2\right)d\vg
    \end{align*}
    where we used the Codazzi identity $\bar{\partial}h_0=g\otimes\partial H$.
    By \eqref{ineq_d21} and \eqref{ineq_d22}, we get
    \begin{align}\label{ineq_d2_final}
        Q_{\phi}(u)\geq \frac{1}{8}\int_{\Sigma}(\Delta_gu)^2d\vg-\int_{\Sigma}(-8\s{\D^{2}_{\partial,\partial}H}{h_0}_{WP}+\Delta_g|A|^2+10|A|^4)u^2d\vg.
    \end{align}
    Since $\phi$ is smooth, we deduce that $-8\s{\D^{2}_{\partial,\partial}H}{h_0}_{WP}+\Delta_g|A|^2+10|A|^4\in L^{\infty}(\Sigma)$, which implies that for 
    \begin{align*}
        \lambda>2\np{-8\s{\D^{2}_{\partial,\partial}H}{h_0}_{WP}+\Delta_g|A|^2+10|A|^4}{\infty}{\Sigma}, 
    \end{align*}
    the bilinear form is coercive
    \begin{align*}
        B_{\lambda}(u,v)=\frac{1}{2}\int_{\Sigma}u\mathcal{L}_gu\,d\vg+\frac{\lambda}{2}\int_{\Sigma}u\,v\,d\vg,
    \end{align*}
    and we deduce by the Lax-Milgram theorem that the minimisation problem \eqref{min_eigenvalue} admits a unique solution $u$. Furthermore, thanks to the Rellich-Kondrachov theorem, if $\ens{u_k}_{k\in\N}$ is a minimising sequence, then $u_k\hookrightarrow u$ weakly as $k\rightarrow \infty$, which shows that for all $v\in W^{2,2}(\Sigma)$, the following equation holds:
    \begin{align}
        \int_{\Sigma}v\,\mathcal{L}u\,d\vg+\lambda\int_{\Sigma}u\,v\,d\vg-\int_{\Sigma}f\,v\,d\vg=0,
    \end{align}
    which shows that $\mathcal{L}_{\lambda}u=\mathcal{L}u+\lambda u=f$.

    Furthermore, the previous analysis shows that for all $\lambda>4C_1^2+2C_2$, the following set
    \begin{align*}
        L^2(\Sigma,d\vg)\cap\ens{u:Q_{\phi}(u)+\frac{\lambda}{2}\int_{\Sigma}u^2\,d\vg\leq 1}
    \end{align*}
    is a compact subset of $L^2(\Sigma)$. Therefore, $\mathcal{L}_{\lambda}^{-1}$ is a compact operator,  which shows that there exists a sequence $\ens{\mu_k}_{k\in\N}\subset (0,\infty)$ such that $\mu_k\conv{k\rightarrow \infty}0$  and $\sigma(\mathcal{L}_{\lambda})=\ens{0}\cup\ens{\mu_k}_{k\in\N}$ (\cite[Théorème VI.$8$]{brezis}) Furthermore, there exists a Hilbert basis $\ens{e_k}_{k\in\N}$ of $L^2(\Sigma,d\vg)$ made of eigenvectors of $\mathcal{L}$ (\cite[Théorème VI.$11$]{brezis}) Finally, for all $k\in\N$, there exists $u_k\in L^2(\Sigma,d\vg)$ such that
    \begin{align*}
        \mathcal{L}_{\lambda}^{-1}u_k=\mu_k\,u_k,
    \end{align*}
    or
    \begin{align*}
        \mathcal{L}u_k=\left(\frac{1}{\mu_k}-\lambda\right)u_k.
    \end{align*}
    Notice that $u_k\in W^{2,2}(\Sigma)$ by the previous elliptic regularity analysis. 
    Therefore, if $\lambda_k=\mu_k^{-1}-\lambda$, we deduce that $\sigma(\mathcal{L})=\ens{\lambda_k}_{k\in\N}$. Furthermore, since $\ens{e_k}_{k\in\N}$ are eigenvectors of $\mathcal{L}$, we deduce that there exists a Hilbertian base of eigenvectors of $\mathcal{L}$. Since $\lambda_k\conv{k\rightarrow \infty} \infty$, we deduce that \begin{align*}
        \mathrm{Ind}_W^{\mathcal{L}}(\phi)+\mathrm{Null}_W^{\mathcal{L}}(\phi)=\dim\bigoplus_{\lambda\leq 0}\mathcal{E}(\lambda)<\infty.
    \end{align*} 

    \textbf{Step 2: First inequality.}
    
    If $\displaystyle V=\bigoplus_{\lambda\leq 0}\mathcal{E}(\lambda)$, then (up to relabelling) there exists $N\in\N$ such that $\ens{e_k}_{k\geq N}$ is a Hilbertian base of $V$. In particular, for all and $u\in V^{\perp}\setminus\ens{0}$, we have the following expansion in $l^2(\N)$
    \begin{align*}
        u=\sum_{k=N}^{\infty}\s{u}{e_k}e_k,
    \end{align*}
    and
    \begin{align*}
        \mathcal{L}u=\sum_{k=N}^{\infty}\lambda_{\varphi(k)}\s{u}{e_k}e_k,
    \end{align*}
    where $\varphi:\N\rightarrow \N$ and $\lambda_{\varphi(k)}>0$ for all $k\geq N$. Therefore, we have
    \begin{align*}
        Q_{\phi}(u)=\frac{1}{2}\int_{\Sigma}u\,\mathcal{L}u\,d\vg=\frac{1}{2}\sum_{k\in\N}\lambda_{\varphi(k)}\s{u}{e_k}^2>0
    \end{align*}
    since $\lambda_k>0$ for all $k\geq N$, and
    \begin{align*}
        \sum_{k=N}^{\infty}\s{u}{e_k}^2=\sum_{k=0}^{\infty}\s{u}{e_k}^2=\int_{\Sigma}u^2\,d\vg>0.
    \end{align*}
    Therefore, the following inequality holds:
    \begin{align}\label{eigen_bound2}
        \left\{\begin{alignedat}{1}
           \mathrm{Ind}_W(\phi)&\leq \mathrm{Ind}_W^{\mathcal{L}}(\phi)<\infty\\
           \mathrm{Null}_W(\phi)&\leq \mathrm{Null}_W^{\mathcal{L}}(\phi)<\infty.
        \end{alignedat}\right.
    \end{align}

    \textbf{Step 3: Second inequality.}
    
    Now, for all $\lambda<0$ and $u\in \mathcal{E}(\lambda)$, we trivially have
    \begin{align*}
        Q_{\phi}(u)=\lambda\int_{\Sigma}u^2\,d\vg<0,
    \end{align*}
    while $Q_{\phi}(u)=0$ for all $u\in \mathcal{E}(0)$. 
    Therefore, the following inequalities hold:
    \begin{align}\label{eigen_bound3}
        \left\{\begin{alignedat}{1}
           \mathrm{Ind}_W^{\mathcal{L}}(\phi)&\leq \mathrm{Ind}_W(\phi)\\
           \mathrm{Null}_W^{\mathcal{L}}(\phi)&\leq \mathrm{Null}_W(\phi).
        \end{alignedat}\right.
    \end{align}
    Therefore, we have equality in \eqref{eigen_bound2}, which completes the proof of the theorem in the case $n=3$.

    \textbf{Case 2: General Codimension $n-2\geq 1$.} We can identify $\mathrm{Var}^N(\phi)$ to $W^{2,2}(\Sigma,\R^{n-2})$ with the help of a locally controlled trivialisation $\n_1,\cdots,\n_{n-2}$ of $\n$. Indeed, for all $p\in \Sigma$, there exists $r=r(p)>0$ and $\n_1(p),\cdots,\n_{n-2}(p):B(p,r)\rightarrow \R^n$ such that
    \begin{align*}
        \n=\n_1(p)\wedge\cdots\wedge\n_{n-2}(p)\qquad \text{in}\;\, B(p,r),
    \end{align*}
    and
    \begin{align*}
        \np{\D\n_{j}(p)}{2}{B(p,r)}\leq C\np{\D\n}{2}{B(p,r)}
    \end{align*}
    where $C<\infty$ is independent of $p$. Therefore, by compactness of $\n$, using a partition of unity subordinated to $(p,r(p))$, there exists a finite number $N<\infty$ such that
    \begin{align*}
        \n=\sum_{i=1}^N\rho_j\n_{1}(p_1)\wedge\cdot\wedge \n_{n-2}(p_n)
    \end{align*}
    where $\rho_j\in C^{\infty}(\Sigma)$. As it will be apparent in computations, it is not restrictive to assume that $N=1$. Furthermore, to make computations easier, assume that $n=4$. Then, $\n=\n_1\wedge\n_2$, where $\s{\n_1}{\n_2}=0$ and 
    \begin{align*}
        \np{d\n_1}{2}{\Sigma}+\np{d\n_2}{2}{\Sigma}\leq C\np{d\n}{2}{\Sigma}.
    \end{align*}
    If $\w$ is a normal variation, there exists $u_1,u_2\in W^{1,\infty}(\Sigma)$ such that 
    \begin{align*}
        \w=u_1\,\n_1+u_2\,\n_2.
    \end{align*}
    We have
    \begin{align*}
        \Delta_g\w=\Delta_gu_1\,\n_1+\Delta_gu_2\,\n_2+2\s{du_1}{d\n_1}_g+2\s{du_2}{d\n_2}_g+u_1\,\Delta_g\n_1+u_2\,\Delta_g\n_2.
    \end{align*}
    Recalling that
    \begin{align*}
        \Delta_g\n+|d\n|_g^2\n=8\,g^{-1}\otimes \Im\left(\star \left(\bar{\partial}\H\wedge\partial\phi\right)\right)+2\,i\,g^{-2}\otimes \left(\bar{\h_0}\wedge \h_0\right),
    \end{align*}
    we deduce that 
    \begin{align*}
        |\Delta_g\n_i|\leq 2|d\n_g|^2+2|d\H|_g^2,
    \end{align*}
    which shows by smoothness of $\phi$ that $\Delta_g^{\perp}\w\in L^2(\Sigma,g)$ if and only $u=(u_1,u_2)\in W^{2,2}(\Sigma,\R^2)$. Furthermore, 
    \begin{align*}
        \int_{\Sigma}|\leb_g\vec{w}|^2d\vg\geq \frac{1}{2}\int_{\Sigma}|\Delta_gu|^2d\vg-C\int_{\Sigma}\left(|A|^4+|d\H|_g^2\right)|u|^2d\vg,
    \end{align*}
    and we can apply the preceding proof to complete the proof of the theorem in the general setting. 
    \end{proof}

\begin{rem}
    We will study the branched case in \cite{morse_willmore_II}. However, an arbitrary $W^{2,2}(\Sigma)$ variation will not necessarily give an \emph{admissible} variation for which the integral
    \begin{align*}
        Q_{\phi}(u)=\frac{1}{2}\int_{\Sigma}u\,\mathcal{L}u\,d\vg
    \end{align*}
    is finite, since $H\notin L^4(\Sigma,d\vg)$ in general. 
\end{rem}
At this point, the reader must be weary and not want to hear about Willmore surfaces for a few hours. Let us therefore move to the eigenvalue estimates of some elegant fourth-order differential operators.

\section{Refined Pointwise Estimate on the Second Fundamental Form}

\subsection{General Lemmas}

We will need some lemmas from the appendix of \cite{riviere_morse_scs}. First, we have the following result. 
\begin{lemme}\label{lemma_harmonic_pointwise}
    Let $0<a\leq 2\,b<\infty$, $\Omega=B_b\setminus\bar{B}_a(0)$ and $u\in W^{1,2}(\Omega)$ be a harmonic function, and assume that for some $a\leq r\leq b$, we have
    \begin{align*}
        \int_{\partial B_{r}(0)}\partial_{\nu}u\,d\mathscr{H}^1=0.
    \end{align*}
    Then, we have for all $z\in \Omega$
    \begin{align}\label{ineq_harmonic_annulus}
        |\D u(z)|\leq \frac{2}{\sqrt{3\pi}}\left(\frac{b}{b^2-|z|^2}+\frac{a}{|z|^2-a^2}\right)\np{\D u}{2}{\Omega},
    \end{align}
\end{lemme}
\begin{proof}
    Let $d\in \R$ and $\ens{a_k}_{k\in\Z}$ such that 
    \begin{align*}
        u(z)=d\,\log|z|+\Re\left(\sum_{k\in\Z}a_kz^k\right).
    \end{align*}
    By assumption, we have $d=0$. Then, we have
    \begin{align*}
        |\D u|^2=4|\p{z}u|^2=\left|\sum_{k\in\Z^{\ast}}ka_kz^{k-1}\right|^2.
    \end{align*}
    Therefore, we have 
    \begin{align}\label{l2_formula}
        \int_{\Omega}|\D u|^2dx=2\pi\sum_{k\in \Z^{\ast}}k^2|a_k|^2\int_{a}^b\rho^{2k-1}d\rho=\pi\sum_{k\in\Z^{\ast}}|k||a_k|^2\left|b^{2k}-a^{2k}\right|.
    \end{align}
    Then, we get by the Cauchy-Schwarz inequality
    \begin{align*}
        &\left|\sum_{k=0}^{\infty}(k+1)a_{k+1}z^{k}\right|\leq\left(\sum_{k=0}^{\infty}(k+1)|a_{k+1}|^2b^{2k}\right)^{\frac{1}{2}}\left(\sum_{k=0}^{\infty}(k+1)\left(\frac{|z|}{b}\right)^{2k}\right)^{\frac{1}{2}}\\
        &=\frac{1}{1-\left(\frac{|z|}{b}\right)^2}\left(\sum_{k=0}^{\infty}(k+1)|a_k|^2b^{2k}\right)^{\frac{1}{2}}
        \leq \frac{b^2}{b^2-|z|^2}\frac{2}{\sqrt{3\pi}}\frac{1}{b}\np{\D u}{2}{\Omega}=\frac{2}{\sqrt{3\pi}}\frac{b}{b^2-|z|^2}\np{\D u}{2}{\Omega},
    \end{align*}
    where we used that
    \begin{align}\label{ineq_harmonic}
        \int_{\Omega}|\D u|^2dx\geq \pi\sum_{k=0}^{\infty}(k+1)|a_{k+1}|^2b^{2k+2}\left(1-\left(\frac{a}{b}\right)^{2k+2}\right)\geq \frac{3\pi}{4}b^2\sum_{k=0}^{\infty}(k+1)|a_{k+1}|^2b^{2k},
    \end{align}
    where the second inequality comes from the hypothesis $\left(\dfrac{a}{b}\right)^2\leq \dfrac{1}{4}$.
    Likewise, we have 
    \begin{align*}
        &\left|\sum_{k=1}^{\infty}-(k+1)a_{-(k+1)}\frac{1}{z^{k+2}}\right|\leq \left(\sum_{k=0}^{\infty}(k+1)|a_{-(k+1)}|^2\frac{1}{a^{2k+2}}\right)^{\frac{1}{2}}\left(\sum_{k=0}^{\infty}(k+1)\frac{a^{2k+2}}{|z|^{2k+4}}\right)^{\frac{1}{2}}\\
        &=\frac{a}{|z|^2}\frac{1}{1-\left(\frac{a}{|z|}\right)^2}\left(\sum_{k=0}^{\infty}(k+1)|a_{-(k+1)}|^2\frac{1}{a^{2k+2}}\right)^{\frac{1}{2}}
        \leq \frac{2}{\sqrt{3\pi}}\frac{a}{|z|^2-a^2}\np{\D u}{2}{\Omega}
    \end{align*}
    using a similar inequality to \eqref{ineq_harmonic}.
    Finally, we get the estimate
    \begin{align*}
        |\D u(z)|\leq \frac{2}{\sqrt{3\pi}}\left(\frac{b}{b^2-|z|^2}+\frac{a}{|z|^2-a^2}\right)\np{\D u}{2}{\Omega},
    \end{align*}
    which concludes the proof of the lemma.
\end{proof}
\begin{rem}
    This inequality is very similar to the following inequality for holomorphic functions $\varphi:\mathbb{D}\rightarrow \C$:
    \begin{align*}
        \sup_{z\in\mathbb{D}}\,(1-|z|^2)|\varphi(z)|\leq \frac{1}{\sqrt{\pi}}\np{\varphi}{2}{\mathbb{D}}
    \end{align*}
    that is proved in a similar fashion.
\end{rem}

\begin{cor}\label{pointwise_annulus_harmonic}
   Let $0<a<16\,b<\infty$, $\Omega=B_b\setminus\bar{B}_a(0)$, and let $u\in W^{1,2}(\Omega)$ be a harmonic function such that for all some $a\leq r\leq b$, we have
   \begin{align*}
       \int_{\partial B_r(0)}\partial_{\nu}u\,d\mathscr{H}^1=0.
   \end{align*}
   Then, we have for all $4a\leq |z|\leq \dfrac{b}{4}$
   \begin{align}\label{ineq_harmonic_average}
       \left(\int_{B_{2|z|}\setminus\bar{B}_{\frac{|z|}{2}}(0)}|\D u|^2dx\right)^{\frac{1}{2}}\leq \frac{8}{3}\sqrt{\frac{5}{2}}\left(\frac{|z|}{b}+\frac{a}{|z|}\right)\np{\D u}{2}{\Omega}.
   \end{align}
\end{cor}
\begin{proof}
    Since $4\,a\leq |z|\leq \dfrac{b}{4}$ for all $w\in B_{2|z|}\setminus\bar{B}_{\frac{|z|}{2}}(0)$, we deduce by the point-wise bound \eqref{ineq_harmonic_annulus} from Lemma \ref{lemma_harmonic_pointwise} that
    \begin{align*}
        |\D u(w)|&\leq \frac{2}{\sqrt{3\pi}}\left(\frac{b}{b^2-|w|^2}+\frac{a}{|w|^2-a^2}\right)\np{\D u}{2}{\Omega}\\
        &=\frac{2}{\sqrt{3\pi}}\left(\frac{b}{b^2\left(1-\left(\frac{|w|}{b}\right)^2\right)}+\frac{a}{|w|^2\left(1-\left(\frac{a}{|w|}\right)^2\right)}\right)\np{\D u}{2}{\Omega}\\
        &\leq \frac{8}{3\sqrt{3\pi}}\left(\frac{1}{b}+\frac{a}{|w|^2}\right)\np{\D u}{2}{\Omega}.
    \end{align*}
    Therefore, we deduce by Minkowski's inequality that
    \begin{align*}
        \np{\D u}{2}{B_{2|z|}\setminus\bar{B}_{\frac{|z|}{2}}(0)}&\leq \frac{8}{3\sqrt{3\pi}}\np{\D u}{2}{\Omega}\left(\np{\frac{1}{b}}{2}{B_{2|z|}\setminus\bar{B}_{\frac{|z|}{2}}(0)}+\np{\frac{a}{|w|^2}}{2}{B_{2|z|}\setminus\bar{B}_{\frac{|z|}{2}}(0)}\right)\\
        &=\frac{8}{3\sqrt{3\pi}}\np{\D u}{2}{\Omega}\left(\sqrt{\frac{2\pi}{b^2}\left(4|z|^2-\frac{|z|^2}{4}\right)}+\sqrt{\pi\,a^2\left(\frac{1}{\left(\frac{|z|}{2}\right)^2}-\frac{1}{(2|z|)^2}\right)}\right)\\
        &\leq \frac{8}{3\sqrt{3\pi}}\np{\D u}{2}{\Omega}\sqrt{\frac{15\pi}{2}}\left(\frac{|z|}{b}+\frac{a}{|z|}\right)\\
        &=\frac{8}{3}\sqrt{\frac{5}{2}}\left(\frac{|z|}{b}+\frac{a}{|z|}\right)\np{\D u}{2}{\Omega},
    \end{align*}
    which concludes the proof of the Corollary.
\end{proof}

We will now use the following elementary lemma that would follow from the mean-value formula in the case of balls without the extra hypothesis \eqref{no_flux3} on the flux.
\begin{lemme}\label{monotonicity_formula_annulus}
    Let $0<a<b<\infty$, $\Omega=B_b\setminus\bar{B}_a(0)$, and for all $0<t<1$ let $\Omega_t=B_{t\,b}\setminus\bar{B}_{t^{-1}a}(0)$. Let $u\in W^{1,2}(\Omega)$ be a harmonic function such that for some $a\leq r\leq b$
    \begin{align}\label{no_flux3}
        \int_{\partial B_r(0)}\partial_{\nu}u\,d\mathscr{H}^1=0.
    \end{align}
    Then, for all $\left(\dfrac{a}{b}\right)^{\frac{1}{2}}<t\leq 1$, we have
    \begin{align}\label{mono}
        \np{\D u}{2}{\Omega_t}\leq t\np{\D u}{2}{\Omega}.
    \end{align}
\end{lemme}
\begin{proof}
    Thanks to the proof of Lemma \ref{lemma_harmonic_pointwise} and the hypothesis \eqref{no_flux3}, there exists $\ens{a_k}_{k\in\Z}\subset \C$ such that for all $\left(\dfrac{a}{b}\right)^{\frac{1}{2}}<t\leq 1$
    \begin{align}
        \int_{\Omega_t}|\D u|^2dx=\pi\sum_{k\in\Z^{\ast}}|k||a_k|^2\left|(t\,b)^{2k}-(t^{-1}a)^{2k}\right|.
    \end{align}
    We directly estimate
    \begin{align}\label{mono1}
        \sum_{k=1}^{\infty}|k||a_k|^2\left|(t\,b)^{2k}-(t^{-1}a)^{2k}\right|&=\sum_{k=1}^{\infty}|k||a_k|^2\left((t\,b)^{2k}-(t^{-1}a)^{2k}\right)=\sum_{k=1}^{\infty}|k||a_k|^2t^{2k}(b^{2k}-t^{-2k}a^{2k})\nonumber\\
        &\leq t^2\sum_{k=1}^{\infty}|k||a_k|^2\left(b^{2k}-a^{2k}\right)
    \end{align}
    since $t^{-2}\geq 1$, so that $-t^{-2k}a^{2k}\leq -a^{2k}$ for all $k\geq 1$.
    Likewise, we have
    \begin{align}\label{mono2}
        \sum_{k=1}^{\infty}|k||a_{-k}|^2\left|\frac{1}{(tb)^{2k}}-\frac{1}{(t^{-1}a)^{2k}}\right|=\sum_{k=1}^{\infty}|k||a_{-k}|^2\left(\left(\frac{t}{a}\right)^{2k}-\left(\frac{t^{-1}}{b}\right)^{2k}\right)\leq t^2\sum_{k=1}^{\infty}|k||a_{-k}|^2\left(\frac{1}{a^{2k}}-\frac{1}{b^{2k}}\right).
    \end{align}
    Gathering the estimates \eqref{mono1} and \eqref{mono2} shows thanks to the identity \eqref{mono} that
    \begin{align*}
        \int_{\Omega_t}|\D u|^2dx\leq t^2\int_{\Omega}|\D u|^2dx,
    \end{align*}
    which concludes the proof of the lemma.
\end{proof}

Finally, recall the following lemma from \cite{pointwise} (see also \cite{lauriv1}).
\begin{lemme}[Lemma $2.3$, \cite{pointwise}]\label{improved_reg_harmonic}
    Let $0<a<4\,b<\infty$, $\Omega=B_b\setminus\bar{B}_a(0)$, and for all $0<t<1$, let $\Omega_t=B_{t\,b}\setminus\bar{B}_{t^{-1}a}(0)$. Let $u\in W^{1,2}(\Omega)$ be a harmonic function such that for some $a\leq r\leq b$
    \begin{align*}
        \int_{\partial B_r(0)}\partial_{\nu}u\,d\mathscr{H}^1=0.
    \end{align*}
    Then, for all $\left(\dfrac{a}{b}\right)^{\frac{1}{2}}<t<1$, we have
    \begin{align}
        \np{\D u}{2,1}{\Omega_t}+\np{\D^2u}{1}{\Omega_t}\leq 64\sqrt{\frac{\pi}{15}}\frac{t}{1-t}\np{\D u}{2}{\Omega}.
    \end{align}
\end{lemme}

The following extention lemma was used in \cite{quanta,quantamoduli, pointwise} and a less precise version stated in \cite[Lemma $7.2$]{pointwise}.
\begin{lemme}[Lemma C.$1$, \cite{riviere_morse_scs}]\label{whitney_extension}
        Let $\epsilon>0$ and $0<(1+\epsilon)r<R$, $\Omega=B_{R}\setminus\bar{B}_r(0)$. Then, there exists a universal constant $C_0(\epsilon)<\infty$ depending only on $\epsilon>0$ with the following property. For all $u\in W^{1,2}(\Omega)$, there exists an extension $\widetilde{u}\in W^{1,2}(\C)$ of $u$ such that $\mathrm{supp}(\D \widetilde{u})\subset B_{(1+\epsilon)R}\setminus\bar{B}_{(1+\epsilon)^{-1}r}(0)$ and satisfying the estimates
    \begin{align}
    \left\{\begin{alignedat}{2}
        \int_{\C}|\D \widetilde{u}|^2dx&\leq C_0(\epsilon)\int_{\Omega}|\D u|^2dx.\\
        \int_{B_{(1+\epsilon)R}\setminus\bar{B}_{(1+\epsilon)^{-1}R}(0)}|\D \widetilde{u}|^2dx&\leq C_0(\epsilon)\int_{B_{R}\setminus\bar{B}_{(1+\epsilon)^{-1}R}(0)}|\D u|^2dx\\
        \int_{B_{(1+\epsilon)r}\setminus\bar{B}_{(1+\epsilon)^{-1}r}(0)}|\D \widetilde{u}|^2dx&\leq C_0(\epsilon)\int_{B_{(1+\epsilon)r}\setminus\bar{B}_{r}(0)}|\D u|^2dx
    \end{alignedat}\right.
    \end{align}
\end{lemme}

We finish this section with an elementary result.

\begin{theorem}\label{Wente_L21}
    There exists a constant $C_{W(2,1)}<\infty$ with the following property. For all $0<R<\infty$, and for all $a,b\in W^{1,2}(\Omega)$, the solution of the following system
    \begin{align}
    \left\{\begin{alignedat}{2}
        \Delta u&=\D^{\perp}a\cdot \D b\qquad&&\text{in}\;\, B(0,R)\\
        u&=0\qquad&& \text{on}\;\, \partial B(0,R)
        \end{alignedat}\right.
    \end{align}
    satisfies the estimate
    \begin{align}
        \np{\D u}{2,1}{B(0,R)}\leq C_{W(2,1)}\np{\D a}{2}{B(0,R)}\np{\D b}{2}{B(0,R)}.
    \end{align}
\end{theorem}
\begin{proof}
    Let $\widetilde{u}:B(0,1)\rightarrow \R$ such that $\widetilde{u}(x)=u(Rx)$ for all $x\in B(0,1)$. Then, $\widetilde{u}$ solves the equation
    \begin{align}
    \left\{\begin{alignedat}{2}
        \Delta \widetilde{u}&=\D^{\perp}\widetilde{a}\cdot \D \widetilde{b}\qquad&&\text{in}\;\, B(0,1)\\
        \widetilde{u}&=0\qquad&&\text{on}\;\, \partial B(0,1),
        \end{alignedat}\right.
    \end{align}
    where $\widetilde{a},\widetilde{b}\in W^{1,2}(B(0,1))$ are defined similarly to $\widetilde{u}$. 
    Therefore, the classical (\cite[Théorème ($3.4.1$)]{helein}) improved Wente estimate shows that
    \begin{align}\label{scaling1}
        \np{\D \widetilde{u}}{2,1}{B(0,1)}\leq C_{W(2,1)}\np{\D\widetilde{a}}{2}{B(0,1)}\np{\D\widetilde{b}}{2}{B(0,1)}.
    \end{align}
    We have
    \begin{align}\label{scaling2}
    \left\{\begin{alignedat}{1}
        &\int_{B(0,1)}|\D \widetilde{a}|^2dx=\int_{B(0,1)}R^2|\D a(Rx)|^2dx=\int_{B(0,R)}|\D a|^2dx,\\
        &\int_{B(0,1)}|\D\widetilde{b}|^2dx=\int_{B(0,R)}|\D b|^2dx.
        \end{alignedat}\right.
    \end{align}
    Finally, for all $t>0$, we have
    \begin{align*}
        \leb^2\left(B(0,1)\cap\ens{x:|\D \widetilde{u}(x)|>t}\right)&=\int_{B(0,1)}\mathbf{1}_{\ens{|\D u(Rx)|>R^{-1}t}}dx=\frac{1}{R^2}\int_{B(0,R)}\mathbf{1}_{\ens{|\D u(y)|>R^{-1}t}}dy\\
        &=\frac{1}{R^2}\leb^2\left(B(0,R)\cap\ens{x:|\D u(x)|>R^{-1}t}\right).
    \end{align*}
    Therefore, we have
    \begin{align}\label{scaling3}
        \np{\D \widetilde{u}}{2,1}{B(0,1)}&=4\int_{0}^{\infty}\left(\leb^2\left(B(0,1)\cap\ens{x:|\D \widetilde{u}(x)|>t}\right)\right)^{\frac{1}{2}}dt\nonumber\\
        &=4\int_{0}^{\infty}\left(\leb^2\left(B(0,R)\cap\ens{x:|\D u(x)|>R^{-1}t}\right)\right)^{\frac{1}{2}}R^{-1}\,dt\nonumber\\
        &=4\int_{0}^{\infty}\left(\leb^2\left(B(0,R)\cap\ens{x:|\D u(x)|>s}\right)\right)^{\frac{1}{2}}ds=\np{\D u}{2,1}{B(0,R)}.
    \end{align}
    Finally, the estimate follows from \eqref{scaling1}, \eqref{scaling2}, and \eqref{scaling3}. 
\end{proof}

\subsection{Improved Epsilon-Regularity for the Mean Curvature}

\begin{theorem}\label{epsilon_reg_H}
    Let $\phi:B(0,1)\rightarrow \R^n$ be a conformal weak Willmore immersion. There exists a constant $\epsilon_0>0$ with the following property. Let $\n:B(0,1)\rightarrow \Lambda^2\R^n$ be the unit normal of the immersion $\phi$. Assume that
    \begin{align}
        \int_{B(0,1)}|\D\n|^2dx\leq \epsilon_0.
    \end{align}
    Then, for all $k\in\N$, there exists a positive constant $C_k<\infty$ such that
    \begin{align}
        \np{e^{\lambda}\D^k\H}{\infty}{B(0,\frac{1}{2})}\leq C_k\np{e^{\lambda}\H}{2}{B(0,1)}.
    \end{align}
\end{theorem}
\begin{rem}
    The classical $\epsilon$-regularity shows that
    \begin{align*}
        \np{\D^k\n}{\infty}{B(0,\frac{1}{2})}\leq C_k\np{\D\n}{2}{B(0,1)}.
    \end{align*}
    A similar result is proved by Nicolas Marque in \cite{marque_epsilon_regularity} by assuming only a smallness of the $L^2$ norm of the mean curvature, but the trade-off is that the result is less general for the improved regularity only holds for $k=0$. A similar result was recently proved assuming a smallness of the trace-free second fundamental form by Bernard-Laurain-Marque, see \cite{marque_epsilon_regularity_tracefree_A}. The precised $\epsilon$-regularity presented here is well-adapted to analysis related to the quantization of energy.
\end{rem}
\begin{proof}
    We adapt the proof of \cite{riviere1} by following the estimates more carefully. The following estimates hold for any Willmore immersion $\phi:B(0,1)\rightarrow \R^n$ (see \cite[(III.$13$, $20$, $21$)]{riviere1}) 
\begin{align}
    &\np{\D\H}{2,\infty}{B(0,\frac{1}{2})}\leq C\left(\int_{B_1\setminus\bar{B}_{\frac{1}{2}}(0)}|\H|^2dx\right)\label{epsilon_H1}\\
    &\np{\D\H}{2}{B(0,\frac{1}{4})}\leq C\np{\D\n}{2}{B(0,\frac{1}{2})}\np{\D\H}{2,\infty}{B(0,\frac{1}{2})}+\np{\D\H}{2,\infty}{B(0,\frac{1}{2})}\label{epsilon_H2}\\
    &\np{\D\H}{2,1}{B(0,\frac{1}{8})}\leq C\np{\D\n}{2}{B(0,\frac{1}{4})}\np{\D\H}{2}{B(0,\frac{1}{4})}+\np{\D\H}{2,\infty}{B(0,\frac{1}{4})}\label{epsilon_H3}.
\end{align}
Therefore, adding all estimates \eqref{epsilon_H1}, \eqref{epsilon_H2}, and \eqref{epsilon_H3}, we find that
\begin{align*}
    \np{\D\H}{2,1}{B(0,\frac{1}{8})}&\leq C\np{\D\n}{2}{B(0,\frac{1}{4})}\left(C\np{\D\n}{2}{B(0,\frac{1}{2})}\np{\D\H}{2,\infty}{B(0,\frac{1}{2})}+\np{\D\H}{2,\infty}{B(0,\frac{1}{2})}\right)\\
    &+\np{\D\H}{2,\infty}{B(0,\frac{1}{2})}\\
    &\leq C\left(1+\np{\D\n}{2}{B(0,\frac{1}{2})}\right)\np{\D\n_k}{2}{B(0,\frac{1}{4})}\np{\D\H}{2,\infty}{B(0,\frac{1}{2})}\\
    &\leq C\left(1+\np{\D\n}{2}{B(0,\frac{1}{2})}\right)\np{\D\n}{2}{B(0,\frac{1}{2}}\np{\H}{2}{B(0,1)}.
\end{align*}
Using the Sobolev embedding $W^{1,(2,1)}(\R^2)\hookrightarrow C^0(\R^2)$, we deduce that
\begin{align}\label{sobolev_embedding_H}
    \np{\H}{\infty}{B(0,\frac{1}{8})}&\leq C\np{\D\H}{2,1}{B(0,\frac{1}{8})}+\np{\H}{2}{B(0,\frac{1}{8})}\nonumber\\
    &\leq C\left(1+\np{\D\n}{2}{B(0,\frac{1}{2})}\right)\np{\D\n}{2}{B(0,\frac{1}{2}}\np{\H}{2}{B(0,1)}+C\np{\H}{2}{B(0,\frac{1}{8})}.
\end{align}
Of course, one can replace the $\dfrac{1}{8}$ by $\dfrac{1}{2}$ up to slightly modifying the cutoff functions, and using the Harnack inequality on the conformal parameter, we finally obtain the following estimate
\begin{align}\label{precise_Linfty_H}
    \np{e^{\lambda}\H}{\infty}{B(0,\frac{1}{2})}\leq C\left(1+\np{\D\n}{2}{B(0,1)}\right)\np{\D\n}{2}{B(0,1)}\np{e^{\lambda}\H}{2}{B(0,1)}.
\end{align}
Furthermore, since $\np{\D\n}{2}{B(0,1)}\leq \sqrt{\epsilon_0}$, we obtain the announced bound for a universal constant. As in \cite{riviere1}, this last estimate can be bootstrapped and gives the announced series of inequalities.           
\end{proof}

Let us now precise \cite[Lemma VI.$1$]{quanta} and its counterpart \cite[Theorem $4.2$]{pointwise}.

\begin{theorem}\label{precise_estimate_neck_region}
    There exists a constant $\epsilon_1(n)>0$ with the following property. Let $0<2\,a\leq b<\infty$ and $\phi:B(0,b)\rightarrow \R^n$ be a weak conformal Willmore immersion such that
    \begin{align}\label{neck_hypothesis}
        \sup_{a<r<b/2}\int_{B_{2r}\setminus\bar{B}_r(0)}|\D\n|^2dx\leq \epsilon_1(n).
    \end{align}
    Let $\Omega=B_b\setminus\bar{B}_{a}(0)$, and for all $\left(\frac{r}{R}\right)^{\frac{1}{2}}<t\leq 1$, define $\Omega_t=B_{t\,b}\setminus\bar{B}_{t^{-1}a}(0)$. Let 
    \begin{align*}
        E=\np{\D\lambda}{2,\infty}{\Omega}+\np{\D\n}{2}{\Omega}.
    \end{align*}
    Then, the exists a universal constant $C_1(n)<\infty$ such that for all $\left(\dfrac{r}{R}\right)^{\frac{1}{2}}<t\leq \dfrac{1}{2}$
    \begin{align}\label{precise1}
    \left\{\begin{alignedat}{2}
        &\np{\D^2 S}{1}{\Omega_t}+\np{\D S}{2,1}{\Omega_t}+\np{S}{\infty}{\Omega_t}\leq C_1e^{C_1E}\np{e^{\lambda}\H}{2}{\Omega}\\
        &\np{\D^2 \vec{R}}{1}{\Omega_t}+\np{\D \vec{R}}{2,1}{\Omega_t}+\np{\vec{R}}{\infty}{\Omega_t}\leq C_1e^{C_1E}\np{e^{\lambda}\H}{2}{\Omega}.
        \end{alignedat}\right.
    \end{align}
    In particular, we have
    \begin{align}\label{precise2}
    \left\{\begin{alignedat}{2}
        &\np{\D^2 S}{1}{\Omega_t}+\np{\D S}{2,1}{\Omega_t}+\np{S}{\infty}{\Omega_t}\leq C_1e^{C_1E}\left(\np{\D S}{2}{\Omega}+\np{\D \vec{R}}{2}{\Omega}\right)\\
        &\np{\D^2 \vec{R}}{1}{\Omega_t}+\np{\D \vec{R}}{2,1}{\Omega_t}+\np{\vec{R}}{\infty}{\Omega_t}\leq C_1e^{C_1E}\left(\np{\D S}{2}{\Omega}+\np{\D \vec{R}}{2}{\Omega}\right).
        \end{alignedat}\right.
    \end{align}
    \begin{align}\label{precise3}
        \np{e^{\lambda}\H}{2,1}{\Omega_t}+\np{e^{\lambda}\D\H}{1}{\Omega_t}\leq C_1e^{C_1E}\np{e^{\lambda}\H}{2}{\Omega}.
    \end{align}
\end{theorem}
\begin{rem}
    The results are stated for $0<8\,a<b$ in \cite{quanta} and $0<100\,a<b$ in \cite{pointwise}, but those are just artifacts of the proof, and we need only assume that the conformal class $\log\left(\dfrac{R}{r}\right)$ is bounded away from $-\infty$. The bound on $R$ from \cite{pointwise} is not necessary either but convenient to apply extension results.
\end{rem}
\begin{proof}
    We need only show the proof for $t=\dfrac{1}{2}$. For all $R>1$ and $z\in \Omega_{\frac{1}{2}}$, define 
    \begin{align}
        A(R,z)=B_{R|z|}\setminus\bar{B}_{R^{-1}|z|}(0).
    \end{align}
    Thanks to the Harnack inequality on the conformal parameter in dyadic annuli of the neck region (\cite[Lemma V.$2$]{quanta}) that we can apply here, we deduce that for all $z\in \Omega_{\frac{1}{2}}$ the estimate
\begin{align}\label{new_Linfty_H}
    e^{\lambda(z)}|\H(z)|&\leq C\left(1+\np{\D\n_k}{2}{A(2,z)}\right)\np{\D\n_k}{2}{A(2,z)}\left(\dashint{A(2,z)}e^{2\lambda}|\H|^2dx\right)^{\frac{1}{2}}\nonumber\\
    &\leq C_0(1+\epsilon_1)\epsilon_1\np{e^{\lambda}\H}{2}{A(2,z)}=C_0'\np{e^{\lambda}\H}{2}{A(2,z)}.
\end{align}
Likewise, we have
\begin{align}\label{new_Linfty_DH}
    e^{\lambda(z)}|\D\H(z)|\leq \frac{C_1'}{|z|}\np{e^{\lambda}\H}{2}{A(2,z)}.
\end{align}
Therefore, if we define for all $0<r<\dfrac{b}{2}$
\begin{align*}
    \delta(r)=\left(\frac{1}{r^2}\int_{A(2,r)}e^{2\lambda}|\H|^2dx\right)^{\frac{1}{2}}
\end{align*}
where $A(2,r)=B_{2r}\setminus\bar{B}_{\frac{r}{2}}(0)$, then we can replace $\delta$ as defined in the equation preceding (VI.$9$) in \cite{quanta} (and also in the proof of the $L^{2,1}$ quantization in \cite[Theorem $4.2$, equation preceding ($4.6$)]{pointwise}) by this new $\delta$ function, up to a few modifications that we will indicate now. Rather than detailing those changes in \cite{quanta}, we will directly show them for \cite{pointwise}.

We will also need the classical $\epsilon$-regularity for $\n$, that states that for all $z\in\Omega_{\frac{1}{2}}$, we have
\begin{align}\label{new_Linfty_n}
    |\D\n(z)|\leq C\np{\D\n}{2}{A(2,z)}.
\end{align}
Then, recall that the Willmore equation implies that there exists $\vec{L}:B_b(0)\rightarrow \R^n$
\begin{align}\label{poinca1}
    2i\,\partial\vec{L}=\partial\H+|\H|^2\partial\phi+g^{-1}\otimes\s{\H}{\h_0}{\bar{\partial}\phi}.
\end{align}
Indeed, the imaginary part of the right-hand side is closed (this is the weak Willmore equation first proved in \cite{riviere1}), so the result follows from the Poincaré lemma. Notice that in real coordinatres, we have
\begin{align}\label{poinca2}
    \D^{\perp}\vec{L}=\D\H-3\D^N\H+\star\left(\D^{\perp}\n\wedge \H\right).
\end{align}
By virtue of \eqref{new_Linfty_H} and \eqref{new_Linfty_n}, we directly from either \eqref{poinca1} or \eqref{poinca2}
\begin{align}\label{}
    e^{\lambda}|\D\L|&\leq C\left(\frac{1}{|z|}\delta(|z|)+\delta^2(|z|)+\np{\D\n}{2}{A(2,z)}\delta(|z|)\right)\leq \frac{C}{|z|}\delta(|z|),
\end{align}
where we used the control $\np{\D\n}{2}{A(2,z)}\leq \epsilon_1(n)$ that follwos from .
Then, as in \cite{quanta}, we have for all $2a\leq r_1\leq r_2\leq \dfrac{b}{2}$
\begin{align}\label{second_quanta1}
    \int_{r_1}^{r_2}r\delta^2(r)\,dr\leq \log(4)(1+\epsilon_1(n))\epsilon_1(n)\int_{B_{r_2}\setminus\bar{B}_{r_1}(0)}e^{2\lambda}|\H|^2dx.
\end{align}
The estimate ($4.12$) from \cite{pointwise} is 
\begin{align}\label{second_quanta2}
    e^{\lambda(z)}|\vec{L}(z)-\vec{L}_{|z|}|\leq e^{\lambda(z)}\int_{\partial B(0,|z|)}|\D\vec{L}|d\mathscr{H}^1\leq Ce^{2\Gamma_1E}\delta(|z|),
\end{align}
Therefore, we deduce by \eqref{second_quanta1} and \eqref{second_quanta2} that the estimate ($4.13$) from \cite{pointwise} becomes
\begin{align}
    \int_{\Omega_{\frac{1}{2}}}|\vec{L}(z)-\vec{L}_{|z|}|^2e^{2\lambda(z)}|dz|^2\leq Ce^{2\Gamma_1E}\int_{\Omega}e^{2\lambda}|\H|^2dx.
\end{align}
Now, no change in the strategy from \cite{quanta} is needed to obtain the pointwise estimate
\begin{align}
    e^{\lambda}|\vec{L}(z)|\leq \frac{C}{|z|}e^{2\Gamma_1E}\np{e^{\lambda}\H}{2}{\Omega}^2,
\end{align}
which implies the bound (that actually precedes it in the proof)
\begin{align*}
    \np{e^{\lambda}\vec{L}}{2,\infty}{\Omega_{1/2}}\leq Ce^{2\Gamma_1E}\np{e^{\lambda}\H}{2}{\Omega}^2.
\end{align*}
Now, the rest of the proof is unchanged, and we deduce thanks to the Jacobian systems
\begin{align}
\left\{\begin{alignedat}{2}
    \Delta S&=-\star\D\n\cdot\D^{\perp}\vec{R}\qquad&&\text{in}\;\, \Omega\\
    \Delta \vec{R}&=(-1)^n\star\left(\D\n\antires\D^{\perp}\vec{R}\right)+\star\D\n\cdot\D^{\perp}S\qquad&&\text{in}\;\, \Omega
    \end{alignedat}\right.
\end{align}
that
\begin{align*}
    \left\{\begin{alignedat}{2}
        &\np{\D^2 S}{1}{\Omega_{1/2}}+\np{\D S}{2,1}{\Omega_{1/2}}+\np{S}{\infty}{\Omega_{1/2}}\leq Ce^{2\Gamma_1E}\left(1+\np{\D\n}{2}{\Omega}\right)^2\np{e^{\lambda}\H}{2}{\Omega}\\
        &\np{\D^2 \vec{R}}{1}{\Omega_{1/2}}+\np{\D \vec{R}}{2,1}{\Omega_{1/2}}+\np{\vec{R}}{\infty}{\Omega_{1/2}}\leq C_1(n)e^{C_1(n)E}\left(1+\np{\D\n}{2}{\Omega}\right)^2\np{e^{\lambda}\H}{2}{\Omega}.
        \end{alignedat}\right.
\end{align*}
Recalling the algebraic identity, we have
\begin{align}\label{algebraic_mean_curvature_R_S}
    e^{2\lambda}\H=\frac{1}{4}\D^{\perp}S\cdot\D\phi-\frac{1}{4}\D\vec{R}\res \D^{\perp}\phi,
\end{align}
which implies by the Minkowski inequality that
\begin{align*}
    \np{e^{\lambda}\H}{2}{\Omega}\leq \frac{1}{2\sqrt{2}}\left(\np{\D S}{2}{\Omega}+\np{\D\vec{R}}{2}{\Omega}\right),
\end{align*}
Finally, we have
\begin{align}
\left\{\begin{alignedat}{1}
    &\np{\D S}{2,1}{\Omega_{1/2}}+\np{\D\vec{R}}{2,1}{\Omega_{1/2}}\leq Ce^{2\Gamma_1E}\left(1+\np{\D\n}{2}{\Omega}\right)^2\left(\np{\D S}{2}{\Omega}+\np{\D\vec{R}}{2}{\Omega}\right)\\
    &\np{S}{\infty}{\Omega_{1/2}}+\np{\vec{R}}{\infty}{\Omega_{1/2}}\leq Ce^{2\Gamma_1E}\left(1+\np{\D\n}{2}{\Omega}\right)^2\left(\np{\D S}{2}{\Omega}+\np{\D\vec{R}}{2}{\Omega}\right),
    \end{alignedat}\right.    
\end{align}
while the estimate on $\H$ and $\D\H$ follow from \eqref{algebraic_mean_curvature_R_S}. Indeed, we first have
\begin{align*}
    \np{e^{\lambda}\H}{2,1}{\Omega_{1/2}}\leq \np{\D S}{2,1}{\Omega_{1/2}}+\np{\D\vec{R}}{2,1}{\Omega_{1/2}}\leq Ce^{2\Gamma_1E}\left(1+\np{\D\n}{2}{\Omega}\right)\np{e^{\lambda}\H}{2}{\Omega}.
\end{align*}
On the other hand, \eqref{algebraic_mean_curvature_R_S} shows that
\begin{align*}
    e^{\lambda}\p{z}\H&=-2(\p{z}\lambda)e^{\lambda}\H+\frac{1}{4}\D^{\perp}(\p{z}S)\cdot e^{-\lambda}\D^{\perp}\phi+\frac{1}{4}\D^{\perp}S\cdot e^{-\lambda}\D(\p{z}\phi)-\frac{1}{4}\D(\p{z}\vec{R})\res \D^{\perp}\phi\\
    &-\frac{1}{4}\D\vec{R}\res e^{-\lambda}\D^{\perp}(\p{z}\phi).
\end{align*}
Recall the following identities
\begin{align*}
    &\H=\frac{1}{2}\Delta_g\phi=\frac{1}{2}e^{-\lambda}\Delta\phi=2\,e^{-2\lambda}\p{z\z}^2\phi\\
    &\h_0=2\,\pi_{\n}(\p{z}^2\phi)=2\,\left(\p{z}^2\phi-2(\p{z}\lambda)\p{z}\phi\right).
\end{align*}
Therefore, we estimate by the identity $|\D^2\phi|^2=2|\p{z}^2\phi|^2+2|\p{z\z}^2\phi|^2$
\begin{align*}
    \np{e^{-\lambda}\D^2\phi}{2,\infty}{\Omega_{{1/2}}}&=\np{e^{\lambda}\H}{2,\infty}{\Omega_{\frac{1}{2}}}+\np{e^{-\lambda}\h_0}{2,\infty}{\Omega_{\frac{1}{2}}}+2\np{\D\lambda}{2}{\Omega_{\frac{1}{2}}}\\
    &\leq 2\left(\np{\D\n}{2}{\Omega_{\frac{1}{2}}}+\np{\D\lambda}{2,\infty}{\Omega_{\frac{1}{2}}}\right)\leq 2E.
\end{align*}
Using the $L^{2,1}/L^{2,\infty}$ duality, we get 
\begin{align*}
    \np{e^{\lambda}\D\H}{1}{\Omega_{\frac{1}{2}}}&\leq 2\np{e^{\lambda}\H}{2,1}{\Omega_{\frac{1}{2}}}\np{\D\lambda}{2,\infty}{\Omega_{\frac{1}{2}}}+\np{\D^2S}{1}{\Omega_{\frac{1}{2}}}+\np{\D S}{2,1}{\Omega_{\frac{1}{2}}}\np{e^{-\lambda}\D^2\phi}{2,\infty}{\Omega_{\frac{1}{2}}}\\
    &\np{\D^2\vec{R}}{1}{\Omega_{\frac{1}{2}}}+\np{\D \vec{R}}{2,1}{\Omega_{\frac{1}{2}}}\np{e^{-\lambda}\D^2\phi}{2,\infty}{\Omega_{\frac{1}{2}}}\\
    &\leq C(1+E)e^{2\Gamma_1E}\left(1+\np{\D\n}{2}{\Omega}\right)\np{e^{\lambda}\H}{2}{\Omega}.
\end{align*}
Finally, all estimates follow by estimating
\begin{align*}
    &(1+E)\leq e^{E}\\
    &\left(1+\np{\D\n}{2}{\Omega}\right)^2\leq \exp\left(2\left(1+\np{\D\n}{2}{\Omega}\right)\right)
\end{align*}
and taking a large enough constant $C_1(n)$.
\end{proof}

\subsection{Pointwise Estimate in Annular Regions for Jacobian Systems}

In this section, we formalise the argument contained in \cite[Lemma III.$1$]{riviere_morse_scs} and prove a general result that allows us to obtain a Hölder-type estimate for all $0<\beta<1$ (it will be important in the asymptotic analysis due to the non-integrable weight that we consider).

First, recall a few technical lemmas from \cite{riviere_morse_scs}. Indeed, our approach from the previous section does not work in the case of a non-trivial logarithmic component as in \cite{lauriv1}.

\begin{lemme}[Lemma F.$1$, \cite{riviere_morse_scs}]\label{lemme_F_1}
    Let $A\in W^{1,2}(B(0,1),\R^n), B\in W^{1,2}(B(0,1),M_n(\R))$. Assume that $u:B(0,1)\rightarrow \R^n$ solves the equation
    \begin{align}
        \Delta u=\D^{\perp}B\cdot \D A\qquad\text{in}\;\, B(0,1).
    \end{align}
    Then, there exists a universal constant $C<\infty$ independent of $n$ such that
    \begin{align}\label{monotonie_optimale}
        \np{\D u}{2}{B_{1/2}\setminus\bar{B}_{1/4}(0)}\leq \frac{1
        }{2}\np{\D u}{2}{B_1\setminus\bar{B}_{1/2}(0)}+C\np{\D A}{2}{B(0,1)}\np{\omega\,\D B}{2}{B(0,1)},
    \end{align}
    where $\omega(x)=|x|\log\left(1+\frac{1}{|x|}\right)\sqrt{\log\left(1+\log\left(\frac{1}{|x|}\right)\right)}$.
\end{lemme}
\begin{proof}
    Write $u=\varphi+\psi$, where 
    \begin{align*}
        \left\{\begin{alignedat}{2}
           \Delta \varphi&=\D^{\perp}B\cdot \D A\qquad&&\text{in}\;\, B(0,1)\\
           \varphi&=0\qquad&&\text{in}\;\, \partial B(0,1).
        \end{alignedat}\right.
    \end{align*}
    Since $\psi$ is a harmonic function on $B(0,1)$, there exists a sequence $\ens{a_k}_{k\in\N}\subset \C^n$ such that
    \begin{align*}
        \psi(z)=\Re\left(\sum_{k=0}^{\infty}a_kz^k\right)
    \end{align*}
    Thanks to the identity \eqref{l2_formula}, we have for all $0\leq a\leq b\leq 1$
    \begin{align*}
        \int_{B_b\setminus\bar{B}_a(0)}|\D \psi|^2dx=\pi\sum_{k=1}^{\infty}k|a_k|^2\left(b^{2k}-a^{2k}\right).
    \end{align*}
    Therefore, taking $b=1$ and $a=1/2$, we deduce that 
    \begin{align*}
        \int_{B_1\setminus\bar{B}_{1/2}(0)}|\D \psi|^2dx=\pi\sum_{k=1}^{\infty}|k||a_k|^2\left(1-\frac{1}{4^k}\right),
    \end{align*}
    while (taking $b=1/2$ and $a=1/4$)
    \begin{align}\label{ineq_harmonic_ball}
        \int_{B_{1/2}\setminus\bar{B}_{1/4}(0)}|\D \psi|^2dx&=\pi\sum_{k=1}^{\infty}|k||a_k|^2\left(\frac{1}{4^k}-\frac{1}{16^k}\right)=\pi\sum_{k=1}^{\infty}\frac{|k|}{4^k}|a_k|^2\left(1-\frac{1}{4^k}\right)\nonumber\\
        &\leq \frac{1}{4}\,\pi\sum_{k=1}^{\infty}|k||a_k|^2\left(1-\frac{1}{4^k}\right)=\frac{1}{4}\int_{B_1\setminus\bar{B}_{1/2}(0)}|\D \psi|^2dx.
    \end{align}
    On the other hand, the weighted Wente inequality of \cite[Lemma E.$1$]{riviere_morse_scs} shows that
    \begin{align*}
        \np{\D\varphi}{2}{B_1\setminus\bar{B}_{1/4}(0)}\leq C_W\np{\D A}{2}{B(0,1)}\np{\omega\,\D B}{2}{B(0,1)}.
    \end{align*}
    Therefore, we deduce that
    \begin{align*}
        \np{\D u}{2}{B_{1/2}\setminus\bar{B}_{1/4}(0)}&\leq \np{\D \varphi}{2}{B_{1/2}\setminus\bar{B}_{1/4}(0)}+\np{\D\psi}{2}{B_{1/2}\setminus\bar{B}_{1/4}(0)}\\
        &\leq \frac{1}{2}\np{\D\psi}{2}{B_1\setminus\bar{B}_{1/2}(0)}+\np{\D\varphi}{2}{B_{1/2}\setminus\bar{B}_{1/4}(0)}\\
        &\leq \frac{1}{2}\np{\D u}{2}{B_{1}\setminus\bar{B}_{1/2}(0)}+\frac{1}{2}\np{\D\varphi}{2}{B_1(0)\setminus\bar{B}_{1/2}(0)}+\np{\D\varphi}{2}{B_1\setminus\bar{B}_{1/2}(0)}\\
        &\leq \frac{1}{2}\np{\D u}{2}{B_1\setminus\bar{B}_{1/2}(0)}+\frac{3}{2}C_W\np{\D A}{2}{B(0,1)}\np{\omega\,\D B}{2}{B(0,1)},
    \end{align*}
    which concludes the proof of the the lemma.
\end{proof}

\begin{lemme}[Lemma F.$2$, \cite{riviere_morse_scs}]\label{lemme_F_2}
     Let $A\in W^{1,2}(B(0,1),\R^n), B\in W^{1,2}(B(0,1),M_n(\R))$. Assume that $u:B(0,1)\rightarrow \R^n$ solves the equation
    \begin{align}
        \Delta u=\D^{\perp}B\cdot \D A\qquad\text{in}\;\, B(0,1).
    \end{align}    
    Then, there exists a universal constant $C<\infty$ such that for all $0<\alpha<1$ and for all $k\in \N$
    \begin{align}\label{ineq_F20}
        \np{\D u}{2}{A_k}\leq \frac{1}{2^k}\np{\D u}{2}{A_0}+\frac{C}{(1-\alpha)^{\frac{5}{2}}}\np{\D A}{2}{B(0,1)}\left(\sum_{l=0}^{\infty}\frac{1}{2^{2\alpha|l-k+1|}}\int_{A_l}|\D B|^2dx\right)^{\frac{1}{2}},
    \end{align}
    where $A_k=B_{2^{-k}}\setminus\bar{B}_{2^{-(k+1)}}(0)$ for all $k\in\N$.
\end{lemme}
\begin{proof}
    Thanks to an immediate scaling argument, we deduce that for all $k\in\N$
    \begin{align}\label{iteration1}
        \np{\D u}{2}{A_{k+1}}\leq \frac{1}{2}\np{\D u}{2}{A_k}+C\np{\D A}{2}{B(0,2^{-k})}\np{\omega_k\,\D B}{2}{B(0,2^{-k}(0)}.
    \end{align}
    where $\omega_k(x)=\omega(2^kx)$. Therefore, an immediate induction shows that
    \begin{align*}
        \np{\D u}{2}{A_{k}}\leq \frac{1}{2^{k}}\np{\D u}{2}{A_0}+C\sum_{j=0}^{k-1}\frac{1}{2^{j}}\np{\D A}{2}{B(0,2^{-(k-1-j)}}\np{\omega_{k-1-j}\D B}{2}{B(0,2^{-(k-1-j)})}.
    \end{align*}
    Using the Cauchy-Schwarz inequality, we deduce that for all $0\leq \alpha\leq 1$
    \begin{align*}
        &\sum_{j=0}^{k-1}\frac{1}{2^j}\np{\D A}{2}{B(0,2^{-(k-1-j)})}\np{\omega_{k-1-j}\D B}{2}{B(0,2^{-(k-1-j)}}\\
        &\leq \left(\sum_{j=0}^{k-1}\frac{1}{2^{2(1-\alpha) j}}\int_{B(0,2^{k-1-j})}|\D A|^2dx\right)^{\frac{1}{2}}\left(\sum_{j=0}^{k-1}\frac{1}{2^{2\alpha\,j}}\int_{B(0,2^{-(k-1-j)})}f(2^{k-1-j}x)|\D B|^2dx\right)^{\frac{1}{2}},
    \end{align*}
    where $f=\omega^2$.
    Now, for all $\ens{\varphi_{k,j}}_{k,j\in\N}\in L^2(B(0,1))$ and for all $0<\alpha<1$, we have
    \begin{align}\label{fubini_discrete}
        &\sum_{j=0}^{k-1}\frac{1}{2^{2\alpha\,j}}\int_{B(0,2^{k-1-j})}|\varphi_{k,j}|^2dx=\sum_{j=0}^{k-1}\frac{1}{2^{2\alpha\,j}}\sum_{l=k-1-j}^{\infty}\int_{A_l}|\varphi_{k,j}|^2dx=\sum_{l=0}^{\infty}\sum_{j=0}^{k-1}\frac{1}{2^{2\alpha j}}\mathbf{1}_{\ens{k-1-j\leq l}}\int_{A_l}|\varphi_{k,j}|^2dx\nonumber\\
        &=\sum_{l=0}^{\infty}\sum_{j=\max\ens{k-1-l,0}}^{k-1}\frac{1}{2^{2\alpha j}}\int_{A_l}|\varphi_{k,j}|^2dx=\sum_{l=0}^{k-1}\sum_{j=k-1-l}^{k-1}\frac{1}{2^{2\alpha j}}\int_{A_l}|\varphi_{k,j}|^2dx+\sum_{l=k}^{\infty}\sum_{j=0}^{k-1}\frac{1}{2^{2\alpha j}}\int_{A_l}|\varphi_{k,j}|^2dx.
    \end{align}
    Applying this inequality to $\ens{\varphi_{k,j}}_{k,j\in\N}=\ens{\D A}_{k,j\in\N}$, we get
    \begin{align}\label{dyadic_final1}
        \sum_{j=0}^{k-1}\frac{1}{2^{2(1-\alpha) j}}\int_{B(0,2^{k-1-j})}|\D A|^2dx&=\sum_{l=0}^{k-1}\int_{A_l}|\D A|^2dx\sum_{j=k-1-l}^{k-1}\frac{1}{2^{2(1-\alpha)j}}+\sum_{l=k}^{\infty}\int_{A_l}|\D A|^2dx\sum_{j=0}^{k-1}\frac{1}{2^{2(1-\alpha)j}}\nonumber\\
        &\leq \frac{2^{2(1-\alpha)}}{2^{2(1-\alpha)}-1}\sum_{l=0}^{k-1}\frac{1}{2^{2(1-\alpha)l}}\int_{A_l}|\D A|^2+\frac{2^{2(1-\alpha)}}{2^{2(1-\alpha)}-1}\sum_{l=k}^{\infty}\int_{A_l}|\D A|^2dx\nonumber\\
        &\leq \frac{4}{2^{2(1-\alpha)-1}}\int_{B(0,1)}|\D A|^2.
    \end{align}
    On the other hand, using that for all $l\in \N$ and $x\in A_l$
    \begin{align*}
        f(2^{k-1-j}x)\leq 2^{2(k-1-j-l)}\log^2(1+2^{l-(k-j)+2})\log\left(1+(l-(k-j)+2)\log(2)\right),
    \end{align*}
    we deduce that
    \begin{align*}
        &\sum_{j=0}^{k-1}\frac{1}{2^{2\alpha j}}\int_{B(0,2^{-(k-1-j)})}f(2^{k-1-j}x)|\D B|^2dx\\
        &\leq \sum_{l=0}^{k-1}\int_{A_l}|\D B|^2\sum_{j=k-1-l}^{k-1}\frac{1}{2^{2\alpha j}}2^{-2(l-(k-j)+1)}\log^2(1+2^{l-(k-j)+2})\log\left(1+(l-(k-j)+2)\log(2)\right)\\
        &+\sum_{l=k}^{\infty}\int_{A_l}|\D B|^2dx\sum_{j=0}^{k-1}\frac{1}{2^{2\alpha j}}2^{-2(l-(k-j)+1)}\log^2(1+2^{l-(k-j)+2})\log\left(1+(l-(k-j)+2)\log(2)\right).
    \end{align*}
    First, we have by the elementary inequality $\log(1+x)\leq x$
    \begin{align*}
        &\sum_{j=k-1-l}^{k-1}\frac{1}{2^{2\alpha j}}2^{-2(l-(k-j)+1)}\log^2(1+2^{l-(k-j)+2})\log\left(1+(l-(k-j)+2)\log(2)\right)\\
        &\leq \log^3(2)\sum_{j=k-1-l}^{k-1}(l-(k-j)+2)^3\frac{1}{2^{2\alpha j}}2^{-2(l-(k-j)+1)}\leq \log^3(2)\frac{1}{2^{2\alpha(k-1-l)}}\sum_{i=0}^{l}(i+1)^3\frac{1}{2^{2(1+\alpha)i}}\\
        &\leq \log^3(2)\frac{1}{2^{2\alpha(k-1-l)}}\frac{2^{-4(1+\alpha)}+4\cdot 2^{-2(1+\alpha))}+1}{(1-2^{-2(1+\alpha))})^4}\leq \frac{176\log^3(2)}{27}\frac{1}{2^{2\alpha(k-1-l)}}
    \end{align*}
    where we used
    \begin{align*}
        \sum_{i=0}^{\infty}(i+1)^3x^i=\frac{x^2+4x+1}{(1-x)^4}.
    \end{align*}
    On the other hand, we have
    \begin{align*}
        &\sum_{j=0}^{k-1}\frac{1}{2^{2\alpha j}}2^{-2(l-(k-j)+1)}\log^2(1+2^{l-(k-j)+2})\log\left(1+(l-(k-j)+2)\log(2)\right)\\
        &\leq \frac{1}{2}\log^3(2)2^{-2(l-k-1)}\sum_{j=0}^{k-1}(l-(k-j)+2)^3\frac{1}{2^{2(1+\alpha)j}}\\
        &=\log^3(2)2^{-2(l-k+1)}\sum_{j=0}^{k-1}\left(j^3+3(l-k+2)j^2+3(l-k+2)^2j+(l-k+2)^3\right)\frac{1}{2^{2(1+\alpha)j}}\\
        &=\log^3(2)2^{-2(l-k+1)}\\
        &\times\left(\frac{(k-1)^3\beta^{k+3}-(3k^3-6k^2+4)\beta^{k+2}+(3k^3-3k^2-3k+1)\beta^{k+1}-k^3\beta^k+\beta^3+4\beta^2+\beta}{(1-\beta)^4}\right.\\
        &\left.+3(l-k+2)\frac{-(k-1)^2\beta^{k+2}+(2k^2-2k-1)\beta^{k+1}-k^2\beta^k+\beta^2+\beta}{(1-\beta)^3}\right.\nonumber\\
        &\left.+3(l-k+2)^2\frac{(k-1)\beta^{k+1}-k\beta^{k}+\beta}{(1-\beta)^4}
        +(l-k+2)^3\frac{1-\beta^k}{1-\beta}\right)
    \end{align*}
    where $\beta=2^{-2(1+\alpha)}$. Since the series $\sum_{k=1}^{\infty}k^3\beta^k$ converges, we deduce that there exists a universal constant independent of $0<\alpha<1$ such that
    \begin{align*}
        &\sum_{j=0}^{k-1}\frac{1}{2^{2\alpha j}}2^{-2(l-(k-j)+1)}\log^2(1+2^{l-(k-j)+2})\log\left(1+(l-(k-j)+2)\log(2)\right)\\
        &\leq C\left(1+(l-k+1)^3\right)\frac{1}{2^{2(l-k+1)}}.
    \end{align*}
    Then, for all $0<\alpha<1$ and for all $l\geq k$, we have
    \begin{align*}
        &(l-k+1)^3\frac{1}{2^{2(1-k-1)}}=\frac{1}{2^{2\alpha(l-k+1)}}\frac{(l-k+1)^3}{2^{2(1-\alpha)(l-k-1)}}\leq \frac{1}{2^{2\alpha(l-k+1)}}\sum_{l=k}^{\infty}\frac{(l-k+1)^3}{2^{2(1-\alpha)(l-k+1)}}\\
        &\leq \frac{1}{2^{2\alpha(l-k+1)}}\frac{2^{-4(1-\alpha)}+4\cdot 2^{-2(1-\alpha)}+1}{(1-2^{-2(1-\alpha)})^4}=\frac{1}{2^{2\alpha(l-k+1)}}\frac{2^{4(1-\alpha)}+4\cdot 2^{6(1-\alpha)}+2^{8(1-\alpha)}}{(2^{2(1-\alpha)}-1)^4}
    \end{align*}
    Therefore, we finally deduce that there exists a universal constant $C<\infty$ such that
    \begin{align}\label{dyadic_final2}
        \sum_{j=0}^{k-1}\frac{1}{2^{2\alpha j}}\int_{B(0,2^{-(k-1-j)})}f(2^{k-1-j}x)|\D B|^2dx\leq \frac{C}{(2^{2(1-\alpha)}-1)^4}\sum_{l=0}^{\infty}\frac{1}{2^{2\alpha|l-k+1|}}\int_{A_l}|\D B|^2dx.
    \end{align}
    Thanks to \eqref{dyadic_final1} and \eqref{dyadic_final2}, we deduce that for all $0<\alpha<1$
    \begin{align}\label{inequality_fundamental2}
        &\sum_{j=0}^{k-1}\frac{1}{2^j}\np{\D A}{2}{B(0,2^{-(k-1-j)})}\np{\omega_{k-1-j}\D B}{2}{B(0,2^{-(k-1-j)}}\nonumber\\
        &\leq \frac{C}{(2^{2(1-\alpha)}-1)^{\frac{5}{2}}}\np{\D A}{2}{B(0,1)}\left(\sum_{l=0}^{\infty}\frac{1}{2^{2\alpha|l-k+1|}}\int_{A_l}|\D B|^2dx\right)^{\frac{1}{2}},
    \end{align}
    and
    \begin{align*}
        \np{\D u}{2}{A_k}\leq \frac{1}{2^k}\np{\D u}{2}{A_0}+\frac{C}{(2^{2(1-\alpha)}-1)^{\frac{5}{2}}}\np{\D A}{2}{B(0,1)}\left(\sum_{l=0}^{\infty}\frac{1}{2^{2\alpha|l-k+1|}}\int_{A_l}|\D B|^2dx\right)^{\frac{1}{2}}.
    \end{align*}
    Using the elementary inequality
    \begin{align*}
        2^{2(1-\alpha)}-1\geq 2(1-\alpha)\log(2),
    \end{align*}
    the lemma is proved.
\end{proof}
\begin{rem}
    The proof is symmetric in $A$ and $B$, and we also get
    \begin{align*}
        \np{\D u}{2}{A_k}\leq \frac{1}{2^k}\np{\D u}{2}{A_0}+\frac{C}{(2^{2(1-\alpha)}-1)^{\frac{5}{2}}}\np{\D B}{2}{B(0,1)}\left(\sum_{l=0}^{\infty}\frac{1}{2^{2\alpha|l-k+1|}}\int_{A_l}|\D A|^2dx\right)^{\frac{1}{2}}.
    \end{align*}
\end{rem}

\begin{cor}\label{iteration_partial}
    Let $A\in W^{1,2}(B(0,1),\R^n), B\in W^{1,2}(B(0,1),M_n(\R))$. Assume that $u:B(0,1)\rightarrow \R^n$ solves the equation
    \begin{align}
        \Delta u=\D^{\perp}B\cdot \D A\qquad\text{in}\;\, B(0,1).
    \end{align}    
    Then, there exists a universal constant $C<\infty$ such that for all $0<\alpha<1$ and for all $k\in \N$ and $0\leq j\leq k$
    \begin{align}\label{ineq_F2}
        \np{\D u}{2}{A_k}\leq \frac{1}{2^j}\np{\D u}{2}{A_{k-j}}+\frac{C}{(1-\alpha)^{\frac{5}{2}}}\np{\D A}{2}{B(0,1)}\left(\sum_{l=0}^{\infty}\frac{1}{2^{2\alpha|l-j+1|}}\int_{A_l}|\D B|^2dx\right)^{\frac{1}{2}},
    \end{align}
    where $A_k=B_{2^{-k}}\setminus\bar{B}_{2^{-(k+1)}}(0)$ for all $k\in\N$.
\end{cor}
\begin{proof}
    Iterating \eqref{iteration1} only $0\leq j\leq k$ times, we get
    \begin{align*}
        \np{\D u}{2}{A_k}\leq \frac{1}{2^{j}}\np{\D u}{2}{A_{k-j}}+C\sum_{i=0}^{j-1}\frac{1}{2^i}\np{\D A}{2}{B(0,2^{-(k-1-i)}}\np{\omega_k\D B}{2}{B(0,2^{-(k-1-j)}},
    \end{align*}
    which shows the expected result thanks to \eqref{inequality_fundamental2}.
\end{proof}

We can now state the main theorem of this section.

\begin{theorem}\label{collar_statement}
    Let $0<4\,a<b\leq 1$, $\Omega=B_b\setminus\bar{B}_a(0)$, and for all $0<t\leq 1$, let $\Omega_t=B_{t\,b}\setminus\bar{B}_{t^{-1}\,a}(0)$. Let $A\in W^{1,(2,1)}(\Omega,\R^n)$ and $B\in W^{1,2}(\Omega,M_n(\R))$, and assume that
    \begin{align}\label{jacobian_system_pointwise}
        \Delta A=\D^{\perp}B\cdot \D A\qquad\text{in}\;\,\Omega.
    \end{align}
    Let 
    \begin{align*}
        \Lambda=\frac{1}{2\pi}\int_{\Omega_{1/2}}\frac{|\D A|}{|x|}dx.
    \end{align*}
    Then, there exists  universal constants $C,\Gamma<\infty$ such that for all $0<\delta<1$,  and for all $0<\beta<1$, if
    \begin{align}\label{epsilon_smallness2}
        \frac{\Gamma}{(1-\beta)^7}\int_{\Omega}|\D B|^2dx\leq \delta^2,
    \end{align}
    then for all $z\in \Omega_{1/2}$, we have
    \begin{align}\label{near_optimal_neck}
        \np{\D A}{2}{B_{2|z|}\setminus\bar{B}_{|z|}(0)}&\leq C\left(\left(\frac{|z|}{b}\right)^{\beta}+\left(\frac{a}{|z|}\right)^{\beta}\right)\np{\D A}{2}{\Omega}
        +\frac{\left(1+C\delta\right)}{\log\left(\frac{b}{4a}\right)}\left(\sqrt{2\pi\log(2)}\,\Lambda+C\np{\D A}{2}{\Omega}\right).
    \end{align}
\end{theorem}
\begin{rem}
    The hypothesis $b\leq 1$ is not restrictive for the previous lemmas apply provided that the conformal class of the annulus is bounded from below.
\end{rem}
\begin{proof}
    Make on $\Omega$ a global expansion
    \begin{align*}
        \D A=\D\varphi+\D\psi, 
    \end{align*}
    where
    \begin{align}
        \left\{\begin{alignedat}{2}
            \Delta \varphi&=\D^{\perp}\widetilde{B}\cdot \D \widetilde{A}\qquad&&\text{in}\;\,B(0,1)\\
            \varphi&=0\qquad&&\text{on}\;\, \partial B(0,1).
    \end{alignedat}\right.
    \end{align}
    and $\widetilde{A}:\C\rightarrow \R^n$ and $\widetilde{B}:\C\rightarrow M_n(\R)$ are controlled extensions of $A$ and $B$ respectively given by Lemma \ref{whitney_extension}. There exists a universal constant $C<\infty$ such that
    \begin{align}\label{lambda_1}
        \np{\D\varphi}{2,1}{B(0,1)}\leq C\np{\D A}{2}{\Omega}\np{\D B}{2}{\Omega}.
    \end{align}
    Now, if 
    \begin{align*}
        d=\frac{1}{2\pi}\int_{\partial B(0,r)}\partial_{\nu}\psi\,d\mathscr{H}^1,
    \end{align*}
    and $\psi_0=\psi-d\,\log|z|$, 
    then thanks to the proof of Corollary \ref{pointwise_annulus_harmonic}, we deduce that for all $z\in \Omega_{1/2}$
    \begin{align*}
        |z||\D \psi_0(z)|\leq \frac{8}{3\sqrt{3\pi}}\left(\frac{|z|}{b}+\frac{a}{|z|}\right)\np{\D\psi_0}{2}{\Omega},
    \end{align*}
    which implies that for all $k\in\N$ such that $A_k=B_{2^{-k}}\setminus\bar{B}_{2^{-(k+1)}}(0)\subset \Omega_{1/2}$
    \begin{align}\label{collar1}
        \np{\D\psi_0}{2}{A_k}\leq 3\left(\frac{2^{-k}}{b}+\frac{a}{2^{-k}}\right)\np{\D\psi_0}{2}{\Omega}.
    \end{align}
    Notice also that thanks to Parseval identity, we have
    \begin{align*}
        \sqrt{\np{\D\psi_0}{2}{\Omega}^2+\np{\D(\psi-\psi_0)}{2}{\Omega}^2}&=\np{\D\psi}{2}{\Omega}\leq \np{\D A}{2}{\Omega}+\np{\D\varphi}{2}{\Omega}\\
        &\leq \left(1+C\np{\D B}{2}{\Omega}\right)\np{\D A}{2}{\Omega}. 
    \end{align*}
    Then, thanks to Lemma \ref{improved_reg_harmonic}, we have
    \begin{align}\label{lambda_2}
        \np{\D \psi_0}{2,1}{\Omega_{1/2}}\leq 64\sqrt{\frac{\pi}{15}}\np{\D\psi_0}{2}{\Omega}\leq 64\sqrt{\frac{\pi}{15}}\left(1+C\np{\D B}{2}{\Omega}\right)\np{\D A}{2}{\Omega}.
    \end{align}
    Therefore, thanks to the $L^{2,1}/L^{2,\infty}$ duality and \eqref{lambda_1}, we deduce that 
    \begin{align}\label{lambda3}
        \int_{\Omega_{1/2}}\frac{|\D\varphi|}{|x|}dx\leq \np{\D\varphi}{2,1}{\Omega_{1/2}}\np{\frac{1}{|x|}}{2,\infty}{\Omega_{1/2}}\leq \sqrt{\pi}C\np{\D A}{2}{\Omega}\np{\D B}{2}{\Omega},
    \end{align}
    while \eqref{lambda_2} shows that 
    \begin{align}\label{lambda4}
        \int_{\Omega_{1/2}}\frac{|\D\psi_0|}{|x|}dx\leq 64\frac{\pi}{\sqrt{15}}\left(1+C\np{\D B}{2}{\Omega}\right)\np{\D A}{2}{\Omega}.
    \end{align}
    Therefore, we have by \eqref{lambda3} and \eqref{lambda4}
    \begin{align}\label{lambda5}
        \int_{\Omega_{1/2}}\frac{|\D u-\D(\psi-\psi_0)|}{|x|}dx\leq \int_{\Omega_{1/2}}\frac{|\D\varphi|}{|x|}dx+\int_{\Omega_{1/2}}\frac{|\D\psi_0|}{|x|}dx\leq C\left(1+\np{\D B}{2}{\Omega}\right)\np{\D A}{2}{\Omega},
    \end{align}
    which implies in particular thanks to the triangle inequality that
    \begin{align}\label{lambda6}
        \left|2\pi\, d\,\log\left(\frac{b}{4a}\right)- 2\pi\Lambda\right|&=\left|\int_{\Omega_{1/2}}\frac{|\D u|}{|x|}dx-\int_{\Omega_{1/2}}\frac{|\D(\psi-\psi_0)|}{|x|}dx\right|\leq  \int_{\Omega_{1/2}}\frac{|\D u-\D(\psi-\psi_0)|}{|x|}dx\nonumber\\
        &\leq C\left(1+\np{\D B}{2}{\Omega}\right)\np{\D A}{2}{\Omega}.
    \end{align}
    Now,
    we deduce thanks to Lemma \ref{lemme_F_2} that for all $k\in\N$ and for all $0<\alpha<1$, there holds 
    \begin{align}\label{collar2}
        \np{\D \varphi}{2}{A_k}\leq \frac{1}{2^k}\np{\D\varphi}{2}{A_0}+\frac{C}{(1-\alpha)^{\frac{5}{2}}}\np{\D \widetilde{B}}{2}{B(0,1)}\left(\sum_{l=0}^{\infty}\frac{1}{2^{2\alpha|l-k+1|}}\int_{A_l}|\D \widetilde{A}|^2dx\right)^{\frac{1}{2}}.
    \end{align}
    Since $\mathrm{supp}(\D \widetilde{A})\subset B_{2b}\setminus\bar{B}_{a/2}(0)$, we have 
    \begin{align*}
        \int_{A_l}|\D \widetilde{A}|^2dx=0
    \end{align*}
    provided that $2^{-l}\leq a/2$, or $l\geq \log_2\left(\dfrac{2}{a}\right)$. Likewise, 
    \begin{align*}
        \int_{A_l}|\D \widetilde{A}|^2dx=0
    \end{align*}
    provided that $2^{-l-1}\geq 2b$, or $l\leq \log_2\left(\dfrac{1}{4b}\right)$. If $N_2=\left[\log_2\left(\dfrac{2}{a}\right)\right]$, and $N_1=\left[\log_2\left(\dfrac{1}{4b}\right)\right]$, then 
    \begin{align*}
        \sum_{l=0}^{\infty}\frac{1}{2^{2\alpha|l-k+1|}}\int_{A_l}|\D \widetilde{A}|^2dx=\sum_{l=N_1}^{N_2}\frac{1}{2^{2\alpha|l-k+1|}}\int_{A_l}|\D \widetilde{A}|^2dx.
    \end{align*}
    Then, provided that $A_k\subset \Omega_{1/2}$, we have by \eqref{lambda_2}, \eqref{collar1}, and \eqref{collar2} the estimate
    \begin{align}\label{collar3}
        &\np{\D A}{2}{A_k}\leq \np{\D\varphi}{2}{A_k}+\np{\D\psi}{2}{A_k}\leq \frac{1}{2^k}\np{\D\varphi}{2}{A_0}+\np{\D\psi_0}{2}{A_k}+\np{\D(\psi-\psi_0)}{2}{A_k}\nonumber\\
        &\leq \frac{C}{2^k}\np{\D A}{2}{\Omega}\np{\D B}{2}{\Omega}+3\left(\frac{2^{-k}}{b}+\frac{a}{2^{-k}}\right)\left(1+C\np{\D B}{2}{\Omega}\right)\np{\D A}{2}{\Omega}\nonumber\\
        &+\frac{C}{(1-\alpha)^{\frac{5}{2}}}\np{\D \widetilde{B}}{2}{B(0,1)}\left(\sum_{l=N_1}^{N_2}\frac{1}{2^{2\alpha|l-k+1|}}\int_{A_l}|\D \widetilde{A}|^2dx\right)^{\frac{1}{2}}+\sqrt{2\pi\log(2)}\frac{\Lambda}{\log\left(\frac{b}{4a}\right)}\nonumber\\
        &+\frac{C}{\log\left(\frac{b}{4a}\right)}\left(1+\np{\D B}{2}{\Omega}\right)\np{\D A}{2}{\Omega}\nonumber\\
        &\leq C\left(\frac{2^{-k}}{b}+\frac{a}{2^{-k}}\right)\left(1+\np{\D B}{2}{\Omega}\right)\np{\D A}{2}{\Omega}\nonumber\\
        &+\frac{C}{(1-\alpha)^{\frac{5}{2}}}\np{\D {B}}{2}{\Omega}\left(\sum_{l=N_1}^{N_2}\frac{1}{2^{2\alpha|l-k+1|}}\int_{A_l}|\D \widetilde{A}|^2dx\right)^{\frac{1}{2}}\nonumber\\
        &+\sqrt{2\pi\log(2)}\frac{\Lambda}{\log\left(\frac{b}{4a}\right)}
        +\frac{C}{\log\left(\frac{b}{4a}\right)}\left(1+\np{\D B}{2}{\Omega}\right)\np{\D A}{2}{\Omega},
    \end{align}
    where we used $0<b<1$, to bound the first term, and thanks to \eqref{lambda6} that
    \begin{align*}
        \np{\D(\psi-\psi_0)}{2}{A_k}&=\sqrt{2\pi\,d^2\int_{2^{-(k+1)}}^{2^{-k}}\frac{d\rho}{\rho}}=\sqrt{2\pi\log(2)}\,|d|\\
        &\leq \sqrt{2\pi\log(2)}\frac{\Lambda}{\log\left(\frac{b}{4a}\right)}+\frac{C}{\log\left(\frac{b}{4a}\right)}\left(1+\np{\D B}{2}{\Omega}\right)\np{\D A}{2}{\Omega}
    \end{align*}
    Now, let us prove a variant of \cite[Lemma G.$1$]{riviere_morse_scs}.
    \begin{lemme}
        Let $0<s<t<1$, $N_1,N_2\in\N\cup\ens
        {\infty}$, and $\ens{a_k}_{k\in\N},\ens{b_k}_{n\in\N}\subset \R_+$ be two sequences. Assume that the following inequality holds for all $N_1\leq k\leq N_2$
        \begin{align}\label{G_1_1}
            a_k\leq b_k+\Gamma\left(\sum_{n=0}^{\infty}s^{|n-k+1|}a_n^2\right)^{\frac{1}{2}},
        \end{align}
        where $\Gamma<\infty$ is a fixed real number independent of $k\in\N$. Then, we have for all $N_1\leq k\leq N_2$
        \begin{align}\label{G_1_2}
            \left(\sum_{l=N_1}^{N_2}t^{|l-k+1|}a_l^2\right)^{\frac{1}{2}}\leq \left(\sum_{l=N_1}^{N_2}t^{|l-k+1|}b_l^2\right)^{\frac{1}{2}}+2\,\Gamma\left(\frac{1}{1-s\,t}+\frac{1}{1-s\,t^{-1}}\right)\left(\sum_{l=0}^{\infty}t^{|l-k+1|}a_l^2\right)^{\frac{1}{2}}.
        \end{align}
    \end{lemme}
    \begin{proof}
        Thanks to Minkowski's inequality and Fubini's theorem, we have
        \begin{align*}
            \left(\sum_{k=N_1}^{N_2}t^{|l-k+1|}a_l^2\right)^{\frac{1}{2}}&\leq \left(\sum_{l=N_1}^{N_2}t^{|l-k+1|}b_k^2\right)^{\frac{1}{2}}+\Gamma\left(\sum_{l=N_1}^{N_2}t^{|l-k+1|}\sum_{n=0}^{\infty}s^{|n-k+1|}a_n^2\right)^{\frac{1}{2}}\\
            &=\left(\sum_{l=N_1}^{N_2}t^{|l-k+1|}b_k^2\right)^{\frac{1}{2}}+\Gamma\left(\sum_{n=0}^{\infty}a_n^2\sum_{l=N_1}^{N_2}s^{|n-k+1|}t^{|l-k+1|}\right)^{\frac{1}{2}},
        \end{align*}
        and thanks to the proof of \cite[Lemma G.$1$]{riviere_morse_scs}, we have
        \begin{align*}
            \sum_{l=N_1}^{N_2}s^{|n-k+1|}t^{|l-k+1|}\leq \left(\frac{2}{1-s\,t}+\frac{1}{1-s\,t^{-1}}\right)\,t^{|n-k+1|},
        \end{align*}
        which concludes the proof of the lemma.
    \end{proof}
    Since $\widetilde{A}=A$ on $\Omega$, if 
    \begin{align*}
    \left\{\begin{alignedat}{1}
        &s=\frac{1}{2^{2\alpha}}\\
        &a_k=\np{\D \widetilde{A}}{2}{A_k}\\
        &b_k=C\left(\frac{2^{-k}}{b}+\frac{a}{2^{-k}}+\frac{1}{\log\left(\frac{b}{4a}\right)}\right)\left(1+\np{\D B}{2}{\Omega}\right)\np{\D A}{2}{\Omega}+\sqrt{2\pi\log(2)}\frac{\Lambda}{\log\left(\frac{b}{4a}\right)}\\
        &\Gamma=\frac{C}{(1-\alpha)^{\frac{5}{2}}}\np{\D B}{2}{\Omega}
        \end{alignedat}\right.
    \end{align*}
    then the estimate \eqref{G_1_1} holds provided that $A_k\subset \Omega_{1/2}$, or 
    \begin{align*}
        2^{-k}\leq \frac{b}{2}\qquad \text{and}\;\, 2^{-k-1}\geq 2a,
    \end{align*}
    which gives
    \begin{align*}
        \log_2\left(\frac{2}{b}\right)\leq k\leq \log_2\left(\frac{1}{4a}\right)
    \end{align*}
    Recalling that $N_1=\left[\log_2\left(\frac{1}{4b}\right)\right]$ and $N_2=\left[\log_2\left(\frac{2}{a}\right)\right]$, we deduce that $A_k\subset \Omega_{1/2}$ if and only if $N_1+3\leq k\leq N_2-3$. Therefore, we deduce thanks to the lemma that for all $0<\beta<\alpha<1$, there holds thanks to Minkowski's inequality
    \begin{align*}
        &\left(\sum_{l=N_1+3}^{N_2-3}\frac{1}{2^{2\beta|l-k+1|}}\int_{A_l}|\D\widetilde{A}|^2dx\right)^{\frac{1}{2}}\\
        &\leq C\left(1+\np{\D B}{2}{\Omega}\right)\np{\D A}{2}{\Omega}\left(\left(\sum_{l=N_1}^{N_2}\frac{1}{2^{2\beta|l-k+1|}}\left(\frac{2^{-l}}{b}\right)^2\right)^{\frac{1}{2}}+\left(\sum_{l=N_1}^{N_2}\frac{1}{2^{2\beta|l-k+1|}}\left(\frac{a}{2^{-l}}\right)^2\right)^{\frac{1}{2}}\right)\\
        &+\left(\sum_{l=N_1}^{N_2}\frac{1}{2^{2\beta|l-k+1|}}\right)^{\frac{1}{2}}\left(\sqrt{2\pi\log(2)}\frac{\Lambda}{\log\left(\frac{b}{4a}\right)}+\frac{C}{\log\left(\frac{b}{4a}\right)}\left(1+\np{\D B}{2}{\Omega}\right)\np{\D A}{2}{\Omega}\right)\\
        &+\frac{C}{(1-\alpha)^{\frac{5}{2}}}\left(\frac{1}{\alpha+\beta}+\frac{1}{\alpha-\beta}\right)\np{\D B}{2}{\Omega}\left(\sum_{l=N_1}^{N_2}\frac{1}{2^{2\beta|l-k+1|}}\int_{A_l}|\D\widetilde{A}|^2dx\right)^{\frac{1}{2}}.
    \end{align*}
    Therefore, if we choose $\alpha=\beta+\frac{1-\beta}{2}=\frac{1+\beta}{2}$, then 
    \begin{align*}
        \frac{C}{(1-\alpha)^{\frac{5}{2}}}\left(\frac{1}{\alpha+\beta}+\frac{1}{\alpha-\beta}\right)\leq \frac{C'}{(1-\beta)^{\frac{7}{2}}},
    \end{align*}
    so for all $0<\delta<1$, assuming that 
    \begin{align}\label{bound_B_neck}
        \frac{C'}{(1-\beta)^{\frac{7}{2}}}\np{\D B}{2}{\Omega}\leq \left(1-\frac{1}{1+\delta}\right),
    \end{align}
    we deduce that
    \begin{align}\label{dyadic_end0}
        &\left(\sum_{l=N_1+3}^{N_2-3}\frac{1}{2^{2\beta|l-k+1|}}\int_{A_l}|\D \widetilde{A}|^2dx\right)^{\frac{1}{2}}\leq C(1+\delta)\left(1+\np{\D B}{2}{\Omega}\right)\np{\D A}{2}{\Omega}\nonumber\\
        &\times\left(\left(\sum_{l=N_1}^{N_2}\frac{1}{2^{2\beta|l-k+1|}}\left(\frac{2^{-l}}{b}\right)^2\right)^{\frac{1}{2}}+\left(\sum_{l=N_1}^{N_2}\frac{1}{2^{2\beta|l-k+1|}}\left(\frac{a}{2^{-l}}\right)^2\right)^{\frac{1}{2}}\right)\nonumber\\
        &+(1+\delta)\left(\sum_{l=N_1}^{N_2}\frac{1}{2^{2\beta|l-k+1|}}\right)^{\frac{1}{2}}\left(\sqrt{2\pi\log(2)}\frac{\Lambda}{\log\left(\frac{b}{4a}\right)}+\frac{C}{\log\left(\frac{b}{4a}\right)}\left(1+\np{\D B}{2}{\Omega}\right)\np{\D A}{2}{\Omega}\right)\nonumber\\
        &+\delta \left(\sum_{l=N_1}^{N_1+2}\frac{1}{2^{2\beta|l-k+1|}}\int_{A_l}|\D\widetilde{A}|^2dx+\sum_{l=N_2-2}^{N_2}\frac{1}{2^{2\beta|l-k+1|}}\int_{A_l}|\D\widetilde{A}|^2dx\right)^{\frac{1}{2}}.
    \end{align}
    From this inequality which corresponds to \cite[Proposition III.$1$]{riviere_morse_scs}, the rest of the proof is unchanged. However, for the sake of completeness, let us give the remaining details.     
    We have
    \begin{align*}
        \sum_{l=N_1}^{N_2}\frac{1}{2^{2\beta|l-k+1|}}\left(\frac{2^{-l}}{b}\right)^2=\sum_{l=N_1}^{k-1}\frac{1}{2^{2\beta(k-1-l)}}\left(\frac{2^{-l}}{b}\right)^2+\sum_{l=k}^{N_2}\frac{1}{2^{2\beta(l+1-k)}}\left(\frac{2^{-l}}{b}\right)^2.
    \end{align*}
    Recalling that $N_1=\left[\log_2\left(\frac{1}{4b}\right)\right]$, we get
    \begin{align}\label{dyadic_end1}
        \sum_{l=N_1}^{k-1}\frac{1}{2^{2\beta(k-1-l)}}\left(\frac{2^{-l}}{b}\right)^2&=\frac{1}{2^{2\beta(k-1)}}\frac{1}{b^2}\sum_{l=N_1}^{k-1}\frac{1}{2^{2(1-\beta)l}}=\frac{1}{2^{2\beta(k-1)}}\frac{1}{b^2}\frac{1}{2^{2(1-\beta)N_1}}\sum_{l'=0}^{k--N_1-1}\frac{1}{2^{2(1-\beta)l'}}\nonumber\\
        &\leq \frac{1}{2^{2\beta(k-1)}}\frac{1}{b^2}\frac{1}{2^{2(1-\beta)N_1}}\frac{2^{2(1-\beta)}}{2^{2(1-\beta)}-1}=\frac{1}{2^{2\beta(k-1)}}\frac{1}{b^2}\frac{1}{2^{2(1-\beta)}\log_2\left(\frac{1}{4b}\right)}\frac{2^{2(1-\beta)}}{2^{2(1-\beta)}-1}\nonumber\\
        &=\frac{1}{2^{2\beta(k-1)}}\frac{1}{b^2}(4b)^{2(1-\beta)}\frac{2^{2(1-\beta)}}{2^{2(1-\beta)}-1}=\frac{2^{2(3-2\beta)}}{2^{2(1-\beta)}-1}\left(\frac{2^{-k}}{b}\right)^{2\beta}.
    \end{align}
    Then, trivially have
    \begin{align}\label{dyadic_end2}
        \sum_{l=k}^{N_2}\frac{1}{2^{2\beta(l+1-k)}}\left(\frac{2^{-l}}{b}\right)^2&\leq \left(\frac{2^{-k}}{b}\right)^2\sum_{l'=0}^{\infty}\frac{1}{2^{2(1+\beta) l'}}
        =\frac{2^{2(1+\beta)}}{2^{2(1+\beta)}-1}\left(\frac{2^{-k}}{b}\right)^2\leq \frac{2^{2(1+\beta)}}{2^{2(1+\beta)}-1}\left(\frac{2^{-k}}{b}\right)^{2\beta}.
    \end{align}
    since $0<\beta<1$. Likewise, we have
    \begin{align}\label{dyadic_end3}
        &\sum_{l=N_1}^{k-1}\frac{1}{2^{2\beta(k-1-l)}}\left(\frac{a}{2^{-l}}\right)^2=a^2\frac{1}{2^{2\beta(k-1)}}2^{2\beta N_1}\sum_{l'=0}^{k-N_1-1}2^{2(1+\beta)l'}=a^2\frac{2^{2\beta}}{2^{2\beta(k-N_1)}}\frac{2^{(1+\beta)(k-N_1)}-1}{2^{2(1+\beta)}-1}\nonumber\\
        &\leq \frac{2^{2\beta}}{2^{2(1+\beta)}-1}\left(\frac{a}{2^{-k}}\right)^2\leq \frac{4}{3}\left(\frac{a}{2^{-k}}\right)^{2\beta}.
    \end{align}
    Finally, since $N_2=\left[\log_2\left(\frac{2}{a}\right)\right]$, we have the estimate
    \begin{align}\label{dyadic_end4}
        \sum_{l=k}^{N_2}\frac{1}{2^{2\beta(l-k+1)}}\left(\frac{a}{2^{-l}}\right)^{2}&=a^22^{2\beta(k-1)}\sum_{l=k}^{N_2}2^{2(1-\beta)l}=a^22^{2\beta(k-1)}2^{2(1-\beta)k}\frac{2^{2(1-\beta)(N_2-k+1)}-1}{2^{2(1-\beta)}-1}\nonumber\\
        &\leq \frac{2^{2(1-2\beta)}}{2^{2(1-\beta)}-1}a^22^{2\beta k}2^{2(1-\beta)N_2}=\frac{2^{2(1-2\beta)}}{2^{2(1-\beta)}-1}a^22^{2\beta k}\left(\frac{2}{a}\right)^{2(1-\beta)}\nonumber\\
        &=\frac{2^{2(2-3\beta)}}{2^{2(1-\beta)}-1}\left(\frac{a}{2^{-k}}\right)^{2\beta}.
    \end{align}
    Therefore, thanks to \eqref{dyadic_end1}, \eqref{dyadic_end2}, \eqref{dyadic_end3}, \eqref{dyadic_end4}, and the elementary inequality $\sqrt{x+y}\leq \sqrt{x}+\sqrt{y}$, we deduce that
    \begin{align}\label{dyadic_end5}
        &\left(\sum_{l=N_1}^{N_2}\frac{1}{2^{2\beta|l-k+1|}}\left(\frac{2^{-l}}{b}\right)^2\right)^{\frac{1}{2}}+\left(\sum_{l=N_1}^{N_2}\frac{1}{2^{2\beta|l-k+1|}}\bigg(\frac{a}{2^{-l}}\bigg)^2\right)^{\frac{1}{2}}\nonumber\\
        &\leq \left(\frac{2^{\beta}}{\sqrt{2^{2(1-\beta)}-1}}+\frac{2^{3-2\beta}}{\sqrt{2^{2(1-\beta)}-1}}\right)\left(\left(\frac{2^{-k}}{b}\right)^{\beta}+\bigg(\frac{a}{2^{-k}}\bigg)^{\beta}\right)\nonumber\\
        &\leq \left(2+4\sqrt{\frac{2}{\log(2)}}\frac{1}{\sqrt{1-\beta}}\right)\left(\left(\frac{2^{-k}}{b}\right)^{\beta}+\bigg(\frac{a}{2^{-k}}\bigg)^{\beta}\right),
    \end{align}
    where we used the elementary inequality $2^{2(1-\beta)}\geq 1+2\log(2)(1-\beta)$. Then, we trivially have
    \begin{align*}
        \left(\sum_{l=N_1}^{N_2}\frac{1}{2^{2\beta|l-k+1|}}\right)^{\frac{1}{2}}\leq \sqrt{\frac{2^{2\beta}}{2^{2\beta}-1}}\leq \sqrt{\frac{2}{\log(2)}}\frac{1}{\sqrt{\beta}}.
    \end{align*}
    Then, we can estimate
    \begin{align}\label{dyadic_end6}
        \sum_{k=N_1}^{N_1+2}\frac{1}{2^{2\beta|l-k+1|}}\int_{A_l}|\D \widetilde{A}|^2dx\leq \frac{3}{2^{2\beta(k-1)}}2^{2\beta N_1}\int_{B_{2a}\setminus\bar{B}_{a}(0)}|\D A|^2dx=3\cdot 2^{2\beta}\left(\frac{2^{-k}}{b}\right)^{2\beta}\int_{B_{2a}\setminus\bar{B}_a(0)}|\D A|^2dx,
    \end{align}
    where we used the extra property from the extension Lemma \ref{whitney_extension}. Finally, we have
    \begin{align}\label{dyadic_end7}
        &\sum_{l=N_2-2}^{N_2}\frac{1}{2^{2\beta|l-k-1|}}\int_{A_l}|\D \widetilde{A}|^2dx\leq 3\cdot 2^{2\beta(k-1)}\frac{1}{2^{2\beta(N_2-2)}}\int_{B_b\setminus\bar{B}_{\frac{b}{2}}(0)}|\D A|^2dx\nonumber\\
        &=3\cdot 2^{2\beta(k-1)}2^{4\beta}\left(\frac{a}{2}\right)^{2\beta}\int_{B_b\setminus\bar{B}_{\frac{b}{2}}(0)}|\D A|^2dx=3\bigg(\frac{a}{2^{-k}}\bigg)^{2\beta}\int_{B_b\setminus\bar{B}_{\frac{b}{2}}(0)}|\D A|^2dx
    \end{align}
    Gathering \eqref{dyadic_end0}, \eqref{dyadic_end5}, \eqref{dyadic_end6}, and \eqref{dyadic_end7}, we deduce that 
    \begin{align}\label{dyadic_end8}
        &\left(\sum_{l=N_1}^{N_2}\frac{1}{2^{2\beta|l-k+1|}}\int_{A_l}|\D\widetilde{A}|^2dx\right)^{\frac{1}{2}}\leq \left(\sum_{l=N_1}^{N_2}\frac{1}{2^{2\beta|l-k+1|}}\int_{A_l}|\D\widetilde{A}|^2dx\right)^{\frac{1}{2}}\nonumber\\
        &+\left(\sum_{l=N_1}^{N_1+2}\frac{1}{2^{2\beta|l-k+1|}}\int_{A_l}|\D\widetilde{A}|^2dx+\sum_{l=N_2-2}^{N_2}\frac{1}{2^{2\beta|l-k+1|}}\int_{A_l}|\D\widetilde{A}|^2dx\right)^{\frac{1}{2}}\nonumber\\
        &\leq C(1+\delta)\left(2+4\sqrt{\frac{2}{\log(2)}}\frac{1}{\sqrt{1-\beta}}\right)\left(\left(\frac{2^{-k}}{b}\right)^{\beta}+\bigg(\frac{a}{2^{-k}}\bigg)^{\beta}\right)\left(1+\np{\D B}{2}{\Omega}\right)\np{\D A}{2}{\Omega}\nonumber\\
        &+(1+\delta)\sqrt{\frac{2^{2\beta}}{2^{2\beta}-1}}\left(\sqrt{2\pi\log(2)}\frac{\Lambda}{\log\left(\frac{b}{4a}\right)}+\frac{C}{\log\left(\frac{b}{4a}\right)}\left(1+\np{\D B}{2}{\Omega}\right)\np{\D A}{2}{\Omega}\right)\nonumber\\
        &+(1+\delta)\sqrt{3}(1+2^{\beta})\left(\left(\frac{2^{-k}}{b}\right)^{\beta}+\bigg(\frac{a}{2^{-k}}\bigg)^{\beta}\right)\left(\np{\D A}{2}{B_b\setminus\bar{B}_{\frac{b}{2}}(0)}+\np{\D B}{2}{B_{2a}\setminus\bar{B}_a(0)}\right).
    \end{align}
    Now, coming back to \eqref{collar3}, and applying the inequality for $\alpha=\beta$ (this step explains why we had to precisely keep track of constants), we finally obtain
    \begin{align*}
        &\np{\D A}{2}{A_k}\leq C\left\{\left(\frac{2^{-k}}{b}+\frac{a}{2^{-k}}\right)+\frac{1}{(1-\beta)^3}\np{\D B}{2}{\Omega}\left(\left(\frac{2^{-k}}{b}\right)^{\beta}+\bigg(\frac{a}{2^{-k}}\bigg)^{\beta}\right)\right\}\\
        &\times \left(1+\np{\D B}{2}{\Omega}\right)\np{\D A}{2}{\Omega}
        +\left(1+\frac{C}{(1-\beta)^{\frac{5}{2}}}\np{\D B}{2}{\Omega}\sqrt{\frac{2^{2\beta}}{2^{2\beta}-1}}\right)\left(\sqrt{2\pi\log(2)}\frac{\Lambda}{\log\left(\frac{b}{4a}\right)}\right.\\
        &\left.+\frac{C}{\log\left(\frac{b}{4a}\right)}\left(1+\np{\D B}{2}{\Omega}\right)\np{\D A}{2}{\Omega}\right)\\
        &\leq C\left(\left(\frac{2^{-k}}{b}+\frac{a}{2^{-k}}\right)+\sqrt{1-\beta}\left(1-\frac{1}{1+\delta}\right)\left(\left(\frac{2^{-k}}{b}\right)^{\beta}+\bigg(\frac{a}{2^{-k}}\bigg)\right)\right)\left(1+\np{\D B}{2}{\Omega}\right)\np{\D A}{2}{\Omega}\\
        &+\left(1+C\sqrt{1-\beta}\left(1-\frac{1}{1+\delta}\right)\right)\left(\sqrt{2\pi\log(2)}\frac{\Lambda}{\log\left(\frac{b}{4a}\right)}+\frac{C}{\log\left(\frac{b}{4a}\right)}\left(1+\np{\D B}{2}{\Omega}\right)\np{\D A}{2}{\Omega}\right),
    \end{align*}
    where we used \eqref{bound_B_neck}.
    Finally, using the elementary inequality 
    \begin{align*}
        \frac{\delta}{2}\leq 1-\frac{1}{1+\delta}\leq \delta \quad \text{for all}\;\, 0<\delta<1,
    \end{align*}
    we deduce by \eqref{bound_B_neck} that there exists universal constants $C,\Gamma<\infty$ such that 
    \begin{align*}
        \np{\D B}{2}{\Omega}\leq \Gamma\,\delta(1-\beta)^{\frac{7}{2}}
    \end{align*}
    implies that 
    \begin{align*}
        \np{\D A}{2}{A_k}&\leq C\left(\left(\frac{2^{-k}}{b}+\frac{a}{2^{-k}}\right)+\delta\sqrt{1-\beta}\left(\left(\frac{2^{-k}}{b}\right)^{\beta}+\bigg(\frac{a}{2^{-k}}\bigg)^{\beta}\right)\right)\left(1+\np{\D B}{2}{\Omega}\right)\np{\D A}{2}{\Omega}\\
        &+\left(1+C\delta\sqrt{1-\beta}\right)\left(\sqrt{2\pi\log(2)}\frac{\Lambda}{\log\left(\frac{b}{4a}\right)}+\frac{C}{\log\left(\frac{b}{a}\right)}\left(1+\np{\D B}{2}{\Omega}\right)\np{\D A}{2}{\Omega}\right),
    \end{align*}
    that we can also rewrite (for all $z\in \Omega_{1/2}$) as
    \begin{align}\label{ineq_neck_optimal}
        &\np{\D A}{2}{B_{2|z|}\setminus\bar{B}_{|z|}(0)}\leq C\left(\left(\left(\frac{|z|}{b}\right)+\frac{a}{|z|}\right)+\delta\sqrt{1-\beta}\left(\left(\frac{|z|}{b}\right)^{\beta}+\left(\frac{a}{|z|}\right)^{\beta}\right)\right)\np{\D A}{2}{\Omega}\nonumber\\
        &+\left(1+C\delta\sqrt{1-\beta}\right)\left(\sqrt{2\pi\log(2)}\,\Lambda+C\np{\D A}{2}{\Omega}\right),
    \end{align}
    which finally concludes the proof of the theorem. 
    \end{proof}
    \begin{rem}
    \begin{enumerate}
        \item[($1$)] For $B=0$, we recover the estimate for harmonic functions from Corollary \eqref{ineq_harmonic_average}. Indeed, if 
        \begin{align*}
            u(z)=d\,\log|z|+\Re\left(\sum_{k\in\Z}a_kz^k\right)=d\,\log|z|+u_0(z),
        \end{align*}
        we can directly estimate thanks to Lemma \ref{improved_reg_harmonic}
        \begin{align*}
            \int_{\Omega_{1/2}}\frac{|\D u_0|}{|x|}dx\leq C\np{\D u_0}{2}{\Omega}\leq C\np{\D u}{2}{\Omega},
        \end{align*}
        while
        \begin{align*}
            |d|\,\log\left(\frac{b}{4a}\right)=\frac{1}{2\pi}\int_{\Omega_{1/2}}\frac{|\D(u-u_0)|}{|x|}dx,
        \end{align*}
        which implies the general bound that includes the logarithm component. 
        \item[($2$)] Another reason why we were careful in the proof is to obtain an optimal constant for the upper semi-continuity result in the case of collars. Let us take the example of harmonic maps into $S^n$ to make the result clearer. If $u:\Sigma\rightarrow S^{n-1}$ is a harmonic map, thanks to the conservation law
        \begin{align*}
            \dive(u_i\D u_j-u_j\D u_i)=0\quad  \text{for all}\;\, 1\leq i,j\leq n
        \end{align*}
        we can (locally) rewrite the equation $-\Delta u=|\D u|^2u$ as
        \begin{align*}
            -\Delta u=\D^{\perp}B\cdot \D u,
        \end{align*}
        where $\mathrm{curl}B_{i,j}=u_i\D u_j-u_j\D u_i$ for all $1\leq i,j\leq n$. Now, let us define the optimal constant in the $\epsilon$-regularity on annuli. We have the following result (\cite{helein}) that follows from the standard one on balls.
        \begin{theorem}
            Let $0<a<b<\infty$, $\Omega=B_b\setminus\bar{B}_a(0)$, and $u:\Omega\rightarrow S^{n-1}$ be a harmonic map. There exists $\epsilon_0(n)>0$ and a universal constant $C(n)<\infty$ such that
            \begin{align}
                \int_{\Omega}|\D u|^2dx\leq \epsilon_0(n)
            \end{align}
            implies that for all $2a<|z|<b/2$
            \begin{align}
                |z||\D u(z)|\leq C(n)\np{\D u}{2}{B_{3|z|/2}\setminus\bar{B}_{|z|/2}(0)}.
            \end{align}
        \end{theorem}
        Recall that the second derivative of the Dirichlet energy for harmonic maps into spheres is given by
        \begin{align*}
            Q_u(w)=\int_{\Sigma}\left(|dw|^2_g-|du|_g^2|w|^2\right)d\vg.
        \end{align*}
        Therefore, if $C(n)<\infty$ is the optimal constant in this inequality, the proof of \cite[Lemma V.$1$]{riviere_morse_scs} and \cite[Lemma IV.$1$]{riviere_morse_scs} (see Theorem \ref{eigenvalue_harmonic_annulus} above) shows that $\Lambda^{\ast}$ can be estimated from below by
        \begin{align*}
            \Lambda^{\ast}\geq \Lambda,
        \end{align*}
        where $\Lambda$ is such that
        \begin{align*}
            \Lambda\sqrt{2\pi\log(2)}C(n)\leq \pi,
        \end{align*}
        which gives us the estimate
        \begin{align*}
            \Lambda^{\ast}\geq \frac{1}{C(n)}\sqrt{\frac{\pi}{2\log(2)}}.
        \end{align*}
        In other words, if $\mathcal{L}_k$ is the maximum length of the \emph{images} of collars by a sequence of harmonic maps $\ens{u_k}_{k\in\N}$ from $\Sigma$ into $S^{n-1}$ is such that
        \begin{align*}
            \limsup_{k\rightarrow \infty}\mathcal{L}_k<\frac{1}{C(n)}\sqrt{\frac{\pi}{2\log(2)}},
        \end{align*}
        then \cite[Theorem I.$2$]{riviere_morse_scs} holds true (assuming that the other hypotheses of the theorem hold).
        We believe that this estimate is optimal. Due to the instability of geodesics of length larger than $\pi$ on the sphere $S^2$, we expect that 
        \begin{align*}
            C(2)=\sqrt{2\pi\log(2)}=2.08690\cdots
        \end{align*}
        Notice that this estimate is verified in all cases covered in \cite{riviere_morse_scs} (although it will be slightly modified depending on the shape of the second derivative) and works for Willmore immersions too.
    \end{enumerate}
    \end{rem}

    \subsection{Extension of the Results for General Divergence Equations}

    In this section, we extend the previous results to an equation of the form
    \begin{align*}
        \Delta u=\D^{\perp}B\cdot \D A+\dive(C),
    \end{align*}
    where $C\in W^{1,2}(B(0,1),M_{n, 2}(\R))$, and $\dive C=(\dive C_1,\cdots,\dive C_n)\in \R^n$, where $C_i\in W^{1,2}(B(0,1),\R^2)$. In the case of collar regions, the non-trivial residue makes such a divergence quantity in the weak Willmore equation (\cite{lauriv1}). 
    
    Let us start by proving a variant of \cite[Lemma E.$1$]{riviere_morse_scs} for another type of divergence equations.

    \begin{lemme}\label{cz_lemme_F1}
        Let $C\in L^2(B(0,1),\R^2)$, and let $u:B(0,1)\rightarrow \R$ be a solution of the equation:
        \begin{align}\label{cz_system}
        \left\{\begin{alignedat}{2}
            \Delta u&=\dive(C)\qquad&&\text{in}\;\, B(0,1)\\
            u&=0\qquad&&\text{on}\;\,\partial B(0,1).
            \end{alignedat}\right.
        \end{align}
        Then, for all $\alpha>2\sqrt{2}=2.828\cdots$, we have
        \begin{align*}
            \int_{B(0,1)}|x|^{\alpha}\frac{u^2}{|x|^2}dx+\int_{B(0,1)}|x|^{\alpha}|\D u|^2dx\leq 25\left(\frac{\alpha^2+1}{\alpha^2-8}\right)^2\int_{B(0,1)}|x|^{\alpha}|C|^2dx.
        \end{align*}
    \end{lemme}
    \begin{proof}
        First, by \cite[Lemma $5.6$]{eigenvalue_annuli}, for all $\beta>-2$, we have
        \begin{align}\label{id_ipp}
            \int_{B(0,1)}|x|^{\beta}u^2dx=-\frac{2}{\beta+2}\int_{B(0,1)}|x|^{\beta}u\,(x\cdot \D u)dx,
        \end{align}
        which also implies by Cauchy-Schwarz inequality that
        \begin{align}\label{basic_id}
            \int_{B(0,1)}|x|^{\beta}u^2dx\leq \frac{4}{(\beta+2)^2}\int_{B(0,1)}|x|^{\beta}\left(x\cdot \D u\right)^2dx\leq \frac{4}{(\beta+2)^2}\int_{B(0,1)}|x|^{\beta+2}|\D u|^2dx.
        \end{align}
        Now, integrating by parts, we have
        \begin{align}\label{basic_id2}
            \int_{B(0,1)}|x|^{\beta+2}|\D u|^2dx&=-\int_{B(0,1)}|x|^{\beta+2}u\,\Delta u\,dx-(\beta+2)\int_{B(0,1)}|x|^{\beta}u(x\cdot \D u)dx\nonumber\\
            &=-\int_{B(0,1)}|x|^{\beta+2}u\,\Delta u\,dx+2\int_{B(0,1)}|x|^{\beta}u^2dx.
        \end{align}
        Using the Euler-Lagrange equation and integrating by parts again, we deduce that
        \begin{align}\label{basic_id3}
            -\int_{B(0,1)}|x|^{\beta+2}u\,\Delta u\,dx&=-\int_{B(0,1)}|x|^{\beta+2}u\,\dive(C)\,dx\nonumber\\
            &=(\beta+2)\int_{B(0,1)}|x|^{\beta}u(x\cdot C)dx+\int_{B(0,1)}|x|^{\beta+2}\D u\cdot C\,dx.
        \end{align}
        Therefore, by \eqref{basic_id2}, \eqref{basic_id3}, and Cauchy's inequality, we deduce that
        \begin{align*}
            &\int_{B(0,1)}|x|^{\beta+2}|\D u|^2dx=(\beta+2)\int_{B(0,1)}|x|^{\beta}u(x\cdot C)\,dx+\int_{B(0,1)}|x|^{\beta+2}\D u\cdot C+2\int_{B(0,1)}|x|^{\beta}u^2dx\\
            &\leq \epsilon\int_{B(0,1)}|x|^{\beta+2}|\D u|^2dx+\frac{1}{4\epsilon}\int_{B(0,1)}|x|^{\beta+2}|C|^2dx+(\beta+2)\int_{B(0,1)}|x|^{\beta}u(x\cdot C)\,dx+2\int_{B(0,1)}|x|^{\beta}u^2dx,
        \end{align*}
        which shows by another application of Cauchy's inequality that
        \begin{align}\label{basic_id4}
            \int_{B(0,1)}|x|^{\beta+2}|\D u|^2dx&\leq \frac{1}{4\epsilon(1-\epsilon)}\int_{B(0,1)}|x|^{\beta+2}|C|^2dx+\frac{\beta+2}{1-\epsilon}\int_{B(0,1)}|x|^{\beta}u(x\cdot C)\,dx\nonumber\\
            &+\frac{2}{1-\epsilon}\int_{B(0,1)}|x|^{\beta}u^2dx\nonumber\\
            &\leq \frac{1+(\beta+2)^2}{4\epsilon(1-\epsilon)}\int_{B(0,1)}|x|^{\beta+2}|C|^2dx+\frac{2+\epsilon}{1-\epsilon}\int_{B(0,1)}|x|^{\beta}u^2dx.
        \end{align}
        Putting together \eqref{basic_id} and \eqref{basic_id4}, we finally deduce that
        \begin{align}\label{basic_id5}
            \int_{B(0,1)}|x|^{\beta}u^2dx\leq \frac{1+(\beta+2)^2}{\epsilon(1-\epsilon)(\beta+2)^2}\int_{B(0,1)}|x|^{\beta+2}|C|^2dx+\frac{2+\epsilon}{1-\epsilon}\frac{4}{(\beta+2)^2}\int_{B(0,1)}|x|^{\beta}u^2dx.
        \end{align}
        Since 
        \begin{align*}
            \frac{8}{(\beta+2)^2}<1
        \end{align*}
        if and only if $\beta>2(\sqrt{2}-1)$, and $0<\epsilon<1$ is arbitrary, we deduce that for all $\beta>2(\sqrt{2}-1)$, there exists $0<\epsilon_{\beta}<1$ such that
        \begin{align*}
            \frac{2+\epsilon}{1-\epsilon}\frac{4}{(\beta+2)^2}<1.
        \end{align*}
        Explicitly, we can choose
        \begin{align*}
            \epsilon_{\beta}=\frac{1}{2}\frac{(\beta+2)^2-8}{4+(\beta+2)^2},
        \end{align*}
        which yields by \eqref{basic_id5} the inequality
        \begin{align}\label{basic_id6}
            \int_{B(0,1)}|x|^{\beta}u^2dx\leq \frac{1+(\beta+2)^2}{\epsilon_{\beta}(1-\epsilon_{\beta})(\beta+2)^2}\left(1-\frac{2+\epsilon_{\beta}}{1-\epsilon_{\beta}}\frac{4}{(\beta+2)^2}\right)^{-1}\int_{B(0,1)}|x|^{\beta+2}|C|^2dx.
        \end{align}
        Make the change of variable $\alpha=\beta+2$. Let us estimate the constant explicitly. We have
        \begin{align*}
            \epsilon=\frac{1}{2}\frac{\alpha^2-8}{\alpha^2+4}, 
        \end{align*}
        which implies that
        \begin{align*}
            &1-\epsilon=\frac{1}{2}\frac{2(\alpha^2+4)-(\alpha^2-8)}{\alpha^2+4}=\frac{1}{2}\frac{\alpha^2+16}{\alpha^2+4}\\
            &\frac{1}{1-\epsilon}=\frac{2(\alpha^2+4)}{\alpha^2+16},
        \end{align*}
        while
        \begin{align*}
            2+\epsilon=\frac{1}{2}\frac{4(\alpha^2+4)+\alpha^2-8}{\alpha^2+4}=\frac{1}{2}\frac{5\alpha^2+8}{\alpha^2+4}
        \end{align*}
        so that
        \begin{align*}
            \frac{2+\epsilon}{1-\epsilon}=\frac{5\alpha^2+8}{\alpha^2+16}.
        \end{align*}
        Therefore, we have
        \begin{align*}
            \frac{2+\epsilon}{1-\epsilon}\frac{4}{\alpha^2}=\frac{4(5\alpha^2+8)}{\alpha^2(\alpha^2+16)},
        \end{align*}
        and finally
        \begin{align*}
            1-\frac{2+\epsilon}{1-\epsilon}\frac{4}{\alpha^2}=\frac{\alpha^4-4\alpha^2-32}{\alpha^2(\alpha^2+16)}=\frac{(\alpha^2-8)(\alpha^2+4)}{\alpha^2(\alpha^2+16)}
        \end{align*}
        Then, we have
        \begin{align*}
            \epsilon(1-\epsilon)=\frac{1}{4}\frac{(\alpha^2+16)(\alpha^2-8)}{(\alpha^2+4)^2}\\
            \frac{1+\alpha^2}{\epsilon(1-\epsilon)\alpha^2}=\frac{4(\alpha^2+1)(\alpha^2+4)^2}{\alpha^2(\alpha^2+16)(\alpha^2-8)},
        \end{align*}
        and finally
        \begin{align*}
            \frac{1+\alpha^2}{\epsilon(1-\epsilon)\alpha^2}\left(1-\frac{2+\epsilon}{1-\epsilon}\frac{4}{\alpha^2}\right)^{-1}&=\frac{4(\alpha^2+1)(\alpha^2+4)^2}{\alpha^2(\alpha^2+16)(\alpha^2-8)}\times \frac{\alpha^2(\alpha^2+16)}{(\alpha^2+4)(\alpha^2-8)}\\
            &=\frac{4(\alpha^2+1)(\alpha^2+4)}{(\alpha^2-8)^2},
        \end{align*}
        which shows that 
        \begin{align}\label{basic_id6bis}
            \int_{B(0,1)}|x|^{\alpha-2}u^2dx\leq \frac{4(\alpha^2+1)(\alpha^2+4)}{(\alpha^2-8)^2}\int_{B(0,1)}|x|^{\alpha}|C|^2dx, 
        \end{align}
        and finally, \eqref{basic_id4} and \eqref{basic_id6bis} show that
        \begin{align}\label{basic_id7}
            \int_{B(0,1)}|x|^{\alpha}|\D u|^2dx&\leq \left(\frac{(\alpha^2+1)(\alpha^2+4)^2}{\alpha^2(\alpha^2+16)(\alpha^2-8)}+\frac{5\alpha^2+8}{\alpha^2+16}\times \frac{4(\alpha^2+1)(\alpha^2+4)}{(\alpha^2-8)^2}\right)\int_{B(0,1)}|x|^{\alpha}|C|^2dx\nonumber\\
            &=\frac{(\alpha^2+1)(\alpha^2+4)}{\alpha^2(\alpha^2+16)(\alpha^2-8)^2}\left((\alpha^2+4)(\alpha^2-8)+4\alpha^2(5\alpha^2+8)\right)\int_{B(0,1)}|x|^{\alpha}|C|^2dx\nonumber\\
            &\leq \frac{(\alpha^2+1)(\alpha^2+4)}{\alpha^2(\alpha^2+16)(\alpha^2-8)^2}\left(21\alpha^4+28\alpha^2-32\right)\int_{B(0,1)}|x|^{\alpha}|C|^2dx.
        \end{align}
        Therefore, \eqref{basic_id6} and \eqref{basic_id7} show that
        \begin{align*}
            \int_{B(0,1)}|x|^{\alpha}\frac{u^2}{|x|^2}dx+\int_{B(0,1)}|x|^{\alpha}|\D u|^2dx\leq \frac{(\alpha^2+1)(\alpha^2+4)}{\alpha^2(\alpha^2+16)(\alpha^2-8)^2}\left(25\alpha^4+102\alpha^2-32\right)\int_{B(0,1)}|x|^{\alpha}|C|^2dx.
        \end{align*}
        Since
        \begin{align*}
            (\alpha^2+4)(25\alpha^4+102\alpha^2)= 25\alpha^6+202\alpha^4+376\alpha^2\leq  25\alpha^6+425\alpha^4+400\alpha^2=25(\alpha^2+1)\alpha^2(\alpha^2+16),
        \end{align*}
        the proof is complete. 
    \end{proof}
    \begin{rem}
        The same proof combined with the $L^{\infty}$ Wente estimate shows that for all $a,b\in W^{1,2}(B(0,1))$, if $u$ solves
        \begin{align}
            \left\{\begin{alignedat}{2}
            \Delta u&=\D^{\perp}a\cdot \D b\qquad&&\text{in}\;\, B(0,1)\\
            u&=0\qquad&&\text{on}\;\,\partial B(0,1),
            \end{alignedat}\right.
        \end{align}
        then for all $\alpha>2\sqrt{2}$ and for all $0\leq \beta\leq 1$, we have
        \begin{align*}
            \int_{B(0,1)}|x|^{\alpha}|\D u|^2dx\leq C_{\alpha}\np{\D a}{2}{\Omega}\np{\D b}{2}{\Omega}\left(\int_{B(0,1)}|x|^{2\alpha\beta}|\D a|^2dx\right)^{\frac{1}{2}}\left(\int_{B(0,1)}|x|^{2\alpha(1-\beta)}|\D b|^2dx\right)^{\frac{1}{2}}
        \end{align*}
    \end{rem}

    We can now extend Lemma \ref{lemme_F_1}.

    \begin{lemme}\label{lemme_F_1_gen}
        Let $A\in W^{1,2}(B(0,1),\R^n)$, $B\in W^{1,2}(B(0,1),M_n(\R))$, and $C\in L^2(B(0,1),M_{n,2}(\R))$ and assume that $u:B(0,1)\rightarrow \R^n$ solves the equation
        \begin{align}
            \Delta u=\D^{\perp}B\cdot \D A+\dive(C)\qquad\text{in}\;\, B(0,1).
        \end{align}
        Then, there exists a universal constant $\Gamma<\infty$ such that for all $\beta>\sqrt{2}$, there holds
        \begin{align}
            \np{\D u}{2}{B_{1/2}\setminus\bar{B}_{1/4}(0)}&\leq \frac{1}{2}\np{\D u}{2}{B_1\setminus\bar{B}_{1/2}(0)}+\Gamma\np{\D A}{2}{B(0,1)}\np{\omega \D B}{2}{B(0,1)}\nonumber\\
            &+2^{2\beta+4}\frac{\beta^2+1}{\beta^2-2}\np{|x|^{\beta}C}{2}{B(0,1)},
        \end{align}
        where $\omega(x)=|x|\log\left(1+\frac{1}{|x|}\right)\sqrt{\log\left(1+\log\left(\frac{1}{|x|}\right)\right)}$.
    \end{lemme}
    \begin{proof}
        Make a decomposition $u=\varphi+\psi+\chi$, where 
        \begin{align*}
            \left\{\begin{alignedat}{2}
                \Delta \varphi&=\D^{\perp}A\cdot \D B\qquad&&\text{in}\;\, B(0,1)\\
                u&=0\qquad&&\text{on}\;\,\partial B(0,1),
            \end{alignedat}\right.
        \end{align*}
        and
        \begin{align*}
            \left\{\begin{alignedat}{2}
                \Delta \chi&=\dive(C)\qquad&&\text{in}\;\, B(0,1)\\
                u&=0\qquad&&\text{on}\;\,\partial B(0,1).
            \end{alignedat}\right.
        \end{align*}
        Since $\psi$ is a harmonic function on the ball $B(0,1)$, the inequality \eqref{ineq_harmonic_ball} shows that
        \begin{align}\label{ext_div1}
            \int_{B_{1/2}\setminus\bar{B}_{1/4}(0)}|\D \psi|^2dx\leq \frac{1}{4}\int_{B_1\setminus\bar{B}_{1/2}(0)}|\D \psi|^2dx.
        \end{align}
        On the other hand, we have by \cite[Lemma E.$1$]{riviere_morse_scs}
        \begin{align}\label{ext_div2}
            \int_{B(0,1)}|x|^2|\D \varphi|^2dx\leq C\int_{B(0,1)}|\D A|^2dx\int_{B(0,1)}|\D B|^2\omega^2dx,
        \end{align}
        which implies that
        \begin{align}\label{ext_div2bis}
            \int_{B_1\setminus\bar{B}_{1/4}(0)}|\D \varphi|^2dx\leq 16\int_{B_1\setminus\bar{B}_{1/4}(0)}|x|^2|\D \varphi|^2dx\leq 16C\int_{B(0,1)}|\D A|^2dx\int_{B(0,1)}|\D B|^2\omega^2dx.
        \end{align}
        Likewise, Lemma \ref{cz_lemme_F1} implies that for all $\beta>2\sqrt{2}$, we have
        \begin{align}\label{ext_div3}
            \int_{B_1\setminus\bar{B}_{1/4}(0)}|\D \chi|^2\leq 2^{2\beta}\int_{B(0,1)}|x|^{\beta}|\D\chi|^2dx\leq 2^{2\beta+5}\left(\frac{\beta^2+1}{\beta^2-8}\right)^2\int_{B(0,1)}|x|^{\beta}|C|^2dx.
        \end{align}
        Gathering \eqref{ext_div1}, \eqref{ext_div2bis}, and \eqref{ext_div3}, we deduce by Minkowski's inequality that
        \begin{align*}
            \np{\D u}{2}{B_{1/2}\setminus\bar{B}_{1/4}(0)}&\leq \np{\D \psi}{2}{B_{1/2}\setminus\bar{B}_{1/4}(0)}+\np{\D\varphi}{2}{B_{1/2}\setminus\bar{B}_{1/4}(0)}+\np{\D\chi}{2}{B_{1/2}\setminus\bar{B}_{1/4}(0)}\\
            &\leq \frac{1}{2}\np{\D\psi}{2}{B_{1}\setminus\bar{B}_{1/2}(0)}+\np{\D\varphi}{2}{B_{1/2}\setminus\bar{B}_{1/4}(0)}+\np{\D\chi}{2}{B_{1/2}\setminus\bar{B}_{1/4}(0)}\\
            &\leq \frac{1}{2}\np{\D u}{2}{B_1\setminus\bar{B}_{1/2}(0)}+\frac{3}{2}\np{\D\varphi}{2}{B_{1/2}\setminus\bar{B}_{1/4}(0)}+\frac{3}{2}\np{\D\chi}{2}{B_{1/2}\setminus\bar{B}_{1/4}(0)}\\
            &\leq \frac{1}{2}\np{\D u}{2}{B_1\setminus\bar{B}_{1/2}(0)}+6\sqrt{C}\np{\D A}{2}{B(0,1)}\np{\omega\D B}{2}{B(0,1)}\\
            &+\frac{3}{2}2^{\beta+\frac{5}{2}}\frac{\beta^2+1}{\beta^2-8}\np{|x|^{\frac{\beta}{2}}C}{2}{B(0,1)}.
        \end{align*}
        Replacing $\beta$ by $2\beta$, the last constant becomes
        \begin{align*}
            \frac{3}{2}2^{2\beta+\frac{5}{2}}\frac{4\beta^2+1}{4\beta^2-8}\leq 2^{2\beta+4}\frac{\beta^2+1}{\beta^2-2},
        \end{align*}
        which concludes the proof of the lemma.
    \end{proof}

    Let us now extend Lemma \ref{lemme_F_2}.

    \begin{lemme}\label{lemme_F_2_gen}
        Let $A\in W^{1,2}(B(0,1),\R^n)$, $B\in W^{1,2}(B(0,1),M_n(\R))$, and $C\in L^{2}(B(0,1),M_{n,2}(\R))$ and assume that $u:B(0,1)\rightarrow \R^n$ solves the equation
        \begin{align}
            \Delta u=\D^{\perp}B\cdot \D A+\dive(C)\qquad\text{in}\;\, B(0,1).
        \end{align}
        Then, there exists a universal constant $\Gamma<\infty$ such that for all $0<\alpha<1$ and for all $\beta>\sqrt{2}$, there holds
        \begin{align}\label{iteration_gen}
            &\np{\D u}{2}{B(0,1)}\leq \frac{1}{2^k}\np{\D u}{2}{A_0}+\frac{\Gamma}{(1-\alpha)^{\frac{5}{2}}}\np{\D A}{2}{B(0,1)}\left(\sum_{l=0}^{\infty}\frac{1}{2^{2\alpha|l-k+1|}}\int_{A_l}|\D B|^2dx\right)^{\frac{1}{2}}\nonumber\\
            &+\Gamma\frac{2^{2\beta}}{\sqrt{1-\alpha}}\left(\frac{\beta^2+1}{\beta^2-2}\right)\left(\sum_{l=0}^{k-1}\frac{1}{2^{2\beta l}}\frac{1}{2^{2\alpha|l-k+1|}}\int_{A_l}|C|^2dx
            +\sum_{l=k}^{\infty}\frac{1}{2^{2\beta|l-k+1|}}\int_{A_l}|C|^2dx\right)^{\frac{1}{2}}
        \end{align}
        where $A_k=B_{2^{-k}}\setminus\bar{B}_{2^{-(k+1)}}(0)$ for all $k\in\N$.
    \end{lemme}
    \begin{proof}
        Fix some $\beta>\sqrt{2}$, and apply Lemma \ref{lemme_F_1_gen} and an immediate scaling argument to deduce that for all $k\in\N$, there holds
        \begin{align*}
            \np{\D u}{2}{A_{k+1}}&\leq \frac{1}{2}\np{\D u}{2}{A_k}+\Gamma\np{\D A}{2}{B(0,2^{-k})}\np{\omega_k\D B}{2}{B(0,2^{-k})}\\
            &+2^{2\beta+4}\frac{\beta^2+1}{\beta^2-2}\np{|2^{k}x|^{\beta}C}{2}{B(0,2^{-k})},
        \end{align*}
        where $\omega_k(x)=\omega(2^{k}x)$. Therefore, an immediate induction shows that
        \begin{align}\label{gen_iteration1}
            \np{\D u}{2}{A_k}&\leq \frac{1}{2^k}\np{\D u}{2}{A_k}+C\sum_{j=0}^{k-1}\frac{1}{2^j}\np{\D A}{2}{B(0,2^{-(k-1-j)})}\np{\omega_{k-1-j}\D B}{2}{B(0,2^{-(k-1-j)})}\nonumber\\
            &+2^{2\beta+4}\frac{\beta^2+1}{\beta^2-2}\sum_{j=0}^{k-1}\frac{1}{2^j}\np{|2^{k-1-j}x|^{\beta} C}{2}{B(0,2^{-(k-1-j)})}.
        \end{align}
        We estimate as in Lemma \ref{lemme_F_2} the second component of the right-hand side of \eqref{gen_iteration1}. Now, using Cauchy-Schwarz inequality, we deduce that for all $0<\delta<1$, there holds
        \begin{align}\label{C_estimate1}
            &\sum_{j=0}^{k-1}\frac{1}{2^j}\np{|2^{k-1-j}x|^{\beta} C}{2}{B(0,2^{-(k-1-j)})}\nonumber\\
            &
            \leq \left(\int_{B(0,1)}\frac{1}{2^{(1-\delta)j}}\right)^{\frac{1}{2}}
            \left(\sum_{j=0}^{k-1}\frac{1}{2^{2\delta j}}\int_{B(0,2^{k-1-j})}|2^{k-1-j}x|^{2\beta}|C|^2dx\right)^{\frac{1}{2}}
            \nonumber\\
            &\leq \sqrt{\frac{2^{1-\delta}}{2^{1-\delta}-1}}\left(\sum_{j=0}^{k-1}\frac{1}{2^{2\delta j}}\int_{B(0,2^{k-1-j})}|2^{k-1-j}x|^{2\beta}|C|^2dx\right)^{\frac{1}{2}}.
        \end{align}
        Now, using the identity \eqref{fubini_discrete}, we get
        \begin{align}\label{C_estimate2}
            \sum_{j=0}^{k-1}\frac{1}{2^{2\delta j}}\int_{B(0,2^{k-1-j})}|2^{k-1-j}x|^{2\beta}|C|^2dx&=\sum_{l=0}^{k-1}\sum_{j=k-1-l}^{k-1}\frac{1}{2^{2\delta j}}\int_{A_l}|2^{k-1-j}x|^{2\beta}|C|^2dx\nonumber\\
            &+\sum_{l=k}^{\infty}\sum_{j=0}^{k-1}\frac{1}{2^{2\delta j}}\int_{A_l}|2^{k-1-j}x|^{2\beta}|C|^2dx.
        \end{align}
        We first estimate
        \begin{align}\label{C_estimate3}
            \sum_{l=0}^{k-1}\sum_{j=k-1-l}^{k-1}\frac{1}{2^{2\delta j}}\int_{A_l}|2^{k-1-j}x|^{2\beta}|C|^2dx\leq \sum_{l=0}^{k-1}\frac{1}{2^{2\beta l}}\int_{A_l}|C|^2dx\sum_{j=k-1-l}^{k-1}\frac{1}{2^{2\delta j}}2^{2\beta(k-1-j)}.
        \end{align}
        Then, we have
        \begin{align}\label{C_estimate4}
            \sum_{j=k-1-l}^{k-1}\frac{1}{2^{2(\beta+\delta)j}}=\frac{1}{2^{2(\beta+\delta)(k-1-l)}}\sum_{i=0}^{l}\frac{1}{2^{2(\beta+\delta)i}}\leq \frac{2^{2(\beta+\delta)}}{2^{2(\beta+\delta)}-1}\frac{1}{2^{2(\beta+\delta)(k-1-l)}},
        \end{align}
        so that
        \begin{align}\label{C_estimate5}
            &\sum_{l=0}^{k-1}\frac{1}{2^{2\beta l}}\int_{A_l}|2^{k-1-j}x|^{2\beta}|C|^2dx\sum_{j=k-1-l}^{k-1}\frac{1}{2^{2\delta j}}2^{2\beta(k-1-j)}\leq \frac{2^{2(\beta+\delta)}}{2^{2(\beta+\delta)}-1}\sum_{l=0}^{k-1}\frac{1}{2^{2\beta l}}\frac{1}{2^{2\delta(k-1-l)}}\int_{A_l}|C|^2dx.
        \end{align}
        On the other hand, there holds
        \begin{align}\label{C_estimate6}
            &\sum_{l=k}^{\infty}\sum_{j=0}^{k-1}\frac{1}{2^{2\delta j}}\int_{A_l}|2^{k-1-j}x|^{2\beta}|C|^2dx\leq \sum_{l=k}^{\infty}\frac{1}{2^{2\beta (l-k+1)}}\int_{A_l}|C|^2dx\sum_{j=0}^{k-1}\frac{1}{2^{2(\beta+\delta)j}}\nonumber\\
            &\leq \frac{2^{2(\beta+\delta)}}{2^{2(\beta+\delta)}-1}\sum_{l=k}^{\infty}\frac{1}{2^{2\beta(l-k+1)}}\int_{A_l}|C|^2dx.
        \end{align}
        Finally, we have by \eqref{C_estimate1}, \eqref{C_estimate2}, \eqref{C_estimate3}, \eqref{C_estimate5}, and \eqref{C_estimate6}
        \begin{align}\label{C_estimate7}
            \sum_{j=0}^{k-1}\frac{1}{2^{2\delta j}}\int_{B(0,2^{k-1-j})}|2^{k-1-j}|^{2\beta}|C|^2dx&\leq \frac{2^{2(\beta+\delta)}}{2^{2(\beta+\delta)}-1}\left(\sum_{l=0}^{k-1}\frac{1}{2^{2\beta l}}\frac{1}{2^{2\delta (l-k+1)}}\int_{A_l}|C|^2dx\right.\nonumber\\
            &\left.+\sum_{l=k}^{\infty}\frac{1}{2^{2\beta(l-k+1)}}\int_{A_l}|C|^2dx\right).
        \end{align}
        Since $\beta>\sqrt{2}$, the theorem follows by taking $\alpha=\delta$.
    \end{proof}

    Finally, we can prove a generalisation of Theorem \ref{collar_statement}.

    \begin{theorem}\label{collar_statement_div}
    Let $0<2\,a<b\leq 1$, $\Omega=B_b\setminus\bar{B}_a(0)$, and for all $0<t\leq 1$, let $\Omega_t=B_{t\,b}\setminus\bar{B}_{t^{-1}\,a}(0)$. Let $A\in W^{1,(2,1)}(\Omega,\R^n)$, $B\in W^{1,2}(\Omega,M_n(\R))$, $C\in L^{2,1}(\Omega,M_{n,2}(\R))$ and assume that
    \begin{align*}
        \Delta A=\D^{\perp}B\cdot \D A+\dive(C)\qquad\text{in}\;\,\Omega.
    \end{align*}
    Let 
    \begin{align*}
        \Lambda=\frac{1}{2\pi}\int_{\Omega_{1/2}}\frac{|\D A|}{|x|}dx.
    \end{align*}
    Then, there exists a universal constants $\Gamma_0,\Gamma_1<\infty$ such that for all $0<\delta<1$,  and for all $0<\beta<1$, if
    \begin{align}\label{epsilon_smallness2_div}
        \frac{\Gamma_0}{(1-\beta)^7}\int_{\Omega}|\D B|^2dx\leq \delta^2,
    \end{align}
    then there exists $C_{\beta}<\infty$ such that for all $z\in \Omega_{1/2}$, we have
    \begin{align}
        \np{\D A}{2}{B_{2|z|}\setminus\bar{B}_{|z|}(0)}&\leq \Gamma\left(\left(\frac{|z|}{b}\right)^{\beta}+\left(\frac{a}{|z|}\right)^{\beta}\right)\left(\np{\D A}{2}{\Omega}+\np{C}{2,1}{\Omega}\right)\nonumber\\
        &
        +\frac{\left(1+C\delta\right)}{\log\left(\frac{b}{4a}\right)}\left(\sqrt{2\pi\log(2)}\,\Lambda+\Gamma_1\left(\np{\D A}{2}{\Omega}+\np{C}{2,1}{\Omega}\right)\right).
    \end{align}
\end{theorem}
\begin{proof}
    On $\Omega$, make a global expansion
    \begin{align*}
        A=\D\varphi+\D\chi+\D\psi,
    \end{align*}
    where
    \begin{align}\label{decomp_gen1}
        \left\{\begin{alignedat}{2}
            \Delta \varphi&=\D^{\perp}\widetilde{B}\cdot \D\widetilde{A}\qquad&&\text{in}\;\, B(0,1)\\
            \varphi&=0\qquad&&\text{on}\;\, \partial B(0,1),
        \end{alignedat}\right.
    \end{align}
    and 
    \begin{align}
        \left\{\begin{alignedat}{2}
            \Delta\chi&=\dive(C)\qquad&&\text{in}\;\, B(0,1)\\
            \chi&=0\qquad&&\text{on}\;\, \partial B(0,1)
        \end{alignedat}\right.
    \end{align}
    where $\widetilde{A}:\C\rightarrow \R^n$ and $\widetilde{B}:\C\rightarrow M_n(\R)$ are controlled extensions of $A$ and $B$ given by Lemma \ref{whitney_extension}, while $\widetilde{C}$ is the extension of $C$ by $0$. We have as in \eqref{lambda_1}
    \begin{align}\label{lambda_1_gen}
        \np{\D\varphi}{2,1}{B(0,1)}\leq \Gamma_0\np{\D A}{2}{\Omega}\np{\D B}{2}{\Omega},
    \end{align}
    while Calder\'{o}n-Zygmund estimates and interpolation theory show that
    \begin{align}\label{lambda_1_extra}
        &\np{\D\chi}{2}{B(0,1)}\leq \Gamma_0\np{C}{2}{\Omega}\nonumber\\
        &\np{\D\chi}{2,1}{B(0,1))}\leq \Gamma_0\np{C}{2,1}{\Omega}.
    \end{align}
    Then, if for some $a\leq r\leq b$
    \begin{align*}
        d=\frac{1}{2\pi}\int_{\partial B(0,r)}\partial_{\nu}\psi\,d\mathscr{H}^1,
    \end{align*}
    and $\psi_0=\psi-d_0\log|z|$, we obtain as in \eqref{collar1}  that all $k\in\N$ such that $A_k=B_{2^{-k}}\setminus\bar{B}_{2^{-(k+1)}}(0)\subset \Omega_{1/2}$, we have
    \begin{align}\label{collar1_gen}
        \np{\D\psi_0}{2}{A_k}\leq 3\left(\frac{2^{-k}}{b}+\frac{a}{2^{-k}}\right)\np{\D\psi_0}{2}{\Omega},
    \end{align}
    while Parseval identity implies that
    \begin{align*}
        \sqrt{\np{\D\psi_0}{2}{\Omega}+\np{\D(\psi-\psi_0)}{2}{\Omega}}&=\np{\D\psi}{2}{\Omega}\leq \np{\D A}{2}{\Omega}+\np{\D\psi}{2}{\Omega}+\np{\D\chi}{2}{\Omega}\\
        &\leq \left(1+\Gamma_0\np{\D B}{2}{\Omega}\right)\np{\D A}{2}{\Omega}+\Gamma_0\np{C}{2}{\Omega}.
    \end{align*}
    As previously, using Lemma \ref{ineq_harmonic_annulus}, we get
    \begin{align}\label{lambda2_gen}
        \np{\D\psi_0}{2,1}{\Omega_{1/2}}\leq 64\sqrt{\frac{\pi}{15}}\np{\D\psi_0}{2}{\Omega}\leq 64\sqrt{\frac{\pi}{15}}\left(1+\Gamma_0\np{\D B}{2}{\Omega}\right)\np{\D A}{2}{\Omega}+\Gamma_0\np{C}{2}{\Omega}.
    \end{align}
    Now, we have
    \begin{align}\label{lambda3_gen_bis}
        \int_{\Omega_{1/2}}\frac{|\D\chi|}{|x|}dx\leq \np{\D\chi}{2,1}{\Omega_{1/2}}\np{\frac{1}{|x|}}{2,\infty}{\Omega_{1/2}}\leq \sqrt{\pi}\,\Gamma_0\np{C}{2,1}{\Omega}.
    \end{align}
    As previously (see \eqref{lambda3} and \eqref{lambda4}), we have
    \begin{align}\label{lambda3_gen}
        \int_{\Omega_{\frac{1}{2}}}\frac{|\D\varphi|}{|x|}dx\leq \sqrt{\pi}\,\Gamma_0\np{\D A}{2}{\Omega}\np{\D B}{2}{\Omega},
    \end{align}
    and
    \begin{align}\label{lambda4_gen}
        \int_{\Omega_{1/2}}\frac{|\D\psi_0|}{|x|}dx\leq 64\frac{\pi
        }{\sqrt{15}}\left(1+\Gamma_0\np{\D B}{2}{\Omega}\right)\np{\D A}{2}{\Omega}+64\frac{\pi}{\sqrt{15}}\Gamma_0\np{C}{2,1}{\Omega}.
    \end{align}
    Finally, we deduce by \eqref{lambda3_gen_bis}, \eqref{lambda3_gen}, and \eqref{lambda4_gen}, we have by the triangle inequality
    \begin{align}\label{lambda6_gen}
        \left|2\pi d\,\log\left(\frac{b}{4a}\right)-2\pi\,\Lambda\right|&=\left|\int_{\Omega_{1/2}}\frac{|\D u|}{|x|}dx-\int_{\Omega_{1/2}}\frac{|\D(\psi-\psi_0)|}{|x|}dx\right|\leq \int_{\Omega_{1/2}}\frac{|\D u-\D(\psi-\psi_0)|}{|x|}dx\nonumber\\
        &\leq \int_{\Omega_{1/2}}\frac{|\D\varphi|}{|x|}dx+\int_{\Omega_{1/2}}\frac{|\D\chi|}{|x|}dx+\int_{\Omega_{1/2}}\frac{|\D\psi_0|}{|x|}dx\nonumber\\
        &\leq \Gamma_1\left(1+\np{\D B}{2}{\Omega}\right)\np{\D A}{2}{\Omega}+\Gamma_1\np{C}{2,1}{\Omega}.
    \end{align}
    Now, using Lemma \ref{lemme_F_2_gen}, we deduce that for all $0<\alpha<1$ and $\beta>\sqrt{2}$, we have
    \begin{align}\label{collar2_gen}
        \np{\D (\varphi+\chi)}{2}{A_k}&\leq \frac{1}{2^k}\np{\D(\varphi+\chi)}{2}{A_0}+\frac{\Gamma}{(1-\alpha)^{\frac{5}{2}}}\np{\D \widetilde{B}}{2}{B(0,1)}\left(\sum_{l=0}^{\infty}\frac{1}{2^{2\alpha|l-k+1|}}\int_{A_l}|\D\widetilde{A}|^2dx\right)^{\frac{1}{2}}\nonumber\\
        &+\Gamma\frac{2^{2\beta}}{\sqrt{1-\alpha}}\left(\frac{\beta^2+1}{\beta^2-1}\right)\left(\sum_{l=0}^{k-1}\frac{1}{2^{2\beta l}}\frac{1}{2^{2\alpha|l-k+1|}}\int_{A_l}|\widetilde{C}|^2dx+\sum_{l=k}^{\infty}\frac{1}{2^{2\beta|l-k+1|}}\int_{A_l}|\widetilde{C}|^2dx\right)^{\frac{1}{2}}.
    \end{align}
    Proceeding as in the proof of Theorem \ref{collar_statement}, we finally deduce that for all $k\in\N$ such that $A_k\subset \Omega_{1/2}$, there holds
    \begin{align*}
        &\np{\D A}{2}{A_k}\leq \np{\D(\varphi+\chi)}{2}{A_k}+\np{\D\psi_0}{2}{A_k}+\np{\D(\psi-\psi_0)}{2}{A_k}\\
        &\leq \frac{1}{2^k}\left(\np{\D\varphi}{2}{B(0,1)}+\np{\D\chi}{2}{B(0,1)}\right)
        +3\left(\frac{2^{-k}}{b}+\frac{a}{2^{-k}}\right)\left((1+\Gamma_0\np{\D B}{2}{\Omega})\np{\D A}{2}{\Omega}\right.\nonumber\\
        &\left.+\Gamma_0\np{C}{2,1}{\Omega}\right)\\
        &+\frac{\Gamma}{(1-\alpha)^{\frac{5}{2}}}\np{\D B}{2}{\Omega}\left(\sum_{l=0}^{\infty}\frac{1}{2^{2\alpha|l-k+1|}}\int_{A_l}|\D\widetilde{A}|^2dx\right)^{\frac{1}{2}}\nonumber\\
        &+\Gamma\frac{2^{2\beta}}{\sqrt{1-\alpha}}\left(\frac{\beta^2+1}{\beta^2-1}\right)\left(\sum_{l=0}^{k-1}\frac{1}{2^{2\beta l}}\frac{1}{2^{2\alpha|l-k+1|}}\int_{A_l}|\widetilde{C}|^2dx+\sum_{l=k}^{\infty}\frac{1}{2^{2\beta|l-k+1|}}\int_{A_l}|\widetilde{C}|^2dx\right)^{\frac{1}{2}}\\
        &+\sqrt{2\pi\log(2)}\frac{\Lambda}{\log\left(\frac{b}{4a}\right)}+\frac{\Gamma_2}{\log\left(\frac{b}{a}\right)}\left(\left(1+\np{\D B}{2}{\Omega}\right)\np{\D A}{2}{\Omega}+\np{C}{2,1}{\Omega}\right)\\
        &\leq \Gamma_2\left(\frac{2^{-k}}{b}+\frac{a}{2^{-k}}\right)\left(\left(1+\np{\D B}{2}{\Omega}\right)\np{\D A}{2}{\Omega}+\np{C}{2,1}{\Omega}\right)\\
        &+\frac{\Gamma}{(1-\alpha)^{\frac{5}{2}}}\np{\D B}{2}{\Omega}\left(\sum_{l=0}^{\infty}\frac{1}{2^{2\alpha|l-k+1|}}\int_{A_l}|\D\widetilde{A}|^2dx\right)^{\frac{1}{2}}\nonumber\\
        &+\Gamma\frac{2^{2\beta}}{\sqrt{1-\alpha}}\left(\frac{\beta^2+1}{\beta^2-1}\right)\left(\sum_{l=0}^{k-1}\frac{1}{2^{2\beta l}}\frac{1}{2^{2\alpha|l-k+1|}}\int_{A_l}|\widetilde{C}|^2dx+\sum_{l=k}^{\infty}\frac{1}{2^{2\beta|l-k+1|}}\int_{A_l}|\widetilde{C}|^2dx\right)^{\frac{1}{2}}\\
        &+\sqrt{2\pi\log(2)}\frac{\Lambda}{\log\left(\frac{b}{4a}\right)}+\frac{\Gamma_2}{\log\left(\frac{b}{a}\right)}\left(\left(1+\np{\D B}{2}{\Omega}\right)\np{\D A}{2}{\Omega}+\np{C}{2,1}{\Omega}\right).
    \end{align*}
    From now on, all steps of the proof of Theorem \ref{collar_statement} are identical and we skip the details.
\end{proof}

\subsection{Pointwise Estimate of the Mean Curvature:Neck Regions Case}

In this section, we refine the work of \cite{pointwise} that gives a pointwise estimate of the metric in the neck regions thanks to the previous general theorem. Recall that in a neck region $\Omega_k(\alpha)=B_{\alpha}\setminus\bar{B}_{\alpha^{-1}\rho_k}(0)$, there exists an integer $m\geq 1$ and $\alpha_0>0$ such that for $0<\alpha<\alpha_0$ and all $k$ large enough,
\begin{align*}
	\np{\D(\lambda_k-(m-1)\log|z|)}{2,1}{\Omega_k(\alpha)}\leq C\left(\sqrt{\alpha}\np{\D\lambda_k}{2,\infty}{\Omega_k(\alpha)}+\int_{\Omega_k(\alpha)}|\D\n_k|^2dx\right).
\end{align*}
We will also need the improved energy quantization
\begin{align}\label{quanta_L21}
	\lim\limits_{\alpha\rightarrow 0}\limsup_{k\rightarrow\infty}\np{\D\n_k}{2,1}{\Omega_k(\alpha)}=0.
\end{align}
Explicitly, we prove in \cite[Theorem $4.1$]{pointwise} that
\begin{align*}
    \np{\D\n_k}{2,1}{\Omega_k(\alpha)}\leq Ce^{C\,E}\left(1+\np{\D\n_k}{2}{\Omega_k(4\alpha)}\right)\np{\D\n_k}{2}{\Omega_k(4\alpha)},
\end{align*}
where
\begin{align*}
    E=\sup_{k\in\N}\left(\np{\D\lambda_k}{2,\infty}{\Omega_k(1)}+\int_{\Omega_k(1)}|\D\n_k|^2dx\right)<\infty,
\end{align*}
which allows us to deduce the $L^{2,1}$ quantization from the $L^2$ quantization first proved in \cite{quanta}.

\begin{theorem}\label{pointwise_H}
    Let $0<2\,a<b<\infty$, $\Omega=B_{b}\setminus\bar{B}_a(0)$, and for all $0<t\leq 1$, defined $\Omega_t=\Omega_{t\,b}\setminus\bar{B}_{t^{-1}a}$. Let $\phi:B(0,b)\rightarrow \R^n$ be a weak conformal Willmore immersion. Define
    \begin{align*}
        E=\np{\D\lambda}{2,\infty}{\Omega}+\int_{\Omega}|\D\n|^2dx.
    \end{align*}
    For all $0<\beta<1$, there exists $\epsilon_1(\beta,n)>0$ and for all $k\in\N$, there exists a constant $C_k(E,n)$ such that the estimate
    \begin{align}\label{smallness_annulus}
        \np{\D \n}{2}{\Omega}\leq \epsilon_1(\beta,n)
    \end{align}
    implies that for all $z\in \Omega_{1/2}$,
    \begin{align}
        e^{\lambda(z)}|\D^k\H(z)|\leq \frac{C_k(E,n)}{|z|^k}\left(\left(\frac{|z|}{b}\right)^{\beta}+\left(\frac{a}{|z|}\right)^{\beta}+\frac{1}{\log\left(\frac{b}{4a}\right)}\right)\np{e^{\lambda}\H}{2}{\Omega}.
    \end{align}
\end{theorem}
\begin{proof}
    Recall that $A=(\vec{R},S):B(0,1)\rightarrow (\Lambda^2\R^n)\times \R=\R^{N(n)}$ solves the system of equations
    \begin{align}
\left\{\begin{alignedat}{2}
    \Delta S&=-\star\D\n\cdot\D^{\perp}\vec{R}\qquad&&\text{in}\;\, \Omega\\
    \Delta \vec{R}&=(-1)^n\star\left(\D\n\antires\D^{\perp}\vec{R}\right)+\star\D\n\cdot\D^{\perp}S\qquad&&\text{in}\;\, \Omega
    \end{alignedat}\right.
\end{align}
    Therefore, $A$ solves a system of the form \eqref{jacobian_system_pointwise} in Theorem \ref{collar_statement} (where the coefficients of $B$ are linear combinations of coefficients of $\n$). Furthermore, the estimate \eqref{epsilon_smallness2} is equivalent to \eqref{smallness_annulus}. Therefore, we can apply Theorem \ref{collar_statement} to $A=(\vec{R},S)$ which gives us the estimate for all $z\in \Omega_{1/4}$
    \begin{align}\label{pointwise_H1}
        \np{\D (\vec{R},S)}{2}{A(2,z)}&\leq C_{\beta}\left(\left(\frac{|z|}{b}\right)^{\beta}+\left(\frac{a}{|z|}\right)^{\beta}\right)\np{\D (\vec{R},S)}{2}{\Omega_{1/2}}\nonumber\\
        &+\frac{C}{\log\left(\frac{b}{4a}\right)}\left(\Lambda+\np{\D\vec{R}}{2}{\Omega}+\np{\D S}{2}{\Omega}\right)\nonumber\\
        &\leq C_{\beta}\left(\left(\frac{|z|}{b}\right)^{\beta}+\left(\frac{a}{|z|}\right)^{\beta}\right)\np{e^{\lambda}\H}{2}{\Omega}+\frac{C}{\log\left(\frac{b}{4a}\right)}\left(\Lambda+\np{e^{\lambda}\H}{2}{\Omega}\right),
    \end{align}
    where 
    \begin{align}\label{def_Lambda}
        \Lambda=\frac{1}{2\pi}\int_{\Omega_{1/2}}\frac{|\D A|}{|x|}dx,
    \end{align}
    where we used the estimate \eqref{precise1} from Theorem \ref{precise_estimate_neck_region}. Thanks to the $L^{2,1}$ bound \eqref{precise2} from Theorem \ref{precise_estimate_neck_region}, we deduce thanks to the $L^{2,1}/L^{2,\infty}$ duality that
    \begin{align}\label{lambda_bound}
        \Lambda&\leq \frac{1}{2\pi}\np{\D A}{2,1}{\Omega_{1/2}}\np{\frac{1}{|x|}}{2,\infty}{\Omega_{1/2}}\leq \frac{1}{2\sqrt{\pi}}\left(\np{\D \vec{R}}{2}{\Omega_{1/2}}+\np{\D S}{2}{\Omega_{1/2}}\right)\nonumber\\
        &\leq C(E)\np{e^{\lambda}\H}{2}{\Omega}.
    \end{align}
    Since
    \begin{align*}
        \np{\D (\vec{R},S)}{2}{A(2,z)}^2=\int_{A(2,z)}\left(|\D \vec{R}|^2+|\D S|^2\right)dx,
    \end{align*}
    and using the algebraic identity
    \begin{align}
        e^{2\lambda}\H=\frac{1}{4}\D^{\perp}S\cdot\D\phi-\frac{1}{4}\D\vec{R}\res \D^{\perp}\phi,
    \end{align}
    we deduce by \eqref{pointwise_H} and \eqref{lambda_bound} that for all $z\in \Omega_{\frac{1}{4}}$,
    \begin{align}
        \np{e^{\lambda}\H}{2}{A(2,z)}\leq C_{\beta}(E)\left(\left(\frac{|z|}{b}\right)^{\beta}+\left(\frac{a}{|z|}\right)^{\beta}+\frac{1}{\log\left(\frac{b}{4a}\right)}\right)\np{e^{\lambda}\H}{2}{\Omega}
    \end{align}
    Finally, thanks to the precised $\epsilon$-regularity from Theorem \ref{epsilon_reg_H}, we
    have for all $z\in \Omega_{1/4}$
    \begin{align}
        e^{\lambda(z)}|\H(z)|\leq \frac{C_1}{|z|}\np{e^{\lambda}\H}{2}{A(2,z)}\leq \frac{C_{\beta}(E)}{|z|}\left(\left(\frac{|z|}{b}\right)^{\beta}+\left(\frac{a}{|z|}\right)^{\beta}+\frac{1}{\log\left(\frac{b}{4a}\right)}\right)\np{e^{\lambda}\H}{2}{\Omega}
    \end{align}
    which concludes the proof of the theorem for $k=1$. The case $k>1$ is immediate thanks to Theorem \ref{epsilon_reg_H}.
\end{proof}
\begin{rem}
    The theorem was stated on $\Omega_{1/2}$ and proved on $\Omega_{1/4}$, but it clearly holds on $\Omega_{\frac{1}{2}}$ provided that one uses \eqref{precise1} and \eqref{precise2} with $t=\dfrac{1}{\sqrt{2}}$ and chooses worst constants in Theorem \ref{collar_statement}.
\end{rem}

\begin{cor}\label{pointwise_H_sequence}
    Under the hypotheses of Theorem \ref{main_theorem}, for all $0<\beta<1$, there exists a universal constant there exists a universal constant $\alpha_0=\alpha_0(\beta)>0$, and $C_{\beta}<\infty$ such that for all $0<\alpha<\alpha_0$ and for all $z\in \Omega_k(\alpha)$
    \begin{align}
        e^{\lambda_k(z)}|\H_k(z)|\leq \frac{C_{\beta}}{|z|^k}\left(\left(\frac{|z|}{\alpha}\right)^{\beta}+\left(\frac{\alpha^{-1}\delta_k
        }{|z|}\right)^{\beta}+\frac{1}{\log\left(\frac{\alpha^2}{\delta_k}\right)}\right)\np{e^{\lambda_k}\H_k}{2}{\Omega_k(2\alpha)}.
    \end{align}
\end{cor}

\subsection{Estimate of the Full Second Fundamental Form}

\subsubsection{The Codimension One Case}

In this case, the unit normal takes values into $S^2\subset \R^3$.

Recall that the following equation holds on $B(0,b)$
\begin{align*}
    \D\n=\n\times \D^{\perp}\n-2\,H\D\phi.
\end{align*}
Therefore, we have
\begin{align*}
    \Delta\n=\D\n\times \D^{\perp}\n-2\,\dive\left(H\D\phi\right).
\end{align*}
We are exactly in the setting of Theorem \ref{collar_statement_div}, which implies the following statement.

\begin{theorem}\label{normal_pointwise_codimension1}
    Let $0<4\,a<b<\infty$, $\Omega=B_{b}\setminus\bar{B}_a(0)$, and for all $0<t\leq 1$, defined $\Omega_t=\Omega_{t\,b}\setminus\bar{B}_{t^{-1}a}$. Let $\phi:B(0,b)\rightarrow \R^3$ be a weak conformal Willmore immersion. Define
    \begin{align*}
        E=\np{\D\lambda}{2,\infty}{\Omega}+\int_{\Omega}|\D\n|^2dx.
    \end{align*}
    For all $0<\beta<1$ there exists $\epsilon_1(\beta,n)>0$ and for all $k\in\N$, there exists a constant $C_k(E,n)<\infty$ such that the estimate
    \begin{align}\label{smallness_annulus2}
        \np{\D \n}{2}{\Omega}\leq \epsilon_1(\beta,n)
    \end{align}
    implies that for all $z\in \Omega_{1/2}$,
    \begin{align}
        |z|^k|\D^k\n(z)|\leq C_k(E,n)\left(\left(\frac{|z|}{b}\right)^{\beta}+\left(\frac{a}{|z|}\right)^{\beta}+\frac{1}{\log\left(\frac{b}{4a}\right)}\right)\np{\D\n}{2}{\Omega}.
    \end{align}
\end{theorem}
\begin{proof}
    Assuming that $\epsilon_1(\beta,n)>0$ is small enough, we deduce by \cite[Theorem C, ($1.5$)]{pointwise} that
    \begin{align*}
        \np{\D\n}{2,1}{\Omega_{1/2}}\leq C_1e^{C_1 E}\left(1+\np{\D\n}{2}{\Omega}\right)\np{\D\n}{2}{\Omega},
    \end{align*}
    while Theorem \ref{collar_statement_div} shows (up to worsening the constants) that for all $z\in \Omega_{1/4}$
    \begin{align*}
        \np{\D\n}{2}{B_{2|z|}\setminus\bar{B}_{\frac{|z|}{2}}(0)}\leq C\left(\left(\frac{|z|}{b}\right)^{\beta}+\left(\frac{a}{|z|}\right)^{\beta}+\frac{\Lambda}{\log\left(\frac{b}{a}\right)}\right)\left(\np{\D\n}{2}{\Omega_{1/2}}+\np{\D\n}{2,1}{\Omega_{1/2}}\right).
    \end{align*}
    We need only control 
    \begin{align*}
        \Lambda=\frac{1}{2\pi}\int_{\Omega_{1/4}}\frac{|\D\n|}{|x|}dx.
    \end{align*}
    As previously, the $L^{2,1}/L^{2,\infty}$ duality implies that
    \begin{align*}
        \Lambda\leq \frac{1}{2\sqrt{\pi}}\np{\D\n}{2,1}{\Omega_{1/2}}\leq \frac{1}{2\sqrt{\pi}}\left(1+\np{\D\n}{2}{\Omega}\right)\np{\D\n}{2}{\Omega},
    \end{align*}
    Finally, thanks to the $\epsilon$-regularity, we deduce that for some universal constant such that for all $z\in \Omega_{1/8}$,
    \begin{align*}
        |z||\D\n(z)|\leq C\np{\D\n}{2}{A_{2|z|}\setminus\bar{B}_{\frac{|z|}{2}}(0)}\leq C\left(\left(\frac{|z|}{b}\right)^{\beta}+\left(\frac{a}{|z|}\right)^{\beta}+\frac{1}{\log\left(\frac{b}{a}\right)}\right)\np{\D\n}{2}{\Omega},
    \end{align*}
    the conclusion follows by replacing dyadic annuli by smaller annuli and worsening the constants. 
\end{proof}

\subsubsection{Argument in Arbitrary Codimension}

Thanks to the analysis of \cite{riviere1} (see also \cite[Lemme $4.1.3$]{helein}), we deduce that the exists a trivialisation $\n=\n_{1}\wedge\n_{2}\wedge \cdots\wedge \n_{n-2}$ satisfying for all $1\leq i,j\leq n-2$ the Coulomb condition
\begin{align}\label{coulomb}
\left\{\begin{alignedat}{2}
    \dive\left(\D\n_{i}\cdot\n_{j}\right)&=0\qquad &&\text{in}\;\, \Omega\\
    \partial_{\nu}\n_{i}\cdot \n_{j}&=0\qquad&&\text{on}\;\, \partial \Omega.
    \end{alignedat}\right.
\end{align}
where
\begin{align*}
    \sum_{i=1}^{n-2}\np{\D \n_{i}}{2}{\Omega}\leq C\np{\D\n}{2}{\Omega}.
\end{align*}
Furthermore, for all $1\leq i\leq n-2$, the following pointwise identity holds
\begin{align}\label{controlled_frame}
    \D\n_{i}=-\ast(\n\wedge \D^{\perp}\n_{i})+\sum_{j=1}^{n-2}\s{\D\n_{i}}{\n_{j}}\,\n_{j}-2\,H_{i}\D\phi,
\end{align}
where we wrote $\displaystyle\H=\sum_{i=1}^{n-2}H_{i}\,\n_{i}$.
Thanks to the Coulomb condition \eqref{coulomb}, we deduce that 
\begin{align*}
    \Delta \n_{i}=-\ast\left(\D\n_k\wedge \D^{\perp}\n_{i}\right)+\sum_{j=1}^{n-2}\bs{\D\n_{i}}{\n_{j}}\cdot\D\n_{j}-2\,\dive\left(H_{i}\D\phi_k\right).
\end{align*}

\begin{theorem}\label{normal_pointwise_gen}
    Let $n\geq 3$, $0<4\,a<b<\infty$, $\Omega=B_{b}\setminus\bar{B}_a(0)$, and for all $0<t\leq 1$, defined $\Omega_t=\Omega_{t\,b}\setminus\bar{B}_{t^{-1}a}$. Let $\phi:B(0,b)\rightarrow \R^3$ be a weak conformal Willmore immersion. Define
    \begin{align*}
        E=\np{\D\lambda}{2,\infty}{\Omega}+\int_{\Omega}|\D\n|^2dx.
    \end{align*}
    For all $0<\beta<1$ there exists $\epsilon_1(\beta,n)>0$ and for all $k\in\N$, there exists a constant $C_k(E,n)<\infty$ such that the estimate
    \begin{align}\label{smallness_annulus3}
        \np{\D \n}{2}{\Omega}\leq \epsilon_1(\beta,n)
    \end{align}
    implies that for all $z\in \Omega_{1/2}$,
    \begin{align}
        |z|^k|\D^k\n(z)|\leq C_k(E,n)\left(\left(\frac{|z|}{b}\right)^{\beta}+\left(\frac{a}{|z|}\right)^{\beta}+\frac{1}{\log\left(\frac{b}{4a}\right)}\right)\np{\D\n}{2}{\Omega}.
    \end{align}
\end{theorem}
\begin{proof}
    The proof is the same as the one of Theorem \ref{normal_pointwise_codimension1} and we skip it. 
\end{proof}

\section{Proof of the Main Theorem for Precompact Conformal Class}\label{proof_neck}

    In this section, we will prove a version of \cite[Lemma IV.$2$]{riviere_morse_scs} in our setting. Fix some $\frac{1}{2}<\beta<1$ and introduce the following weight on $\Sigma$:
    \begin{align}\label{weight}
        \omega_{a,b}(x)=\left\{\begin{alignedat}{2}
        &\frac{1}{|x|^{4m}}\left(\left(\frac{|x|}{b}\right)^{4\beta}+\left(\frac{a}{|x|}\right)^{4\beta_1}+\frac{1}{\log^2\left(\frac{b}{a}\right)}\right)\qquad&&\text{for all}\;\, x\in \Omega=B_b\setminus\bar{B}_a(0)\\
        &\frac{1}{b^{4m}}\left(1+\left(\frac{a}{b}\right)^{4\beta}+\frac{1}{\log^2\left(\frac{b}{a}\right)}\right)\qquad&&\text{for all}\;\, x\in \Sigma\setminus\bar{B}(0,a)\\
        &\frac{1}{a^{4m}}\left(\left(1+\left(\frac{a}{b}\right)^{4\beta}+\frac{1}{\log^2\left(\frac{b}{a}\right)}\right)\right)\qquad&&\text{for all}\;\, x\in B(0,a)
        \end{alignedat}\right.
    \end{align}

    \begin{theorem}\label{neck_positive}
        Let $0<\beta_1,\beta_2<1$ be such that
        \begin{align}\label{beta12_def0}
        \left\{\begin{alignedat}{1}
            &\frac{1}{2}<\beta_1<1\\
            &\sqrt{2}-1<\beta_2<1
            \end{alignedat}\right.
        \end{align}
        and $\{\phi_k\}_{k\in\N}$ be a sequence of Willmore immersions from Theorem \ref{main_theorem}, and let $\Omega_k(\alpha)=B_{\alpha}\setminus\bar{B}_{\alpha^{-1}\delta_k}(0)$ be a neck region. Then, there exists $\lambda_1(\beta_1),\lambda_2(\beta_2)>0$ and $\alpha_0>0$ such that for all $0<\alpha\leq \alpha_0$ and for all $k\in\N$ large and for all $u\in W^{2,2}_0(\Omega_k(\alpha))$
        \begin{align}\label{neck_inequality}
            Q_{\phi_k}(u)&\geq \lambda_1(\beta_1)\int_{\Omega_k(\alpha)}\frac{u^2}{|x|^4}\left(\left(\frac{|x|}{\alpha}\right)^{4\beta_1}+\left(\frac{\alpha^{-1}\delta_k}{|x|}\right)^{4\beta_1}+\frac{1}{\log^2\left(\frac{\alpha^2}{\delta_k}\right)}\right)dx\\
            &+\lambda_2(\beta_2)\int_{\Omega_k(\alpha)}\frac{|\D u|^2}{|x|^2}\left(\left(\frac{|x|}{\alpha}\right)^{2\beta_2}+\left(\frac{\alpha^{-1}\delta_k}{|x|}\right)^{2\beta_2}+\frac{1}{\log^2\left(\frac{\alpha^2}{\delta_k}\right)}\right)dx.
        \end{align}
    \end{theorem}
    \begin{proof}
        For technical reasons, we first let $0<\beta_1,\beta_2<1$ be such that
        \begin{align}\label{beta12_def}
            \left\{\begin{alignedat}{1}
            &\frac{1}{2}<\beta_1<1-\frac{1}{2\sqrt{2}}\\
            &\sqrt{2}-1<\beta_2=2\beta_1-2+\sqrt{2}<1
            \end{alignedat}\right.
        \end{align}
        Since we have assumed that the limiting map and the bubbles were immersions, we deduce by the main theorem of \cite{pointwise} that there exists $\alpha_0>0$ such that for all 
        \begin{align}\label{proof_main_1}
            e^{2\lambda_k}=e^{2\mu_k},
        \end{align}
        where $\mu_k\in C^0(\Omega_k(B(0,\alpha_0))$ and for all $0<\alpha<\alpha_0$,
        \begin{align}\label{proof_main_2}
            \np{\D u_k}{2,1}{B(0,\alpha)}+\np{\mu_k-c_k}{\infty}{\Omega_k(\alpha)}\leq C(n)\left(\sqrt{\alpha}\np{\D\lambda_k}{2,\infty}{\Omega_k(2\alpha)}+\int_{\Omega_k(2\alpha)}|\D\n_k|^2dx\right).
        \end{align}
        where $c_k\conv{k\rightarrow \infty}c\in\R$. Then, using the two main estimates from \cite{eigenvalue_annuli}, we deduce that for all $1/2<\beta_1<1$ and for all $\sqrt{2}-1<\beta_2<1$ there exists $0<C_{\beta_1,1},C_{\beta_2,2}<\infty$ such that for all $u\in \Omega_k(\alpha)$,
        \begin{align}\label{fine_neck_estimate}
            \int_{\Omega_k(\alpha)}(\Delta u)^2dx&\geq C_{\beta_1,1}\int_{\Omega_k(\alpha)}\frac{u^2}{|x|^4}\left(\left(\frac{|x|}{\alpha}\right)^{4\beta_1}+\left(\frac{\alpha^{-1}\delta_k}{|x|}\right)^{4\beta_1}+\frac{1}{\log^2\left(\frac{\alpha^2}{\delta_k}\right)}\right)dx\nonumber\\
            &+C_{\beta_2,2}\int_{\Omega_k(\alpha)}\frac{|\D u|^2}{|x|^2}\left(\left(\frac{|x|}{\alpha}\right)^{2\beta_2}+\left(\frac{\alpha^{-1}\delta_k}{|x|}\right)^{2\beta_2}+\frac{1}{\log^2\left(\frac{\alpha^2}{\delta_k}\right)}\right)dx.
        \end{align}
        Then, using inequality \eqref{ineq_d21}, we deduce that
        \begin{align*}
            Q_{\phi_k}(u)&\geq \frac{1}{4}\int_{\Omega_k(\alpha)}(\Delta_{g_k}u)^2d\mathrm{vol}_{g_k}-2\int_{\Omega_k(\alpha)}|\D u|^2|A_k|^2dx-2\int_{\Omega_k(\alpha)}|A_k|^4u^2d\mathrm{vol}_{g_k}\\
            &-4\int_{\Omega_k(\alpha)}|\D u||u||\D H_k||A_k|dx.
        \end{align*}
        Furthermore, thanks to estimates \eqref{proof_main_1} and \eqref{proof_main_2}, we deduce that 
        \begin{align*}
            \int_{\Omega_k(\alpha)}(\Delta_{g_k}u)^2d\mathrm{vol}_{g_k}=\int_{\Omega_k(\alpha)}e^{-2\mu_k}(\Delta u)^2dx\geq \frac{1}{2}e^{-2c}\int_{\Omega_k(\alpha)}(\Delta u)^2dx
        \end{align*}
        for $k$ large enough. Therefore we deduce that 
        \begin{align}\label{proof_main_3}
            Q_{\phi_k(u)}&\geq \frac{1}{8}e^{-2c}\int_{\Omega_k(\alpha)}(\Delta u)^2dx-2\int_{\Omega_k(\alpha)}|\D u|^2|A_k|^2dx-2\int_{\Omega_k(\alpha)}|A_k|^4u^2d\mathrm{vol}_{g_k}\nonumber\\
            &-4\int_{\Omega_k(\alpha)}|\D u||u||\D H_k||A_k|dx.
        \end{align}
        Then, Theorem \ref{normal_pointwise_gen} shows that there exists a universal constant $C<\infty$ (depending only on $\beta$ and $n$) such that
        \begin{align}\label{pointwise_second_fundamental}
            |A_k(x)|\leq \frac{C}{|x|}\left(\left(\frac{|x|}{\alpha}\right)^{\beta}+\left(\frac{a}{|x|}\right)^{\beta}+\frac{1}{\log\left(\frac{\alpha^2}{\delta_k}\right)}\right)\np{\D\n_k}{2}{\Omega_k(2\alpha)}.
        \end{align}
        Therefore, we get (by the elementary inequality $(a+b+c)^2\leq 4a^2+4b^2+2c^2\leq 4(a^2+b^2+c^2)$
        \begin{align}\label{proof_main_4}
            2\int_{\Omega_k(\alpha)}|\D u|^2|A_k|^2dx\leq 4C^2\np{\D\n_k}{2}{\Omega_k(2\alpha)}^2\int_{\Omega_k(\alpha)}\frac{|\D u|^2}{|x|^2}\left(\left(\frac{|x|}{\alpha}\right)^{2\beta}+\left(\frac{a}{|x|}\right)^{2\beta}+\frac{1}{\log^2\left(\frac{\alpha^2}{\delta_k}\right)}\right)dx.
        \end{align}
        On the other hand, we have for $k$ large enough
        \begin{align*}
            e^{2\lambda_k}|A_k|^4\leq 32e^{2c}\np{\D\n_k}{2}{\Omega_k(2\alpha)}^4\frac{C^4}{|x|^4}\left(\left(\frac{|x|}{\alpha}\right)^{4\beta}+\left(\frac{\alpha^{-1}\delta_k}{|x|}\right)^{4\beta}+\frac{1}{\log^4\left(\frac{b}{a}\right)}\right)dx,
        \end{align*}
        which implies that (by the inequality $(a+b+c)^4\leq 16(a^4+b^4+c^4)$)
        \begin{align}\label{proof_main_5}
            2\int_{\Omega_k(\alpha)}|A_k|^2u^2d\mathrm{vol}_{g_k}\leq 64e^{2c}C^2\np{\D\n_k}{2}{\Omega_k(2\alpha)}^4\int_{\Omega_k(\alpha)}\frac{u^2}{|x|^4}\left(\left(\frac{|x|}{\alpha}\right)^{4\beta}+\left(\frac{\alpha^2}{\delta_k}\right)^{4\beta}+\frac{1}{\log^4\left(\frac{\alpha^2}{\delta_k}\right)}\right)dx.
        \end{align}
        Then, we have for all $1/2<\beta<1$
        \begin{align}\label{proof_main_6}
            &\int_{\Omega_k(\alpha)}|\D u||u||\D H_k||A_k|dx\nonumber\\
            &\leq C'_{\beta}\np{\D\n_k}{2}{\Omega_k(2\alpha)}\int_{\Omega_k(\alpha)}\frac{|\D u|}{|x|}\frac{|u|}{|x|^2}\left(\left(\frac{|x|}{\alpha}\right)^{2\beta}+\left(\frac{\alpha^{-1}\delta_k}{|x|}\right)^{2\beta}+\frac{1}{\log^2\left(\frac{\alpha^2}{\delta_k}\right)}\right)dx.
        \end{align}
        We trivially estimate by virtue of the Cauchy-Schwarz inequality
        \begin{align}\label{proof_main_7}
            \int_{\Omega_k(\alpha)}\frac{|\D u|}{|x|}\frac{|u|}{|x|^2}\frac{1}{\log^2\left(\frac{\alpha^2}{\delta_k}\right)}dx\leq \left(\int_{\Omega_k(\alpha)}\frac{|\D u|^2}{|x|^2}\frac{1}{\log^2\left(\frac{\alpha^2}{\delta_k}\right)}dx\right)^{\frac{1}{2}}\left(\int_{\Omega_k(\alpha)}\frac{u^2}{|x|^4}\frac{1}{\log^2\left(\frac{\alpha^2}{\delta_k}\right)}dx\right)^{\frac{1}{2}}.
        \end{align}
        Then, we have for all $0<\epsilon<1$
        \begin{align}\label{proof_main_8}
            &\int_{\Omega_k(\alpha)}\frac{|\D u|}{|x|}\frac{|u|}{|x|^2}\left(\left(\frac{|x|}{\alpha}\right)^{2\beta}+\left(\frac{\alpha^{-1}\delta_k}{|x|}\right)^{2\beta}\right)dx
            \leq 2\left(\int_{\Omega_k(\alpha)}\frac{|\D u|^2}{|x|^2}\left(\left(\frac{|x|}{\alpha}\right)^{4\epsilon\beta}+\left(\frac{\alpha^{-1}\delta_k}{|x|}\right)^{4\epsilon\beta}\right)dx\right)^{\frac{1}{2}}\nonumber\\
            &\times\left(\int_{\Omega_k(\alpha)}\frac{u^2}{|x|^4}\left(\left(\frac{|x|}{\alpha}\right)^{4(1-\epsilon)\beta}+\left(\frac{\alpha^{-1}\delta_k}{|x|}\right)^{4(1-\epsilon)\beta}\right)dx\right)^{\frac{1}{2}}.
        \end{align}
        In order to control this quantity thanks to \eqref{fine_neck_estimate}, we must choose $1/2<\beta<1$ and $0<\epsilon<1$ such that
        \begin{align}\label{fine_cond2}
            \left\{\begin{alignedat}{1}
                4\epsilon \beta>2(\sqrt{2}-1)\\
                4(1-\epsilon)\beta>2.
            \end{alignedat}\right.
        \end{align}
        Summing both equations, we deduce that $\beta>\dfrac{\sqrt{2}}{2}=0.707\cdots$. Then, we get
        \begin{align}\label{fine_cond2bis}
            \left\{\begin{alignedat}{1}
                \epsilon >\frac{\sqrt{2}-1}{2\beta}\\
                \epsilon<1-\frac{1}{2\beta}.
            \end{alignedat}\right.
        \end{align}
        Therefore, we must have
        \begin{align*}
            \frac{\sqrt{2}-1}{2\beta}<1-\frac{1}{2\beta},
        \end{align*}
        which is indeed equivalent to $\beta>\dfrac{\sqrt{2}}{2}$. For such a $\beta$, there exists $0<\epsilon<1$ solving \eqref{fine_cond2}. Explicitly, we can choose
        \begin{align*}
            \epsilon=\frac{1}{2}\left(\frac{\sqrt{2}-1}{2\beta}+1-\frac{1}{2\beta}\right)=\frac{1}{2}-\frac{(2-\sqrt{2})}{4\beta}.
        \end{align*}
        Then, we get
        \begin{align}\label{beta2_def}
            2\beta_2=4\epsilon\beta=2\beta-(2-\sqrt{2})>2(\sqrt{2}-1)
        \end{align}
        since we have assume that $\beta>\dfrac{\sqrt{2}}{2}$. On the other hand,
        \begin{align}\label{beta1_def}
            4\beta_1=4(1-\epsilon)\beta=2\beta+2-\sqrt{2}>2
        \end{align}
        since $\beta>\dfrac{\sqrt{2}}{2}$. Gathering \eqref{proof_main_6}, \eqref{proof_main_7}, and \eqref{proof_main_8}, we deduce that
        \begin{align}\label{proof_main_9}
            &4\int_{\Omega_k(\alpha)}|\D u||u||\D H_k||A_k|dx\nonumber\\
            &\leq 4C_{\beta}'\np{\D\n_k}{2}{\Omega_k(2\alpha)}\int_{\Omega_k(\alpha)}\frac{u^2}{|x|^4}\left(\left(\frac{|x|}{\alpha}\right)^{4\beta_1}+\left(\frac{\alpha^{-1}\delta_k}{|x|}\right)^{4\beta_1}+\frac{1}{\log^2\left(\frac{\alpha^2}{\delta_k}\right)}\right)dx\nonumber\\
            &+4C_{\beta}'\np{\D\n_k}{2}{\Omega_k(2\alpha)}\int_{\Omega_k(\alpha)}\frac{|\D u|^2}{|x|^2}\left(\left(\frac{|x|}{\alpha}\right)^{2\beta_2}+\left(\frac{\alpha^{-1}\delta_k}{|x|}\right)^{2\beta_2}+\frac{1}{\log^2\left(\frac{\alpha^2}{\delta_k}\right)}\right)dx
        \end{align}
        Therefore, if $0<\beta_1,\beta_2<1$ are chosen as in the statement of the theorem, we deduce by \eqref{proof_main_3}, \eqref{proof_main_4}, \eqref{proof_main_5}, and \eqref{proof_main_9} that
        \small
        \begin{align*}
            &Q_{\phi_k}(u)\geq \left(\frac{1}{8}e^{-2c}C_{\beta_1,1}-\left(4C_{\beta}'+64e^{2c}C^2\np{\D\n_k}{2}{\Omega_k(2\alpha)}^2\right)\np{\D\n_k}{2}{\Omega_k(2\alpha)}^2\right)\\
            &\times\int_{\Omega_k(\alpha)}\frac{u^2}{|x|^4}\left(\left(\frac{|x|}{\alpha}\right)^{4\beta_1}+\left(\frac{\alpha^{-1}\delta_k}{|x|}\right)^{4\beta_1}+\frac{1}{\log^2\left(\frac{\alpha^2}{\delta_k}\right)}\right)dx\\
            &+\left(\frac{1}{8}e^{-2c}C_{\beta_2,2}-4\left(C^2+C_{\beta}'\right)\np{\D\n_k}{2}{\Omega_k(2\alpha)}\right)\int_{\Omega_k(2\alpha)}\frac{|\D u|^2}{|x|^2}\left(\left(\frac{|x|}{\alpha}\right)^{2\beta_2}+\left(\frac{\alpha^{-1}\delta_k}{|x|}\right)^{2\beta_2}+\frac{1}{\log^2\left(\frac{\alpha^2}{\delta_k}\right)}\right)dx,
        \end{align*}
        \normalsize
        where we assumed without loss of generality that $\delta_k\leq \alpha^2 e$ in order to get the trivial estimate
        \begin{align*}
            \frac{1}{\log^4\left(\frac{\alpha^2}{\delta_k}\right)}\leq \frac{1}{\log^2\left(\frac{\alpha^2}{\delta_k}\right)}.    
        \end{align*}
        Therefore, the quantization of energy (\cite{beriviere}) 
        \begin{align*}
            \lim_{\alpha\rightarrow 0}\limsup_{k\rightarrow \infty}\np{\D\n_k}{2}{\Omega_k(2\alpha)}=0
        \end{align*}
        completes the proof of the theorem. Now, if $1\leq \beta_1\leq \beta<1$, then
        \begin{align*}
            \left(\frac{|x|}{\alpha}\right)^{4\beta}+\left(\frac{\alpha^{-1}\delta_k}{|x|}\right)^{4\beta}\leq \left(\frac{|x|}{\alpha}\right)^{4\beta_1}+\left(\frac{\alpha^{-1}\delta_k}{|x|}\right)^{4\beta_1}.
        \end{align*}
        Therefore, coming back to \eqref{fine_cond2}, we see that as $\epsilon,\beta>0$ varies, $\beta_1$ or $\beta_2$ can be chosen arbitrarily (as long as \eqref{beta12_def0} is satisfied). Therefore, we deduce that \eqref{neck_inequality} holds for all $0<\beta_1,\beta_2<1$ satisfying the conditions in \eqref{beta12_def0}.
    \end{proof}
    Now, recall the formula
    \begin{align}\label{der2_redite}
        Q_{\phi}(u)&=\frac{1}{2}\int_{\Sigma}\left(\Delta_gu+|A|^2u\right)^2d\vg+\int_{\Sigma}\left(|du|_g^2+4|h_0|_{WP}^2u^2\right)H^2d\vg-8\int_{\Sigma}\s{\partial u\otimes\partial u}{h_0}_{WP}H\,d\vg\nonumber\\
        &-16\int_{\Sigma}\s{\partial u\otimes\partial H}{h_0}_{\mathrm{WP}}u\,d\vg.
    \end{align}
    We want to introduce two weights, both adapted to derivatives of different order.
    \begin{align*}
        \omega_{1,\alpha,k}(x)=\frac{1}{|x|^4}\left(\left(\frac{|x|}{\alpha}\right)^{4\beta_1}+\left(\frac{\alpha^{-1}\delta_k}{|x|}\right)^{4\beta_1}+\frac{1}{\log^2\left(\frac{\alpha^2}{\delta_k}\right)}\right)
    \end{align*}
    and
    \begin{align*}
        \omega_{2,k}(x)=\frac{1}{|x|^2}\left(\left(\frac{|x|}{\alpha}\right)^{2\beta_2}+\left(\frac{\alpha^{-1}\delta_k}{|x|}\right)^{2\beta_2}+\frac{1}{\log^2\left(\frac{\alpha^2}{\delta_k}\right)}\right).
    \end{align*}
    We introduce a third weight
    \begin{align*}
        \omega_{3,k}=\sqrt{\omega_{1,k}\omega_{2,k}}.
    \end{align*}
    This is justified in the proof of Theorem \ref{neck_positive} since $\beta$ is given thanks to \eqref{beta2_def} and \eqref{beta1_def} by
    \begin{align*}
        \beta=\beta_1+\frac{1}{2}\beta_1,
    \end{align*}
    so that
    \begin{align*}
        \left(\frac{|x|}{\alpha}\right)^{2\beta}=\left(\frac{|x|}{\alpha}\right)^{2\beta_1}\left(\frac{|x|}{\alpha}\right)^{\beta_2}=\sqrt{\left(\frac{|x|}{\alpha}\right)^{4\beta_1}\left(\frac{|x|}{\alpha}\right)^{2\beta_2}}
    \end{align*}
    and shows that this is the weight that appeared in the proof of Theorem \ref{neck_positive}. Then, we write
    \begin{align*}
        \frac{1}{2}\int_{\Sigma}(\Delta_gu)^2d\vg=\frac{1}{2}\int_{\Sigma}u\,\Delta_g^2u\,d\vg=\frac{1}{2}\int_{\Sigma}\left((\omega_{1,k}+\omega_{2,k})^{-1}\Delta_g^2u\right)u(\omega_{1,k}+\omega_{2,k})d\vg.
    \end{align*}
    For the next term, we write
    \begin{align*}
        \int_{\Sigma}|A|^2u\,\Delta_g u\,d\vg=\int_{\Sigma}(\omega_{1,k}^{-1}|A|^2\Delta_gu)u\,\omega_{1,k}d\vg.
    \end{align*}
    The next term is
    \begin{align*}
        \int_{\Sigma}|du|_g^2H^2d\vg=\int_{\Sigma}u\,d^{\ast_g}(H^2\,du)d\vg=\int_{\Sigma}\omega_{2,k}^{-1}d^{\ast_g}\left(H^2\,du\right)u\,\omega_{2,k}\,d\vg.
    \end{align*}
    We then rewrite
    \begin{align*}
        \int_{\Sigma}4|h_0|_{\mathrm{WP}}^2H^2u^2d\vg=\int_{\Sigma}\left(4\,\omega_{1,k}^{-1}|h_0|_{WP}^2H^2u\right)u\,\omega_{1,k}\,d\vg.
    \end{align*}
    The last two terms are rewritten the following way
    \begin{align*}
        -8\int_{\Sigma}\s{\partial u\otimes\partial u}{h_0}_{\mathrm{WP}}H\,d\vg=8\int_{\Sigma}\omega_{2,k}^{-1}\ast_g\,d\,\Re\left(H\,g^{-1}\otimes\h_0\otimes\partial u\right)u\,\omega_{2,k}\,d\vg,
    \end{align*}
    and
    \begin{align*}
        -16\int_{\Sigma}\s{\partial u\otimes \partial H}{h_0}_{\mathrm{WP}}u\,d\vg=-16\int_{\Sigma}\left(\omega_{3,k}^{-1}\s{\partial u\otimes\partial H}{h_0}_{\mathrm{WP}}\right)u\,\omega_{3,k}\,d\vg.
    \end{align*}
    Now, we write thanks to \eqref{operator3_normal}
    \begin{align}\label{second_derivative_weight}
        Q_{\phi_k}(u)=\frac{1}{2}\int_{\Sigma}\left(\mathcal{L}_{\alpha,k}u\right)u\,\omega_{1,\alpha,k}\,d\mathrm{vol}_{g_k},
    \end{align}
    where $\mathcal{L}_{\alpha,k}=\omega_{1,\alpha,k}^{-1}\mathcal{L}_{g_k}$, and
    \begin{align*}
        \mathcal{L}_{g_k}&=\Delta_{g_k}^2+|A_k|^2\Delta_{g_k}+2\s{d|A_k|^2}{d(\,\cdot\,)}_{g_k}+2\,d^{\ast_{g_k}}\left(H_k^2\,d(\,\cdot\,)\right)+\left(|A_k|^4+\Delta_g|A_k|^2+24|H_k|^2|h_{0,k}|_{\mathrm{WP}}^2\right)\\
        &+16\,\ast_{g_k}\,d\,\Re\left(g_k^{-1}\otimes h_{0,k}\otimes\partial(H_k\,\cdot\,)\right)-16\s{\partial(\,\cdot\,)\otimes\partial H_k}{h_{0,k}}_{\mathrm{WP}}.
    \end{align*}
    Thanks to \eqref{second_derivative_weight}, we deduce that $\mathcal{L}_{\alpha,k}$ is a self-adjoint operator with respect to the weight $\omega_{1,\alpha,k}$. Then, we generalise \cite[Lemma IV.$3$]{riviere_morse_scs} to out setting.

    \begin{lemme}\label{lemme_IV.3}
        For all $\lambda\in\R$, let 
        \begin{align}\label{weighted_eigenspace}
            \mathscr{E}_{\alpha,k}(\lambda)=W^{2,2}(\Sigma)\cap\ens{u:\mathcal{L}_{\alpha,k}u=\lambda\,u}.
        \end{align}
        We have
        \begin{align*}
            \mathrm{Ind}_W(\phi_k)=\dim\bigoplus_{\lambda<0}\mathscr{E}_{\alpha,k}(\lambda).
        \end{align*}
    \end{lemme}
    \begin{proof}
        The proof is identical as the one in \cite{riviere_morse_scs} and we omit it. 
    \end{proof}

    Now, recall the following theorem from \cite[Theorem $4.5$]{eigenvalue_annuli}, which is a generalisation of the afore-mentioned estimate for functions that do \emph{not} necessarily vanish on the boundary.

    \begin{theorem}\label{interpolation_weighted_poincare}
            Let $0<a<b<\infty$ and $\Omega=B_b\setminus\bar{B}_a(0)$. For all $\dfrac{1}{2}<\beta<1$ and $\sqrt{2}-1<\gamma<1$ there exists a universal constant $C_{\beta,\gamma}<\infty$ such that
            for all $u\in W^{2,2}(\Omega)$, provided that
            \small
            \begin{align}\label{conformal_class_cor}
            \log\left(\frac{b}{a}\right)\geq \max\ens{2,\frac{1}{2\beta-1}\log\left(4\beta\right),\frac{1}{4\beta}\log\left(\frac{2}{2-\sqrt{3}}\right), \frac{1}{4(1-\beta)}\log\left(1+\frac{8\beta(1-\beta)}{(2\beta-1)^2}\right), \frac{1}{4\beta}\log(8\beta(\beta+1))},
            \end{align}
            \normalsize
            we have
            \begin{align}\label{interpolation_weight}
                \int_{\Omega}\frac{|\D u|^2}{|x|^2}\left(\left(\frac{|x|}{b}\right)^{2\gamma}+\left(\frac{a}{|x|}\right)^{2\gamma}\right)dx\leq C_{\beta,\gamma}\left(\int_{\Omega}\frac{u^2}{|x|^4}\left(\left(\frac{|x|}{b}\right)^{4\beta}+\left(\frac{a}{|x|}\right)^{4\beta}\right)dx+\int_{\Omega}|\D^2u|^2dx\right).
            \end{align}
    \end{theorem}
    The generalisation of the next result (\cite[Lemma IV.$4$]{riviere_morse_scs}) will be considerably more involved for it makes use of Theorem \ref{interpolation_weighted_poincare}.

    \begin{lemme}\label{lemme_IV.4}
        There exists $\alpha_0>0$ and $\mu_0>0$ and a family of constants $\ens{\mu_{\alpha,k}}_{0<\alpha<\alpha_0,k\in\N}\subset (0,\mu_0)$ such that
        \begin{align*}
            \lim_{\alpha\rightarrow 0}\limsup_{k\rightarrow \infty}\mu_{\alpha,k}=0.
        \end{align*}
        Furthermore, for all $\lambda\in\R$, we have
        \begin{align*}
            \mathrm{dim}\,\mathscr{E}_{\alpha,k}(\lambda)>0\implies \lambda\geq -\mu_{\alpha,k}\geq -\mu_{0}.
        \end{align*}
    \end{lemme}
    \begin{proof}
        First, we let $\dfrac{1}{2}<\beta_1<1$ and $\sqrt{2}-1<\beta_2<1$ to be defined later.
        Write $\omega_{\alpha,k}=\omega_{1,\alpha,k}$ for simplicity. We have
        \begin{align*}
            \omega_{1,\alpha,k}=\left\{\begin{alignedat}{2}
                 &\frac{1}{\alpha^4}\left(1+\left(\frac{\alpha^{-1}\delta_k}{\alpha}\right)^{4\beta_1}+\frac{1}{\log^2\left(\frac{\alpha^2}{\delta_k}\right)}\right)\qquad&& \text{for all}\;\, x\in \Sigma\setminus\bar{B}(0,\alpha)\\
                 &\frac{1}{|x|^4}\left(\left(\frac{|x|}{\alpha}\right)^{4\beta_1}+\left(\frac{\alpha^{-1}\delta_k}{|x|}\right)^{4\beta_1}+\frac{1}{\log^2\left(\frac{\alpha^2}{\delta_k}\right)}\right)\qquad&& \text{for all}\;\, x\in \Omega_k(\alpha)=B_{\alpha}\setminus\bar{B}_{\alpha^{-1}\delta_k}(0)\\
                 &\frac{1}{\alpha^4}\left(1+\left(\frac{\alpha^{-1}\delta_k}{\alpha}\right)^{4\beta_1}+ \frac{1}{\log^2\left(\frac{\alpha^{2}}{\delta_k}\right)}\right)\qquad&&\text{for all}\;\, x\in B(0,\alpha^{-1}\delta_k),
            \end{alignedat}\right.
        \end{align*}
        and
        \begin{align*}
            \omega_{2,\alpha,k}=\left\{\begin{alignedat}{2}
                 &\frac{1}{\alpha^2}\left(1+\left(\frac{\alpha^{-1}\delta_k}{\alpha}\right)^{2\beta_2}+\frac{1}{\log^2\left(\frac{\alpha^2}{\delta_k}\right)}\right)\qquad&& \text{for all}\;\, x\in \Sigma\setminus\bar{B}(0,\alpha)\\
                 &\frac{1}{|x|^2}\left(\left(\frac{|x|}{\alpha}\right)^{2\beta_2}+\left(\frac{\alpha^{-1}\delta_k}{|x|}\right)^{2\beta_2}+\frac{1}{\log^2\left(\frac{\alpha^2}{\delta_k}\right)}\right)\qquad&& \text{for all}\;\, x\in \Omega_k(\alpha)=B_{\alpha}\setminus\bar{B}_{\alpha^{-1}\delta_k}(0)\\
                 &\frac{1}{\alpha^2}\left(1+\left(\frac{\alpha^{-1}\delta_k}{\alpha}\right)^{2\beta_2}+ \frac{1}{\log^2\left(\frac{\alpha^{2}}{\delta_k}\right)}\right)\qquad&&\text{for all}\;\, x\in B(0,\alpha^{-1}\delta_k).
            \end{alignedat}\right.
        \end{align*}
        First, using \eqref{pointwise_second_fundamental}, we deduce that 
        \begin{align}\label{weight_limit1}
            \lim_{\alpha\rightarrow 0}\limsup_{k\rightarrow \infty}\np{\frac{A_k}{\omega_k^{\frac{1}{4}}}}{\infty}{\Omega_k(\alpha)}=0.
        \end{align}
        Assume without loss of generality that there are no bubbles in $\Sigma\setminus \bar{B}(0,\alpha)$. Then, we have on $\Sigma\setminus\bar{B}(0,\alpha)$
        \begin{align*}
            \omega_{\alpha,\infty}=\lim_{k\rightarrow \infty}\omega_{1,\alpha,k}=\frac{1}{\alpha^4}.
        \end{align*}
        Thanks to the strong convergence on $\Sigma\setminus\bar{B}(0,\alpha)$ and the absence of branch point at the limit (which implies that $A_{\infty}\in L^{\infty}(\Sigma)$), we deduce that 
        \begin{align}\label{weight_limit2}
            \lim_{\alpha\rightarrow 0}\limsup_{k\rightarrow \infty}\np{\frac{A_k}{\omega_k^{\frac{1}{4}}}}{\infty}{\Sigma\setminus\bar{B}_{\alpha}(0)}=\lim_{\alpha\rightarrow 0}\alpha\np{A_{\infty}}{\infty}{\Sigma\setminus\bar{B}(0,\alpha)}=0.
        \end{align}
        Likewise, the strong convergence of the bubble (and the absence of branch point of the bubble at the point of concentration) finally shows by \eqref{weight_limit1} and \eqref{weight_limit2} that 
        \begin{align}
            \lim_{\alpha\rightarrow 0}\limsup_{k\rightarrow \infty}\np{\frac{A_k}{\omega_k^{\frac{1}{4}}}}{\infty}{\Sigma}=0.
        \end{align}
        Then, we introduce on $\Sigma$ the weight
        \begin{align*}
            \omega_{\alpha,\infty}=\left\{\begin{alignedat}{2}
                &\frac{1}{\alpha^4}\qquad&&\text{in}\;\, \Sigma\setminus\bar{B}(0,\alpha)\\
                &\frac{1}{|x|^4}\left(\frac{|x|}{\alpha}\right)^{4\beta_1}\qquad&&\text{for all}\;\, x\in B(0,\alpha).
            \end{alignedat}\right.
        \end{align*}
        Notice that since $\beta_1>\dfrac{1}{2}$, there exists $p>1$ such that $\omega_{\alpha,\infty}\in L^p(\Sigma)$. 
        Previous estimate \eqref{weight_limit2} shows that
        \begin{align}\label{weight_limit3}
            \np{\frac{A_{\infty}}{\omega_{\infty}}}{\infty}{\Sigma\setminus\bar{B}_{\alpha}(0)}=\limsup_{k\rightarrow \infty}\np{\frac{A_{k}}{\omega_{k}^{\frac{1}{4}}}}{\infty}{\Sigma\setminus\bar{B}(0,\alpha)}=O(\alpha).
        \end{align}
        On the other hand, since $|\mathring{A}_k|^2e^{2\lambda_k}$ is a pointwise conformal invariant, and $\ens{\lambda_k}_{k\in\N}$ is bounded in $L^{\infty}(\Sigma)$ due to the absence of branch points, we deduce that we can assume that 
        \begin{align*}
            \vec{\Psi}_k(z)=\phi_k(\delta_kz)\conv{k\rightarrow \infty}\vec{\Psi}_{\infty}\quad \text{in}\;\, C^l_{\mathrm{loc}}(\widehat{\C})\;\, \text{for all}\;\, l\in\N.
        \end{align*}
        Then, we get
        \begin{align*}
            A_{\vec{\Psi}_k}(x)=\delta_k^{-1}A_k(\delta_k x).
        \end{align*}
        Now, recall that on $B(0,\alpha^{-1}\delta_k)$, we have
        \begin{align*}
            \omega_k(x)=\frac{1}{(\alpha^{-1}\delta_k)^4}\left(1+\left(\frac{\alpha^{-1}\delta_k}{\alpha}\right)^{4\beta_1}+\frac{1}{\log^2\left(\frac{\alpha^2}{\delta_k}\right)}\right).
        \end{align*}
        In particular, we have for all $x\in B(0,\alpha^{-1})$ 
        \begin{align}
            \frac{A_k(\delta_k x)}{\omega_k(\delta_k x)^{\frac{1}{4}}}=\frac{\delta_kA_{\vec{\Psi}_k}(x)}{(\alpha^{-1}\delta_k)}\left(1+\left(\frac{\delta_k}{\alpha^2}\right)^{4\beta_1}+\frac{1}{\log^2\left(\frac{\alpha^2}{\delta_k}\right)}\right)^{-\frac{1}{4}}\conv{k\rightarrow \infty}\alpha A_{\vec{\Psi}_{\infty}}(x),
        \end{align}
        which implies that
        \begin{align}\label{weight_limit4}
            \limsup_{k\rightarrow \infty}\np{\frac{A_k}{\omega_k^{\frac{1}{4}}}}{\infty}{B(0,\alpha^{-1}\delta_k)}=\alpha \np{A_{\vec{\Psi}_{\infty}}}{\infty}{B(0,\alpha^{-1})}=O(\alpha)
        \end{align}
        since $A_{\vec{\Psi}_{\infty}}\in L^{\infty}(S^2)$ by hypothesis since $\vec{\Psi}_{\infty}$ is an immersion (which allows us to apply the classical $\epsilon$-regularity \cite{riviere1}).
        Therefore, using the three inequalities \eqref{weight_limit1}, \eqref{weight_limit2}, and \eqref{weight_limit3}, defining 
        \begin{align}
            \mu_{\alpha,k}=\np{\frac{A_k}{\omega_k^{\frac{1}{4}}}}{\infty}{\Sigma},
        \end{align}
        we deduce that the two conditions below hold for some $\alpha_0>0$ and some $\mu_0<\infty$
        \begin{align}
        \left\{\begin{alignedat}{1}
            &\lim_{\alpha\rightarrow 0}\limsup_{k\rightarrow \infty}\mu_{\alpha,k}=0,\\
            &\text{For all}\;\, k\in\N,\;\, \text{for all}\;\,  0<\alpha<\alpha_0, \;\, \text{we have}\;\,  0<\mu_{\alpha,k}<\mu_0.
            \end{alignedat}\right.
        \end{align}
        Let $\lambda\in\R$ and assume that 
        \begin{align*}
            \dim\mathscr{E}_{\alpha,k}(\lambda)>0,
        \end{align*}
        where we recall that $\mathscr{E}_{\alpha,k}(\lambda)$ is defined in \eqref{weighted_eigenspace}.
        Then, there exists $u\in W^{2,2}(\Sigma)\setminus\ens{0}$ such that 
        \begin{align*}
            Q_{\phi_k}(u)=\frac{\lambda}{2}\int_{\Sigma}u^2\,\omega_{\alpha,k}\,d\mathrm{vol}_{g_k},
        \end{align*}
        or
        \begin{align}\label{vp_equality}
            &\frac{1}{2}\int_{\Sigma}\left(\Delta_{g_k}u+|A_k|^2u\right)^2d\mathrm{vol}_{g_k}+\int_{\Sigma}\left(|du|_{g_k}^2+4|h_{0,k}|_{\mathrm{WP}}^2u^2\right)H_k^2d\mathrm{vol}_{g_k}-8\int_{\Sigma}\s{\partial u\otimes\partial u}{h_{0,k}}_{\mathrm{WP}}H_k\,d\mathrm{vol}_{g_k}\nonumber\\
            &-16\int_{\Sigma}\s{\partial u\otimes \partial H_k}{h_{0,k}}_{\mathrm{WP}}u\,d\mathrm{vol}_{g_k}=\frac{\lambda}{2}\int_{\Sigma}u^2\,\omega_{1,\alpha,k}\,d\mathrm{vol}_{g_k}.
        \end{align}
        Now, let us prove that
        \begin{align*}
            \lim_{\alpha\rightarrow 0}\limsup_{k\rightarrow \infty}\np{\frac{A_k\D A_k}{\sqrt{\omega_{1,\alpha,k}\omega_{2,\alpha,k}}}}{\infty}{\Sigma}=0.
        \end{align*}
        This estimate holds in $\Omega_k(\alpha)$ thanks to the proof of Theorem \ref{neck_positive}, provided that $\dfrac{\sqrt{2}}{2}<\beta<1$ and $0<\epsilon<1$ satisfy the conditions \eqref{fine_cond2}. Likewise, thanks to the strong convergence on $\Sigma\setminus\bar{B}(0,\alpha)$, we get
        \begin{align*}
            \lim_{k\rightarrow \infty}\np{\frac{A_k\D A_k}{\sqrt{\omega_{1,\alpha,k}\omega_{2,\alpha,k}}}}{\infty}{\Sigma}=\np{\frac{A_{\infty}\D A_{\infty}}{\sqrt{\omega_{1,\alpha,\infty}\omega_{2,\alpha,\infty}}}}{\infty}{\Sigma\setminus\bar{B}(0,\alpha)}.
        \end{align*}
        On $\Sigma\setminus\bar{B}(0,\alpha)$, we have
        \begin{align*}
            \omega_{2,\alpha,\infty}(x)=\frac{1}{\alpha^2},
        \end{align*}
        which implies since $A_{\infty}\in C^{\infty}(\Sigma)$ that
        \begin{align*}
            \np{\frac{A_{\infty}\D A_{\infty}}{\sqrt{\omega_{1,\alpha,\infty}\omega_{2,\alpha,\infty}}}}{\infty}{\Sigma\setminus\bar{B}(0,\alpha)}=O(\alpha^3).
        \end{align*}
        Then, on the bubble, we have
        \begin{align*}
            \D A_{\vec{\Psi}_k}(x)=\delta_k^{-2}\D A_k(\delta_kx),
        \end{align*}
        while
        \begin{align*}
            \sqrt{\omega_{1,\alpha,k}\omega_{2,\alpha,k}}=\frac{1}{(\alpha^{-1}\delta_k)^3}\sqrt{1+\left(\frac{\alpha^{-1}\delta_k}{\alpha}\right)^{4\beta_1}+\frac{1}{\log^2\left(\frac{\alpha^2}{\delta_k}\right)}}\sqrt{1+\left(\frac{\alpha^{-1}\delta_k}{\alpha}\right)^{2\beta_2}+\frac{1}{\log^2\left(\frac{\alpha^2}{\delta_k}\right)}},
        \end{align*}
        which implies that for all $x\in B(0,\alpha^{-1})$,
        \begin{align*}
            &\frac{A_k(\delta_k x)\D A_k(\delta_k x)}{\sqrt{\omega_{1,\alpha,k}(\delta_kx)\omega_{2,\alpha,k}(\delta_kx)}}=\alpha^3 A_{\vec{\Psi_k}}(x)\D A_{\vec{\Psi}_k}(x)\left(1+\left(\frac{\alpha^{-1}\delta_k}{\alpha}\right)^{4\beta_1}+\frac{1}{\log^2\left(\frac{\alpha^2}{\delta_k}\right)}\right)^{-\frac{1}{2}}\\
            &\times \left(1+\left(\frac{\alpha^{-1}\delta_k}{\alpha}\right)^{2\beta_2}+\frac{1}{\log^2\left(\frac{\alpha^2}{\delta_k}\right)}\right)^{-\frac{1}{2}}\conv{k\rightarrow \infty}\alpha^3 A_{\vec{\Psi}_{\infty}}(x)\D A_{\vec{\Psi}_{\infty}}(x)=O(\alpha^3).
        \end{align*}
        Therefore, if 
        \begin{align}
            \nu_{\alpha,k}=\np{\frac{A_k\D A_k}{\sqrt{\omega_{1,\alpha,k}\omega_{2,\alpha,k}}}}{\infty}{\Sigma},
        \end{align}
        we deduce that the two conditions below hold for some $\alpha_0>0$ and some $\nu_0<\infty$
        \begin{align}
        \left\{\begin{alignedat}{1}
            &\lim_{\alpha\rightarrow 0}\limsup_{k\rightarrow \infty}\nu_{\alpha,k}=0,\\
            &\text{For all}\;\, k\in\N,\;\, \text{for all}\;\,  0<\alpha<\alpha_0, \;\, \text{we have}\;\,  0<\nu_{\alpha,k}<\nu_0.
            \end{alignedat}\right.
        \end{align}
        Then, we get
        \begin{align}\label{delicate_term1}
            \left|\int_{\Sigma}\s{\partial u\otimes\partial H_k}{h_{0,k}}_{\mathrm{WP}}\,u\,d\mathrm{vol}_{g_k}\right|\leq \nu_{\alpha,k}\left(\int_{\Sigma}u^2\,\omega_{1,\alpha,k}\,d\mathrm{vol}_{g_k}\right)^{\frac{1}{2}}\left(\int_{\Sigma}|du|_{g_k}^2\omega_{2,\alpha,k}\,d\mathrm{vol}_{g_k}\right)^{\frac{1}{2}}.
        \end{align}
        Since $\omega_{2,\alpha,k}$ is constant and Lipschitzian on $\Sigma\setminus\Omega_k(\alpha)$, integrating by parts, we get
        \begin{align}\label{unmerry_terms1}
            \int_{\Sigma}|du|_{g_k}^2\,\omega_{2,\alpha,k}\,d\mathrm{vol}_{g_k}=-\int_{\Sigma}u\left(\Delta_{g_k}u\right)\,\omega_{2,\alpha,k}\,d\mathrm{vol}_{g_k}-\int_{\Omega_k(\alpha)}u\,\D u\cdot \D\omega_{2,\alpha,k}\,dx.
        \end{align}
        We first estimate thanks to Cauchy-Schwarz inequality
        \begin{align}\label{unmerry_terms2}
            \left|\int_{\Sigma}u\left(\Delta_{g_k}u\right)\,\omega_{2,\alpha,k}\,d\mathrm{vol}_{g_k}\right|\leq \left(\int_{\Sigma}u^2\,\omega_{2,\alpha,k}^2\,d\mathrm{vol}_{g_k}\right)^{\frac{1}{2}}\left(\int_{\Sigma}(\Delta_{g_k}u)^2d\mathrm{vol}_{g_k}\right)^{\frac{1}{2}}.
        \end{align}
        Then, if we choose $\beta_2\geq \beta_1$, we deduce by the elementary inequality $(a+b+c)^2\leq 4(a^2+b^2+c^2)$ that 
        \begin{align*}
            \omega_{2,\alpha,k}^2\leq \frac{4}{|x|^4}\left(\left(\frac{|x|}{\alpha}\right)^{4\beta_2}+\left(\frac{\alpha^{-1}\delta_k}{|x|}\right)^{4\beta_2}+\frac{1}{\log^4\left(\frac{\alpha^2}{\delta_k}\right)}\right)\leq 4\,\omega_{1,\alpha,k}\qquad\text{in}\;\, \Omega_k(\alpha).
        \end{align*}
        The same inequality on $\Sigma\setminus\bar{\Omega_k(\alpha)}$ implies that
        \begin{align}\label{weight_relation}
            \omega_{2,\alpha,k}^2\leq 4\,\omega_{1,\alpha,k}\qquad \text{in}\;\, \Sigma.
        \end{align}
        Therefore, \eqref{unmerry_terms1} shows that
        \begin{align}\label{unmerry_terms3}
            \left|\int_{\Sigma}u\left(\Delta_{g_k}u\right)\,\omega_{2,\alpha,k}\,d\mathrm{vol}_{g_k}\right|\leq 2\left(\int_{\Sigma}u^2\,\omega_{1,\alpha,k}\,d\mathrm{vol}_{g_k}\right)^{\frac{1}{2}}\left(\int_{\Sigma}\left(\Delta_{g_k}u\right)^2d\mathrm{vol}_{g_k}\right)^{\frac{1}{2}}.
        \end{align}
        Now, we have
        \begin{align*}
            \D\omega_{2,\alpha,k}=-2\frac{x}{|x|^2}\omega_{2,\alpha,k}+2\beta_2\frac{x}{|x|^4}\left(\left(\frac{|x|}{\alpha}\right)^{2\beta_2}-\left(\frac{\alpha^{-1}\delta_k}{|x|}\right)^{2\beta_2}\right).
        \end{align*}
        Therefore, we have
        \begin{align*}
            |\D \omega_{2,\alpha,k}|\leq \frac{2(1+\beta_2)}{|x|}\omega_{2,\alpha,k}.
        \end{align*}
        We first have by Cauchy's inequality 
        \begin{align}\label{unmerry_terms4}
            &2(1+\beta_2)\left|\int_{\Omega_k(\alpha)}u\,\D u\cdot \frac{x}{|x|^4}\frac{1}{\log^2\left(\frac{\alpha}{\delta_k}\right)}dx\right|\leq\frac{1}{2}\int_{\Omega_k(\alpha)}\frac{|\D u|^2}{|x|^2}\frac{1}{\log^2\left(\frac{\alpha^2}{\delta_k}\right)}dx\nonumber\\
            &+2(1+\beta_2)^2\int_{\Omega_k(\alpha)}\frac{u^2}{|x|^4}\frac{1}{\log^2\left(\frac{\alpha^2}{\delta_k}\right)}dx
            \leq \frac{1}{2}\int_{\Sigma}|du|_{g_k}^2\,\omega_{2,\alpha,k}\,d\mathrm{vol}_{g_k}+C\int_{\Sigma}u^2\,\omega_{1,\alpha,k}\,d\mathrm{vol}_{g_k}.
        \end{align}
        Then, using inequality \eqref{proof_main_8} (where $\beta$ is replaced by $\beta_2$), provided that $\beta_2>\dfrac{\sqrt{2}}{2}$, there exists $\dfrac{1}{2}<\beta_1'<1$ and $\sqrt{2}-1<\beta_2'<1$ such that
        \begin{align}\label{unmerry_terms5}
            &\int_{\Omega_k(\alpha)}\frac{|\D u|}{|x|}\frac{|u|}{|x|^2}\left(\left(\frac{|x|}{\alpha}\right)^{2\beta_2}+\left(\frac{\alpha^{-1}\delta_k}{|x|}\right)^{2\beta_2}\right)dx
            \leq 2\left(\int_{\Omega_k(\alpha)}\frac{|\D u|^2}{|x|^2}\left(\left(\frac{|x|}{\alpha}\right)^{2\beta_2'}+\left(\frac{\alpha^{-1}\delta_k}{|x|}\right)^{2\beta_2'}\right)dx\right)^{\frac{1}{2}}\nonumber\\
            &\times \left(\int_{\Omega_k(\alpha)}\frac{u^2}{|x|^4}\left(\left(\frac{|x|}{\alpha}\right)^{4\beta_1'}+\left(\frac{\alpha^{-1}\delta_k}{|x|}\right)^{4\beta_1'}\right)dx\right)^{\frac{1}{2}}
        \end{align}
        Furthermore, the chosen $0<\epsilon<1$ in \eqref{fine_cond2} is such that $\epsilon<\dfrac{1}{2}$ (which is in fact trivial since $\beta_2<1$ and $4\beta_1'=4(1-\epsilon)\beta_2>2$). In particular, we have $\beta_1'\geq\beta_2\geq \beta_1$, so by the elementary inequality
        \begin{align*}
            x^{4\beta_1'}\leq x^{4\beta_1}\quad \text{for all}\;\, 0<x<1,
        \end{align*}
        we deduce that 
        \begin{align}\label{unmerry_terms6}
            \int_{\Omega_k(\alpha)}\frac{u^2}{|x|^4}\left(\left(\frac{|x|}{b}\right)^{4\beta_1'}+\left(\frac{a}{|x|}\right)^{4\beta_1'}\right)dx\leq C\int_{\Sigma}u^2\,\omega_{1,\alpha,k}\,d\mathrm{vol}_{g_k}.
        \end{align}
        Now, this is where the weighted Poincaré estimate comes in, since $\beta_2'<\beta_2$, which prevents one from controlling the associated weighted integral by the initial integral. Thanks to Theorem \ref{interpolation_weighted_poincare}, we deduce that
        \begin{align}\label{unmerry_terms7}
            \int_{\Omega_k(\alpha)}\frac{|\D u|^2}{|x|^2}\left(\left(\frac{|x|}{b}\right)^{2\beta_2'}+\left(\frac{a}{|x|}\right)^{2\beta_2'}\right)dx&\leq C_{\beta_1,\beta_2'}\int_{\Omega_k(\alpha)}\frac{u^2}{|x|^4}\left(\left(\frac{|x|}{\alpha}\right)^{4\beta_1}+\left(\frac{\alpha^{-1}\delta_k}{|x|}\right)^{4\beta_1}\right)dx\nonumber\\
            &+C_{\beta_1,\beta_2'}\int_{\Omega_k(\alpha)}|\D^2u|^2dx.
        \end{align}
        Now, by the uniform boundedness of $\ens{g_k}_{k\in\N}$ (see \cite[Remark $2.36$ p. $56$]{aubin} for example), we deduce that  that
        \begin{align*}
            \int_{\Omega_k(\alpha)}|\D^2u|^2dx\leq \int_{\Sigma}|\mathrm{Hess}(u)|^2_{g_0}d\mathrm{vol}_{g_0}\leq C\int_{\Sigma}(\Delta_{g_k}u)^2d\mathrm{vol}_{g_k},
        \end{align*}
        where $g_0$ is any smooth metric chosen flat on $\Omega_k(\alpha)$, and $C<\infty$ is independent of $\alpha>0$ and $k\in\N$. Therefore, \eqref{unmerry_terms5}, \eqref{unmerry_terms6}, and \eqref{unmerry_terms7} show that 
        \begin{align}\label{unmerry_term8}
            &\int_{\Omega_k(\alpha)}\frac{|\D u|}{|x|}\frac{|u|}{|x|^2}\left(\left(\frac{|x|}{\alpha}\right)^{2\beta_2}+\left(\frac{\alpha^{-1}\delta_k}{|x|}\right)^{2\beta_2}\right)dx\leq C\int_{\Omega_k(\alpha)}u^2\,\omega_{1,\alpha,k}\,d\mathrm{vol}_{g_k}\nonumber\\
            &+C\left(\int_{\Omega_k(\alpha)}u^2\,\omega_{1,\alpha,k}\,d\mathrm{vol}_{g_k}\right)^{\frac{1}{2}}\left(\int_{\Sigma}\left(\Delta_{g_k}u\right)^2d\mathrm{vol}_{g_k}\right)^{\frac{1}{2}}.
        \end{align}
        We deduce by \eqref{unmerry_terms1}, \eqref{unmerry_terms3}, \eqref{unmerry_terms4}, and \eqref{unmerry_term8} that
        \begin{align*}
            \int_{\Sigma}|du|_{g_k}^2\omega_{2,\alpha,k}d\mathrm{vol}_{g_k}&\leq \frac{1}{2}\int_{\Sigma}|du|_{g_k}^2\,\omega_{2,\alpha,k}\,d\mathrm{vol}_{g_k}+C\int_{\Omega_k(\alpha)}u^2\,\omega_{1,\alpha,k}\,d\mathrm{vol}_{g_k}\nonumber\\
            &+C\left(\int_{\Omega_k(\alpha)}u^2\,\omega_{1,\alpha,k}\,d\mathrm{vol}_{g_k}\right)^{\frac{1}{2}}\left(\int_{\Sigma}\left(\Delta_{g_k}u\right)^2d\mathrm{vol}_{g_k}\right)^{\frac{1}{2}},
        \end{align*}
        and finally
        \begin{align}\label{ipp_weighted}
            \int_{\Sigma}|du|_{g_k}^2\omega_{2,\alpha,k}d\mathrm{vol}_{g_k}&\leq C\int_{\Omega_k(\alpha)}u^2\,\omega_{1,\alpha,k}\,d\mathrm{vol}_{g_k}
            +C\left(\int_{\Omega_k(\alpha)}u^2\,\omega_{1,\alpha,k}\,d\mathrm{vol}_{g_k}\right)^{\frac{1}{2}}\left(\int_{\Sigma}\left(\Delta_{g_k}u\right)^2d\mathrm{vol}_{g_k}\right)^{\frac{1}{2}}
        \end{align}
        where we used the fact that $\ens{g_k}_{k\in\N}$ is a uniformly bounded metric. Therefore, we get by \eqref{delicate_term1} and \eqref{ipp_weighted}
        \begin{align}\label{main_control1}
            16\left|\int_{\Sigma}\s{\partial u\otimes\partial H_k}{h_{0,k}}_{\mathrm{WP}}u\,d\mathrm{vol}_{g_k}\right|\leq C_{\beta_1,\beta_2}\,\nu_{\alpha,k}\int_{\Sigma}u^2\,\omega_{1,\alpha,k}\,d\mathrm{vol}_{g_k}+C_{\beta_1,\beta_2}\,\nu_{\alpha,k}\int_{\Sigma}(\Delta_{g_k}u)^2d\mathrm{vol}_{g_k},
        \end{align}
        and if $\kappa_{\alpha,k}$ is defined like $\mu_{\alpha,k}$ but with $\beta_1$ replaced by $\beta_2$, we get by \eqref{ipp_weighted}
        \begin{align}\label{main_control2}
            &16\left|\int_{\Sigma}\s{\partial u\otimes\partial u}{h_{0,k}}_{\mathrm{WP}}H_kd\mathrm{vol}_{g_k}\right|\leq \kappa_{\alpha,k}^2\int_{\Sigma}|du|_{g_k}^2\omega_{2,\alpha,k}\,d\mathrm{vol}_{g_k}\nonumber\\
            &\leq C_{\beta_1,\beta_2}\,\kappa_{\alpha,k}^2\int_{\Sigma}u^2\,\omega_{1,\alpha,k}\,d\mathrm{vol}_{g_k}+C_{\beta_1,\beta_2}\,\kappa_{\alpha,k}^2\int_{\Sigma}(\Delta_{g_k}u)^2d\mathrm{vol}_{g_k}.
        \end{align}
        Then, we have
        \begin{align}\label{main_control3}
            &\frac{1}{2}\int_{\Sigma}\left(\Delta_{g_k}u+|A_k|^2u\right)^2d\mathrm{vol}_{g_k}\geq \frac{1}{4}\int_{\Sigma}(\Delta_{g_k}u)^2d\mathrm{vol}_{g_k}-\int_{\Sigma}|A_k|^4u^2d\mathrm{vol}_{g_k}\nonumber\\
            &\geq \frac{1}{4}\int_{\Sigma}(\Delta_{g_k}u)^2d\mathrm{vol}_{g_k}-\mu_{\alpha,k}^4\int_{\Sigma}u^2\,\omega_{1,\alpha,k}\,d\mathrm{vol}_{g_k}.
        \end{align}
        Therefore, by \eqref{der2_redite}, \eqref{main_control1}, \eqref{main_control2}, and \eqref{main_control3}, we deduce that
        \begin{align}\label{main_control4}
            Q_{\phi_k}(u)&\geq \left(\frac{1}{4}-C_{\beta_1,\beta_2}\left(\nu_{\alpha,k}+\kappa_{\alpha,k}^2+\mu_{\alpha,k}^4\right)\right)\int_{\Sigma}(\Delta_{g_k}u)^2d\mathrm{vol}_{g_k}\nonumber\\
            &-C_{\beta_1,\beta_2}\left(\nu_{\alpha,k}+\kappa_{\alpha,k}^2+\mu_{\alpha,k}^4\right)\int_{\Sigma}u^2\,\omega_{1,\alpha,k}\,d\mathrm{vol}_{g_k}.
        \end{align}
        Now, let $\alpha_0>0$ be such that 
        \begin{align*}
            \limsup_{k\rightarrow\infty}\left(\nu_{\alpha_0,k}+\kappa_{\alpha_0,k}^2+\mu_{\alpha_0,k}^4\right)< \frac{1}{4}\frac{1}{C_{\beta_1,\beta_2}}.
        \end{align*}
        Then, for all $0<\alpha\leq \alpha_0$, if \eqref{vp_equality} holds, we get for all $k\in\N$ large enough by \eqref{main_control4}
        \begin{align}
            \frac{\lambda}{2}\int_{\Sigma}u^2\,\omega_{1,\alpha,k}\,d\mathrm{vol}_{g_k}=Q_{\phi_k}(u)\geq -C_{\beta_1,\beta_2}\left(\nu_{\alpha,k}+\kappa_{\alpha,k}^2+\mu_{\alpha,k}^4\right)\int_{\Sigma}u^2\,\omega_{1,\alpha,k}\,d\mathrm{vol}_{g_k},
        \end{align}
        which concludes the proof of the theorem. 
    \end{proof}

    We now generalise \cite[Lemma B.$1$]{riviere_morse_scs} to our setting. 

    \begin{lemme}\label{Lemma_B.1}
        Let $\phi:\Sigma \rightarrow \R^n$ be a smooth Willmore immersion. Let $p\in\Sigma$, let $\omega\in C^{\infty}(\Sigma\setminus\ens{p})$ such that for some $\dfrac{1}{2}<\beta<1$, 
        \begin{align}
            0\leq \omega_0<\omega(x)\leq \frac{C_0}{\mathrm{dist}(x,p)^{4(1-\beta)}}\quad \text{and}\quad |\D \omega_0(x)|\leq \frac{C_l}{\mathrm{dist}(x,p)^{4(1-\beta)+l}}\;\, \text{for all}\;\, 0\leq l\in \N,
        \end{align}
        where $\ens{C_l}_{l\in\N}\subset (0,\infty)$. Denote
        \begin{align}\label{weighted_L2}
            L^2_{\omega}(\Sigma)=L^2(\Sigma)\cap\ens{u:\int_{\Sigma}u^2\,\omega\,d\mathrm{vol}_{g_0}<\infty},
        \end{align}
        where $g_0$ is any smooth metric on $\Sigma$. If $g=\phi^{\ast}g_{\R^n}$ is the induced metric on $\Sigma$. Consider the operator
        \begin{align*}
            \mathcal{L}_{g,\omega}&=\omega^{-1}\mathcal{L}_g=\omega^{-1}\left(\Delta_g^2+|A|^2\Delta_g+2\s{d|A|^2}{d(\,\cdot\,)}_g+2\,d^{\ast_g}\left(H^2\,d(\,\cdot\,)\right)\right.\\
            &\left.+\left(|A|^4+\Delta_g|A|^2+24\,H^2|h_0|_{\mathrm{WP}}^2\right)+16\,\ast_g\,d\,\Re\left(g^{-1}\otimes h_0\otimes\partial(H\,\cdot\,)\right)-16\s{\partial(\,\cdot\,)\otimes\partial H}{h_0}_{\mathrm{WP}}\right)
        \end{align*}
        acting on $W^{2,2}(\Sigma)$ functions. Then, there exists a Hilbertian base of $L^2_{\omega}(\Sigma)$ made of eigenvectors of $\mathcal{L}_{g,\omega}$ whose eigenvalues satisfy
        \begin{align*}
            \lambda_1\leq \lambda_2\leq \cdots\leq \lambda_k\conv{k\rightarrow \infty}\infty. 
        \end{align*}
    \end{lemme}
    \begin{rem}
        Notice that the definition of the weighted $L^2$ space $L^2_{\omega}$ does not depend on the choice of the metric $g_0$ on $\Sigma$.
    \end{rem}
    \begin{proof}
        The proof is similar to the one of Theorem \ref{diagonalisation}. First, as $\phi$ is smooth, $\leb_{g,\omega}u$ is a well-defined distribution for all $u\in L^2_{\omega}(\Sigma)$. For all $\lambda>0$, let $\mathcal{L}_{g,\omega}^{\lambda}=\mathcal{L}_{g,\omega}+\lambda\mathrm{Id}$. Then, let $f\in L^2_{\omega}(\Sigma)$ and $u\in W^{2,2}(\Sigma)$ is such that
        \begin{align*}
            \mathcal{L}_{g,\omega}^{\lambda}u=f.
        \end{align*}
        Now, define the bilinear form $B_{\lambda}$ on $W^{2,2}(\Sigma)$
        \begin{align*}
            B_{\lambda}(u,v)=\frac{1}{2}\int_{\Sigma}u\,\mathcal{L}_gu\,d\vg+\frac{\lambda}{2}\int_{\Sigma}u\,v\,d\vg,
        \end{align*}
        and let $Q_{\lambda}(u)=B_{\lambda}(u,u)$. 
        Recalling \eqref{ineq_d2_final}, we deduce that 
        \begin{align*}
            \int_{\Sigma}f\,u\,\omega\,d\vg\geq \int_{\Sigma}(\Delta_gu)^2d\vg-\int_{\Sigma}\left(\Delta_g|A|^2+10|A|^4-8\s{\D^2_{\partial,\partial}H}{h_0}_{\mathrm{WP}}\right)u^2d\vg
            &+\frac{\lambda}{2}\int_{\Sigma}u^2\,\omega\,d\vg.
        \end{align*}
        Therefore, taking $\lambda>\lambda_0=2\np{\Delta_g|A|^2+10|A|^4-8\s{\D^2_{\partial,\partial}H}{h_0}_{\mathrm{WP}}}{\infty}{\Sigma}$, we deduce by Cauchy-Schwarz inequality that
        \begin{align*}
            &\frac{1}{8}\int_{\Sigma}(\Delta_gu)^2d\vg+\frac{\lambda-\lambda_0}{2}\int_{\Sigma}u^2\,\omega\,d\vg\leq \left(\int_{\Sigma}u^2\,\omega\,d\vg\right)^{\frac{1}{2}}\left(\int_{\Sigma}f^2\,\omega\,d\vg\right)^{\frac{1}{2}}\\
            &\leq \frac{\lambda-\lambda_0}{4}\int_{\Sigma}u^2\,\omega\,d\vg+\frac{4}{\lambda-\lambda_0}\int_{\Sigma}f^2\,\omega\,d\vg,
        \end{align*}
        which finally implies that
        \begin{align*}
            \frac{1}{8}\int_{\Sigma}(\Delta_gu)^2d\vg+\frac{\lambda-\lambda_0}{4}\int_{\Sigma}u^2\,\omega\,d\vg\leq \frac{4}{\lambda-\lambda_0}\int_{\Sigma}f^2\,\omega\,d\vg,
        \end{align*}
        The rest of the proof is exactly the same as the one of Theorem \ref{diagonalisation} and we omit it. 
    \end{proof}

    Now, introduce as in \cite{riviere_morse_scs}
    \begin{align*}
        \mathscr{E}_{\alpha,\infty}^0=\bigoplus_{\lambda\leq 0}\mathscr{E}_{\alpha,\infty}(\lambda),
    \end{align*}
    where $\mathcal{L}_{\alpha,\infty}^0=\omega_{1,\alpha,\infty}^{-1}\mathcal{L}_{g_{\phi_{\infty}}}$, and $\mathscr{E}_{\alpha,\infty}(\lambda)$ is defined as in \eqref{weighted_eigenspace}.

    Then, we also get
    \begin{align*}
        \mathrm{dim}(\mathscr{E}_{\alpha,\infty})\leq \mathrm{Ind}_W(\phi_{\infty})+\mathrm{Null}_{W}(\phi_{\infty}).
    \end{align*}

    We can finally move to the proof of the main theorem. 

    \begin{proof}
        Consider the finite-dimensional sphere
        \begin{align}\label{spectrum_sphere}
            \mathscr{S}_{\alpha,k}=\bigoplus_{\lambda\leq 0}\mathscr{E}_{\alpha,k}(\lambda)\cap\ens{u:\int_{\Sigma}u^2\,\omega_{1,\alpha,k}\,d\mathrm{vol}_{g_k}=1}.
        \end{align}
        Let $u_k\in \mathscr{S}_{\alpha,k}$. Thanks to \eqref{main_control4}, for $k$ large enough, we have
        \begin{align*}
            Q_{\phi_k}(u_k)\geq \frac{1}{8}\int_{\Sigma}(\Delta_{g_k}u_k)^2d\mathrm{vol}_{g_k}-\mu_{\alpha,k}\int_{\Sigma}u^2\,\omega_{1,\alpha,k}\,d\mathrm{vol}_{g_k},
        \end{align*}
        while for some $\lambda>-\mu_{\alpha,k}\geq -\mu_0$, we have
        \begin{align*}
            Q_{\phi_k}(u_k)=\frac{\lambda}{2}\int_{\Sigma}u_k^2\,\omega_{1,\alpha,k}\,d\mathrm{vol}_{g_k}\geq -\frac{\mu_0}{2}\int_{\Sigma}(\Delta_{g_k}u)^2d\mathrm{vol}_{g_k}.
        \end{align*}
        Therefore, we get
        \begin{align*}
            \int_{\Sigma}(\Delta_{g_k}u_k)^2d\mathrm{vol}_{g_k}\leq \frac{\mu_0}{2}.
        \end{align*}
        Then, since $\Sigma$ is closed, we have by the Cauchy-Schwarz inequality
        \begin{align*}
            \int_{\Sigma}|du_k|_{g_k}^2d\mathrm{vol}_{g_k}&=-\int_{\Sigma}u_k\,\Delta_{g_k}u_k\,d\mathrm{vol}_{g_k}\leq \left(\int_{\Sigma}u_k^2\,d\mathrm{vol}_{g_k}\right)^{\frac{1}{2}}\left(\int_{\Sigma}(\Delta_{g_k}u_k)^2d\mathrm{vol}_{g_k}\right)^{\frac{1}{2}}\\
            &\leq C_{\alpha}\mu_0.
        \end{align*}
        Therefore, $\ens{u_k}_{k\in\N}$ is bounded in $W^{2,2}(\Sigma)$. Up to a subsequence, we deduce that 
        \begin{align*}
            u_k\underset{k\rightarrow \infty}{\hookrightarrow}u_{\infty}\quad\text{in}\;\,W^{2,2}(\Sigma)\qquad \text{and}\qquad u_k(\delta_kz)\underset{k\rightarrow \infty}{\hookrightarrow}v_{\infty}\quad \text{in}\;\, W^{2,2}_{\mathrm{loc}}(\C).
        \end{align*}

        \textbf{Claim 1: $u_{\infty}\neq 0$ or $v_{\infty}\neq 0$.} 
        
        Let $\ens{\varphi_{k,j}}_{1\leq j\leq N_k}$ be an orthonormal base of 
        \begin{align*}
            \mathscr{E}_{\alpha,k}^0=\bigoplus_{\lambda\leq 0}\mathscr{E}_{\alpha,k}(\lambda).
        \end{align*}
        For all $1\leq j\leq N_k$, we have
        \begin{align*}
            \mathcal{L}_{g_k}\varphi_{k,j}=\lambda_{k,j}\,\omega_{1,\alpha,k}\,\varphi_{k,j}.
        \end{align*}
        Since $u_k\in \mathscr{S}_{\alpha,k}\simeq S^{N_k-1}$, we have a decomposition
        \begin{align*}
            u_k=\sum_{j=1}^{N_k}a_{k,j}\varphi_{k,j},
        \end{align*}
        where $a=\ens{a_{k,j}}_{1\leq j\leq N_k}\in S^{N_k-1}$. Then, we get as $-\mu_0\leq \lambda_{k,j}\leq 0$ the inequality
        \begin{align*}
            \int_{\Sigma}\frac{1}{\omega_{1,\alpha,k}}|\mathcal{L}_{g_k}u_k|^2d\mathrm{vol}_{g_k}&=\int_{\Sigma}\sum_{i,j=1}^{N_k}a_{k,i}a_{k,j}\lambda_{k,i}\lambda_{k,j}\varphi_{k,i}\varphi_{k,j}\,\omega_{1,\alpha,k}d\mathrm{vol}_{g_k}\\
            &=\sum_{j=1}^{N_k}|a_{k,j}|^2|\lambda_{k,j}|^2\leq \mu_0^2.
        \end{align*}
        Therefore, there exists a uniformly bounded sequence of functions $\ens{f_k}_{k\in\N}\subset L^2(\Sigma)$ such that
        \begin{align*}
            \mathcal{L}_{g_k}u_k=\sqrt{\omega_{\alpha,k}}f_k.
        \end{align*}
        By classical elliptic estimates, as $\ens{g_k}_{k\in\N}$ is uniformly bounded and smooth, we deduce that
        \begin{align*}
            \np{\D^4u_k}{2}{\Sigma\setminus\bar{B}(0,\alpha)}\leq C_{\alpha}.
        \end{align*}
        Likewise, making a linear change of variable, we deduce that 
        \begin{align*}
            \np{\D^4v_k}{2}{B(0,\alpha^{-1})}\leq C_{\alpha},
        \end{align*}
        where $v_k(x)=u_k(\rho_kx)$. Therefore, we have
        \begin{align}\label{w42_convergence}
            u_k\underset{k\rightarrow \infty}{\rightharpoonup}u_{\infty}\quad \text{in}\;\, W^{4,2}_{\mathrm{loc}}(\Sigma\setminus\ens{0})\quad \text{and}\quad v_k\underset{k\rightarrow \infty}{\rightharpoonup}v_{\infty}\quad \text{in}\;\, W^{4,2}_{\mathrm{loc}}(\C). 
        \end{align}
        Assume by contradiction that $u_{\infty}=0$ and $v_{\infty}=0$. Let $\chi_{\alpha,k}\in C^{\infty}(\Sigma)$ be a cutoff function such that $\chi_{\alpha,k}=1$ in $\Sigma\setminus\bar{\Omega_k(\alpha)}$ and $\chi_{\alpha,k}=0$ in $\Omega_k(\alpha/2)$ and $\widetilde{u}_k=\chi_{\alpha,k}u_k$. We can also assume that for some universal constant $C<\infty$ independent of $\alpha$ and $k$, we have
        \begin{align*}
            |\D^l\chi_{\alpha,k}|\leq \frac{C}{|x|^l}\quad \text{for all}\;\,x\in \left(B_{\alpha}\setminus\bar{B}_{\alpha/2}(0)\right)\cup\;\,\left(B_{2\alpha^{-1}\delta_k}\setminus\bar{B}_{\alpha^{-1}\delta_k}(0)\right) \text{for all}\;\,l=0,1,2.
        \end{align*}
        First, we trivially get by the contradiction hypothesis and the strong convergence that
        \begin{align*}
            \np{\D^2u_k(1-\chi_{\alpha},k)}{2}{\Sigma}\conv{k\rightarrow \infty}0.
        \end{align*}
        Then, we get
        \begin{align*}
            \int_{\Sigma}|du_k|_{g_k}|^2|d\chi_{\alpha,k}|_{g_k}^2d\mathrm{vol}_{g_k}&\leq C^2\int_{B_{\alpha}\setminus\bar{B}_{\frac{\alpha}{2}}(0)}\frac{|\D u_k|^2}{|x|^2}dx+C^2\int_{B_{2\alpha^{-1}}\setminus\bar{B}_{\alpha^{-1}}(0)}\frac{|\D v_k|^2}{|x|^2}dx\\
            &\conv{k\rightarrow \infty}C^2\int_{B_{\alpha}\setminus\bar{B}_{\frac{\alpha}{2}}(0)}\frac{|\D u_{\infty}|^2}{|x|^2}dx+C^2\int_{B_{2\alpha^{-1}}\setminus\bar{B}_{\alpha^{-1}}(0)}\frac{|\D v_{\infty}|^2}{|x|^2}dx=0
        \end{align*}
        by assumption. Likewise, we have
        \begin{align*}
            \limsup_{k\rightarrow \infty}\int_{\Sigma}u_k^2|\D^2\chi_{\alpha,k}|^2_{g_k}d\mathrm{vol}_{g_k}\leq C^2\int_{B_{\alpha}\setminus\bar{B}_{\frac{\alpha}{2}}(0)}\frac{u_{\infty}^2}{|x|^4}dx+C^2\int_{B_{2\alpha^{-1}}\setminus\bar{B}_{\alpha^{-1}}(0)}\frac{v_{\infty}^2}{|x|^4}dx=0.
        \end{align*}
        Therefore, we deduce that
        \begin{align}\label{w22_conv}
            \lim_{k\rightarrow \infty}\wp{u_k-\widetilde{u}_k}{2,2}{\Sigma}=0.
        \end{align}
        Now, recall the formula
        \begin{align*}
            Q_{\phi_k}(u_k)&=\frac{1}{2}\int_{\Sigma}\left(\Delta_{g_k}u_k+|A_k|^2u_k\right)^2d\mathrm{vol}_{g_k}+\int_{\Sigma}\left(|du_k|_{g_k}^2+4|h_{0,k}|_{WP}^2u_k^2\right)H_k^2d\vg\\
            &-8\int_{\Sigma}\s{\partial u_k\otimes\partial u_k}{h_{0,k}}_{WP}H_k\,d\mathrm{vol}_{g_k}
            -16\int_{\Sigma}\s{\partial u_k\otimes\partial H_k}{h_{0,k}}_{\mathrm{WP}}u_k\,d\mathrm{vol}_{g_k}.
        \end{align*}
        Thanks to the boundedness of $\ens{g_k}_{k\in\N}$ and $\ens{A_k}_{k\in\N}$, we deduce by \eqref{w22_conv} that 
        \begin{align}\label{conv_second_der}
            \lim_{k\rightarrow \infty}|Q_{\phi_k}(u_k)-Q_{\phi_k}(\widetilde{u}_k)|=0.
        \end{align}
        Then, we have
        \begin{align*}
            &\left|\int_{\Sigma}u_k^2\,\omega_{1,\alpha,k}\,d\mathrm{vol}_{g_k}-\int_{\Sigma}\widetilde{u}_k^2\,\omega_{1,\alpha,k}\,d\mathrm{vol}_{g_k}\right|\leq \int_{B_{\alpha}\setminus\bar{B}_{\frac{\alpha}{2}}(0)}u_k^2\,\omega_{1,\alpha,k}\,d\mathrm{vol}_{g_k}
            +\int_{B_{2\alpha^{-1}}\setminus\bar{B}_{\frac{\alpha}{2}}(0)}v_k^2\,\widetilde{\omega_{1,\alpha,k}}\,d\mathrm{vol}_{g_{\vec{\Psi}_k}}\\
            &\conv{k\rightarrow \infty}\int_{B_{\alpha}\setminus\bar{B}_{\frac{\alpha}{2}}(0)}u_{\infty}^2\,\omega_{1,\alpha,\infty}\,d\mathrm{vol}_{g_{\infty}}
            +\int_{B_{2\alpha^{-1}}\setminus\bar{B}_{\frac{\alpha}{2}}(0)}v_{\infty}^2\,\frac{1}{\alpha^4}\,d\mathrm{vol}_{g_{\vec{\Psi}_{\infty}}}=0.
        \end{align*}
        Thanks to the normalisation of $u_k\in \mathscr{S}_{\alpha,k}$ (see \eqref{spectrum_sphere}), we deduce that 
        \begin{align}\label{unit_weight}
            \lim_{k\rightarrow \infty}\int_{\Sigma}\widetilde{u}_k^2\,\omega_{1,\alpha,k}\,d\mathrm{vol}_{g_k}=1.
        \end{align}
        Therefore, we deduce by Theorem \ref{neck_positive}, \eqref{unit_weight}, and the boundedness of the sequence of metrics $\ens{g_k}_{k\in\N}$ that 
        \begin{align*}
            Q_{\phi_k}(\widetilde{u}_k)\geq \lambda_1(\beta_1)\int_{\Omega_k(\alpha)}\widetilde{u}_k^2\,\omega_{1,\alpha,k}\,d\mathrm{vol}_{g_k}\conv{k\rightarrow \infty}\lambda_1(\beta_1)>0.
        \end{align*}
        However, this contradicts the fact that $Q_{\phi_k}(u_k)\leq 0$ due to the convergence in \eqref{conv_second_der}.

        Therefore, we deduce that either $u_{\infty}\neq 0$ or $v_{\infty}\neq 0$ and \textbf{Claim 1} is proven.

        \textbf{End of the proof of the main theorem.} As in the proof of \textbf{Claim 1}, let $\ens{\varphi_{k,j}}_{1\leq j\leq N_k}$ be an orthonormal base of 
        \begin{align}
            \bigoplus_{\lambda\leq 0}\mathscr{E}_{\alpha,k}(\lambda),
        \end{align}
        and let $\ens{\lambda_{k,j}}_{k\in\N,j\in \ens{1,\cdots,N_k}}$ be such that
        \begin{align*}
            \mathcal{L}_{g_k}\varphi_{k,j}=\lambda_{k,j}\,\omega_{1,\alpha,k}\,\varphi_{k,j}.
        \end{align*}
        Thanks to \eqref{w42_convergence}, for all $1\leq j\leq N=\limsup_{k\rightarrow\infty}N_k$ we deduce that there exists $\lambda_{\infty,j}\leq 0$ and $\varphi_{\infty,j}\in W^{2,2}(\Sigma)$ and $\psi_{\infty,j}\in W^{2,2}_{\mathrm{loc}}(\C)$ such that
        \begin{align}\label{limiting_eigenvalues}
        \left\{\begin{alignedat}{3}
            &\mathcal{L}_{g_{\infty}}\varphi_{\infty,j}&&=\lambda_{\infty,j}\,\omega_{1,\alpha,\infty}\,\varphi_{\infty,j}\qquad&& \text{in}\;\, \Sigma\\
            &\mathcal{L}_{g_{\Psi_{\infty}}}\psi_{\infty,j}&&=\lambda_{\infty,j}\,\widehat{\omega_{1,\alpha,\infty}}\,\psi_{\infty,j}\qquad&& \text{in}\;\, \C,
            \end{alignedat}\right.
        \end{align}
        where 
        \begin{align*}
            \omega_{1,\alpha,\infty}(x)=\left\{\begin{alignedat}{2}
                &\frac{1}{|x|^{4(1-\beta)}}\frac{1}{\alpha^{4\beta}}\qquad&&\text{for all}\;\, x\in B(0,\alpha)\\
                &\frac{1}{\alpha^4}\qquad&&\text{for all}\;\, x\in \Sigma\setminus\bar{B}(0,\alpha)
            \end{alignedat}\right.
        \end{align*}
        and
        \begin{align*}
            \widehat{\omega_{1,\alpha,\infty}}(x)=\left\{\begin{alignedat}{2}
                &\frac{1}{|x|^{4(1-\beta)}}\frac{1}{\alpha^{4\beta}}\qquad&&\text{for all}\;\, x\in \C\setminus\bar{B}(0,\alpha^{-1})\\
                &\frac{1}{\alpha^4}\qquad&&\text{for all}\;\, x\in B(0,\alpha^{-1}).
            \end{alignedat}\right.
        \end{align*}
        Indeed, since the branch point is of multiplicity $1$, there exists a bounded sequence $\ens{\mu_k}_{k\in\N}$ such that for all $\alpha>0$ small enough
        \begin{align*}
            \vec{\Psi}_k(z)=e^{-\mu_k}\phi_k(\rho_k\,z)\conv{k\rightarrow \infty} \vec{\Psi}_{\infty}(z)\quad \text{for all}\;\, C^l(B(0,\alpha^{-1}))\;\,\text{for all}\;\, l\in\N.
        \end{align*}
        Furthermore, an elementary computation shows that for all $z\in B(0,\alpha^{-1})$,
        \begin{align*}
            e^{2\lambda_{\vec{\Psi}_k}(z)}=\rho_k^2e^{-2\mu_k}e^{2\lambda_{\phi_k}(\rho_kz)}.
        \end{align*}
        Therefore, if $\psi_{k,j}(z)=e^{\mu_k}\varphi_{k,j}(\rho_k\,z)$, we get
        \begin{align*}
            &\int_{B(0,\alpha^{-1}\rho_k)}\left(\Delta_{g_k}\varphi_{k,j}\right)^2d\mathrm{vol}_{g_k}=\int_{B(0,\alpha^{-1}\rho_k)}\left(\Delta \varphi_{k,j}(w)\right)^2e^{-2\lambda_{\phi_k}(w)}|dw|^2\\
            &=\int_{B(0,\alpha^{-1}\rho_k)}\left(\rho_k^{-2}\Delta \psi_{k,j}(\rho_k^{-1}w)\right)^2\rho_k^2e^{2\lambda_{\vec{\Psi}_k}(\rho_k^{-1}w)}|dw|^2=\int_{B(0,\alpha^{-1})}\left(\Delta_{g_{\vec{\Psi}_k}}\psi_{k,j}\right)^2d\mathrm{vol}_{g_{\vec{\Psi}_k}},
        \end{align*}
        which shows by the weak convergence that
        \begin{align}\label{bounded_delta}
            \int_{B(0,\alpha^{-1})}\left(\Delta_{g_{\vec{\Psi}_{\infty}}}\psi_{\infty,j}\right)^2d\mathrm{vol}_{g_{\vec{\Psi}_{\infty}}}\leq \liminf_{k\rightarrow\infty}\int_{B(0,\alpha^{-1}\rho_k)}\left(\Delta_{g_k}\varphi_{k,j}\right)^2d\mathrm{vol}_{g_k}<\infty.
        \end{align}
        Since all terms in the second derivative enjoy the same scaling, we deduce that $\psi_{\infty,j}\in W^{2,2}(\C)$.
        Furthermore, if our limiting immersion is parametrised by $\C$, and $\pi:S^2\rightarrow \C$ is the stereographic projection, by conformal invariance of the Willmore energy, if $\widetilde{\vec{\Psi}}_{\infty}=\vec{\Psi}_{\infty}\circ \pi:S^2\setminus\ens{N}\rightarrow \R^3$, then 
        \begin{align*}
            \mathscr{W}(\widetilde{\Psi}_{\infty})=\mathscr{W}({\Psi}_{\infty}),
        \end{align*}
        where for all $\phi:\Sigma\rightarrow \R^n$, we define
        \begin{align*}
            \mathscr{W}(\phi)=\int_{\Sigma}\left(|\H|^2-K_g\right)d\vg.
        \end{align*}
        If $\widetilde{\vec{\Psi}}_{\infty}:S^2\setminus{N}\rightarrow \R^3$ extend to a continuous function on $S^2$, and $\psi_{\infty,j}$ extends to a $W^{2,2}(S^2)$ function. Otherwise, making an inversion $\vec{\chi}_{\infty}=\dfrac{\widetilde{\vec{\Psi}}_{\infty}}{|\widetilde{\vec{\Psi}}_{\infty}|^2}:S^2\rightarrow \R^3$, we deduce by \cite{indexS3} (see also \cite{index4}) that for all $v\in W^{2,2}(S^2)$
        \begin{align}\label{formula_inversion_d2}
            Q_{\vec{\chi}_{\infty}}(v)=Q_{\widetilde{\Psi}_{\infty}}\left(|\widetilde{\vec{\Psi}}_{\infty}|^2v\right).
        \end{align}
        In particular, we have
        \begin{align*}
            \mathrm{Ind}_{W}(\vec{\chi}_{\infty})=\mathrm{Ind}_{W}(\widetilde{\vec{\Psi}}_{\infty})\qquad \text{and}\qquad \mathrm{Null}_{W}(\vec{\chi}_{\infty})=\mathrm{Null}_{W}(\widetilde{\vec{\Psi}}_{\infty}),
        \end{align*}
        and 
        \begin{align}\label{weight_index_bound}
            \mathrm{dim}\,\mathscr{E}_{\alpha,\infty}^0(\mathscr{L}_{g_{\vec{\Psi}_{\infty}}})\leq \mathrm{Ind}_{W}(\vec{\chi}_{\infty})+\mathrm{Null}_{W}(\vec{\chi}_{\infty}).
        \end{align}
        Indeed, the function $\vec{\Psi}_{\infty}:\C\rightarrow \R^3$ admits the following expansion as $|z|\rightarrow \infty$
        \begin{align*}
            \vec{\Psi}_{\infty}(z)=\Re\left(\vec{A}_0z\right).
        \end{align*}
        Therefore, the metric at infinity is flat, and \eqref{bounded_delta} is equivalent to
        \begin{align*}
            \int_{\C}\left(\Delta \psi_{\infty,j}\right)^2|dz|^2<\infty.
        \end{align*}
        Since the metric on $\widetilde{C}=\C\cup\ens{\infty}\simeq S^2$ is given by 
        \begin{align*}
            g_{\widehat{\C}}=\frac{4|dz|^2}{\left(1+|z|^2\right)^2},
        \end{align*}
        we deduce that
        \begin{align*}
            \int_{\widehat{\mathbb{C}}}\left(1+|z|^2\right)^2\left(\Delta_{g_{\widehat{\mathbb{C}}}}\psi_{\infty,j}\right)^2d\mathrm{vol}_{g_{\widehat{\C}}}<\infty.
        \end{align*}
        Finally, if $u_{\infty,j}=|\vec{\chi}_{\infty}|^2\psi_{\infty,j}$, we have $u_{\infty,j}\in W^{2,2}(S^2)$, and the previous formula \eqref{formula_inversion_d2} and the Lemma \ref{Lemma_B.1} show that \eqref{weight_index_bound} holds.
        Notice that in higher codimension, if $\vec{\Psi}:\Sigma\rightarrow \R^n$ and $\vec{\Psi}=\frac{\phi}{|\phi|^2}$, we have for all $\vec{v}\in W^{2,2}(\Sigma,\R^n)$ (see \cite{index4})
        \begin{align*}
            Q_{\vec{\Psi}}(\vec{v})=Q_{\phi}\left(|\phi|^2\vec{v}-2\s{\phi}{\vec{v}\,}\phi\right)
        \end{align*}
        and the same proof applies.

        Notice that \textbf{Claim 1} implies that $(\varphi_{\infty,j},\psi_{\infty,j})\neq (0,0)$. Now, assume by contradiction that
        \begin{align*}
            N=\limsup_{k\rightarrow \infty}N_k>\mathrm{dim}\,\mathscr{E}_{\alpha,\infty}^0+\mathrm{dim}\,\mathscr{E}_{\alpha,\infty}^0(\mathscr{L}_{g_{\vec{\Psi}_{\infty}}}).
        \end{align*}
        Then, we deduce that the family $(\varphi_{\infty,j},\psi_{\infty,j})_{1\leq j\leq N}$ is linearly independent, which shows that there exists $(c_{\infty,1},\cdots,c_{\infty,N})\in \R^N\setminus\ens{0}$ (almost all $0$) such that
        \begin{align*}
            \left\{\begin{alignedat}{1}
                &\sum_{j=1}^Nc_{\infty,j}\,\varphi_{\infty,j}=0\\
                &\sum_{j=1}^Nc_{\infty,j}\,\psi_{\infty,j}=0.
            \end{alignedat}\right.
        \end{align*}
        Finally, considering the variation $\displaystyle u_k=\sum_{j=1}^Nc_{\infty,j}\,\varphi_{k,j}$, we are led to a contradiction. Therefore, we have
        \begin{align*}
            \limsup_{k\rightarrow \infty}\left(\mathrm{Ind}_W(\phi_k)+\mathrm{Null}_W(\phi_k)\right)\leq \left(\mathrm{Ind}_W(\phi_{\infty})+\mathrm{Null}_W(\phi_{\infty})\right)+\left(\mathrm{Ind}_W(\vec{\Psi}_{\infty})+\mathrm{Null}_W(\vec{\Psi}_{\infty})\right)
        \end{align*}
        and this concludes the proof of the main theorem, up to a last remark on non-compact bubbles. 
    \end{proof}

\section{Improved Energy Quantization for Degenerating Riemann Surfaces}

\subsection{First Estimates on the Residues}

In this section, we make more precise estimates from \cite{lauriv1} and find a criterium under which $L^{2,1}$ energy quantization holds, or more generally, is sufficiently small. This stronger quantization is necessary to get pointwise estimates of the second fundamental form in neck regions. First, recall that
\begin{align*}
    d\H-3\pi_{\n}(d\H)+\star(\H\wedge d\n)&=-4\,\Im\left(g^{-1}\otimes\left(\bar{\partial}^N-\bar{\partial}^{\top}\right)\h_0-|\h_0|_{\mathrm{WP}}^2\right)\\
    &=-4\,\Im\left(\partial\H+|\H|^2\partial\phi+\,g^{-1}\otimes\s{\H}{\h_0}\otimes\bar{\partial}\phi\right),
\end{align*}
the four residues can be defined as (\cite{classification})
\begin{align*}
    &\vec{\gamma}_0(\phi,p)=\frac{1}{4\pi}\Im\int_{\gamma}\partial\H+|\H|^2\partial\phi+\,g^{-1}\otimes\s{\H}{\h_0}\otimes\bar{\partial}\phi\\
    &\vec{\gamma}_1(\phi,p)=\frac{1}{4\pi}\Im\int_{\gamma}\phi \wedge\left(\partial\H+|\H|^2\partial\phi+\,g^{-1}\otimes\s{\H}{\h_0}\otimes\bar{\partial}\phi\right)+g^{-1}\otimes \h_0\wedge \bar{\partial}\phi\\
    &\gamma_2(\phi,p)=\frac{1}{4\pi}\Im\int_{\gamma}\left(\partial\H+|\H|^2\partial\phi+\,g^{-1}\otimes\s{\H}{\h_0}\otimes\bar{\partial}\phi\right)\cdot\phi\\
    &\vec{\gamma}_3(\phi,p)=\frac{1}{4\pi}\Im\int_{\gamma}\mathscr{I}_{\phi}\left(\partial\H+|\H|^2\partial\phi+\,g^{-1}\otimes\s{\H}{\h_0}\otimes\bar{\partial}\phi\right)-g^{-1}\otimes\left(\bar{\partial}|\phi|^2\otimes\h_0-2\s{\phi}{\h_0}\otimes\bar{\partial}\phi\right),
\end{align*}
where $\mathscr{I}_{\phi}(\vec{X})=|\phi|^2\vec{X}-2\s{\phi}{\vec{X}}\phi$, while $\gamma$ is a closed simple curve around $p$. The residues are invariant under isometries, $\vec{\gamma}_1$ is invariant under all conformal transformation, while an inversion changes the sign of $\gamma_2$ and inverts $\vec{\gamma}_0$ and $\vec{\gamma}_4$, the residues associated to the invariance under translations and inversions respectively. 

If $\phi:\Omega=B_R\setminus\bar{B}_r(0)\rightarrow \R^n$ is a weak conformal immersion, we have
\begin{align*}
    \dive\left(\D \H-3\D^N\H+\star(\D^{\perp}\n\wedge \H)\right)=0\qquad \text{in}\;\, B_R\setminus\bar{B}_r(0).
\end{align*}
Thanks to the Poincaré lemma, if
\begin{align}\label{gamma_0}
    \vec{\gamma}_0=\frac{1}{2\pi}\int_{\partial B(0,\rho)}d\H-3\pi_{\n}(d\H)+\star(\H\wedge d\n)=-8\,\vec{\gamma}_0(\phi,0)
\end{align}
for some $r\leq \rho\leq R$, then by the Poincaré lemma, there exists $\vec{L}:\Omega\rightarrow\R^n$ such that
\begin{align*}
    \D^{\perp}\vec{L}=\D\H-3\pi_{\n}(\D\H)+\star\left(\D^{\perp}\n\wedge \H\right)-\vec{\gamma}_0\D\log|z|
\end{align*}
As in \cite{lauriv1}, we have 
\begin{align*}
    &\dive\left(\vec{L}\cdot\D^{\perp}\phi-\s{\vec{\gamma}_0}{\phi}\D\log|z|\right)=0\\
    &\dive\left(\vec{L}\wedge \D^{\perp}\phi+2(-1)^{n-1}(\star(\n\res \H)\res \D^{\perp}\phi-\vec{\gamma}_0\wedge\phi\,\D\log|z|\right)=0.
\end{align*}
Therefore, if 
\begin{align*}
    \widetilde{\vec{\gamma}_1}&=-\frac{1}{2\pi}\int_{\partial B(0,r)}\left(\vec{L}\wedge\partial_{\tau}\phi+2(-1)^{n-1}\left(\star(\n\res\H)\res\partial_{\tau}\phi\right)-\vec{\gamma}_0\wedge\phi\right)d\mathscr{H}^1\\
    \widetilde{\gamma_2}&=-\frac{1}{2\pi}\int_{\partial B(0,r)}\left(\s{\vec{L}}{\partial_{\tau}\phi}+\s{\vec{\gamma_0}}{\phi}\right)\,d\mathscr{H}^1,
\end{align*}
there exists $S:\Omega\rightarrow \R$ and $\vec{R}:\Omega\rightarrow \Lambda^2\R^n$ such that
\begin{align}\label{jacobian_RS_residue}
    \left\{\begin{alignedat}{2}
    \D\vec{R}&=\vec{L}\wedge \D\phi+2(-1)^{n-1}\left(\star(\n\res\H)\res\D^{\perp}\phi\right)+\vec{\gamma}_1\D^{\perp}\log|z|\qquad&&\text{in}\;\,\Omega\\
    \D S&=\vec{L}\cdot\D\phi+\gamma_2\D^{\perp}\log|z|\qquad&&\text{in}\;\,\Omega,
    \end{alignedat}\right.
\end{align}
where
\begin{align}\label{log_constants}
    \left\{\begin{alignedat}{1}
        \vec{\gamma}_1&=\widetilde{\vec{\gamma}_1}+\vec{\gamma}_0\wedge \phi\\
        \gamma_2&=\widetilde{\gamma_2}+\vec{\gamma}_0\cdot\phi
    \end{alignedat}\right.
\end{align}

Using our precised $\epsilon$-regularity result of Theorem \ref{epsilon_reg_H}, we deduce that \cite[Lemma $4.1$]{lauriv1} can be more precisely stated as
follows.

\begin{lemme}\label{lemma_4_1}
    There exists a universal constant $C_1(\Lambda)<\infty$ such that for all $z\in\Omega_{1/2}$,
    \begin{align*}
        e^{\lambda(z)}|\vec{L}(z)|\leq \frac{C_1(\Lambda)}{|z|}\np{e^{\lambda}\H}{2}{\Omega}
    \end{align*}
\end{lemme}

The next lemma allows one to control the residue of $S$.

\begin{lemme}\label{lemma_4_2}
    There exists a universal constant $C_2(\Lambda)<\infty$ such that
    \begin{align*}
        \np{\gamma_2\D\log|z|}{2,1}{\Omega_{1/2}}\leq C_2(\Lambda)\np{\D\n}{2}{\Omega}\np{e^{\lambda}\H}{2}{\Omega}.
    \end{align*}
\end{lemme}
\begin{proof}
    Indeed, using the estimate just after equation ($28$) in \cite{lauriv1} and our improved $\epsilon$-regularity from Theorem \ref{epsilon_reg_H}, the estimate follows immediately since $|\phi(z)-\phi(w)|\leq e^{\lambda}|z-w|$, and
    \begin{align}\label{pointwise_gamma_2}
    |\gamma_2(z)|&\leq C\left(|z|\np{\phi-\phi(z)}{\infty}{\partial B(0,|z|)}\np{\D\n}{\infty}{\partial B(0,|z|)}\np{\H}{\infty}{\partial B(0,|z|)}\right.\nonumber\\
        &\left.+|z|^3\np{\D\n}{\infty}{\partial B(0,|z|)}\np{e^{\lambda}\H}{\infty}{\partial B(0,|z|)}\right)\nonumber\\
        &\leq C'\np{\D\n}{2}{A(2,z)}\np{e^{\lambda}\H}{2}{A(2,z)},
    \end{align}
    which concludes the proof of the lemma (notice that we used the Harnack inequality for the conformal parameter in this step). The conclusion follows from an generatisation of Lemma $4.3$ from \cite{lauriv1} that we state below.
    \end{proof}

    \begin{lemme}\label{lorentz_estimate_average}
    Let $0<4\,r<R<\infty$ and $\Omega=B_{R}\setminus\bar{B}_r(0)$, and for all $0<t<1$, let $\Omega_t=B_{t\,R}\setminus\bar{B}_{t^{-1}\,r}(0)$. For all $2r<|z|<R/2$, and for all $f\in L^2(B_{R}\setminus\bar{B}_r(0)$ define $\delta(f):\Omega_{1/2}\rightarrow \R$ such that for all $z\in \Omega_{1/2}$
    \begin{align*}
        \delta                                                                        (f)(z)=\left(\frac{1}{|z|^2}\int_{B_{2|z|}\setminus\bar{B}_{|z|/2}(0)}|f|^2dx\right)^{\frac{1}{2}}.
    \end{align*}
    Then, there exists a universal constant independent $C<\infty$ of $0<4\,r<R<\infty$ such that for all $f,g\in L^2(\Omega)$
    \begin{align}
         \np{|z|\delta(f)\delta(g)}{2,1}{\Omega_{1/2}}\leq C\np{f}{2}{\Omega}\np{g}{2}{\Omega}.            
    \end{align}
    \end{lemme}
    \begin{proof}
        Let $\varphi(r)=r\delta^2(f)(r)$. We have
        \begin{align*}
            \p{r}\varphi=-\frac{1}{r^2}\int_{A(2,r)}|f|^2dx+\frac{1}{r}\int_{\partial A(2,r)}|f|^2d\mathscr{H}^1.
        \end{align*}
        We have thanks to Fubini's theorem
        \begin{align}\label{lemma4.3_1}
            &\int_{\Omega_{1/2}}\frac{1}{|x|^2}\left(\int_{A(2,x)}|f(y)|^2dy\right)dx=\int_{2r}^{R/2}\frac{1}{\rho}\left(\int_{B_{2\rho}\setminus\bar{B}_{\rho/2}(0)}|f(y)|^2dy\right)d\rho\nonumber\\
            &=\int_{B_R\setminus\bar{B}_r(0)}|f(y)|^2\left(\int_{2r}^{R/2}\frac{\mathbf{1}_{\ens{\rho/2\leq |y|\leq 2\rho}}}{\rho}d\rho\right)dy=\int_{B_R\setminus\bar{B}_r(0)}|f(y)|^2\left(\int_{|y|/2}^{2|y|}\frac{d\rho}{\rho}\right)dy\nonumber\\
            &=\log(4)\int_{B_{R}\setminus\bar{B}_r(0)}|f(y)|^2dy.
        \end{align}
        Then, using the co-area formula (\cite[$\mathbf{3.2.11}$]{federer} twice
        \begin{align}\label{lemma4.3_2}
            &\int_{\Omega_{1/2}}\frac{1}{|x|}\left(\int_{\partial A(2,x)}|f(y)|^2d\mathscr{H}^1(y)\right)dx=\int_{2r}^{R/2}\left(\int_{\partial B(0,2\rho)}|f(y)|^2d\mathscr{H}^1(y)\right)d\rho\nonumber\\
            &+\int_{2r}^{R/2}\left(\int_{\partial B(0,\rho/2)}|f(y)|^2d\mathscr{H}^1(y)\right)d\rho\nonumber\\
            &=\int_{B_R\setminus\bar{B}_{4r}(0)}|f(x)|^2dx+\int_{B_{R/4}\setminus\bar{B}_r(0)}|f(x)|^2dx
            \leq 2\int_{B_{R}\setminus\bar{B}_r(0)}|f(x)|^2dx.
        \end{align}
        Therefore, since $\varphi$ is a radial function, we deduce thanks to \eqref{lemma4.3_1} and \eqref{lemma4.3_2}
        \begin{align}\label{lemma4.3_3}
            \int_{\Omega_{1/2}}|\D\varphi|dx\leq \left(2+\log(4)\right)\int_{\Omega}|f(x)|^2dx.
        \end{align}
        Furthermore, using Fubini's theorem and the co-area formula, we get
        \begin{align}\label{lemma4.3_4}
            \int_{\Omega}|\varphi|dx&=\int_{\Omega_{1/2}}\frac{1}{|x|}\left(\int_{A(2,x)}|f(y)|^2dy\right)dx=\int_{2r}^{R/2}\left(\int_{B_{2\rho}\setminus\bar{B}_{\rho/2}(0)}|f(y)|^2dy\right)d\rho\nonumber\\
            &=\int_{B_R\setminus\bar{B}_r(0)}|f(y)|^2\left(\int_{|y|/2}^{2|y|}d\rho\right)dy=\frac{3}{2}\int_{\Omega}|y||f(y)|^2dy.
        \end{align}
        Now, thanks to the improved Sobolev embedding $W^{1,1}(\R^2)\hookrightarrow L^{2,1}(\R^2)$, we deduce by a trivial scaling argument that there exists a universal constant $C<\infty$ such that for all $u\in W^{1,1}(B(0,R))$
        \begin{align}\label{sobolev_ball}
            \np{u}{2,1}{B(0,R)}\leq C_1\left(\int_{B(0,R)}|\D u|dx+\frac{1}{R}\int_{B(0,R)}|u|dx\right).
        \end{align}
        for some universal constant $C_2<\infty$.
        Furthermore, since the conformal class of $\Omega$ is bounded away from $-\infty$, using \cite[Lemma $7.2$]{pointwise}, we deduce that there exists an extension operator $T:W^{1,1}(\Omega)\rightarrow W^{1,1}(B(0,R))$ and a universal constant independent of $0<4\,r<R<\infty$ such that for all $u\in W^{1,1}(\Omega)$, 
        \begin{align}\label{extension}
            \wp{T u}{1,1}{B(0,R)}\leq C_3\wp{u}{1,1}{\Omega}.
        \end{align}
        Combining \eqref{extension} and \eqref{sobolev_ball} we deduce that there exists a universal constant $C_4<\infty$ such that for all $u\in W^{1,1}(\Omega)$,
        \begin{align}\label{last_sobolev}
            &\np{u}{2,1}{\Omega}\leq C_4\left(\int_{\Omega}|\D u|dx+\frac{1}{R}\int_{\Omega}|u|dx\right).
        \end{align}
        Gathering the estimates \eqref{lemma4.3_3}, \eqref{lemma4.3_4}, and \eqref{last_sobolev}, we deduce that
        \begin{align}\label{lemma4.3_5}
            \np{\varphi}{2,1}{\Omega_{1/2}}&\leq C_4\left(\int_{\Omega}|\D\varphi|dx+\frac{1}{R}\int_ {\Omega}|\varphi|dx\right)\nonumber\\
            &\leq C_4\left((2+\log(4))\int_{\Omega}|f(x)|^2dx+2\int_{\Omega}\frac{|x|}{R}|f(x)|^2dx\right)\nonumber\\
            &\leq C_4\left(4+\log(4)\right)\int_{\Omega}|f(x)|^2dx.
        \end{align}
        In other words
        \begin{align}\label{final_lemma4.3}
            \np{|z|\delta^2(f)}{2,1}{\Omega_{1/2}}\leq C\np{f}{2}{\Omega}^2.
        \end{align}
        Then, using the Hölder's inequality for Lorentz spaces, we deduce by \eqref{final_lemma4.3} that
        \begin{align}\label{ineq_finale}
            \np{|z|\delta(f)\delta(g)}{2,1}{\Omega_{1/2}}&\leq C_{H}\np{|z|^{\frac{1}{2}}\delta(f)}{4,2}{\Omega_{1/2}}\np{|z|^{\frac{1}{2}}\delta(g)}{4,2}{\Omega_{1/2}}\nonumber\\
            &\leq C_H'\np{|z|\delta^2(f)}{2,1}{\Omega_{1/2}}^{\frac{1}{2}}\np{|z|\delta^2(g)}{2,1}{\Omega_{1/2}}^{\frac{1}{2}}\nonumber\\
            &\leq C_H'C_4(4+\log(4))\np{f}{2}{\Omega_{1/2}}\np{g}{2}{\Omega_{1/2}},
        \end{align}
        where the constants $C_H$ and $C_H'$ are independent of the domain, since the inequalities are verified for a general measured space $(X,\mu)$ We also used the fact that the $L^{p,q}$ norm is equivalent to the following semi-norm:
        \begin{align*}
            |f|_{\mathrm{L}^{p,q}(X)}=\left(\int_{0}^{\infty}t^{q}\lambda_f(t)^{\frac{q}{p}}\frac{dt}{t}\right)^{\frac{1}{q}},
        \end{align*}
        where $\lambda_f(t)=\mu\left(X\cap\ens{x:|f(x)|>t}\right)$ for all $t>0$. Noticing that $\lambda_{\sqrt{|f|}}(t)=\lambda_f(t^2)$, we deduce that
        \begin{align}
            \left|\sqrt{|f|}\right|_{\mathrm{L}^{p,q}(X)}=\left(\int_{0}^{\infty}t^q\lambda_f(t^2)^{\frac{q}{p}}\frac{dt}{t}\right)^{\frac{1}{q}}=\left(\int_{0}^{\infty}s^{\frac{q}{2}}\lambda_f(s)^{\frac{q}{p}}\frac{ds}{2s}\right)^{\frac{1}{q}}=\frac{1}{2^{\frac{1}{q}}}\left|f\right|_{\mathrm{L}^{\frac{p}{2},\frac{q}{2}}(X)}^{\frac{1}{2}}.
        \end{align}
        This remark closes the proof of the lemma.
    \end{proof}

    Now, we make more precise \cite[Lemma $4.4$]{lauriv1}. 
    \begin{lemme}\label{lemma_4_4}
        There exists a universal constant $C_3(\Lambda)<\infty$ such that
        \begin{align}
            \np{\vec{\gamma}_1\D\log|z|}{2}{\Omega_{1/2}}\leq C_3(\Lambda)\np{e^{\lambda}\H}{2}{\Omega}.
        \end{align}
    \end{lemme}
    \begin{proof}
        We have
        \begin{align*}
            \vec{\gamma}_1(z)&=\frac{1}{2\pi}\int_{\partial B(0,|z|)}\left(\left(\pi_{T}\left(\p{r}\H\right)-2\pi_{\n}\left(\p{r}\H\right)-\star\left(\p{\tau}\n\wedge\H\right)\right)\wedge (\phi(z)-\phi(y))\right.\\
            &\left.+2(-1)^m\left(\star\left(\n\res \H\right)\right)\res \p{\tau}\phi\right)\,d\mathscr{H}^1(y),
        \end{align*}
        that we estimate as in \cite{lauriv1} as
        \begin{align}\label{pointwise_gamma_1}
            |\vec{\gamma}_1(z)|&\leq C\left(|z|+|z|^2\np{\D\n}{\infty}{\partial B(0,|z|)}\right)\np{e^{\lambda}\H}{\infty}{\partial B(0,|z|)}\nonumber\\
            &\leq C'\left(1+\np{\D\n}{2}{\Omega}\right)\np{e^{\lambda}\H}{2}{A(2,z)},
        \end{align}
        which shows that
        \begin{align}
            \np{\vec{\gamma}_1\D\log|z|}{2}{\Omega_{1/2}}\leq C'\sqrt{\log(4)}\left(1+\np{\D\n}{2}{\Omega}\right)\np{e^{\lambda}\H}{2}{\Omega},
        \end{align}
        and concludes the proof of the lemma.
    \end{proof}

    \subsection{Finer Estimates of Residues}
    
    Now, recall that the following system holds true:
    \begin{align}\label{jacobian_RS_residues}
    \left\{\begin{alignedat}{2}
    \D\vec{R}&=(-1)^n\star\left(\n\antires\D^{\perp}\vec{R}\right)+\star\,\n\cdot\D^{\perp}S+\vec{\gamma}_1\D^{\perp}\log|z|\\
    &+(-1)^n\star\left(\n\antires\vec{\gamma}_1\right)\D\log|z|+\gamma_2\D\log|z|\star\n\qquad&&\text{in}\;\,\Omega\\
    \D S&=-\star \n\cdot\D^{\perp}\vec{R}+\gamma_2\D^{\perp}\log|z|-\star\,\n\cdot \vec{\gamma}_1\D\log|z|\qquad&&\text{in}\;\,\Omega
    \end{alignedat}\right.
    \end{align}
    Thanks to the previous analysis, since $\gamma_2\D\log|z|\in L^{2,1}(\Omega_{1/2})$ while $\vec{\gamma_1}\D\log|z|\in L^2(\Omega_{1/2})$, we need to prove finer estimates for algebraic components featuring $\vec{\gamma}_1$.
    
    Furthermore, the following algebraic identity is verified:
    \begin{align}
        e^{2\lambda}\H&=\frac{1}{4}\D^{\perp}S\cdot\D\phi-\frac{1}{4}\D\vec{R}\res \D^{\perp}\phi+\frac{1}{4}\frac{\gamma_2}{|z|}\p{r}\phi+\frac{1}{4}\frac{\vec{\gamma}_1}{|z|}\res \p{r}\phi.
    \end{align}
    First, let us make precise \cite[Lemma $5.1$]{lauriv1}.

    \begin{lemme}\label{lemma_5_1}
        There exists a universal constant $C_4(\Lambda)<\infty$ such that
        \begin{align*}
            \np{\star\,\n\cdot\vec{\gamma}_1\D\log|z|}{2,1}{\Omega_{1/2}}\leq C_4(\Lambda)\np{\D\n}{2}{\Omega}\np{e^{\lambda}\H}{2}{\Omega}.
        \end{align*}
    \end{lemme}
    \begin{proof}
        Recall that 
        \begin{align}\label{residue_def}
            \vec{\gamma}_1(z)&=\frac{1}{2\pi}\int_{\partial B(0,|z|)}\left(\pi_{T}\left(\p{r}\H\right)-2\,\pi_{\n}(\p{r}\H)-\star\left(\p{\tau}\n\wedge\H\right)\right)\wedge (\phi(z)-\phi(y))d\mathscr{H}^1(y)\nonumber\\
            &+(-1)^n\frac{1}{\pi}\int_{\partial B(0,|z|)}\left(\star(\n\res\H)\right)\res\p{\tau}\phi\,d\mathscr{H}^1.
        \end{align}
        Notice that in a conformal chart, we have
        \begin{align*}
            \pi_{T}(\p{z}\H)&=2\,e^{-2\lambda}\s{\p{z}\H}{\p{\z}\phi}\p{z}\phi+2e^{-2\lambda}\s{\p{z}\H}{\p{z}\phi}\p{\z}\phi\\
            &=-2\,e^{-2\lambda}\s{\H}{\p{z\z}^2\phi}\p{z}\phi-2\,e^{-2\lambda}\s{\H}{\p{z}^2\phi}\p{\z}\phi\\
            &=-|\H|^2\p{z}\phi-\s{\H}{\H_0}\p{\z}\phi.
        \end{align*}
        Therefore, we have
        \begin{align*}
            |\pi_{T}(\D\H)|\leq \left(e^{\lambda}|\H|+e^{\lambda}|\H_0|\right)|\H|\leq 2|\D\n||\H|.
        \end{align*}
        Therefore, we deduce that
        \begin{align}
            \left|\frac{1}{2\pi}\int_{\partial B(0,|z|)}\pi_{T}\left(\p{r}\H\right)\wedge \left(\phi(z)-\phi(y)\right)d\mathscr{H}^1(y)\right|&\leq C|z|^2\np{\D\n}{\infty}{\partial B(0,|z|)}\np{e^{\lambda}\H}{\infty}{\partial B(0,|z|)}\nonumber\\
            &\leq C\np{\D\n}{2}{A(2,z)}\np{e^{\lambda}\H}{2}{A(2,z)}.
        \end{align}
        \emph{A fortiori}, we have
        \begin{align}\label{res_est1}
            \left|\star\,\n(z)\cdot \frac{1}{2\pi}\int_{\partial B(0,|z|)}\pi_{T}\left(\p{r}\H\right)\wedge \left(\phi(z)-\phi(y)\right)d\mathscr{H}^1(y)\right|\leq C\np{\D\n}{2}{A(2,z)}\np{e^{\lambda}\H}{2}{A(2,z)}.
        \end{align}
        On the other hand, we directly estimate
        \begin{align}\label{res_est2}
            \left|\star\n(z)\cdot \frac{1}{2\pi}\star\left(\p{\tau}\n\wedge\H\right)\wedge (\phi(z)-\phi(y))d\mathscr{H}^1(y)\right|&\leq C|z|^2\np{\D\n}{\infty}{\partial B(0,|z|)}\np{e^{\lambda}\H}{\infty}{\partial B(0,|z|)}\nonumber\\
            &\leq C\np{\D\n}{2}{A(2,z)}\np{e^{\lambda}\H}{2}{A(2,z)}.
        \end{align}
        Since $\pi_{\n}(\p{r}\H)$ is a normal vector, we deduce that 
        \begin{align}\label{algebraic_normal}
            \bs{\star\, \n}{\pi_{\n}(\p{r}\H)\wedge \vec{v}}=\bs{\e_1\wedge\e_2}{\pi_{\n}(\p{r}\H)\wedge \vec{v}}=0\quad \text{for all}\;\, \vec{v}\in \R^n.
        \end{align}
        Therefore, we have
        \begin{align}\label{res_est3}
            &\left|\star\,\n(z)\cdot\frac{1}{2\pi}\int_{\partial B(0,|z|)}\pi_{\n}\left(\p{r}\H\right)\wedge(\phi(z)-\phi(y))d\mathscr{H}^1(y)\right|\nonumber\\
            &=\frac{1}{2\pi}\left|\int_{\partial B(0,|z|)}\left(\star\,\n(z)-\star\,\n(y)\right)\cdot \left(\pi_{\n}(\p{r}\H)\wedge \left(\phi(z)-\phi(y)\right)\right)d\mathscr{H}^1(y)\right|\nonumber\\
            &\leq C|z|^3\np{\D\n}{\infty}{\partial B(0,|z|)}\np{e^{\lambda}\D\H}{\infty}{\partial B(0,|z|)}\leq C\np{\D\n}{2}{A(2,z)}\np{e^{\lambda}\H}{2}{A(2,z)}
        \end{align}
        where we used once more the improved $\epsilon$-regularity from Theorem \ref{epsilon_reg_H} together with the classical $\epsilon$-regularity from \cite{riviere1}. Finally, noticing as in \eqref{algebraic_normal} holds for $\H\wedge \p{x_i}\phi$, we deduce that
        \begin{align}\label{res_est4}
            \left|\star\,\n(z)\cdot\frac{1}{2\pi}\int_{\partial B(0,|z|)}\left(\star\left(\n\res \H\right)\res\p{\tau}\phi\right)d\mathscr{H}^1\right|&=\frac{1}{2\pi}\left|\int_{\partial B(0,|z|)}\left(\star\,\n(z)-\star\,\n(y)\right)\cdot\left(\frac{1}{r}\p{\theta}\phi\wedge\H\right)d\mathscr{H}^1(y)\right|\nonumber\\
            &\leq C|z|^2\np{\D\n}{\infty}{\partial B(0,|z|)}\np{e^{\lambda}\H}{\infty}{\partial B(0,|z|)}\nonumber\\
            &\leq C\np{\D\n}{2}{A(2,z)}\np{e^{\lambda}\H}{2}{A(2,z)}.
        \end{align}
        Gathering estimates \eqref{res_est1}, \eqref{res_est2}, \eqref{res_est3}, and \eqref{res_est4}, we deduce by the algebraic identity \eqref{residue_def} that
        \begin{align}\label{pointwise_n_times_gamma_1}
            \left|\star\,\n(z)\cdot \vec{\gamma}_1(z)\right|\leq C\np{\D\n}{2}{A(2,z)}\np{e^{\lambda}\H}{2}{A(2,z)}.
        \end{align}
        Therefore, the proof is completed by virtue of Lemma \eqref{lorentz_estimate_average}.
    \end{proof}
    The next lemma is straightforward (see \cite[Lemma $5.2$]{lauriv1}.
    \begin{lemme}\label{lemma_5_2}
        There exists a universal constant $C_5(\Lambda)<\infty$ such that
        \begin{align*}
            \np{\D\vec{\gamma}_1}{2}{\Omega_{1/2}}\leq C_5(\Lambda)\np{e^{\lambda}\H}{2}{\Omega}.
        \end{align*}
    \end{lemme}
    \begin{proof}
        Recall by \eqref{gamma_0} that
        \begin{align*}
            \vec{\gamma_0}=\frac{1}{2\pi}\int_{\partial B(0,\rho)}d\H-3\pi_{\n}(d\H)+\star(\H\wedge d\n).
        \end{align*}
        Since $\vec{\gamma}_1(z)=\widetilde{\gamma_1}+\vec{\gamma}_0\wedge \phi(z)$, we have $\D\vec{\gamma}_1=\vec{\gamma}\wedge \D\phi$, which implies that
        \begin{align}\label{pointwise_grad_gamma_1}
            |\D\vec{\gamma}_1(z)|&\leq C|z|\left(\np{e^{\lambda}\H}{\infty}{\partial B(0,|z|)}+\np{e^{\lambda}|\H||\D\n|}{\infty}{\partial B(0,|z|)}\right)\nonumber\\
            &\leq \frac{C}{|z|}\left(1+\np{\D\n}{2}{A(2,z)}\right)\np{e^{\lambda}\H}{2}{A(2,z)}
            \leq C\left(1+\np{\D\n}{2}{\Omega}\right)\delta(z),
        \end{align}
        which shows that
        \begin{align*}
            \np{\D\vec{\gamma}_1}{2}{\Omega_{1/2}}\leq C\sqrt{\log(4)}\left(1+\np{\D\n}{2}{\Omega}\right)\np{e^{\lambda}\H}{2}{\Omega},
        \end{align*}
        and concludes the proof of the theorem.
    \end{proof}
    We skip the proof of the next estimate that is immediately derived from the second inequality in the proof of \cite[Lemma $5.3$]{lauriv1}
    \begin{lemme}\label{lemma_5_3}
        There exists a universal constant $C_6(\Lambda)<\infty$ such that
        \begin{align}\label{pointwise_Delta_gamma_1}
            |\Delta \vec{\gamma}_1(z)|\leq \frac{C_6(\Lambda)}{|z|^2}\np{\D\n}{2}{A(2,z)}\np{e^{\lambda}\H}{2}{A(2,z)}=C_6(\Lambda)\delta_0(z)\,\delta(z).
        \end{align}
    \end{lemme}
    Then, the follow result holds.
    \begin{lemme}\label{lemma_5_4}
        There exists a universal constant $C_7(\Lambda)<\infty$ such that
        \begin{align*}
            \np{\D \vec{\gamma}_1}{2}{\Omega_{1/2}}\leq C_7(\Lambda)\left(R+\np{\D\n}{2}{\Omega}\right)\np{e^{\lambda}\H}{2}{\Omega}\sup_{2r<\rho<R/2}\np{e^{\lambda}\H}{2}{A(2,\rho)}.
        \end{align*}
    \end{lemme}
    \begin{proof}
        Integrating by parts and using \eqref{pointwise_gamma_1}, \eqref{pointwise_grad_gamma_1}, and \eqref{pointwise_Delta_gamma_1}, we obtain thanks to Cauchy-Schwarz inequality
        \begin{align*}
            &\int_{\Omega_{1/2}}|\D \vec{\gamma}_1|^2dx=-\int_{\Omega_{1/2}}\vec{\gamma}_1\cdot \Delta\vec{\gamma}_1\,dx+\int_{\partial \Omega_{1/2}}\vec{\gamma}_1\cdot \partial_{\nu}\vec{\gamma}_1\,d\mathscr{H}^1\\
            &\leq C\int_{\Omega_{1/2}}\frac{1}{|z|^2}\np{\D\n}{2}{A(2,z)}\np{e^{\lambda}\H}{2}{A(2,z)}^2|dz|^2+C\int_{\Omega_{1/2}}\frac{1}{|z|}\np{e^{\lambda}\H}{2}{A(2,z)}^2|dz|^2\\
            &\leq C\left(\left(\int_{\Omega_{1/2}}\frac{1}{|z|^2}\np{\D\n}{2}{A(2,z)}^2|dz|^2\right)^{\frac{1}{2}}\left(\int_{\Omega_{1/2}}\frac{1}{|z|^2}\np{e^{2\lambda}\H}{2}{A(2,z)}^2|dz|^2\right)^{\frac{1}{2}}\right.\\
            &\left.+\left(\int_{\Omega_{1/2}}\frac{1}{|z|^2}\np{e^{\lambda}\H}{2}{A(2,z)}^2|dz|^2\right)^{\frac{1}{2}}\sqrt{\pi(R^2-r^2)}\right)\sup_{2r<\rho<R/2}\np{e^{\lambda}\H}{2}{A(2,\rho)}\\
            &=C\left(\log(4)\np{\D\n}{2}{\Omega}\np{e^{\lambda}\H}{2}{\Omega}+\sqrt{\log(4)}\sqrt{\pi(R^2-r^2)}\np{e^{\lambda}\H}{2}{\Omega}\right)\sup_{2r<\rho<R/2}\np{e^{\lambda}\H}{2}{A(2,\rho)}.
        \end{align*}
        Notice that we also estimate directly
        \begin{align*}
            &\int_{\Omega_{1/2}}\frac{1}{|z|^2}\np{\D\n}{2}{A(2,z)}\np{e^{\lambda}\H}{2}{A(2,z)}^2|dz|^2+\int_{\Omega_{1/2}}\frac{1}{|z|}\np{e^{\lambda}\H}{2}{A(2,z)}^2|dz|^2\\
            &\leq \log(4)\sup_{2r<\rho<R/2}\np{\D\n}{2}{A(2,z)}\int_{\Omega}e^{2\lambda}|\H|^2dx+\frac{3}{2}\int_{\Omega_{1/2}}|x|e^{2\lambda}|\H|^2dx.
        \end{align*}
    \end{proof}

    \subsection{Estimates on the Jacobian System}

    \subsubsection{Marcinkiewicz Space Estimate}

    Recalling system \eqref{jacobian_RS_residue}
    \begin{align}\label{jacobian_RS_residue2}
    \left\{\begin{alignedat}{2}
    \D\vec{R}&=\vec{L}\wedge \D\phi+2(-1)^{n-1}\left(\star(\n\res\H)\res\D^{\perp}\phi\right)+\vec{\gamma}_1\D^{\perp}\log|z|\qquad&&\text{in}\;\,\Omega\\
    \D S&=\vec{L}\cdot\D\phi+\gamma_2\D^{\perp}\log|z|\qquad&&\text{in}\;\,\Omega,
    \end{alignedat}\right.
\end{align}
we deduce thanks to Lemma \ref{lemma_4_1} that for all $z\in \Omega_{1/2}$
    \begin{align}\label{pointwise_RS}
    \left\{\begin{alignedat}{1}
        |\D \vec{R}(z)|&\leq 2e^{\lambda(z)}|\vec{L}(z)|+4e^{\lambda(z)}|\H(z)|+\frac{|\gamma_1(z)|}{|z|}
        \leq \frac{C(\Lambda)}{|z|}\np{e^{\lambda}\H}{2}{\Omega}+\frac{|\gamma_1(z)|}{|z|}\\
        |\D S(z)|&\leq 2e^{\lambda}|\vec{L}(z)|+\frac{|\gamma_2(z)|}{|z|}
        \leq \frac{C(\Lambda)}{|z|}\np{e^{\lambda}\H}{2}{\Omega}+\frac{|\gamma_2(z)|}{|z|}
        \end{alignedat}\right.
    \end{align}
    where we  used the $\epsilon$-regularity from Theorem \ref{epsilon_reg_H}. Thanks to Lemma \ref{lemma_4_1} Lemma \ref{lemma_4_2}, we deduce that 
    \begin{align}
        \np{\D S}{2,\infty}{\Omega_{1/2}}&\leq 2\np{e^{\lambda}\vec{L}}{2,\infty}{\Omega_{1/2}}+2\np{\gamma_2\D\log|z|}{2}{\Omega_{1/2}}\nonumber\\
        &\leq 2\sqrt{\pi}C_1(\Lambda)\np{e^{\lambda}\H}{2}{\Omega}+2C_2(\Lambda)\np{\D\n}{2}{\Omega}\np{e^{\lambda}\H}{2}{\Omega}\nonumber\\
        &\leq C(\Lambda)\np{e^{\lambda}\H}{2}{\Omega},
    \end{align}
    while Lemma \ref{lemma_4_1} and Lemma \ref{lemma_4_4} shows that 
    \begin{align}
        \np{\D\vec{R}}{2,\infty}{\Omega_{1/2}}\leq C(\Lambda)\np{e^{\lambda}\H}{2}{\Omega}.
    \end{align}

    \subsubsection{Lebesgue Space Estimate}
    Using the same argument as in \cite{lauriv1}, and the pointwise estimates \eqref{pointwise_gamma_1} and \eqref{pointwise_gamma_2}, we deduce that
    \begin{align}
        \left|\frac{d \vec{R}}{d\rho}\right|^2+\left|\frac{dS_{\rho}}{d\rho}\right|^2\leq \frac{C}{\rho^2}\np{e^{\lambda}\H}{2}{A(2,\rho)}
    \end{align}
    where for all $\varphi\in C^{\infty}(\R^2)$ and $0<\rho<\infty$,
    \begin{align*}
        \varphi_{\rho}=\dashint{\partial B(0,\rho)}\varphi\,d\mathscr{H}^1.
    \end{align*}
    Therefore, we have
    \begin{align}
        \int_{\Omega_{1/2}}\left(|\D \vec{R}_{|z|}|^2+|\D S_{|z|}|^2\right)|dz|^2=\int_{2r}^{R/2}\left(\left|\frac{d \vec{R}}{d\rho}\right|^2+\left|\frac{dS_{\rho}}{d\rho}\right|^2\right)\rho\,d\rho\leq C\int_{\Omega}e^{2\lambda}|\H|^2dx.
    \end{align}
    Now, let $\vec{\psi}_{\vec{R}}:\Omega\rightarrow \Lambda^2\R^n$ be the solution of 
    \begin{align}
        \left\{\begin{alignedat}{2}
            \Delta \vec{\psi}_{\vec{R}}&=\dive\left(\vec{\gamma}_1\D^{\perp}\log|z|+(-1)^n\star\left(\n\antires\vec{\gamma}_1)\right)\D\log|z|+\gamma_2\D\log|z|\star\n\right)\qquad&& \text{in}\;\, \Omega_{1/2}\\
            \vec{\psi}_{\vec{R}}&=0\qquad&&\text{on}\;\,\partial\Omega_{1/2},
        \end{alignedat}\right.
    \end{align}
    and $\psi_S:\Omega\rightarrow \R$ be the solution of 
    \begin{align}
        \left\{\begin{alignedat}{2}
            \Delta {\psi}_{\vec{R}}&=\dive\left(\gamma_2\D^{\perp}\log|z|-\star\,\n\cdot\vec{\gamma}_1\D\log|z|\right)\qquad&& \text{in}\;\, \Omega_{1/2}\\
            {\psi}_{\vec{R}}&=0\qquad&&\text{on}\;\,\partial\Omega_{1/2}.
        \end{alignedat}\right.
    \end{align}
    Integrating by parts, we deduce that thanks to Lemma \ref{lemma_4_2} and \eqref{lemma_4_4}
    \begin{align*}
        \int_{\Omega_{1/2}}|\D\psi_S|^2dx&=-\int_{\Omega_{1/2}}\psi_S\Delta \psi_S\,dx=\int_{\Omega_{1/2}}\D\psi_S\cdot \left(\gamma_2\D^{\perp}\log|z|+\star\,\n\cdot\vec{\gamma_1}\D\log|z|\right)dx\\
        &\leq \left(\int_{\Omega_{1/2}}|\D\psi_S|^2dx\right)^{\frac{1}{2}}\left(\np{\gamma_2\D\log|z|}{2}{\Omega_{1/2}}\np{\vec{\gamma}_1\D\log|z|}{2}{\Omega_{1/2}}\right)\\
        &\leq \left(\int_{\Omega_{1/2}}|\D\psi_S|^2dx\right)\left(C_2(\Lambda)\np{\D\n}{2}{\Omega}\np{e^{\lambda}\H}{2}{\Omega}+C_3(\Lambda)\np{e^{\lambda}\H}{2}{\Omega}\right)\\
        &\leq C(\Lambda)\left(\int_{\Omega_{1/2}}|\D\psi_S|^2dx\right)^{\frac{1}{2}}\np{e^{\lambda}\H}{2}{\Omega}.
    \end{align*}
    Therefore, we have
    \begin{align*}
        \np{\D\psi_S}{2}{\Omega_{1/2}}\leq C(\Lambda)\np{e^{\lambda}\H}{2}{\Omega}.
    \end{align*}
    Likewise, 
    \begin{align*}
        \np{\D\psi_{\vec{R}}}{2}{\Omega_{1/2}}\leq C(\Lambda)\np{e^{\lambda}\H}{2}{\Omega}.
    \end{align*}
    Therefore, we finally deduce thanks to \cite[Lemma $5.5$]{pointwise}
    \begin{align}\label{est_S_2}
        \np{\D S}{2}{\Omega_{1/4}}&\leq C\left(\np{\vec{R}}{2,\infty}{\Omega_{1/2}}\np{\D\n}{2}{\Omega_{1/2}}+\np{\D S_{|z|}}{2}{\Omega_{1/2}}+\np{\D S}{2,\infty}{\Omega_{1/2}}\right)+\np{\D\psi_S}{2}{\Omega_{1/2}}\nonumber\\
        &\leq C(\Lambda)\np{e^{\lambda}\H}{2}{\Omega}.
    \end{align}
    Likewise, we deduce that
    \begin{align}\label{est_R_2}
        \np{\D \vec{R}}{2}{\Omega_{1/4}}\leq C(\Lambda)\np{e^{\lambda}\H}{2}{\Omega}.
    \end{align}

    \subsubsection{Lorentz Space Estimate}

    Using the same argument as in \cite[($60$)]{riviere1} and \eqref{pointwise_n_times_gamma_1}, we deduce that
    \begin{align*}
        \left|\frac{dS_{\rho}}{d\rho}\right|\leq \left(\int_{\partial B(0,\rho)}|\D\n|^2d\mathscr{H}^1\right)^{\frac{1}{2}}\left(\int_{\partial B(0,\rho)}|\D\vec{R}|^2d\mathscr{H}^1\right)^{\frac{1}{2}}+\frac{C}{|z|}\np{\D\n}{2}{A(2,\rho)}\np{e^{\lambda}\H}{2}{A(2,\rho)}.
    \end{align*}
    Therefore, using the Cauchy-Schwarz inequality and the co-area formula, we deduce thanks to \eqref{est_R_2} that
    \begin{align*}
        \int_{8r}^{R/8}\left|\frac{d S_{\rho}}{d\rho}\right|d\rho\leq C\np{\D\n}{2}{\Omega}\np{e^{\lambda}\H}{2}{\Omega},
    \end{align*}
    and finally
    \begin{align*}
        \np{S}{\infty}{\Omega_{1/8}}\leq C\np{\D\n}{2}{\Omega}\np{e^{\lambda}\H}{2}{\Omega}.
    \end{align*}
    Reasoning similarly for $\vec{R}$, we deduce that 
    \begin{align*}
        \np{\vec{R}}{\infty}{\Omega_{1/8}}\leq C\np{\D\n}{2}{\Omega}\np{e^{\lambda}\H}{2}{\Omega}.
    \end{align*}
    Finally, using \cite[Lemma $5.6$]{lauriv1} (or rather, the more precise \cite[Lemma $4.5$]{pointwise}), and the $L^{2,1}$ estimates from Lemma \ref{lemma_4_2} and Lemma \ref{lemma_5_1} that 
    \begin{align*}
        \np{\psi_{S}}{\infty}{\Omega_{1/8}}\leq C\np{\D\psi_S}{2,1}{\Omega_{1/8}}&\leq C\left(\np{\gamma_2\D^{\perp}\log|z|}{2,1}{\Omega_{1/8}}+\np{\star\,\n\cdot\vec{\gamma}_1\D\log|z|}{2,1}{\Omega_{1/2}}\right)\\
        &\leq C(\Lambda)\np{\D\n}{2}{\Omega}\np{e^{\lambda}\H}{2}{\Omega}.
    \end{align*}
    Finally, we deduce that
    \begin{align*}
        \np{\D S}{2,1}{\Omega_{1/8}}\leq C(\Lambda)\np{\D\n}{2}{\Omega}\np{e^{\lambda}\H}{2}{\Omega}.
    \end{align*}
    As in \cite{lauriv1}, the $L^{2,1}$ estimate does not hold for $\vec{R}$. However, following the same steps of the proof, we deduce thanks to \eqref{jacobian_RS_residue} 
    \begin{align*}
        \np{\vec{R}}{2,1}{\Omega_{1/2}}\leq C(\Lambda) +|\vec{\gamma}_1(\phi,0)|\np{\D\log|z|}{2,1}{\Omega_{1/8}}
    \end{align*}
    As
    \begin{align*}
        &\np{\D\log|z|}{2,1}{B_1\setminus\bar{B}_{e^{-l_k}}(0)}=4\sqrt{\pi}\left(\log\left(\frac{1}{e^{-l_k}}\right)+\log\left(1+\sqrt{1-e^{-2l_k}}\right)\right)\\
        &=4\sqrt{\pi}\left(l_k+\log\left(1+\sqrt{1-e^{-2l_k}}\right)\right)=4\sqrt{\pi}\left(l_k+O\left(\sqrt{l_k}\right)\right)=4\sqrt{\pi}\left(1+O\left(\frac{1}{\sqrt{l_k}}\right)\right)l_k,
    \end{align*}
    where $l_k\rightarrow \infty$, we deduce that the $L^{2,1}$ quantization holds if 
    \begin{align*}
        \lim_{k\rightarrow \infty}|\vec{\gamma}_1(\phi_k)|l_k=0.
    \end{align*}

\section{Proof of the Main Theorem in the General Setting}

The goal of this setion is to show the following result. Recall that for all Willmore immersion $\phi$, the second residue is defined by 
\begin{align*}
    \vec{\gamma}_1(\phi,p)=\frac{1}{4\pi}\Im\int_{\gamma}\phi\wedge \left(\partial\H+|\H|^2\partial\phi+g^{-1}\otimes\s{\H}{\h_0}\otimes \bar{\partial\phi}\right)+g^{-1}\otimes\h_0\wedge\bar{\partial}\phi.
\end{align*}

\begin{theorem}\label{main}
    For all $n\in\N$, there exists a universal constant $\Lambda_n>0$ with the following property. 
    Let $(\Sigma,h_k)$ be sequence of compact Riemann surfaces of fixed genus that converges to a punctured $(\bar{\Sigma},\bar{h})$ whose connected components are $(\bar{\Sigma}_1,\cdots,\bar{\Sigma}_p)$, while we denote by $\ens{\gamma_k^1,\cdots\gamma_k^q}$ the finite sequence of pinching geodesics. For all $k\in\N$, let $\phi_k:(\Sigma,h_k)\rightarrow \R^n$ be a sequence of smooth Willmore immersions such that
    \begin{align}\label{finite_energy}
        \limsup_{k\rightarrow\infty}W(\phi_k)<\infty,
    \end{align}
    and assume that for all $1\leq j\leq q$, if $l_k^j=\leb(\gamma_k^j)$ is the length of $\gamma_j^j$, the
    \begin{align}\label{residue_bound}
        \limsup_{k\rightarrow \infty}\frac{|\vec{\gamma}_1(\phi_k,\gamma_{k}^j)|}{l_k^j}\leq \Lambda_n,
    \end{align}
    then there exists branched Willmore immersions $\phi_{\infty}^i:\Bar{\Sigma}_i\rightarrow \R^n$ and branched Willmore spheres $\vec{\Psi}_1,\cdots \vec{\Psi}_r:S^2\rightarrow \R^n$, and $\vec{\chi}_1,\cdots\vec{\chi}_s:S^2\rightarrow \R^n$ such that
    \begin{align}\label{scs}
        \limsup_{k\rightarrow\infty}\left(\mathrm{Ind}_W(\phi_k)+\mathrm{Null}_W(\phi_k\right)&\leq \sum_{i=1}^p\left(\mathrm{Ind}_W(\phi_{\infty}^i)+\mathrm{Null}_W(\phi_{\infty}^i)\right)+\sum_{j=1}^r\left(\mathrm{Ind}_W(\vec{\Psi}_j)+\mathrm{Null}_W(\vec{\Psi}_j)\right)\nonumber\\
        &+\sum_{l=1}^{s}\left(\mathrm{Ind}_W(\vec{\chi}_l)+\mathrm{Null}_W(\vec{\chi}_l)\right).
    \end{align}
\end{theorem}
\begin{rem}
    Notice that the $L^2$ energy quantization holds since \eqref{residue_bound} implies that
    \begin{align*}
        \lim_{k\rightarrow \infty}\frac{\vec{\gamma}_1(\phi_k,\gamma_k^j)}{\sqrt{l_k^j}}=0,
    \end{align*}
    which allows us to use the analysis of \cite{quantamoduli}.
\end{rem}
\begin{proof}
    The proof is the same as the one given in Section \ref{proof_neck} with the exception of Theorem \ref{neck_positive} that must be refined. The previous analysis shows the result holds provided that
    \begin{align*}
        \np{\D\n_k}{2,1}{\Omega_k(\alpha)}
    \end{align*}
    is small enough, and since
    \begin{align*}
        \np{\D\n_k}{2,1}{\Omega_k(\alpha)}\leq C\left(\np{\D\n_k}{2}{\Omega_2(2\alpha)}+|\vec{\gamma}_1(\phi_k,0)|\np{\D\log|z|}{2,1}{\Omega_k(2\alpha)}\right).
    \end{align*}
    Since
    \begin{align*}
        \np{\D\log|z|}{2,1}{\Omega_k(2\alpha)}=4\sqrt{\pi}\left(l_k+2\log(\alpha)+\log\left(1+\sqrt{1-\alpha^2e^{-2l_k}}\right)\right),
    \end{align*}
    where $l_k\conv{k\rightarrow\infty}\infty$
    we deduce thanks to the $L^2$ quantization of energy that for all $\epsilon>0$, provided that
    \begin{align*}
        |\vec{\gamma}_1(\phi_k,0)|l_k\leq \epsilon,
    \end{align*}
    then
    \begin{align*}
        \lim_{\alpha\rightarrow 0}\limsup_{k\rightarrow \infty}\np{\D\n_k}{2,1}{\Omega_k(\alpha)}\leq 4\sqrt{\pi}C\epsilon.
    \end{align*}
    Therefore, using \eqref{fine_neck_estimate}, we deduce that the theorem holds provided that $\epsilon>0$ is chosen small enough in the above estimate, which concludes the proof of the theorem.
\end{proof}

    \nocite{}
	 \bibliographystyle{plain}
	 \bibliography{biblio_full}

    \end{document}